\documentclass[11pt,reqno]{amsart}
\usepackage{}
\usepackage{esint}
\usepackage{bbm}
\usepackage{amssymb}
\usepackage{mathrsfs}
\usepackage{amsfonts}
\usepackage{amsfonts,amssymb,amsmath,amsthm}
\usepackage{url}
\usepackage{enumerate}
\usepackage[pdftex,bookmarksnumbered]{hyperref}
\usepackage{graphicx,xcolor}
\usepackage{pgfplots}
\usepackage{float}
\usepackage{caption}
\usepackage{ulem}
\usepackage{amsmath}
\usepackage{mathtools}
\makeatletter 
\def\@captype{figure} 
\makeatother  

\usetikzlibrary{patterns,arrows,positioning,calc,fadings,shapes,decorations.markings}
\usepackage{tikz}

\usepackage{multicol}

\newcommand{\ds}{\displaystyle}
\newtheorem{theorem}{Theorem}[section]
\newtheorem{lemma}[theorem]{Lemma}
\newtheorem{proposition}[theorem]{Proposition}
\newtheorem{corollary}[theorem]{Corollary}
\theoremstyle{definition}
\newtheorem{definition}[theorem]{Definition}
\newtheorem{remark}[theorem]{Remark}
\newtheorem{observation}{Observation}
\numberwithin{equation}{section}

\newtheorem{example}{Example}

\newtheorem{notation}[theorem]{Notation}

\usepackage[left=2.0cm, right=2.0cm, top=2.0cm, bottom=2.0cm]{geometry}

\usepackage{graphicx}%
\usepackage{color}
\usepackage{cite}
\usepackage{hyperref}
\hypersetup{
	colorlinks = true,%
	citecolor = blue,
	filecolor=red,%
	linkcolor = [rgb]{0.65,0.0,0.0},%
	anchorcolor = red,
	pagecolor = red,
	urlcolor= [rgb]{0.65,0.0,0.0}
	linktocpage=true,
	pdfpagelabels=true,
	bookmarksnumbered=true,
}

\makeatother

\makeatletter
\@namedef{subjclassname@2020}{\textup{2020} Mathematics Subject Classification}
\makeatother

\DeclareMathOperator{\m}{\mathfrak{m}}
\DeclareMathOperator{\dm}{d\mathfrak{m}}

\DeclareMathOperator{\diam}{Diam}

\DeclareMathOperator{\dd}{d}

\DeclareMathOperator{\sgn}{sgn}

\DeclareMathOperator{\supp}{supp}

\DeclareMathOperator{\E}{\mathsf{E}}
\DeclareMathOperator{\PD}{\mathscr{O}}

\DeclareMathOperator{\FMMM}{\mathsf{FMMM}}

\DeclareMathOperator{\Rno}{\mathbb{R}_{\,\circ}^{\it n}}

\DeclareMathOperator{\Omegao}{\Omega_{\circ}}

\newcommand{\Lip}{\mathsf{Lip}}
\newcommand{\dil}{\mathsf{dil}}

\author{Benling Li}
\address{
School of Mathematics and Statistics\\
Ningbo University\\
315211 Ningbo, China}
\email{libenling@nbu.edu.cn}

\author{Wei Zhao}
\address{
School of Mathematics\\
East China University of Science and Technology\\
200237 Shanghai, China}
\email{wzhao@ecust.edu.cn; szhao\underline{ }wei@yahoo.com}

\date{November 14, 2025.}


\begin{document}

\title[]{Hilbert's fourth problem in the constant curvature setting}

\begin{abstract}

Hilbert's fourth problem seeks the classification of metric geometries where straight lines are shortest paths. Its regular case  identifies the projectively flat Finsler manifolds.
This broader framework breaks the equivalence between projective flatness and constant curvature that holds in the Riemannian setting,
creating a more intricate classification problem.

This paper resolves the long-standing question of how the local structure determines the global topology for such manifolds of constant flag curvature, where flag curvature is the natural generalization of Riemannian sectional curvature.
We derive explicit distance formulas for all cases of constant flag curvature.
For non-positive constant curvature, we establish a global classification of forward complete manifolds,  a uniqueness theorem for forward complete metrics, and a characterization of maximal domains of metrics where exotic examples are constructed. For positive constant curvature, we prove a maximum diameter theorem and show the completion of such manifold is a sphere.
 A fundamental connection is revealed between Sobolev space nonlinearity and backward incompleteness.
This work provides a complete characterization of the global geometry for the regular case of Hilbert's fourth problem with constant flag curvature.
\end{abstract}

\maketitle
\tableofcontents

\section{Introduction}

\subsection{Background and motivation}
Hilbert's fourth problem, posed at the 1900 International Congress of Mathematicians in Paris, was formulated as the
``{\it problem of the straight lines as the shortest distance between two points}" \cite{Hilbert}. In modern terms, it seeks to characterize all metric geometries in which straight lines are shortest paths.
Hilbert himself illustrated this concept through his celebrated distance on strictly convex domains
 $\Omega \subset \mathbb{R}^n$:
\begin{equation}\label{Hilbertdis}
d_{\mathsf{H}}(x_1,x_2) = \frac{1}{2} \left\{ \ln \frac{|x_1 - \bar{x}_1|}{|x_2 - \bar{x}_1|} + \ln \frac{|x_2 - \bar{x}_2|}{|x_1 - \bar{x}_2|} \right\}, \quad \forall x_1,x_2\in \Omega,
\end{equation}
where  $\bar{x}_1$, $\bar{x}_2$ are the intersection points of the boundary $\partial \Omega$ with the rays  from $x_1$ through $x_2$ and from $x_2$ through $x_1$, respectively. Hereafter, $|\cdot|$ denotes the Euclidean norm. Building upon Hilbert's construction, Funk \cite{Funk1}
introduced an asymmetric distance on $\Omega$:
\begin{equation}\label{Funkdis}
d_{\mathsf{F}}(x_1,x_2) = \ln \frac{|x_1 - \bar{x}_1|}{|x_2 - \bar{x}_1|}, \quad  \forall x_1,x_2\in \Omega,
\end{equation}
whose symmetrization yields the Hilbert distance:
\begin{equation}\label{symmdishoilfunk}
d_{\mathsf{H}}(x_1,x_2) =\frac{1}{2} \big\{ d_{\mathsf{F}}(x_1,x_2) + d_{\mathsf{F}}(x_2,x_1) \big\}.
\end{equation}
Hilbert emphasized that exploring such geometries, which generalize Minkowski space while maintaining close ties to it, represents a direction of profound mathematical significance.

Over the past century, numerous significant results concerning both symmetric and asymmetric distances have been established, as detailed in the works of  Busemann \cite{Busemann}, Pogorelov \cite{Pogorelov}, Szab{\'o} \cite{Szabo}, {\'A}lvarez Paiva \cite{Alvarez}, Papadopoulos--Troyanov \cite{PT}, and others.
Nevertheless,
 the classification problem
 remains as remarkably challenging. This difficulty  stems from the fact that Hilbert's fourth problem imposes no specific constraints on the distance functions.
Among various distances,  those induced by  families of pseudo-norm functions
have attracted considerable attention.
Specifically, consider a continuous function $(x,y)\longmapsto \|y\|_x$   on the tangent bundle $TM$ of a manifold $M$ that satisfying  the following properties for each $x\in M$ and $y,v\in T_xM$:
\begin{enumerate}[\rm (i)]
\item non-negativity:   $\|y\|_x\geq 0$ with equality $y=0$;

\item positive $1$-homogeneity:  $\| \lambda y\|_x=\lambda \|y\|_x$ for any $\lambda> 0$;

\item \label{Introtriangleine} triangle inequality:   $\|y+v\|_x\leq \|y\|_x+\|v\|_x$.
\end{enumerate}
Such a function $\|\cdot\|$ is called a {\it pseudo-norm function}.  The associated {\it distance function} is defined as
\begin{equation*}
d(x_1,x_2):=\inf \int_0^1 \|\gamma'(t)\|_{\gamma(t)} {\dd}t,
\end{equation*}
where the infimum is taken over all piecewise smooth curves $\gamma:[0,1]\rightarrow M$ with $\gamma(0)=x_1$ and $\gamma(1)=x_2$.
This distance satisfies non-negativity and the triangle inequality but may lack symmetry.

In 1941, H. Busemann and W. Mayer \cite{BM}  established the  differentiability of the corresponding minimizing curves  under  a Lipschitz  condition which they called the weakest condition: for every $x\in M$, there exists a chart $U(x)$ of $x$ and a number $C=C(x)>0$ such that
\begin{equation}\label{weaklIP}
|\|y\|_{x_1}-\|y\|_{x_2}|\leq C |y||x_1-x_2|, \quad \forall x_1,x_2\in U(x).
\end{equation}
Conversely, for  a distance function $d$ on $\mathbb{R}^n$ whose  minimizing curves are straight lines, we find that $d$ is induced by a pseudo-norm function satisfying \eqref{weaklIP} if and only if for every $x\in \mathbb{R}^n$,  the following inequalities hold in some neighborhood of $x$:
\begin{align}\label{baseequ2}
 c|x_2-x_1|\leq d(x_1,x_2)\leq C|x_2-x_1|, \quad \left|d(x_1,x_1+y)-d(x_2,x_2+y) \right|\leq  C |y| |x_2-x_1|,
 \end{align}
where $c,C$ are two positive numbers depending on $x$ (see Theorem \ref{dandFequ}). According to  Berwald \cite{Be2} and Funk \cite{Funk1}, if the boundary of $\Omega$ is smooth and strongly convex, both the distance functions defined by \eqref{Hilbertdis} and \eqref{Funkdis}  satisfy Condition \eqref{baseequ2}.
Their associated pseudo-norm functions are respectively designated as the Hilbert metric and Funk metric in contemporary literature.

Condition \eqref{weaklIP} indicates that the pseudo-norm function satisfies a local Lipschitz property.
Of central importance are strongly regular  pseudo-norm functions, which include Riemannian metrics, the Hilbert metric and Funk metric as   prominent subclasses. A pseudo-norm function $\|\cdot\|$ is called {\it strongly regular} or  a {\it Finsler metric}, if the following strong regularity properties hold for any $x\in M$ and $y\in T_xM\backslash\{0\}$:
\begin{enumerate}[(i)]
\setcounter{enumi}{3} 
\item smoothness: the map  $(x,y)\longmapsto \|y\|_x$ is smooth;
\item \label{introv} positive definiteness: the quadratic form  $Q_y(v):=\frac12\frac{\dd^2}{{\dd}t^2}|_{t=0} \|y+tv\|^2_x$ is positive definite on $T_xM$.
\end{enumerate}
A Finsler metric $\|\cdot\|$ is  {\it Riemannian} if $\|\cdot\|_x$ is induced by an inner product at every $x\in M$.
Note that (\ref{introv}) is stronger than  \eqref{Introtriangleine} and enables the definition of additional geometric quantities, particularly curvature.
For simplicity of presentation,  a Finsler metric is denoted by $F$ in the sequel, i.e., $F(x,y)=\|y\|_x$.

A Finsler metric on an open subset of $\mathbb{R}^n$ is said to be {\it projectively flat} if its geodesics  are straight lines.
In this context, the regular case of Hilbert's fourth problem reduces to the classification of projectively flat Finsler manifolds.
In 1903, G. Hamel \cite{Hamel} gave a local characterization of  such metrics $F$ by the system of differential equations:
\begin{align*}
F_{x^i}-y^j  F_{x^jy^i}=0,
\end{align*}
where $(x^i,y^j)$ are the canonical coordinates on  $T\mathbb{R}^n \cong\mathbb{R}^n\times \mathbb{R}^n$. Hereafter, subscripts such as $x^i$ and $y^j$ denote partial derivatives with respect to these coordinates.
 The geodesics $\gamma(t)$ of such metrics are characterized by the equation
\begin{equation*}
\frac{{\dd}^2 \gamma^i}{{\dd}t^2} + 2 P \Big(\gamma, \frac{{\dd} \gamma}{{\dd}t} \Big) \frac{{\dd} \gamma^i}{{\dd}t} = 0,
\end{equation*}
where the {\it projective factor} $P= P(x,y)$ is defined as
$P:=\frac{1}{2}  y^k F_{x^k} / F$.
Notably, under the parameter transformation satisfying
$ \frac{{\dd}^2 s}{{\dd}t^2} = - 2 P(\gamma,\frac{{\dd} \gamma}{{\dd}t}) \frac{{\dd} s}{{\dd}t}$,
the image $\gamma(t(s))$ becomes a linear function.
Furthermore, the curvature of a projectively flat Finsler manifold can also be expressed in terms of the projective factor.

In the Riemannian setting, Beltrami's theorem \cite{Belt} shows that projective flatness is equivalent to constant sectional curvature $\mathbf{K}=\lambda$, leading to the classical models:
\begin{align} \label{Rie}
F_{\lambda}(x,y) = \frac{ \sqrt{|y|^2 + \lambda ( |x|^2|y|^2 - \langle x, y \rangle^2}) }{1+\lambda|x|^2}, \quad (x,y) \in T\Omega,
\end{align}
on $\Omega = \{x : 1+\lambda|x|^2 > 0\}$.
The Finslerian setting, freed from the constraint of quadratic metrics, exhibits a far richer structure.
The regularity hypotheses (iv) and (v) permit the extension of sectional curvature to Finsler geometry via the notion of {\it flag curvature}.
In contrast to the Riemannian case, projective flatness and constant flag curvature are independent conditions in Finsler geometry; neither implies the other.
The classification of Finsler metrics of constant flag curvature is far from understood.
While a complete classification has been achieved for Randers metrics (cf.  Bao--Robles--Shen \cite{BRS}), comprehensive characterizations for other classes have been more limited.
Moreover, even when both conditions are simultaneously satisfied, they admit a vast and largely unexplored family of metrics.

As indicated previously, if $\Omega$ is a $1$-sublevel set of a Minkowski norm $\varphi$ (i.e., $\partial\Omega$ is smooth and strongly convex),  the Hilbert distance \eqref{Hilbertdis} and the Funk distance \eqref{Funkdis} are induced by two Finsler metrics, the Hilbert metric $\mathsf{H}$ and the Funk metric $\mathsf{F}$, which possess
constant flag curvature $\mathbf{K}=-1$ and $\mathbf{K}=-\frac{1}{4}$, respectively. In particular, the relation \eqref{symmdishoilfunk} is equivalent to
\begin{align*}
\mathsf{H}(x,y) = \frac{1}{2} \left\{ \mathsf{F}(x, y) + \mathsf{F}(x, -y) \right\},
\end{align*}
and  the Funk metric  is the unique solution to the following equation
\begin{equation}\label{pdeodefqueIn}
\mathsf{F}(x, y) = \varphi\left(y + x \mathsf{F}(x,y)\right).
\end{equation}
When $\varphi$ is the Euclidean norm (i.e., $\Omega$ is the unit Euclidean ball $\mathbb{B}^n_1(\mathbf{0})$ in $\mathbb{R}^n$), the Funk metric becomes
 \begin{align}\label{Funkex1}
 \mathsf{F}(x,y) = \frac{ \sqrt{(1-|x|^2)|y|^2 + \langle x, y \rangle^2}}{1-|x|^2} + \frac{\langle x, y \rangle}{1-|x|^2},
 \end{align}
and the corresponding Hilbert metric $\mathsf{H}$ coincides
with the Riemannian metric
$F_{-1}$ defined in \eqref{Rie}.
In 1926,  Berwald \cite{Be2}
derived the local equations characterizing projectively flat Finsler metrics with flag curvature $\mathbf{K}$:
\begin{align}\label{Berequeq}
F_{x^k} = [P F]_{y^k},\qquad
P_{x^k} = P P_{y^k} - \frac{1}{3 F} [\mathbf{K} F^3]_{y^k}.
\end{align}
He also exhibited a non-Riemannian projectively flat metric with $\mathbf{K}=0$ on $\mathbb{B}^n_1(\mathbf{0})$,  called the Berwald's metric:
\begin{align} \label{Berwaldmetric}
\mathsf{B}(x,y) = \frac{ (\sqrt{ (1-|x|^2) |y|^2 + \langle x, y \rangle^2} +\langle x, y \rangle )^2}{ (1-|x|^2)^2 \sqrt{ (1-|x|^2) |y|^2 + \langle x, y \rangle^2} }.
\end{align}
The case of positive constant flag curvature remained open until Bryant's breakthrough in 1996--1997~\cite{Bry,Bry2}, who constructed the first non-Riemannian, irreversible, projectively flat Finsler metric with $\mathbf{K} = 1$ on $\mathbb{S}^2$.
 He later discussed broader generalizations in \cite{Bryant}.
 Building on this foundation, Shen \cite{Sh1} used an algebraic approach to extend Bryant's metric to higher-dimensional spheres $\mathbb{S}^n$, observing that such metrics arise from projectively flat metrics on $\mathbb{R}^n$.

Projectively flat Finsler manifolds of constant flag curvature provide a rich class of geometries satisfying Hilbert's criterion, exhibiting fundamental differences between symmetric and asymmetric settings. The \emph{local} theory of such metrics was completed in the 21st century: Shen's seminal work~\cite{Sh1} established sufficient conditions and constructed extensive examples, while Li~\cite{Li} later derived the necessary conditions, culminating in a full local classification.

In stark contrast, the \emph{global} geometry of these spaces remains largely uncharted, with no classification results available in the literature. The central, long-standing obstacle has been the following:

\medskip
 {\it How does the local structure of a projectively flat metric determine the global topology of a manifold?}
\medskip

\noindent
In this paper, we resolve this question in the constant flag curvature setting.

\subsection{Main results}
We develop a comprehensive global theory of projectively flat Finsler manifolds of constant flag curvature. Our principal contributions are:
\begin{enumerate}[{\rm (I)}]
\item \label{Contributiona-1} \textbf{Distance formula.}
An explicit formula for the distance function, covering both symmetric and asymmetric cases (Theorem~\ref{DisInto}).

\item \label{Contributiona} \textbf{Global classification (non-positive curvature).}
 A complete classification of   forward complete  projectively flat Finsler manifolds with $\mathbf{K} \leq 0$ (Theorems~\ref{thmK=0globalintro} and~\ref{thmK=-1globalintro}).

\item \label{Contributiona-3} \textbf{Uniqueness of complete metrics.}
A uniqueness theorem for forward complete Finsler metrics on a given domain when $\mathbf{K} \leq 0$ (Theorem~\ref{uniquenssFinsler}).

\item \label{Contributionb} \textbf{Maximal domains and exotic examples.}
A characterization of maximal domains for Finsler metrics with $\mathbf{K} \leq 0$ in dimension $2$,
 revealing novel metrics whose maximal domains evolve from connected to disconnected structures
(Theorems~\ref{K=0conneccomconvex} and \ref{thmK=-1convex}, Proposition~\ref{K=0Ex1}, Example~\ref{K=-1ExampleAmazing}).

\item \label{Contributiona-posit} \textbf{Positive curvature: diameter and completion.}
A maximum diameter and completion theorem for $\mathbf{K} > 0$ (Theorem~\ref{thmK=1globalintro}).

\item  \label{Contributiond} \textbf{Connection to nonlinear analysis.}
A fundamental link between the completeness of based Finsler manifolds and the nonlinearity of Sobolev spaces (Theorem~\ref{thmSobolevIntro}).
\end{enumerate}

Before stating the main results, we briefly recall and introduce some necessary notations. The {\it reversibility} of a Finsler manifold $(M,F)$, introduced by Rademacher\cite{Rade1, Rade}, is defined by
\[
\lambda_F(M):=\sup_{x\in M}\bigg(\sup_{y\in T_xM\backslash\{0\}}\frac{F(x,-y)}{F(x,y)}\bigg).
\]
Note that $\lambda_F (M) \geq 1$, with equality if and only if $F$ is reversible (i.e., symmetric).
While every Riemannian metric is reversible, this property often fails for Finsler metrics. For example, the Funk metric
$\mathsf{F}$ is irreversible with infinite reversibility.

A Finsler manifold is said to be {\it forward complete} (resp., {\it backward complete}) if  every geodesic can
be extended       infinitely     in the forward (resp., backward) direction.
Although these two types of completeness  are equivalent for Riemannian manifolds, they are generally distinct in the Finsler setting. For instance, every Funk metric space $(\Omega,\mathsf{F})$  is forward complete but not backward complete.

A  pseudo-norm $\varphi$ on $\mathbb{R}^n$ is called a {\it weak Minkowski norm} if it is smooth on $\mathbb{R}^n\backslash\{\mathbf{0}\}$.
Moreover, a solution to equation \eqref{pdeodefqueIn} is referred to as a weak Funk metric whenever
$\varphi$ is a weak Minkowski norm. In the rest of this section, $\Omega$  denotes a strictly convex connected domain in $\mathbb{R}^n$ containing the origin $\mathbf{0}$, and $F$ represents a Finsler metric on $\Omega$.

We now present our first main result, which provides
a direct answer to Hilbert's fourth problem from the distance point of view.

\begin{theorem}\label{DisInto}
Let $(\Omega,F)$ be a projectively flat Finsler manifold of constant flag curvature $\mathbf{K}$ with  projective factor $P$.
Then for any pair of distinct points $x_1, x_2\in \Omega$,
 the distance $d_F(x_1, x_2)$ is given by
 \begin{align*}
d_F(x_1,x_2)=
	\begin{cases}
	     \frac{F(x_1, x_2-x_1)}{1-P(x_1, x_2-x_1)}, & \text{ if }\mathbf{K}=0;\\[12pt]
		  \ln \sqrt{\frac{1-P(x_1, x_2-x_1)+F(x_1, x_2-x_1)}{1-P(x_1, x_2-x_1)-F(x_1, x_2-x_1)}}, & \text{ if }\mathbf{K}=-1;\\[12pt]
		 \arctan{   \frac{F^2(x_1, x_2-x_1)+ P^2(x_1, x_2-x_1)-P(x_1, x_2-x_1)}{F(x_1, x_2-x_1)}  }
           + \arctan{ \frac{P(x_1, x_2-x_1)}{ F(x_1, x_2-x_1) }} , & \text{ if }\mathbf{K}=1.
	\end{cases}
\end{align*}
\end{theorem}
It can be verified directly that the distance formulas in Theorem \ref{DisInto} recover \eqref{Hilbertdis} and \eqref{Funkdis} via appropriate scaling. This result is established based on the following global classifications for $\mathbf{K} = 0, -1$ and Theorem \ref{lemK=1dx_0xdx_0} for the case $\mathbf{K} = 1$.
In the sequel, we adopt the convention $P = \frac{1}{2} y^k F_{x^k}/F$, which is  the projective factor when $F$ is projectively flat.

\begin{theorem}\label{thmK=0globalintro}
Let $(\Omega,F)$ be a Finsler manifold and set
\begin{equation*}
\psi(y):=F(\mathbf{0},y), \qquad \phi(y):= P(\mathbf{0},y).
\end{equation*}
Then $(\Omega,F)$ is a  forward complete  projectively flat Finsler manifold with $\mathbf{K} =0$ if and only if the following two conditions hold:
\begin{enumerate}[\rm (i)]
\item\label{K=0basciciont}  $\phi$ is non-negative, $\Omega = \big\{ x \in \mathbb{R}^n \,:\, \phi(x) < 1 \big\}$ is either bounded or $\mathbb{R}^n$, the projective factor $P$ is the unique solution to
\[
P(x,y)=\phi(y+xP(x,y)),\quad \forall (x,y)\in \Omega\times \mathbb{R}^n,
\]
and the Finsler metric $F$ is given by
\begin{equation*}
F(x,y) = \psi ( y + x P(x,y)) \,\left\{ 1 +   x^k  P_{y^k}(x,y)\right\},    \quad \forall (x,y)\in \Omega\times \mathbb{R}^n;
\end{equation*}

\item\label{K=0additoncomp} exactly one of the following occurs:
\begin{enumerate}[{\rm (1)}]
\item  \label{globalK=0(i)}
$\phi=0$, in which case $(\Omega,F)=(\mathbb{R}^n,\psi)$ is a Minkowski space (backward complete with $\lambda_F(\Omega)<+\infty$);

\item \label{globalK=0(ii)}
$\phi$ is a weak Minkowski norm, in which case
  $(\Omega, F)$ is not backward complete with $\lambda_F(\Omega)=+\infty$ and particularly,  $\Omega$ is a bounded domain in $\mathbb{R}^n$ and $P$ is the unique weak Funk metric on $\Omega$.
\end{enumerate}
\end{enumerate}
\end{theorem}

\begin{theorem}\label{thmK=-1globalintro}
Let $(\Omega,F)$ be a Finsler manifold and set
\begin{equation*}
\psi(y):=F(\mathbf{0},y), \qquad \phi(y):= P(\mathbf{0},y).
\end{equation*}
Then $(\Omega,F)$ is
a forward complete projectively flat Finsler manifold with $\mathbf{K} =-1$ if and only if the following two conditions hold:
\begin{enumerate}[\rm (i)]

\item\label{K=-1bascicondition1}  $\phi + \psi$ is a weak Minkowski norm  with $\ds\Omega = \big\{ x \in \mathbb{R}^n \,:\, (\phi + \psi)(x) < 1 \big\}$ and
\[
F(x,y) = \frac{1}{2}\big\{ \Phi_{+}(x,y) - \Phi_{-}(x,y) \big\}, \qquad P(x,y) = \frac{1}{2}\big\{ \Phi_{+}(x,y) + \Phi_{-}(x,y) \big\},
\]
where  $\Phi_\pm(x,y)$ are the unique solutions to
\begin{equation*} 
\Phi_{\pm}(x,y) = (\phi \pm \psi)\big(y + x \Phi_{\pm}(x,y)\big), \quad \forall (x,y)\in \Omega\times \mathbb{R}^n;
\end{equation*}

\item\label{K=-1additoncondtion}
exactly one of the following occurs:
\begin{enumerate}[{\rm (1)}]
\item  \label{globalK=-1(i2)2} $\psi$ is reversible and $\phi$ is odd with $-\psi<\phi<\psi$ on $\mathbb{R}^n\backslash\{\mathbf{0}\}$, in which case $(\Omega,F)$ is backward complete  with $\lambda_F(\Omega)=1$ and  $F$ is a Hilbert metric with
\[
 -F(x,y)\leq P(x,y)\leq   F(x,y) \text{ with equality if and only if $y=\mathbf{0}$};
 \]

\item  \label{K=-1(2)2} $\psi=\phi$, in which case  $(\Omega, F)$ is not backward complete with $\lambda_F(\Omega)=+\infty$ and particularly,
 $2F$ is the unique Funk metric on $\Omega$ with $F=P$;

\item  \label{PgreatFK=-12}
$\psi<\phi$ on $\mathbb{R}^n\backslash\{\mathbf{0}\}$, in which case  $(\Omega, F)$ is not backward complete with $\lambda_F(\Omega)=+\infty$ and
\[
F(x,y)\leq P(x,y) \text{ with equality if and only if $y=\mathbf{0}$};
\]

\item  \label{PlessFK=-12} either $\psi$ is irreversible or $\phi$ is not odd but
$-\psi<\phi<\psi$ on $\mathbb{R}^n\backslash\{\mathbf{0}\}$, in which case  $(\Omega, F)$ is not backward complete with $\lambda_F(\Omega)=+\infty$ and
\[
 -F(x,y)\leq P(x,y)\leq   F(x,y) \text{ with equality if and only if $y=\mathbf{0}$}.
 \]
\end{enumerate}
\end{enumerate}

\end{theorem}

Note that Conditions (i) in Theorem~\ref{thmK=0globalintro} and Theorem~\ref{thmK=-1globalintro} already
provide the characterization of forward complete projectively flat Finsler manifolds with
$\mathbf{K}=0, -1$. Condition (ii) thus gives a more refined characterization of such manifolds:

$\bullet$ $\mathbf{K}=0$: The standard Riemannian case is achieved  by setting $\psi=|\cdot|$ and $\phi=0$, which is a trivial Minkowski space.  A nontrivial non-Minkowski is obtained by taking $\psi = \phi = |\cdot|$, resulting in the space $\mathbb{B}^n_1(\mathbf{0})$ endowed with Berwald's metric $\mathsf{B}$ given in \eqref{Berwaldmetric}.

 $\bullet$ $\mathbf{K}=-1$: A representative example is given by $\psi=|\cdot|$ and $\phi=c|\cdot|$ with constant $c\geq 0$, where  Theorem \ref{thmK=-1globalintro} \eqref{K=-1additoncondtion}/(1)--(4) correspond respectively to $c=0$, $c=1$, $c\in (1,+\infty)$ and $c\in (0,1)$.
Notably, the first two cases recover the standard Riemannian metric $F_{-1}$ defined in \eqref{Rie} and the Funk metric $2\mathsf{F}$ given in \eqref{Funkex1}.
\vskip 2mm

In light of  Li \cite{Li}, Shen \cite{Sh1}, Theorem  \ref{thmK=0globalintro}, and Theorem \ref{thmK=-1globalintro}, the initial data of $F$ and $P$ at the origin play a fundamental role in classifying projectively flat Finsler metrics of constant flag curvature. Indeed, under the assumption of forward completeness, such a metric is uniquely determined by these data.
\begin{theorem}\label{uniquenssFinsler}
Given a Minkowski norm $\psi$  and a positively $1$-homogeneous function $\phi$ on $\mathbb{R}^n$, there exists at most one Finsler metric $F$ on a strictly convex domain $\Omega$ such that $(\Omega,F)$ is a forward complete projectively flat Finsler manifold  of   constant  (non-positive) flag curvature  with
\[
\psi(y)=F(\mathbf{0},y),\qquad \phi(y)=P(\mathbf{0},y),\quad \forall  y\in \mathbb{R}^n.
\]
\end{theorem}

It is noteworthy that even when $\psi$ and $\phi$ satisfy conditions (i) and (ii) of Theorems \ref{thmK=0globalintro} or \ref{thmK=-1globalintro}, the key hypothesis ``$(\Omega,F)$  being  a Finsler manifold" is indispensable, since otherwise such a metric may fail to exist globally on $\Omega$.
While many choices of
$\psi$ and $\phi$ do yield well-defined forward complete metrics on the entire
$\Omega$---such as the metrics of Hilbert, Funk, and Berwald---we discover that certain choices lead not only to incompleteness of manifolds but also to  disconnection of  metrics' domains.
We systematically study such phenomenon in dimension two (Theorems \ref{K=0conneccomconvex} and \ref{thmK=-1convex}).
Moreover,
 explicit examples  are constructed
 with disconnected domains:
 one with zero  flag curvature (Proposition~\ref{K=0Ex1}) and another with negative constant flag curvature  (Example~\ref{K=-1ExampleAmazing}).
Figure~1 illustrates
how domains of a metric family with zero flag curvature evolve from connectivity to disconnection; see Proposition \ref{K=0Ex1} for more details.

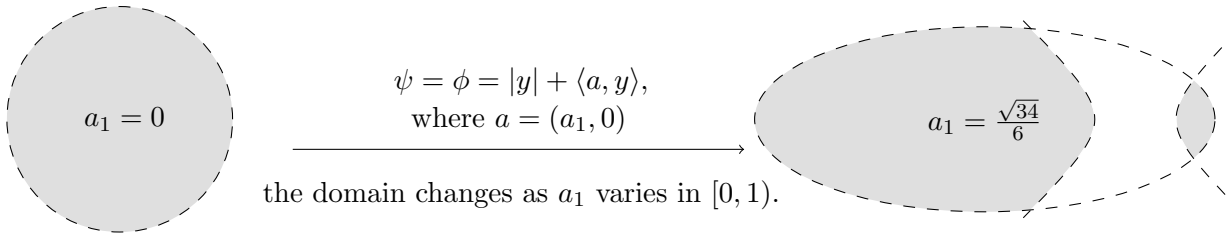
\begin{figure}[h]
\begin{tikzpicture}[scale=1]
   \filldraw[fill = gray!50,fill opacity=0.5, draw = black][dash pattern={on 4.5pt off 4.5pt}] (0,0) ellipse (1.5  and 1.5 );
\draw (0,0) node [inner sep=0.6pt]    {$\  a_1 = 0$};


\draw (17.8*0.3,1) node [inner sep=0.6pt]    {};
\draw (17.8*0.3,0.5) node [inner sep=0.6pt] {$\psi=\phi=|y|+ \langle a, y\rangle$, };
\draw (17.8*0.28,0) node [inner sep=0.6pt] {\ \ \ \ \ where $a=(a_1,0)$};
\draw  [->]  (2.3,-0.4) -- (8.3,-0.4) ;
\draw (17.8*0.3,-1) node [inner sep=0.6pt] {the domain changes  as $a_1$ varies in $[0, 1)$.};

    \begin{scope} 
   \clip   (11.5,0) ellipse (18*0.17  and 7.242*0.17 );
    \filldraw[fill = gray!50,fill opacity=0.5, domain = -1.2 : 1.2, draw=white] plot ( {11.5 + 17.49*0.157 + 0.17*(1.5*1.414*sec(\x r) - 2.915)/0.667 } ,{ 0.17* 3*tan(\x r)} );
    \filldraw[fill = gray!50,fill opacity=0.5, domain = -1.5 : 1.5, draw=white] plot ( {11.5 + 17.49*0.157 + 0.17*(-1.5*1.414*sec(\x r) - 2.915)/0.667 } ,{ -0.17* 3*tan(\x r)} );
    \end{scope}
     \draw[ draw = black][dash pattern={on 4.5pt off 4.5pt}] (11.5,0) ellipse (18*0.17  and 7.242*0.17 );
     \draw[ domain = -1.1 : 1.1][dash pattern={on 4.5pt off 4.5pt}] plot ( {11.5 + 17.49*0.157 + 0.17*(1.5*1.414*sec(\x r) - 2.915)/0.667 } ,{ 0.17* 3*tan(\x r)} );
     \draw[ domain = -1.2 : 1.2][dash pattern={on 4.5pt off 4.5pt}] plot ( {11.5 + 17.49*0.157 + 0.17*(-1.5*1.414*sec(\x r) - 2.915)/0.667 } ,{ -0.17* 3*tan(\x r)} );
\draw (11.5,0) node [inner sep=0.6pt]    {$a_1 = \frac{ \sqrt{34}}{6}$};

\end{tikzpicture}
\caption{ \small Evolution of the domain (shaded) within $\Omega$.}
\label{K=0figIntro}
\end{figure}

For the positive curvature case, since $\Omega$ is noncompact, the Bonnet--Myers theorem (cf.~Bao--Chern--Shen \cite[Theorem 7.7.1]{BCS}) implies that every projectively flat Finsler manifold $(\Omega,F)$ with $\mathbf{K}=1$ must be incomplete.
This motivates the study of its completion.
 Let $\mathbb{S}^n_+$ denote the upper hemisphere and  $\mathfrak{p}: \mathbb{R}^n\rightarrow \mathbb{S}^n_+\subset \mathbb{R}^{n+1}$  the natural diffeomorphism
\begin{equation*}
\mathfrak{p}(x):=\bigg( \frac{x}{\sqrt{1+|x|^2}}, \frac{  1}{\sqrt{1+|x|^2}}  \bigg).
\end{equation*}
Every Finsler metric $F$ on $\mathbb{R}^n$ can be pulled back to a metric $(\mathfrak{p}^{-1})^*F$ on $\mathbb{S}^n_+$ under $\mathfrak{p}$.
This observation leads to the following theorem, which establishes both a diameter bound and a completion result.

\begin{theorem}\label{thmK=1globalintro}
Let $(\Omega,F)$ be a projectively flat Finsler manifold with $\mathbf{K}=1$. Then the following hold:
\begin{enumerate}[\rm (i)]

\item\label{K=1theo1} $(\Omega,F)$ is neither forward complete nor backward complete;

\item\label{K=1theo2}  the diameter of $(\Omega,F)$ is at most $\pi$, and equals $\pi$ when $\Omega=\mathbb{R}^n$;

\item\label{K=1theo3} provided that $\Omega=\mathbb{R}^n$ and  $(\mathfrak{p}^{-1})^*F$ can be extended to a Finsler metric on $\overline{\mathbb{S}^n_+}$,  there exists a complete  locally projectively flat Finsler metric $\mathscr{F}$ on $\mathbb{S}^n$  such that:
\begin{itemize}

\item the flag curvature of $(\mathbb{S}^n,\mathscr{F})$ remains $1$,  the diameter is $\pi$ and $\mathscr{F}|_{\mathbb{S}^n_+}=\mathfrak{p}^*F$;

 \item  every geodesic of length $2\pi$ in $(\mathbb{S}^n,\mathscr{F})$  is  a great circle.

  \end{itemize}
 \end{enumerate}

\end{theorem}


It is well-known that $F_1$ in \eqref{Rie} is the unique projectively flat Riemannian metric on $\mathbb{R}^n$ with $\mathbf{K}=1$, and its pull-back $(\mathfrak{p}^{-1})^*F_1$ is the standard sphere metric. Thus, Theorem \ref{thmK=1globalintro} is natural for Riemannian metrics. For non-Riemannian ones, however, the assumption that $(\mathfrak{p}^{-1})^*F$ extends to $\partial \mathbb{S}^n_+$ is essential, as there exist projectively flat Finsler metrics with $\mathbf{K} = 1$ for which this fails; see Remark \ref{K=1remarkconditionextension}.

Theorem \ref{thmK=1globalintro}/\eqref{K=1theo3} implies that all geodesics are closed, of length $2\pi$, and mutually intersecting. This configuration falls into the first case of Bryant--Foulon--Ivanov--Matveev--Ziller \cite[Theorem 2]{BFIMZ},
but does not belong to the class of Finsler metrics constructed in Katok \cite{Katok}. For further related studies on closed geodesics in Finsler geometry, we refer to Bangert--Long \cite{BL}, Duan--Long--Wang \cite{DLW}, and Duan--Qi \cite{DQ}, and references therein.

Building upon the aforementioned results, we investigate  Sobolev spaces in the Finslerian setting. While Sobolev spaces over Riemannian manifolds are vector spaces, Krist\'aly--Rudas~\cite{KR} showed that this fails for the Funk metric space defined by \eqref{Funkex1}.
We further prove the following theorem, which
highlights a fundamental analytic distinction  between Finsler and Riemannian geometries.
\begin{theorem}\label{thmSobolevIntro}
Let $(\Omega,F,\mathscr{L}^n)$ be an $n$-dimensional forward complete projectively flat Finsler manifold endowed with the Lebesgue measure. If $(\Omega, F)$ has constant flag curvature, then the following statements are equivalent:
\begin{enumerate}[{\rm (i)}]
\item\label{soli1} the reversibility $\lambda_F(\Omega)= +\infty$;

\item\label{soli2}
the Finsler manifold $(\Omega,F)$ is not backward complete;

\item\label{soli3}
the Sobolev space $W^{1,p}_0(\Omega,F,\mathscr{L}^n)$ is nonlinear for some $p\in (1, +\infty)$;

\item\label{soli4}
the Sobolev space $W^{1,p}_0(\Omega,F,\mathscr{L}^n)$ is nonlinear for every $p\in (1, +\infty)$.
\end{enumerate}
\end{theorem}
This result shows that the linearity of Sobolev spaces is equivalent to the backward completeness  or the finite reversibility of based Finsler manifolds. By Theorems  \ref{thmK=0globalintro}, \ref{thmK=-1globalintro}, and \ref{thmK=1globalintro},   only two  classes satisfying these conditions: Minkowski spaces for $\mathbf{K}=0$ and Hilbert metric space for $\mathbf{K}=-1$, which together encompass all Riemannian cases.

\medskip

\noindent\textbf{Organization.}
Section~\ref{Preliminaries} develops algebraic properties of positively homogeneous functions and   Finsler geometry preliminaries. Sections~\ref{sectionK=0}--\ref{sectionK=1} constitute the core geometric analysis: Section~\ref{sectionK=0} addresses the global structure and existence of projectively flat Finsler metrics with $\mathbf{K}=0$, proving Theorem \ref{thmK=0globalintro}; Section~\ref{sectionK=-1} provides parallel results for $\mathbf{K}=-1$, establishing Theorems \ref{thmK=-1globalintro} and \ref{uniquenssFinsler}. Notably, in these two sections, we provide a complete description of maximal domains of metrics in the two-dimensional case (Theorem~\ref{K=0conneccomconvex}, Theorem~\ref{thmK=-1convex}) and exhibit new metrics with disconnected domains (Proposition~\ref{K=0Ex1}, Example~\ref{K=-1ExampleAmazing}). Section~\ref{sectionK=1} treats the $\mathbf{K}=1$ case, deriving explicit distance formulas (Theorem \ref{DisInto}) and proving
the completion of such manifold is a sphere (Theorem \ref{thmK=1globalintro}). Section~\ref{SecSobolev} investigates the linear/nonlinear structure of Sobolev spaces over these manifolds, linking it to backward completeness and proving the main characterization Theorem \ref{thmSobolevIntro}. Section~\ref{sectiondistance} broadens the perspective to asymmetric metrics on Euclidean spaces induced by pseudo-norm functions, where the fundamental equivalence Theorem \ref{dandFequ} is established.
Complementary proofs and technical results are collected in the Appendix (Sections \ref{propergenerlengappex0}--\ref{propesoboleve}), completing the detailed arguments for all results presented in the main text.

\medskip
\noindent\textbf{Acknowledgements.} The authors are grateful to Robert L.~Bryant and Patrick Foulon for their guidance regarding relevant references.
The first author is supported by Natural Science Foundation of Zhejiang Province (No. LY23A010012) and Natural Science Foundation of Ningbo (No.~2024J017). The second author
is supported by Natural Science Foundation of China (No.~12471045) and Science and Technology Project of Xinjiang Production and Construction Corps (No.~2023CB008-12).

\vskip 10mm
\section{Preliminaries}\label{Preliminaries}
This section introduces the fundamental definitions and establishes new results essential for proving our main theorems.
We begin by developing the theory of positively homogeneous functions, where we obtain new algebraic properties that form the foundation for our geometric framework.
We then review essential elements from Finsler geometry, establishing notations and conventions.  In the following part, we discuss projectively flat Finsler metrics, and present several new lemmas and theorems that will be instrumental in proving our main results.

For convenience, we introduce some notations and conventions, which are used throughout the paper: $\Omega$  always denotes a domain (i.e., connected open subset) in $\mathbb{R}^n$, which  usually contains the origin $\mathbf{0}$.
The tangent bundle
$T\Omega$ is identified with $\Omega\times \mathbb{R}^n$ .
Set
\[
\Rno:=\mathbb{R}^n\backslash\{\mathbf{0}\} \quad \text{ and  }\quad \Omegao:=\Omega\backslash\{\mathbf{0}\}.
\]
Given a subset $U\subset \mathbb{R}^n$ and a point $x\in \mathbb{R}^n$, the translation of $U$ by $x$ is denoted by
\begin{equation*}
U + x=\{y + x\,:\,y\in U \}.
\end{equation*}
As usual, $\mathbb{B}^n_r(x)$ denotes the Euclidean open ball in $\mathbb{R}^n$ centered at $x$ with radius $r$.

\subsection{Positively homogeneous functions  }\label{convprofun}
As shown by Li \cite{Li} and Shen \cite{Sh1}, every projectively flat Finsler metric of constant flag curvature is locally determined by two
positively $1$-homogeneous functions. Hence, it is essential to investigate the properties of such functions.


A function $\varphi \colon \mathbb{R}^n\rightarrow \mathbb{R}$ is called {\it positively $\kappa$-homogeneous} for some $\kappa\in \mathbb{Z}$ if for any $y\in \mathbb{R}^n$,
\[
\varphi(\alpha y) = \alpha^\kappa \varphi(y), \quad  \forall \alpha>0,
\]
which implies $\varphi(\mathbf{0})=0$ when $\kappa>0$.
If we further  assume $\varphi|_{\Rno} \in C^1(\Rno)$, then Euler's theorem (cf. Bao--Chern--Shen~\cite[Theorem 1.2.1]{BCS}) yields
\begin{equation}\label{Eulerhter}
y^i \varphi_{y^i}(y)=\kappa\, \varphi(y).
\end{equation}
The following definition is due to Busemann--Mayer~\cite[p.\,181]{BM}:
\begin{definition}\label{regualrdef}
Let $\varphi \colon \mathbb{R}^n\rightarrow \mathbb{R}$ be a positively $1$-homogeneous function.
\begin{enumerate}[\ \ \ \ (1)]

\item The level set
$\varphi^{-1}(1)=\{y\in \mathbb{R}^n\,:\, \varphi(y)=1\}$
is called  the {\it indicatrix} of $\varphi$.

\item $\varphi$ is called {\it quasi-regular} if
\begin{equation}\label{phiquasiconvex0}
\varphi(y_1+y_2)\leq \varphi(y_1)+\varphi(y_2),\quad \forall y_1,y_2\in \mathbb{R}^n.
\end{equation}

\item A quasi-regular function $\varphi$ is said to be {\it regular} if and only if the equality in \eqref{phiquasiconvex0} holds only for
\[
\alpha_1 y_1=\alpha_2y_2, \quad \alpha_1\geq 0, \quad \alpha_2\geq 0, \quad \alpha_1+\alpha_2>0.
\]

\end{enumerate}
\end{definition}

A quasi-regular function is always continuous due to its convexity.
Moreover, every compact convex set in $\mathbb{R}^n$ corresponds to a quasi-regular function (cf.~H\"ormander \cite[Theorem 2.1.23]{Hl}).
In addition,
the following result holds (cf.~\cite[pp. 178--179]{BM}).
\begin{proposition}[\cite{BM}] \label{quasitostrictlyconvex} Let $\varphi \colon \mathbb{R}^n \rightarrow \mathbb{R}$ be
a positively $1$-homogeneous continuous function with $\varphi|_{\Rno} >0$.
Then the indicatrix of  $\varphi$ is  convex (resp., strictly convex) if and only if $\varphi$ is quasi-regular (resp., regular).
 \end{proposition}

In the sequel, we investigate the solution $\Phi(x,y)$ to the following equation
\begin{equation}\label{pdeodefque}
\Phi(x, y) = \varphi(y + x \Phi(x, y)),
\end{equation}
where $\varphi$ is a positively $1$-homogeneous function on $\mathbb{R}^n$. This equation is closely related to projectively flat Finsler metrics (cf.~\cite{Li, Sh1}). For instance, the celebrated Funk metric arises as a solution, while the Hilbert metric is obtained by symmetrizing the former.
 Equation \eqref{pdeodefque} can also be regarded as a special case of Zermelo deformation when $\varphi$ is a Minkowski norm, a framework known to have profound connections with flag curvature in Finsler geometry, as explored by Foulon \cite{Foulon}, Huang--Mo \cite{HM},
Bryant--Huang--Mo \cite{BHM}, and Foulon--Matveev \cite{FM}. In this paper, however, we focus on the more general case where $\varphi$ is merely positively $1$-homogeneous or quasi-convex, which naturally involves the underlying manifold's boundary.
\begin{observation}\label{K=-1remarkindicatrix}
 Assume further that $\varphi|_{\Rno} > 0$. For fixed $x\in \{\varphi<1\}$,
any solution  of \eqref{pdeodefque} satisfies $\Phi(x, y)\geq 0$, with equality if and only if  $y = \mathbf{0}$.
Hence, for $y\in \Rno$,
equation \eqref{pdeodefque} can be equivalently rewritten as
 \begin{equation} \label{K=-1FunkEq_equi}
1 = \varphi\Big(\frac{y}{\Phi(x,y)} + x \Big).
\end{equation}
This identity has a clear geometric interpretation: for $y \neq \mathbf{0}$, the ray emanating from $x$ in the direction $y \neq \mathbf{0}$  intersects the indicatrix ${\varphi = 1}$  at the point $x + y/\Phi(x,y)$.
Moreover, if the solution $\Phi(x,y)$ is {\it unique}, then its unit-level set is given by
\begin{equation}\label{intdaixdd}
\{ y \in \mathbb{R}^n\, :\, \Phi(x, y) =1 \} = \{ y \in \mathbb{R}^n\, :\, \varphi(x+y) =1 \}= \{ y \in \mathbb{R}^n\, :\, \varphi(y) =1 \} - x.
\end{equation}
In other words,
the indicatrix of $\Phi(x, \cdot)$ is a translation by $-x$ of the indicatrix of $\varphi$.
\end{observation}

The following lemma is a direct consequence of  Observation \ref{K=-1remarkindicatrix}, which describes the number of solutions by the number of intersections of a ray with the indicatrix..

\begin{lemma}\label{K=-1lemintersectpoints}
Let $\varphi$ be a positively $1$-homogeneous function on $\mathbb{R}^n$  with $\varphi|_{\Rno}>0$.
Given $x\in \{\varphi<1\}$ and $y\in \Rno$, define two sets by
\begin{align*}
S_{x,y}:=\big\{ \zeta\in \mathbb{R}\,:\,\zeta=\varphi(y+x\zeta) \big\},\qquad
I_{x,y}:= \big\{ x+ky\in \varphi^{-1}(1)\,:\, k\in[0,+\infty) \big\}.
\end{align*}
Thus, the cardinalities of $S_{x,y}$ and $I_{x,y}$  are equal.
\end{lemma}

In particular, Shen proved the local uniqueness of solutions to \eqref{pdeodefque} (cf. \cite[Lemma 5.1]{Sh1}).
\begin{lemma}[{\cite{Sh1}}] \label{Sh1Lemma5.1}
Given a positively $1$-homogeneous continuous function $\varphi \colon \mathbb{R}^n\rightarrow \mathbb{R}$ with
$\varphi|_{\Rno}\in C^\infty(\Rno)$,
there is a number $\delta>0$ such that the solution $\Phi(x,y)$ to the following equation
 \begin{equation} \label{K=-1FunkEq}
\Phi(x, y) = \varphi(y + x \Phi(x,y)), \quad \forall (x,y)\in \mathbb{B}^n_{\delta}(\mathbf{0}) \times \mathbb{R}^n
\end{equation}
exists uniquely. Moreover, $\Phi(x,y)$ is positively $1$-homogeneous in $y$.
\end{lemma}


We now establish a global existence and uniqueness theorem for solutions to  \eqref{pdeodefque}.
\begin{theorem}\label{lemsolexi}
Let $\varphi \colon \mathbb{R}^n \rightarrow \mathbb{R} $ be a (quasi-)regular function with
$0<\varphi|_{\Rno} \in C^\infty(\Rno)$, and define the domain $\Omega := \{ x \in \mathbb{R}^n : \varphi(x) < 1 \}$.
 Then there exists  a unique solution $\Phi(x,y)$ to
  the equation
\begin{equation} \label{MainEq1}
\Phi(x, y) = \varphi(y + x \Phi(x,y)), \quad \forall (x,y)\in \Omega \times \mathbb{R}^n.
\end{equation}
Moreover,
$\Phi$ has the following properties:
\begin{itemize}
\item $\Phi|_{\Omega \times \Rno}$ is a positive smooth function;

\item $\Phi(x, y)$ is a (quasi-)regular function in $y$;

\item $\{ y\in \mathbb{R}^n \, : \, \Phi(x,y) <1 \}=\Omega-x$.
\end{itemize}
\end{theorem}
\begin{proof}  By Observation \ref{K=-1remarkindicatrix}, we have $\Phi(x,\mathbf{0})=0$.
Given $(x,y) \in \Omega \times \Rno$, the (quasi-)regularity and homogeneity of  $\varphi$ imply
\begin{equation*}
h(t):=h_{x,y}(t):=t -  \varphi(y + x t)\geq t ( 1 - \varphi( x ) ) - \varphi( y ), \quad \forall t\in [0,+\infty).
\end{equation*}
Since $0 \leq \varphi(x) < 1$  and $\varphi(y)> 0$,
there exists $t_0=t_0(x,y) >0$ such that
\begin{equation}\label{topropoerty}
h(t_0) = 0 \quad \text{and} \quad h(t)>0, \quad \forall  t\in (t_0,+\infty).
\end{equation}
Defining $\Phi(x,y) := t_0$ then yields a solution to \eqref{MainEq1}.

We now prove uniqueness. The function $h$ is concave due to the convexity of $\varphi$, which implies
\begin{equation}\label{hnonnegative}
h''(t) =  - \varphi_{y^i y^j}(y + x t)x^i x^j \leq 0,\quad \forall  t\in [0,+\infty).
\end{equation}

First, we show that $h(t)<0$ for $t$ arbitrarily close to $t_0$ from the left. Suppose,  for contradiction,  that $h(t) \geq 0$ on some interval $(t_0-\delta, t_0)$. Combined with \eqref{topropoerty}, this would make $t_0$ a local minimum, implying $h'(t_0) = 0$. However, since $h(t) > 0$ for $t > t_0$, there exists $t_* > t_0$ with $h'(t_*) > 0$. By the mean value theorem, there exists $\xi \in (t_0, t_*)$ such that
\[
h''(\xi) = \frac{h'(t_*) - h'(t_0)}{t_* - t_0} > 0,
\]
contradicting \eqref{hnonnegative}. Hence, there exists $\delta > 0$ such that
\begin{equation}\label{claimroo1}
h(t) < 0 \quad \text{for all } t \in (t_0-\delta, t_0).
\end{equation}

Suppose that $t_1(<t_0)$ is the second largest root of $h=0$.
 Then \eqref{claimroo1} implies $h(t) < 0$ on $(t_1, t_0)$, which is impossible for a concave function with $h(t_1) = h(t_0) = 0$. Therefore, $t_0$ is the unique root, establishing the uniqueness of $\Phi$. The homogeneity $\Phi(x, \alpha y) = \alpha \Phi(x, y)$ for $\alpha > 0$ follows since both sides satisfy \eqref{MainEq1} at $(x, \alpha y)$.

For fixed $x \in \Omega$, the homogeneity of $\varphi$ and \eqref{intdaixdd} imply that
\[
\{ y \in \mathbb{R}^n : \Phi(x,y) < 1 \} = \{ y \in \mathbb{R}^n : \varphi(y) < 1 \} - x = \Omega - x.
\]
The (quasi-)regularity of $\Phi$ in $y$ then follows from Proposition \ref{quasitostrictlyconvex}.

Finally, we prove smoothness. Since $\varphi$ is convex, Taylor's theorem yields
$\varphi(y-w)\geq \varphi(y)-\varphi_{y^i}(y)w^i$ for any $y \neq \mathbf{0}$.
Given  $x\in \Omega$, choosing $w:=y-x$ and applying  Euler's theorem \eqref{Eulerhter} gives
\begin{equation}\label{phsixismal11}
\varphi_{y^i}(y)x^i\leq \varphi(x)<1.
\end{equation}
This, together with the implicit function theorem and \eqref{MainEq1}, establishes the smoothness of $\Phi$.
\end{proof}

The following lemma provides a comparison principle for the solutions.

\begin{lemma}\label{varphi_minus} Let $\varphi$, $\Phi$, and $\Omega$ be as in Theorem \ref{lemsolexi}. Assume that
 $\tilde{\varphi}\colon \mathbb{R}^n\rightarrow \mathbb{R}$ is a continuous positively $1$-homogeneous function
satisfying
\begin{equation}   \label{relationvartildephi}
   - \varepsilon \varphi|_{\Rno}< \tilde{\varphi}|_{\Rno} < \varphi|_{\Rno},
\end{equation}
for some $\varepsilon \geq 0$. Then
there exists a function $\tilde{\Phi} = \tilde{\Phi}(x,y)$ satisfying the equation
\begin{equation} \label{EqtildePhivarph}
 \tilde{\Phi}(x,y) = \tilde{\varphi}(y+ x \tilde{\Phi}(x,y)), \quad \forall (x,y)\in \Omega\times \mathbb{R}^n,
\end{equation}
and the inequality
\begin{equation*}
\tilde{\Phi}(x,y)  < \Phi(x,y), \quad   \forall (x,y)\in   \Omega\times \Rno.
\end{equation*}
\end{lemma}
\begin{proof}
Given $(x, y) \in \Omega \times \mathbb{R}^n$, we construct a solution $\tilde{\Phi}(x, y)$ to \eqref{EqtildePhivarph} by studying the function
\[
\tilde{h}(t):= \tilde{h}_{x,y}(t):=t - \tilde{\varphi}(y+ t x).
\]
Each root of $\tilde{h}(t)=0$ corresponds to a solution.
The rest of  the proof proceeds by case analysis on the sign of $\tilde{\varphi}(y)$.

\textbf{Case (i): $\tilde{\varphi}(y) = 0$.}
Then $\tilde{h}(0) = 0$.
Set  $\tilde{\Phi}(x,y): =0 \leq  \Phi(x,y)$, with equality if and only if $y=\mathbf{0}$.

\textbf{Case (ii): $\tilde{\varphi}(y) < 0$.}
We have $y\neq \mathbf{0}$ and $\tilde{h}(0) = -\tilde{\varphi}(y) > 0$. Moreover, \eqref{relationvartildephi} and the quasi-regularity of $\varphi$ imply
\[
\tilde{h}(t) \leq t + \varepsilon \varphi(y + t x) \leq t(1 + \varepsilon \varphi(x)) + \varepsilon \varphi(y),
\]
which gives $\lim_{t \to -\infty} \tilde{h}(t) = -\infty$. Consequently, there exists $\tilde{t}_0 \in (-\infty, 0)$ such that $\tilde{h}(\tilde{t}_0) = 0$. Defining $\tilde{\Phi}(x, y) := \tilde{t}_0$ gives a solution satisfying $\tilde{\Phi}(x, y) < 0 < \Phi(x, y)$.

\textbf{Case (iii): $\tilde{\varphi}(y) > 0$.}
Then $y \neq \mathbf{0}$ and  $\tilde{h}(0) = -\tilde{\varphi}(y) < 0$. In this case, \eqref{relationvartildephi} yields
\begin{align*}
\tilde{h}(t) - h(t) = \varphi(y + t x) - \tilde{\varphi}(y + t x) > 0 \quad \text{for } y + t x \neq \mathbf{0},
\end{align*}
where $h(t) := t - \varphi(y + t x)$. By choosing $t_0=\Phi(x,y)>0$, we have $h(t_0)=0$ and $y+t_0 x\neq\mathbf{0}$. Then
\[
 \tilde{h}(t_0) = \varphi(y+ t_0 x) - \tilde{\varphi}(y+ t_0 x)  > 0.
 \]
Hence, there is  $t_* \in (0, t_0)$ with
$ \tilde{h}(t_*) =0$ and therefore,
 $\tilde{\Phi}(x,y): = t_* < t_0 = \Phi(x,y)$.
\end{proof}

The following corollary collects several properties of the solution $\Phi$ that will be used repeatedly.
\begin{corollary}\label{basicproper}
Let $\varphi \colon \mathbb{R}^n\rightarrow \mathbb{R} $ be a continuous positively $1$-homogeneous  function with $\varphi|_{\Rno}\neq 0$,
 and let $\Omega \subset \mathbb{R}^n$ be a domain containing the origin such that
there exists a unique solution $\Phi(x,y)$ to the equation
\begin{equation}\label{expre11}
\Phi(x,y)=\varphi(y+x\Phi(x,y)), \quad \forall (x,y)\in \Omega\times \mathbb{R}^n.
\end{equation}
Then, $\Phi$ is continuous on $\Omega\times \mathbb{R}^n$ and has the following two properties:
\begin{enumerate}[\ \ \ \ \rm (a)]
\item\label{noezerperto} $y+x\Phi(x,y)=\mathbf{0}$ if and only if $y=\mathbf{0}$;

\item\label{postiperoo} $1+\mu\Phi(\mu y,y)>0$ for any  $\mu\in \mathbb{R}$ and $y\in \mathbb{R}^n$ with $\mu y\in \Omega$.
\end{enumerate}
 In addition, provided that
 $\varphi$ and $\Omega$ satisfy the hypotheses of Theorem \ref{lemsolexi}, then
\begin{equation*}
1+x^k\Phi_{y^k}(x,y)>0,\quad  \forall (x,y)\in \Omega\times \Rno.
\end{equation*}
\end{corollary}
\begin{proof}
The proof addresses Property \eqref{noezerperto}, continuity and Property \eqref{postiperoo} in sequence.

The uniqueness of  $\Phi$ implies $\Phi(x,\alpha y)=\alpha\Phi(x,y)$ for any $\alpha>0$, which also indicates $\Phi(x,\mathbf{0})=0$.
Now, suppose $y_o+x_o\Phi(x_o,y_o)=\mathbf{0}$ for some $(x_o,y_o)\in \Omega \times \mathbb{R}^n$. Then \eqref{expre11} gives
\begin{equation*}
 \Phi(x_o,y_o)=\varphi(y_o+x_o\Phi(x_o,y_o))=\varphi(\mathbf{0})=0.
\end{equation*}
Consequently, $y_o=y_o+x_o\Phi(x_o,y_o)=\mathbf{0}$.
This establishes Property \eqref{noezerperto}.

We now prove continuity.  First, assume $\varphi|_{\Rno} > 0$.
 Then \eqref{K=-1FunkEq_equi} and Property \eqref{noezerperto} imply
\begin{equation}\label{versinta}
\Big(\frac{y}{\Phi(x,y)}+x\Big) \in \varphi^{-1}(1), \quad \forall (x,y)\in \Omega\times \Rno.
\end{equation}
Let $(x_i, y_i) \to (x, y) \in \Omega \times \Rno$ be a convergent sequence. By \eqref{versinta}, the sequence $\left( \frac{y_i}{\Phi(x_i, y_i)} + x_i \right)_i$ lies in the compact set $\varphi^{-1}(1)$,  so the sequence $(\Phi(x_i, y_i))_i$ has a limit point $\zeta$. Passing to a subsequence, we assume $\lim_{i \to +\infty} \Phi(x_i, y_i) = \zeta$. By the continuity of $\varphi$, we obtain
\[
\zeta = \lim_{i \to +\infty} \varphi(y_i + x_i \Phi(x_i, y_i)) = \varphi(y + x \zeta).
\]
The uniqueness of $\Phi$ forces $\zeta = \Phi(x, y)$, proving continuity on $\Omega \times \Rno$. The homogeneity of $\Phi$ in $y$ then extends continuity to $\Omega \times \mathbb{R}^n$.

If $\varphi|_{\Rno} < 0$, consider the equation
\begin{equation}\label{negativecase}
-\Phi(x, y) = (-\varphi)(y - x (-\Phi(x, y))).
\end{equation}
The continuity of $\Phi$ follows by applying the same argument to $-\varphi$.

To prove Property \eqref{postiperoo}, suppose that $1+\mu_o \Phi(\mu_o y_o,y_o)=0$ for some $y_o\in \mathbb{R}^n$ and $\mu_o\in \mathbb{R}$ with $\mu_o y_o\in \Omega$.
 Set $x_o := \mu_o y_o$. An argument analogous to that for Property \eqref{noezerperto} shows $\Phi(x_o, y_o) = 0$, which contradicts $1 + \mu_o \cdot 0 = 0$. Hence, $1 + \mu \Phi(\mu y, y) \neq 0$ for all admissible $\mu$ and $y$. The continuity of $\Phi$ implies that this expression is positive for small $|\mu|$, and hence for all $\mu,y$ with $\mu y \in \Omega$.

Finally, assume that $\varphi$ and $\Omega$ satisfy the hypotheses of Theorem \ref{lemsolexi}.
Set $\xi:=y+x\Phi(x,y)$. Then  \eqref{expre11} yields $\Phi_{y^k}=\varphi_{y^k}(\xi)+\varphi_{y^i}(\xi) x^i\Phi_{y^k}$, which together with \eqref{phsixismal11} furnishes
\[
1+x^k\Phi_{y^k}(x,y)=1+\frac{x^k\varphi_{y^k}(\xi)}{1-x^i\varphi_{y^i}(\xi)}=\frac{1}{1-x^i\varphi_{y^i}(\xi)}>0.
 \]
This completes the proof.
 \end{proof}

This subsection concludes with two propositions revealing key properties of solutions to \eqref{pdeodefque}: their translation invariance and the structure of their domain.

 \begin{proposition}\label{transformforK=-1}
 Let $\varphi \colon \mathbb{R}^n\rightarrow \mathbb{R} $ be a  positively $1$-homogeneous  function with $\varphi|_{\Rno}\neq0$, and let $\Omega \subset \mathbb{R}^n$ be a domain  containing the origin
such that  there exists a unique solution $\Phi(x,y)$ to the equation
 \begin{equation}\label{standPHIqe}
 \Phi(x,y)=\varphi(y+x\Phi(x,y)), \quad \forall (x,y)\in \Omega\times \mathbb{R}^n.
 \end{equation}
For an arbitrary point $\bar{x} \in \Omega$, define $\bar{\varphi}(y) := \Phi(\bar{x}, y)$. Then there exists a unique solution $\bar{\Phi}=\bar{\Phi}(x,y)$ to the following equation
\begin{equation}\label{standPHIqe2}
\bar{\Phi}(x,y)=\bar{\varphi}(y+x\bar{\Phi}(x,y)),\quad \forall (x,y)\in (\Omega-\bar{x})\times \mathbb{R}^n.
\end{equation}
Moreover, the solutions are related by the translation identity
\begin{equation}\label{tranK=-1}
\Phi(x+\bar{x},y)=\bar{\Phi}(x,y), \quad \forall (x,y)\in  (\Omega-\bar{x})\times \mathbb{R}^n.
\end{equation}
\end{proposition}
\begin{proof}
We assume $\varphi|_{\Rno} > 0$ for convenience; the case $\varphi|_{\Rno} < 0$ is handled by considering \eqref{negativecase}.
 By Observation \ref{K=-1remarkindicatrix}, we have $\Phi(x, y) \geq 0$, with equality if and only if $y = \mathbf{0}$. Hence, the solution $\Phi(x,y)$ uniquely exists to the following equation
\[
\varphi\Big( \frac{y}{\Phi(x,y)}+x  \Big)=1,\quad \forall (x,y)\in \Omega\times \Rno.
\]
Therefore, for  a fixed point $\bar{x}\in \Omega$, the function  $\Delta:=\Phi(x+\bar{x},y)$ is the unique solution to the equation
\begin{equation}\label{lastpsta11}
\varphi\left( \frac{y}{\Delta}+(x+\bar{x})  \right)=1,\quad \forall (x,y)\in (\Omega-\bar{x})\times \Rno.
\end{equation}
The rest of the proof proceeds in two steps.

\textbf{Step 1. Existence.}
We show that $\Delta$ solves \eqref{standPHIqe2}. Given $(x, y) \in (\Omega - \bar{x}) \times \Rno$, define $\bar{y} := y + x \Delta$. By \eqref{lastpsta11}, $\Delta$ satisfies
\begin{equation}\label{lalqua11}
\Delta = \varphi(\bar{y} + \bar{x} \Delta).
\end{equation}
Alternatively, the assumption implies that $\Phi(\bar{x},\bar{y})$ is the unique solution to equation \eqref{lalqua11}, which gives
\[
\Delta = \Phi(\bar{x}, \bar{y}) = \Phi(\bar{x}, y + x \Delta) = \bar{\varphi}(y + x \Delta).
\]
Hence, $\Delta$ is a solution to equation \eqref{standPHIqe2}. This argument remains valid when $y = \mathbf{0}$.

\textbf{Step 2. Uniqueness.}
We now prove uniqueness.
Note that every solution $\bar{\Phi}$ to  equation \eqref{standPHIqe2} satisfies $\bar{\Phi}(x,y)\geq 0$ with equality if and only if $y=\mathbf{0}$. Hence, \eqref{standPHIqe2} furnishes
 \[
 1=\bar{\varphi}\left( \frac{y}{\bar{\Phi}(x,y)}+x  \right),\quad \forall (x,y)\in (\Omega-\bar{x})\times \Rno,
 \]
which together with \eqref{standPHIqe} yields
\begin{align*}
1&=\bar{\varphi}\left( \frac{y}{\bar{\Phi}(x,y)}+x  \right)=\Phi\left( \bar{x}, \frac{y}{\bar{\Phi}(x,y)}+x  \right)\\
&=\varphi\left(     \frac{y}{\bar{\Phi}(x,y)}+x+\bar{x}\,\Phi{\left( \bar{x}, \frac{y}{\bar{\Phi}(x,y)}+x  \right) }     \right)=\varphi\left(    \frac{y}{\bar{\Phi}(x,y)}+x+\bar{x}\right).
\end{align*}
Since $\Delta$ is the unique solution to equation \eqref{lastpsta11}, we obtain $\bar{\Phi}(x,y)=\Delta=\Phi(x+\bar{x},y)$.
 \end{proof}

\begin{proposition}\label{propMainEq}
Let $\varphi \colon \mathbb{R}^n \rightarrow \mathbb{R}$ be a positively $1$-homogeneous function, and let $\Phi=\Phi(x,y)$ be the solution  to the following equation
\begin{equation} \label{MainEq}
\Phi(x, y) = \varphi(y + x \Phi(x,y)).
\end{equation}
Suppose that $\mathcal {D}(\Phi)$ is a connected component of the set
 \[ \{x\in \mathbb{R}^n\,:\, \Phi(x,y) \text{ is well-defined for all $y\in \mathbb{R}^n$}\}\]
such that $\mathcal {D}(\Phi)$ contains the origin.
Then the following hold:
\begin{enumerate}[{\rm (i)}]
\item  \label{MainEqi}
if  $\varphi|_{\Rno} >0$,
then
$\mathcal {D}(\Phi)\subset  \{ x \in \mathbb{R}^n \, : \, 0 \leq \varphi(x) < 1   \}$;

\smallskip

\item  \label{MainEqii}
if $\varphi|_{\Rno} < 0$, then $
\mathcal {D}(\Phi)\subset \{ x \in \mathbb{R}^n\,:\, -1 < \varphi(-x) \leq 0   \}.$
\end{enumerate}
\end{proposition}
\begin{proof}
We prove only \eqref{MainEqii}, as the argument for \eqref{MainEqi} is similar but simpler.
Since $\varphi \leq 0$,   \eqref{MainEq} implies that
\begin{equation*}
0 \geq \Phi(x,-x)= \varphi( - x + x \Phi(x,-x)) = (1 - \Phi(x,-x)) \varphi( - x),\quad \forall x\in \mathcal {D}(\Phi),
\end{equation*}
which yields
\[ 0 \geq \varphi(-x) = \frac{\Phi(x,-x)}{1 - \Phi(x,-x)} > -1,\quad \forall x\in \mathcal {D}(\Phi). \qedhere \]
\end{proof}

\subsection{Elements from Finsler geometry}\label{elemFinslergeo}

This subsection recalls essential definitions and properties from Finsler geometry; for details, we refer to Bao--Chern--Shen
\cite{BCS},   Ohta \cite{Ohta1} and Shen \cite{ShenSpray,ShenLecture}.

Let $V$ be an $n$-dimensional vector space with $n \geq 2$. A {\it pseudo-norm} on $V$ is a nonnegative function $\psi: V \rightarrow \mathbb{R}$ satisfying:
\begin{enumerate}[{\rm (i)}]
\item\label{minkownor1} $\psi(y)\geq 0$ with equality if and only if $y=\mathbf{0}$;
\item\label{minkownor2} $\psi$ is positively $1$-homogeneous, i.e., $\psi(\alpha y) = \alpha \psi(y)$ for all $\alpha > 0$;
\item  $\psi(y_1+y_2)\leq \psi(y_1)+\psi(y_2)$ for any $y_1,y_2\in V$.
\end{enumerate}
A pseudo-norm $\psi$ is  a {\it weak Minkowski norm} if it is smooth on $V\backslash\{\mathbf{0}\}$ and regular on $V$ (see Definition \ref{regualrdef}/(3)).
It is a {\it Minkowski norm} if, in addition, the Hessian matrix $\left( [\psi^2]_{y^i y^j}(y) \right)$ is positive definite for all $y \in V \backslash\{\mathbf{0}\}$.
The pair $(V, \psi)$ is then called a ({\it weak}) {\it Minkowski space}.

Every pseudo-norm is quasi-regular and every Minkowski norm is regular. Furthermore, the class of weak Minkowski norms is strictly weaker, as there exist examples that are not Minkowski norms (see \cite[Exercise 1.2.7]{BCS}).
A fundamental example of  Minkowski norm on $\mathbb{R}^n$ is the {\it Randers norm}
$\psi(y):=|y|+\langle a,y \rangle$,
where $a$ is a fixed vector with $|a|<1$.

Let $M$ be an $n$-dimensional connected smooth manifold with tangent bundle $TM$.
 A ({\it weak}) {\it Finsler metric} $F = F(x,y)$ on  $M$ is a $C^{\infty}$ function defined on $TM\backslash\{0\}$ such
 that $F(x,\cdot)$ is a (weak) Minkowski  norm on $T_{x}M$ for each $x\in M$. The pair $(M,F)$ is called a ({\it weak}) {\it Finsler manifold}.
 The {\it indicatrix} of $F$ at $x\in M$ is the set
$S_xM:=\{y\in T_xM\,:\, F(x,y)=1 \}$,
 which is strongly convex for a Finsler metric and strictly convex for a weak one.

Henceforth, unless stated otherwise, $F$ denotes a Finsler metric.
The associated \textit{fundamental tensor} $g = (g_{ij}(x, y))$ is defined as
\begin{align}\label{defbasictensor}
    g_{ij}(x, y) := \frac{1}{2} [F^2]_{y^i y^j}(x, y), \quad \forall (x, y) \in TM \backslash \{0\},
\end{align}
which induces a Riemannian metric on $T_x M\backslash \{0\}$. The metric $F$ is called \textit{Riemannian} if $g$ is independent of the direction $y$, and \textit{Minkowskian} if it is independent of the position $x$. As usual, $(g^{ij})$ denotes the inverse  of the matrix $(g_{ij})$.

According to Rademacher \cite{Rade1,Rade}, the {\it reversibility}    is defined by
\begin{equation}\label{def_reversibility}
\lambda_F(M) := \sup_{x \in M} \lambda_{F}(x), \quad \text{where} \quad \lambda_{F}(x) = \sup_{y \in T_xM \backslash \{ 0 \}} \frac{F(x, -y)}{F(x,y)}.
\end{equation}
It follows that $\lambda_F(M) \geq 1,$ with equality if and only if $F$  is {\it reversible}, i.e., $F (x, y) = F(x, -y)$ for all $(x,y)\in TM$.
Furthermore, the function $\lambda_F(x)$ is continuous on $M$.

The \textit{Legendre transformation} $\mathfrak{L}: TM \rightarrow T^*M$ is defined by
\[
\mathfrak{L}(x, y) :=
\begin{cases}
g_y(y, \cdot) = g_{ij}(x, y) y^i {\dd}x^j, & \text{if } y \in T_x M\backslash \{\mathbf{0}\}, \\
0, & \text{if } y = \mathbf{0}.
\end{cases}
\]
Given a smooth function $u: M \rightarrow \mathbb{R}$, its \textit{gradient} is defined as $\nabla u(x) := \mathfrak{L}^{-1}({\dd} u(x))$. For any smooth vector field $X$, provided $\nabla u(x)\neq0$, the following relation holds
\[
\langle X, {\dd}u\rangle=X(u)=g_{\nabla u}(\nabla u,X),
\]
where  $\langle y,\xi\rangle:=\xi(y)$
 denotes the canonical pairing between $T_xM$
and $T^*_xM$.

The  {\it co-metric} (or {\it dual metric}) $F^{*}$ of $F$ on $M$ is defined as
\begin{equation}\label{dualmetric}
F^{*}(x,\xi):=\sup_{y \in T_xM \backslash \{0\}} \frac{\xi(y)}{F(x,y)}, \quad \forall \xi \in T^{*}_xM,
\end{equation}
which is a ``Finsler metric" on the cotangent bundle $T^{*}M$.
It satisfies the inequality
\begin{equation}\label{dualff*}
\langle y,\xi\rangle\leq F(x,y)F^*(x,\xi), \quad  \forall  y\in T_xM,\ \xi\in T^*_xM,
\end{equation}
with equality if and only if $\xi=\alpha \mathfrak{L}(x,y)$ for some $\alpha\geq 0$. Moreover, $F(x,y)=F^*(\mathfrak{L}(x,y))$ for all $y\in T_xM$. 

A smooth curve $t\mapsto\gamma(t)$ in $(M, F)$ is called a ({\it constant-speed}\,) {\it geodesic} if it satisfies
\begin{equation}\label{geodesequ}
\frac{{\dd}^2 \gamma^i}{{\dd}t^2} + 2 G^i \Big(\gamma, \frac{{\dd} \gamma}{{\dd}t}\Big) = 0,
\end{equation}
where the {\it geodesic coefficients} $ G^i = G^i(x,y)$ are
given by
\begin{equation}\label{goedcooff}
G^i := \frac{1}{4} g^{ij} \left\{  [F^2]_{x^k y^j}  y^k - [F^2]_{x^j}  \right\}.
\end{equation}

 Given  $y\in  T_xM\backslash\{0\}$ and a plane $  \Pi=\text{span}\{y,v\}\subset T_xM$, the
{\it flag curvature}  is defined as
\[
\mathbf{K} (y;\Pi):=\mathbf{K}(y;v):=\frac{g_{ij}(x,y)\,R^i_{\ k}(x,y)\, v^kv^j}{F^2(x,y)\,g_{ij}(x,y)\,v^iv^j-[g_{ij}(x,y)\,y^iv^j]^2},
\]
where
$
 R^i_{\ k} := 2 G^i_{x^k}   -y^j   G^i_{x^jy^k} +2G^j  G^i_{y^jy^k}  -G^i_{y^j}G_{y^k}^j.
$
A Finsler metric $F$ is of \textit{scalar flag curvature} if the curvature depends only on the position and direction, namely, $\mathbf{K}(y; \Pi) = \mathbf{K}(x, y)$. It is of \textit{constant flag curvature} $k$ if $\mathbf{K} \equiv k$. For a Riemannian metric, the flag curvature specializes to the sectional curvature.



Given a piecewise smooth curve $c : [0, 1] \rightarrow M$, its {\it length} is defined as
 \begin{equation}\label{lengthinducedbyF}
 L_{F}(c) : = \int_0^1 F(c(t), c'(t)) {\dd}t.
 \end{equation}
 This induces a \textit{distance function} $d_F: M \times M \rightarrow [0, +\infty)$ defined by
 \begin{equation}\label{distanceinducedbyF}
d_{F}(x_1, x_2):= \inf L_F(c),
\end{equation}
where the infimum is taken over all piecewise smooth curves
$c:[0,1] \rightarrow M$ with $c(0) =x_1$ and $c(1)=x_2$.
For any $x_1,x_2,x_3\in M$, one has
\begin{itemize}
\item $d_F(x_1,x_2)\geq 0,$ with equality if and only if $x_1=x_2$;
\item $d_F(x_1,x_2)\leq d_F(x_1,x_3)+d_F(x_3,x_2)$.
\end{itemize}
Usually $d_F(x_1, x_2) \neq d_F(x_2, x_1),$ unless $F$ is reversible. In fact, for any two distinct  $x_1,x_2\in M$,
\begin{equation}\label{distandreversi}
	\frac{d_F (x_1, x_2)}{d_F(x_2, x_1)} \leq \lambda_F(M).
\end{equation}
For $R > 0$, the \textit{forward} and \textit{backward metric balls} are defined respectively as
\begin{equation}\label{forward/backwradball}
    B^+_R(x) := \{ z \in M : d_F(x, z) < R \}, \qquad B^-_R(x) := \{ z \in M : d_F(z, x) < R \}.
\end{equation}
They coincide when $F$ is reversible, and are then denoted simply by $B_R(x)$.

A Finsler manifold $(M, F)$ is called {\it forward complete} if every geodesic $t\mapsto \gamma(t)$, $t\in [0,1)$, can be
extended to a geodesic defined on $t\in [0,+\infty)$; similarly, $(M, F)$ is called {\it backward complete} if every geodesic $t\mapsto \gamma(t)$, $t\in (0,1]$, can be extended to a geodesic on $ t\in (-\infty,1]$. These two kinds of completeness are equivalent when $\lambda_F(M) < +\infty$.
And $(M, F)$ is  said to {\it complete} if it is both forward
 and backward complete. A compact Finsler manifold is always complete.

 If $(M,F)$ is either forward or backward complete, then for every two points $x_1,x_2\in M$, there exists a minimal geodesic $\gamma$ from $x_1$ to $x_2$ with $d_F(x_1,x_2)=L_F(\gamma)$. Moreover, the closure of a forward (resp., backward) metric ball with finite radius is compact if $(M,F)$ is forward (resp., backward) complete.

On a forward complete $(M, F)$, the {\it exponential map} $\exp:TM\rightarrow M$ is defined as
\begin{equation}\label{expon}
\exp(x,y):=\exp_x(y):=\gamma_y(1),
\end{equation}
where $\gamma_y(t)$, $t\in [0,+\infty)$ denotes a geodesic with  $\gamma'_y(0)=y\in T_xM$. In particular, $\exp$ is $C^1$ at the zero section of $TM$ and $C^\infty$ away from it.

Given a point $x\in M$,  define the {\it cut value} $i_y$ of $y \in S_x M$  and the {\it injectivity radius} $\mathfrak{i}_x$ at $x$ as
\begin{align}\label{injectradia}
	i_y := \sup \{ t>0: \text{the geodesic }  \gamma_y|_{[0,t]}   \text{ is globally minimizing} \}, \quad
	\mathfrak{i}_x  := \inf_{y \in S_xM} i_y>0.
\end{align}
The {\it injectivity radius of a Finsler manifold} $(M,F)$ is defined as $\mathfrak{i}_M:=\inf_{x\in M}\mathfrak{i}_x$.
This is positive and finite if $M$ is compact, and infinite for Cartan--Hadamard manifolds (i.e., simply connected, forward complete, with non-positive flag curvature).

The {\it reverse Finsler metric} $\overleftarrow{F}$ is defined as $\overleftarrow{F}(x,y):=F(x,-y)$.  One can verify that $(M,F)$ is forward (resp., backward) complete   if and only if $(M,\overleftarrow{F})$ is backward (resp., forward) complete. For this reason, in the sequel, we restrict our attention to forward complete Finsler manifolds.

Now we investigate the measure theory on a Finsler manifold. In the sequel, let $\m$ be a smooth  positive measure  on a forward complete Finsler manifold $(M,F)$ and let $\gamma_y(t)$ denote  a geodesic with $\gamma'_y(0)=y$.
In a local coordinate system ($x^i$), $\dm$ can be expressed as
\begin{equation*}
\dm = \sigma\, {\dd}x^1 \wedge \cdot\cdot\cdot \wedge {\dd}x^n,
\end{equation*}
where $\sigma= \sigma(x)$ denotes the {\it density function} of $\dm$. The {\it distortion} $\tau$ and the {\it S-curvature} $\mathbf{S}$ of $(M,F,\m)$ are defined by
\begin{equation}\label{distsdef}
\tau(x,y):= \ln \frac{\sqrt{\det g_{ij}(x,y)}}{\sigma(x)}, \qquad \mathbf{S}(x,y):=\left.\frac{\dd}{{\dd}t}\right|_{t=0}\tau(\gamma_y(t), {\gamma}'_y(t)), \quad \forall  y\in T_xM \backslash \{0\}.
\end{equation}
According to Shen \cite{ShenLecture}, we have
\begin{equation}\label{Scurvature}
\mathbf{S}(x,y) = G^i_{y^i}(x,y) - y^i [\ln \sigma(x)]_{x^i}.
\end{equation}

Let $(r,y)$ be the {\it polar coordinate system} around some point $o\in M$. That is, if $x=(r,y)$, then $x=\exp_o(ry)$ with $ r=r(x)=d_F(o,x)$ and $y=y(x)\in S_oM$  (cf. Zhao and Shen \cite[Section 3]{ZS}). Hence, for any $y\in S_oM$ and $r\in (0,i_y)$,
\begin{equation}\label{geommeaingofr}
\gamma_y(r)=(r,y),\quad \gamma'_y(r)=\nabla r|_{(r,y)}.
\end{equation}
Moreover,
it follows from  \cite[Lemma 3.2.3]{ShenLecture} that $F(\nabla r) = F^{*}({\dd}r) =1$ for  $\m$-a.e. $x\in M$.


Set 
\[
 \mathfrak{s}_k(t):=\left\{
	\begin{array}{lll}
				\ \ \ \ t, && \text{ if }k=0,\\
		\frac{\sinh(\sqrt{-k }t)}{\sqrt{-k}}, && \text{ if }k<0.
	\end{array}
	\right.
\]
In the polar coordinate system $(r,y)$ around $o$, the measure $\m$ decomposes as
\begin{equation*}
\dm|_{(r,y)}=:\hat{\sigma}_o(r,y)\,{\dd}r \wedge {\dd}\nu_o(y),
\end{equation*}
where ${\dd}\nu_o$ denotes the Riemannian volume form of the indicatrix $S_oM$, i.e., $\nu_o(S_oM):=\int_{S_oM}{\dd}\nu_o$
is its Riemannian volume. Thus,
\begin{equation}\label{inffexprexx}
\int_M f \dm=\int_{S_oM} {\dd}\nu_o(y) \int_0^{i_y} f(r,y) \,\hat{\sigma}_o(r,y)\,{\dd}r,\quad \forall  f\in L^1(M,\m).
\end{equation}
Moreover,  $\nu_o(S_oM)$ plays an important role in the Gauss-Bonnet formula; see Bao--Chern \cite{BC} and Shen \cite{ShenGauss}.

The following volume comparison result is  essential to our analysis of Sobolev spaces in Section \ref{SecSobolev} (cf. \cite[Theorem 3.6]{ZS}).
\begin{theorem}[\cite{ZS}]\label{bascivolurcompar} Let $(M,F,\m)$ be an $n$-dimensional forward complete Finsler manifold endowed with a smooth positive measure $\m$. If $\mathbf{K}\leq  k$ for some $k \in \mathbb{R}$, then for every $o\in M$ and $y\in  S_oM$, the function
\begin{equation*} 
H_y(r):=\frac{\hat{\sigma}_o(r,y)}{ e^{-\tau\big(\gamma_y(r),{\gamma}_y'(r)\big)}  \mathfrak{s}_k^{n-1}(r)},\quad \forall r\in (0,i_y),
\end{equation*}
is monotonically increasing in r and satisfies $\lim_{r \to 0^+} H_y(r) = 1$,
where $(r,y)$ is the polar coordinate system around $o$.
\end{theorem}
Further comparison results can be found in the works of Shen \cite{ShenAdv},  Burago--Ivanov \cite{BI}, and Wu--Xin \cite{WX}, and references therein.
We conclude this subsection by recalling the
 {\it integral of distortion}
\begin{equation}\label{cocont}
\mathscr{I}_{\m}(x):=\int_{S_xM} e^{-\tau(x,y)}{\dd}\nu_x(y)<+\infty,
\end{equation}
which is useful to characterize the Busemann--Hausdorff measure (cf. Huang--Krist\'aly--Zhao \cite{HKZ}).

\subsection{Projectively flat Finsler metrics}

This subsection covers basic properties and new results on projectively flat Finsler metrics, which are the regular pseudo-norm functions concerning Hilbert's fourth problem.  See Berwald \cite{Be1,Be2}, Funk \cite{Funk1}, Papadopoulos--Troyanov \cite{PT}, and Shen \cite{ShenSpray}  for further details.

A Finsler metric $F$ on a domain $\Omega \subset \mathbb{R}^n$ is
 said to be \textit{projectively flat} if the images of all geodesics are straight lines.
 This condition is equivalent to  the geodesic coefficients  taking the form
\begin{equation}\label{prjoactor}
G^i(x,y) = P(x,y) y^i,
\end{equation}
 where  $P (x,y):=\frac{y^k}{2 F(x,y)} F_{x^k}(x,y)$ is called the {\it projective factor} of $F$.
 Note that $P$ is positively $1$-homogeneous in $y$.
For a projectively flat Finsler metric, the flag curvature is a scalar function on $T\Omega \backslash \{0\}$ given by the formula \begin{equation}\label{flagC} \mathbf{K}(x,y) = \frac{P^2(x,y) - y^i P_{x^i}(x,y)}{F^2(x,y)}. \end{equation} This result is due to Berwald \cite{Be2}.
The following characterization of flag curvature via a PDE system was also established by Berwald.
\begin{lemma}[\cite{Be2}]\label{lemBerwald}
Let $F=F(x,y)$ be a Finsler metric on a domain $\Omega \subset \mathbb{R}^n$. Then $F$ is projectively flat if and only if
there is a positively $1$-homogeneous function in $y$, say $P = P(x, y)$, and a
positively $0$-homogeneous function in $y$, say $\mathbf{K} = \mathbf{K}(x, y)$, on $T \Omega$
such that
\begin{align}
F_{x^k} = [P F]_{y^k},\qquad
P_{x^k} = P P_{y^k} - \frac{1}{3 F} [\mathbf{K} F^3]_{y^k}.\label{Berwald2}
\end{align}
In this case, $P$ is the projective factor while $\mathbf{K}$ is the flag curvature of $F$.
\end{lemma}

It is Berward who first observed a relationship between  flag curvatures and geodesics \cite{Be2}. We reformulate this connection in an alternative manner,
 which will serve as a pivotal tool for investigating  projectively flat metrics in the following sections.
\begin{theorem}\label{lemgeodesicK}
Let $(\Omega, F)$ be a forward complete projectively flat Finsler manifold, where $\Omega\subset \mathbb{R}^n$ is a domain.
Then a (constant-speed) geodesic $\gamma(t)$ starting at a point $x \in \Omega$ with initial velocity $y\in T_x\Omega\backslash\{ \mathbf{0}\}$ can be expressed as
\begin{equation} \label{geodesiceq}
\gamma(t) = f(t)y +x, \quad t\in (-\varepsilon, +\infty),
\end{equation}
for some $\varepsilon > 0$, where the function $f$  satisfies
\begin{equation}\label{geodesicintial}
  f(0) = 0, \quad f'(0)= 1,  \   \text{ and } \  \ f'(t) > 0.
\end{equation}
Moreover, the flag curvature   along $\gamma(t)$ satisfies
\begin{equation}\label{geodesicK}
\mathbf{K}(\gamma(t),\gamma'(t))\,F^2(\gamma(t),\gamma'(t)) = \frac{2 f'''(t) f'(t) -  3 [ f''(t)]^2}{4  [f'(t)]^2}.
\end{equation}
\end{theorem}
\begin{proof}
For a projectively flat metric, the geodesic equation \eqref{geodesequ} reduces to
\begin{equation} \label{geodesicODE}
\frac{{\dd}^2 \gamma^i}{{\dd}t^2} +2P  \Big(\gamma, \frac{{\dd} \gamma}{{\dd}t}\Big) \frac{{\dd} \gamma^i}{{\dd}t} =0
\end{equation}
by virtue of \eqref{prjoactor}.
Since every geodesic is a straight line, we may assume (after a suitable reparameterization) that the geodesic $\gamma$ with $\gamma(0) = x$ and $\gamma'(0) = y$ takes the form
\begin{equation*}
\gamma(t) = f(t){y} +{x}, \quad t\in (-\varepsilon, \varepsilon),
\end{equation*}
for some $\varepsilon > 0$ guaranteed by the standard existence theory for ODEs.
The initial conditions \eqref{geodesicintial} follow directly from $\gamma(0) = x$ and $\gamma'(0) = y$.
The forward completeness of $(\Omega, F)$ allows us to extend $\gamma$ to the interval $(-\varepsilon, +\infty)$, which completes the proof of the first statement.

To establish the second statement, substituting \eqref{geodesiceq} into  \eqref{geodesicODE} and using  ${y}=(y^i)\neq\mathbf{0}$ gives
\begin{equation}\label{geodesicBackf''1}
f''(t)  + 2 P (\gamma(t), \gamma'(t) ) f'(t) = 0.
\end{equation}
Differentiating \eqref{geodesicBackf''1} with respect to $t$ and using \eqref{geodesiceq} again yields
\begin{equation}\label{geodesicf'''}
f'''(t) + 2 P_{x^k} (\gamma(t), \gamma'(t) )  y^k  [f'(t)]^2+ 2 P_{y^k} (\gamma(t), \gamma'(t) ) y^k  f'(t)  f''(t) + 2 P (\gamma(t), \gamma'(t) ) f''(t)=0.
\end{equation}

On the other hand,
by means of the homogeneity of $P$, $\gamma'(t)=f'(t){y}$, \eqref{Eulerhter}, and  \eqref{flagC},
we have
\begin{align*}
P_{y^k} (\gamma(t), \gamma'(t) ) f'(t) y^k& = P(\gamma(t), \gamma'(t) ),  \\
  P_{x^k} (\gamma(t), \gamma'(t) ) f'(t)y^k & =P^2(\gamma(t), \gamma'(t) ) - \mathbf{K}(\gamma(t), \gamma'(t) )\, F^2(\gamma(t), \gamma'(t) ).
 \end{align*}
Substituting these into \eqref{geodesicf'''} yields
\begin{equation}\label{geodesicf'''1}
f'''(t) + 2 P^2(\gamma(t), \gamma'(t) ) f'(t)+ 4 P (\gamma(t), \gamma'(t) )  f''(t) - 2   f'(t)\, \mathbf{K}(\gamma(t), \gamma'(t) )\, F^2(\gamma(t), \gamma'(t) ) = 0.
\end{equation}
Then $2 f'(t) \times \eqref{geodesicf'''1} - \left( 2   f'(t) P (\gamma(t), \gamma'(t) )+ 3 f''(t) \right) \times \eqref{geodesicBackf''1}$ equals
\begin{equation*}\label{geodesicf'''3}
2 f'''(t) f'(t) -  3 [ f''(t)]^2 - 4   [f'(t)]^2\, \mathbf{K}(\gamma(t), \gamma'(t) )\, F^2(\gamma(t), \gamma'(t) ) =0,
\end{equation*}
which gives \eqref{geodesicK} due to $f'(t)>0$.
\end{proof}

The structure of geodesics emanating from a point encodes the global ``shape" of a projectively flat Finsler manifold.

\begin{lemma} \label{lemfortocom}
Let $(\Omega, F)$ be a  forward complete projectively flat Finsler manifold, where $\Omega\subset\mathbb{R}^n$ is a domain.  Given  $x \in \Omega$ and $y\in S_{x}\Omega$, let
\begin{equation*} 
 \gamma_{y}(t) = f(t;x,y) y + x, \quad t\in (\delta(x,y),+\infty),
\end{equation*}
be a (unit-speed) geodesic  satisfying   $\gamma_{{y}}'(0)={y}$. Here, $(\delta({x},{y}),+\infty)$ is the  maximal domain of  $\gamma_{{y}}$, i.e.,
\begin{equation*}
(\delta({x},{y}),+\infty):=\left\{ t\in \mathbb{R}\,:\, \gamma_{{y}}(t) \text{ is well-defined in } \Omega  \right\}.
\end{equation*}
If $\ds\lim_{t\rightarrow +\infty} f(t;{x},{y}) = +\infty$ for every ${y}\in S_{{x}}\Omega$, then $\Omega=\mathbb{R}^n$ and particularly,
\begin{equation}\label{leftlimoff}
\lim_{t\rightarrow\delta({x},{y})} f(t;{x},{y}) = -\infty,
\end{equation}
\end{lemma}
\begin{proof}
By the assumption,    we have  $\lim_{t\rightarrow +\infty} |\gamma_{{y}}(t)-{x}|= +\infty$ in every direction ${y}\in S_{{x}}\Omega$,
which means
$\Omega = \mathbb{R}^n$. Thus, \eqref{leftlimoff} follows since $\gamma_{{y}}$ is a straight line in  $\Omega=\mathbb{R}^n$.
\end{proof}

The completeness of a  manifold provides a finer property of projectively flat Finsler manifold.

 \begin{proposition}\label{completenessofprojeflat}
Let $(\Omega,F)$ be a  projectively flat Finsler manifold, where $\Omega$ is a bounded domain  in $\mathbb{R}^n$  containing the origin. Denote by $d_{\E}(x,\partial\Omega):=\inf_{z\in \partial\Omega}|z-x|$   the standard Euclidean distance between $x$ and $\partial\Omega$. Then,
\begin{enumerate}[{\rm (i)}]
\item\label{forcompleprf} $(\Omega,F)$ is forward complete if and only if $d_F(\mathbf{0},x)\rightarrow +\infty$ as $d_{\E}(x,\partial\Omega)\rightarrow 0$;

\item\label{backcompleprf} $(\Omega,F)$ is backward complete if and only if $d_F(x,\mathbf{0})\rightarrow +\infty$ as $d_{\E}(x,\partial\Omega)\rightarrow 0$.
\end{enumerate}
\end{proposition}
\begin{proof}
We establish \eqref{forcompleprf}, as \eqref{backcompleprf} follows by a symmetric argument employing the reverse Finsler metric $\overleftarrow{F}$.

We first prove the necessity.  Assume that $(\Omega,F)$ is forward complete. Suppose, for contradiction, that  there exists a sequence $(x_i)\subset \Omega$ such that
\begin{equation}\label{twocondleve}
\lim_{i\rightarrow +\infty}d_{\E}(x_i,\partial\Omega)= 0, \qquad R:=\sup_id_F(\mathbf{0},x_i)<+\infty.
\end{equation}
Owing to the compactness of $\overline{B^+_R({\mathbf{0}})}$, we may assume that
$(x_i)$ converges to  $x_{\infty}\in \overline{B^+_R({\mathbf{0}})}\subset \Omega$  (under the natural topology $\Omega\subset \mathbb{R}^n$). Hence, $x_{\infty}\notin \partial\Omega$, which contradicts \eqref{twocondleve}$_1$ (i.e., the first equality in \eqref{twocondleve}).

We now prove the sufficiency. Suppose $d_F(\mathbf{0},x)\rightarrow +\infty$ as $d_{\E}(x,\partial\Omega)\rightarrow 0$. Given $x_0\in \Omega$ and $y\in S_{x_0}\Omega$, let $l_{x_0,x_1}$ be a straight segment from $x_0$ in the direction $y$, which intersects $\partial\Omega$ at $x_1$. On the other hand, by the ODE theory one can obtain a geodesic $\gamma_y(t)$, $t\in [0,\epsilon)$ with $\gamma'_y(0)=y$.
 Since the image of $\gamma_y([0,\epsilon))$ is a straight segment in $\Omega$, the compactness of $\overline{\Omega}\subset \mathbb{R}^n$ implies that $\gamma_y(t)$ can be extended to $t=\epsilon$ unless $\gamma_y(\epsilon)=x_1\in \partial\Omega$. By considering the geodesic initiating from $\gamma_y(\epsilon)$ in the direction $y$,
an easy induction gives that   $\gamma_y$ can be defined on  the whole $l_{x_0,x_1}\backslash\{x_1\}$.

The assumption implies
$d_F(\mathbf{0},\gamma_y(t))\rightarrow +\infty$ as  $d_{\E}(\gamma_y(t),x_1)\rightarrow 0$,
which combined with
the triangle inequality  yields
\[
t=d_F(x_0,\gamma_y(t))\geq   d_F(\mathbf{0},\gamma_y(t))-d_F(\mathbf{0},x_0)\rightarrow +\infty,
\]
 as  $d_{\E}(\gamma_y(t),x_1)\rightarrow 0$. Then $\gamma_y(t)$ is defined on $[0,+\infty)$, which implies that $(\Omega,F)$ is forward complete.
\end{proof}

Following \eqref{twocondleve}$_1$, we adopt
 the convention ($*$)$_i$  to denote  the $i$-th equality/equation in ($*$) throughout this manuscript.
We now recall the definition of Funk metric (cf. Funk \cite{Funk1}).

\begin{definition}\label{funkdef}
A Finsler metric $\mathsf{F}$ on a bounded domain $\Omega\subset \mathbb{R}^n$ is called a  {\it Funk metric} if
\begin{equation}\label{funkdef}
x + \frac{y}{\mathsf{F}(x,y)}  \in \partial \Omega,
\end{equation}
for  every $x\in \Omega$ and  $ y \in T_x \Omega\backslash\{\mathbf{0}\}$. The pair $(\Omega,\mathsf{F})$ is called a  {\it Funk metric space}.
\end{definition}
We remark that in \eqref{funkdef} both $x$ and $y/\mathsf{F}(x,y)$ are viewed as  vectors in $\mathbb{R}^n$ with their sum emanating from  the origin.  From a local point of view,
a Finsler metric  $\mathsf{F}$ is a Funk metric if and only if
\begin{equation}\label{Funk}
\mathsf{F}_{x^k}=\mathsf{F}\,\mathsf{F}_{y^k}.
\end{equation}

The following characterization of Funk metrics is due to Li~\cite{Li} and Shen~\cite{Sh1}.
\begin{theorem}[\cite{Li, Sh1}]\label{funkmetrisufficnecessy}
A   pair $(\Omega,\mathsf{F})$ is a Funk metric space if and only if
there is a (unique) Minkowski norm $\psi$ on $\mathbb{R}^n$ such that
\begin{equation} \label{phiF}
	\partial\Omega=\psi^{-1}(1),\qquad \mathsf{F}(x,y) = \psi{\left(y + x \mathsf{F}(x,y)\right)}.
\end{equation}
In particular, $\Omega$ is a bounded strongly convex domain in $\mathbb{R}^n$ given by
\begin{equation*}
	\Omega=\{x\in \mathbb{R}^n\,:\, \psi(x)<1\}.
\end{equation*}
\end{theorem}
Additionally, we have the following completeness result (cf. Krist\'aly--Li--Zhao \cite[Proposition 6.1]{KLZ}).
\begin{proposition}[\cite{KLZ}] \label{funkmetriccompletness}
Every Funk metric space  is forward complete but not backward complete.
\end{proposition}

In view of Theorem \ref{lemsolexi} and  Theorem \ref{funkmetrisufficnecessy}, we introduce the following definition.
\begin{definition} Let  $\varphi$ be a weak Minkowski norm on $\mathbb{R}^n$ and set $\Omega:=\{\varphi<1\}$.  The (unique) solution $\Phi(x,y)$ to the following equation
\[
\Phi(x,y)=\varphi{\left(y+x\Phi(x,y)\right)},\quad \forall (x,y)\in \Omega\times \mathbb{R}^n,
\]
is called a {\it weak Funk metric} on $\Omega$.
\end{definition}

\begin{remark}\label{weakFuniswekmonk} Theorem~\ref{lemsolexi} consequently implies that the weak Funk metric $\Phi$ exists uniquely on $\Omega$ and has the property that $\Phi(x, \cdot)$ is a weak Minkowski norm on $T_x \Omega \cong \mathbb{R}^n$ for every $x \in \Omega$.
\end{remark}

\begin{definition}\label{Hildef}
A Finsler metric $\mathsf{H}$ on $\Omega$ is called a {\it Hilbert metric} if there exists a Funk metric $\mathsf{F}$ on $\Omega$ such that
\begin{equation}\label{Hilbertmetric}
\mathsf{H}(x,y)=\frac12\left\{ \mathsf{F}(x,y)+\mathsf{F}(x,-y)   \right\}.
\end{equation}
\end{definition}

From above, equation \eqref{phiF}$_2$ plays an key role in the investigation of  Funk metrics and Hilbert metrics. This is the reason why we study equation \eqref{pdeodefque} in Section \ref{convprofun}.
For more details on Hilbert and Funk geometries, see  Faifman \cite{Fai}, Funk \cite{Funk1} and Papadopoulos--Troyanov \cite{PT}.

We conclude the subsection with  the following result.
\begin{lemma}\label{Regionstrconvex}
Let $(\Omega,F)$ be a forward complete  projectively flat Finsler manifold on a domain $\Omega \subset \mathbb{R}^n$.
Suppose that  a positively $1$-homogeneous function $\varphi:\mathbb{R}^n\rightarrow \mathbb{R}$ satisfies
\[
0<\varphi|_{\Rno}\in  C^\infty(\Rno)\quad \text{ and } \quad \Omega =\{ x \in \mathbb{R}^n \,:\, \varphi(x) < 1 \}.
\]
 Then,
 $\varphi$ is a weak Minkowski norm  and
 $\Omega$ is a bounded strictly convex domain containing the origin with smooth boundary $\partial\Omega=  \varphi^{-1}( 1 )$.
\end{lemma}
\begin{proof}
By the positive 1-homogeneity of $\varphi$,
we have   $\varphi(\mathbf{0})=0$ and $\lim_{|y|\rightarrow +\infty}\varphi(y)=+\infty$. Thus, $\Omega$ is a bounded domain with smooth boundary
$\partial\Omega=  \varphi^{-1}( 1 )$.
The strict convexity of $\Omega$ in $\mathbb{R}^n$ follows directly by the forward completeness and the projective flatness of $(\Omega,F)$. Now the regularity of $\varphi$ is derived from
 Proposition \ref{quasitostrictlyconvex} immediately.
\end{proof}

It is natural to ask whether $\varphi$ is a Minkowski norm in Lemma \ref{Regionstrconvex}. Unfortunately, this may fail; see \cite[Exercise 1.2.7]{BCS} for an example. In fact,  $\varphi$ is a Minkowski norm if and only if
$\partial\Omega$ is strongly convex, i.e., its  Gaussian curvature is   positive everywhere (cf.~\cite[p.\,25]{PT}).

\vskip 10mm
\section{Projectively flat Finsler manifolds with $\mathbf{K}=0$} \label{sectionK=0}

This section addresses projectively flat Finsler manifolds of zero flag curvature. Our treatment has two main goals: first, to study their global structure and prove Theorem~\ref{thmK=0globalintro}; second, to investigate the existence of such metrics on a given domain.


\subsection{Global characterization of manifolds} \label{subsK=0_1}
This subsection begins by recalling a local characterization theorem due to Shen~\cite[Theorem 1.3]{Sh1} and Li~\cite[Theorem 1.1]{Li} for the sufficiency and necessity, respectively.
\begin{theorem}[\cite{Li,Sh1}] \label{thmK=0}
Let $F=F(x,y)$ be a Finsler metric on a neighborhood $\mathcal {U}$ of the origin $\mathbf{0}\in \mathbb{R}^n$. Then
 $F$ is   projectively flat  with $\mathbf{K}=0$ if and only if there exist  a Minkowski norm $\psi = \psi(y)$ and
a positively $1$-homogeneous function $\phi=\phi(y)$ on $\mathbb{R}^n$, with  $\phi|_{\Rno}\in C^\infty(\Rno)$, such that
\begin{equation}\label{thmK=0Eq}
F = \psi ( y + x P) \left\{ 1 +  P_{y^k} x^k \right\},
\end{equation}
where the projective factor $P = P(x,y)$  satisfies the following equation
\begin{equation} \label{thmK=0P}
P = \phi (y + x P).
\end{equation}
In particular, both $F$ and $P$ are locally uniquely determined near the origin by their initial data:
\begin{equation*}
\psi(y) = F(\mathbf{0},y), \qquad     \phi(y) = P(\mathbf{0},y).
\end{equation*}
\end{theorem}

\begin{remark}
By Lemma \ref{Sh1Lemma5.1},  $P$ and hence $F$ are unique on a sufficiently small ball $\mathbb{B}^n_\delta(\mathbf{0}) \subset \mathcal {U}$.
However, whether this uniqueness extends to the entire $\mathcal{U}$ remains an open problem.
\end{remark}

On the one hand, Theorem \ref{thmK=0} provides a local expression for projectively flat Finsler metrics with $\mathbf{K}=0$.
On the other hand, this local formulation suggests an abundance of such metrics on a small domain. However, by imposing the additional requirement of forward completeness, we discover that there are only two fundamental types.
\begin{theorem}\label{thmK=0global}
Let $(\Omega,F)$ be an $n$-dimensional forward complete  projectively flat Finsler manifold with $\mathbf{K}=0$, where $\Omega\subset \mathbb{R}^n$ is a domain containing the origin.  Define $\psi(y): = F(\mathbf{0},y)$ and  $\phi(y) := P(\mathbf{0},y)$.
Then the following hold:
\begin{enumerate}[{\rm (i)}]
\item\label{K=0Casei} if $(\Omega,F)$ is also backward complete, then $F=\psi$ is a Minkowski norm, $\phi=0$ and $\Omega=\mathbb{R}^n$;

 \smallskip

\item \label{K=0Caseii} if $(\Omega,F)$ is not backward complete, then $\phi$ is a weak Minkowski norm and $\Omega$ is a bounded strictly convex domain in $\mathbb{R}^n$  with
\begin{equation}\label{K=0region}
    \Omega = \{ x \in \mathbb{R}^n \,:\, \phi(x) < 1 \}.
\end{equation}

\end{enumerate}
\end{theorem}


\begin{proof}
 Let $\gamma(t)$ be a geodesic starting from a point $x\in \Omega$ with initial velocity $y\neq \mathbf{0}$.
 It follows by Theorem \ref{lemgeodesicK} and $\mathbf{K}=0$ that $ \gamma(t) = f(t)y+ x$, $t\in (-\varepsilon,+\infty)$ with
\begin{equation}\label{K=0intial}
2 f'''(t) f'(t) -  3 [ f''(t)]^2 =0,\quad  f(0)=0, \quad f'(0)=1.
 \end{equation}
Solving   \eqref{K=0intial} yields  the explicit form of the geodesic:
\begin{equation*}
 \text{either} \quad \gamma(t)=t {y} + {x}\quad   \text{or} \quad  \gamma(t)= \left( \frac{c_{{x},{y}} t}{c_{{x},{y}}+t}\right) {y} + {x},
\end{equation*}
where $c_{{x},{y}}>0$ is a constant depending continuously on   $ {x}$ and ${y}$, as required by forward completeness (i.e., $t\in [0,+\infty)$). Since the expression of $\gamma(t)=\gamma(t;{x},{y})$ depends continuously on ${x}$ and ${y}$,
\begin{equation}\label{claimtherekindscon}
\text{these two types of geodesics  cannot appear simultaneously on $(\Omega,F)$}.
\end{equation}
Thus,
owing to  the forward completeness and \eqref{geodesicBackf''1}, we have
\begin{enumerate}[{\quad\rm (a)}]
\item\label{conda} if $ \gamma(t)={y}t + {x}$, then Lemma \ref{lemfortocom} implies that $\Omega=\mathbb{R}^n$,  $\gamma(t)$ is defined on $(-\infty,+\infty)$ and
\begin{equation}\label{star1}
0=P (\gamma(t), \gamma'(t)) = P(t{y} + {x},{y});
\end{equation}

\item\label{condb} if $\gamma(t)={y} \left( \frac{c_{{x},{y}}t}{c_{{x},{y}}+t}\right) + {x}$, then the maximal domain of $\gamma(t)$ is $(-c_{{x},{y}},+\infty)$, and
\begin{equation}\label{star2}
\frac1{c_{{x},{y}}+t}=P \bigg(\left( \frac{c_{{x},{y}}t}{c_{ {x},{y}}+t}\right){y}+ {x}, \left( \frac{c_{{x},{y}} }{c_{ {x},{y}}+t}\right)^2{y}  \bigg).
\end{equation}
\end{enumerate}

\textbf{Case \eqref{K=0Casei}.}   Suppose $(\Omega,F)$ is backward complete. In view of \eqref{conda} and \eqref{condb}, the only possible geodesic is of the form
$\gamma(t)=ty+x$ for all $t\in (-\infty,+\infty)$, which implies $\Omega = \mathbb{R}^n$.
A classical result of Berwald~\cite{Be1} states that every complete projectively flat Finsler manifold of zero flag curvature is Minkowskian. Therefore,
$F(x,y)=F(\mathbf{0},y)=\psi(y)$.
Furthermore, \eqref{star1} gives $\phi({y})=P(\mathbf{0},{y})=0$.

\smallskip

\textbf{Case \eqref{K=0Caseii}.}   Suppose $(\Omega,F)$ is not backward complete.
 Due to \eqref{claimtherekindscon}, the only possible geodesic starting from ${x}:=\mathbf{0}$ with initial velocity ${y}\in S_{\mathbf{0}}\Omega$ is of the form
\begin{equation*}
\gamma(t)=\left( \frac{c_{{y}}t}{c_{{y}}+t} \right){y},
\end{equation*}
where $c_{{y}}:=c_{\mathbf{0},y}>0$ is a constant depending continuously on ${y}$.
Thus,
  \eqref{star2} yields
  \[
  \lim_{t\rightarrow +\infty}\phi(\gamma(t))=\phi(c_{{y}}{y})=P(\mathbf{0},c_{{y}}{y})= 1, \quad \forall {y}\in S_{\mathbf{0}}\Omega,
  \]
  which combined with the homogeneity implies $\phi|_{ \Rno}>0$ and
 \[
\phi(\gamma(t))=\phi\left(\left( \frac{c_{{y}}t}{c_{{y}}+t} \right){y}\right)= \frac{ t}{c_{{y}}+t} \phi(c_y{y})<  \phi(c_{{y}}{y})= 1, \quad \forall {y}\in S_{\mathbf{0}}\Omega.
 \]
Hence, $\Omega = \{ x \in \mathbb{R}^n \,:\, \phi(x) < 1 \}$. Now it follows from Lemma \ref{Regionstrconvex} that $\phi$ is a weak Minkowski norm and $\Omega$ is a bounded
strictly convex domain.
 \end{proof}





As an application of Theorem \ref{thmK=0global},  we obtain the unique structure of  projectively flat  manifolds with $\mathbf{K}=0$ in the Riemannian setting.
For a non-Riemannian example, see Example \ref{prlxeK=0nonR} below.

\begin{example}[Riemannian case] Let $(\Omega,F)$ be a forward complete projectively flat Finsler manifold with $\mathbf{K}=0$. If $F$ is Riemannian, then $F$ is  reversible,  and hence $(\Omega,F)$ is also backward complete.
It follows by Theorem \ref{thmK=0global}/\eqref{K=0Casei} that  $\Omega = \mathbb{R}^n$ and $F(x,y) = \psi(y)$ is a Minkowski norm induced by an inner product. Consequently,  by a suitable coordinate transformation, we have
$(\Omega ,F)= (\mathbb{R}^n,|\cdot|)$.
\end{example}



\begin{corollary}\label{reverprojK=0}
Let $(\Omega,F)$ be as in Theorem \ref{thmK=0global}. Then the following hold:
\begin{enumerate}[\rm (i)]
\item\label{K=0reversinfty1} if $(\Omega,F)$ is also backward complete, then $\lambda_F(\Omega) < +\infty$ and  $P=0$;

\item \label{K=0reversinfty2} if $(\Omega,F)$ is not backward complete, then $\lambda_F(\Omega) = +\infty$ and  $P$ is the unique weak Funk metric on $\Omega$.
\end{enumerate}
\end{corollary}
\begin{proof} \eqref{K=0reversinfty1}
Assume $(\Omega,F)$ is backward complete. By Theorem \ref{thmK=0global}/\eqref{K=0Casei}, $F=\psi$ is a Minkowski norm, and hence $\lambda_F(\Omega)=\sup_{y\in \mathbb{S}^{n-1}}\psi(-y)/\psi(y) < +\infty$. Moreover, $P=0$ follows by $\phi=0$ and \eqref{thmK=0P}.

\eqref{K=0reversinfty2} Suppose $(\Omega,F)$ is not backward complete. Since it is forward complete, $\lambda_F(\Omega)$ is infinite. It follows from Theorem \ref{thmK=0global}/\eqref{K=0Caseii}, \eqref{thmK=0P} and Remark \ref{weakFuniswekmonk} that  $P(x,y)$ is the unique weak Funk metric on $\Omega$.
\end{proof}

\begin{corollary}\label{existsK=0metric} Let $\Omega\subset \mathbb{R}^n$ be a domain containing the origin. Given a Minkowski norm $\psi$  and a positively $1$-homogeneous function $\phi$ on $\mathbb{R}^n$, there exists at most one Finsler metric $F$ on $\Omega$ such that $(\Omega,F)$ is a forward complete projectively flat Finsler manifold with $\mathbf{K}=0$ satisfying the initial conditions
\[
\psi(y)=F(\mathbf{0},y),\quad \phi(y)=P(\mathbf{0},y),\quad \forall  y\in T_{\mathbf{0}}\Omega.
\]
\end{corollary}
\begin{proof} According to Theorem \ref{thmK=0global}, the existence of $F$ implies one of the following cases occurs:
\begin{enumerate}[\rm \ \ \ \  (a)]
\item\label{F=01} $\phi=0$, $F=\psi$ is a Minkowski norm and $\Omega=\mathbb{R}^n$;

\item\label{F=02} $\phi$ is a weak Minkowski norm  and $\Omega=\{\phi<1\}$.
\end{enumerate}

Since the uniqueness of Case \eqref{F=01} is obvious, we focus on Case \eqref{F=02}. Suppose that there exist two Finsler metrics $F$ and $\widetilde{F}$ with the stated properties. Then the associated projective factors $P$ and $\widetilde{P}$ satisfy equation \eqref{thmK=0P}.  Theorem \ref{lemsolexi} yields $P=\widetilde{P}$, which together with \eqref{thmK=0Eq} establishes $F=\widetilde{F}$.
\end{proof}

We remark that
 even when  $\psi$ and $\phi$ are both Minkowski norms,  there may exist no projective flat Finsler metrics on $\Omega=\{\phi<1\}$ with $\mathbf{K}=0$.
  For a detailed discussion, see Corollary \ref{K=0Ex1} below.

In what follows, we turn to study the distance functions  {\it without} assuming completeness.
This approach yields a direct answer to Hilbert's fourth problem.
\begin{theorem} \label{K=0lemdx_0xdxx_0}Let $(\Omega,F)$ be a projectively flat Finsler manifold  with $\mathbf{K}=0$, where $\Omega\subset \mathbb{R}^n$ is a  strictly convex domain containing the origin.
Then  for any $x_1,x_2\in \Omega$,
\begin{equation}\label{K=0eqdx_0xdxx_0}
d_F(x_1,x_2) =  \frac{F(x_1, x_2-x_1)}{1-P(x_1, x_2-x_1)}.
\end{equation}
In particular, defining $\psi(y):=F(\mathbf{0},y)$ and $\phi(y):= P(\mathbf{0},y) $, this implies
\begin{equation}\label{coreqd0xdx0}
d_F(\mathbf{0},x) =  \frac{\psi(x)}{1-\phi(x)}, \qquad d_F(x, \mathbf{0})= \frac{\psi(-x)}{1+ \phi(-x)}.
\end{equation}
\end{theorem}
\begin{proof}
Equation \eqref{K=0eqdx_0xdxx_0} is immediate for $x_1 = x_2$.

Given a distinct point $x_2 \in \Omega \backslash \{x_1\}$, the convexity of $\Omega$ implies the existence of a unit-speed geodesic $\gamma(t)$ connecting $x_1$ to $x_2$. According to the proof of Theorem~\ref{thmK=0global}, $\gamma(t)$ admits one of two forms:
\begin{enumerate}[\rm (i)]
\item \label{K=0disi}
$\ds\gamma(t)   = t \frac{x_2-x_1}{F(x_1, x_2-x_1)}+x_1$;

\smallskip

\item \label{K=0disii}
$\ds\gamma(t) =  \Big( \frac{c  t}{c + t}\Big) \frac{x_2-x_1}{F(x_1, x_2-x_1)} +x_1$,  where  $c=c(x_1, x_2)>0$  is a constant dependent on $x_1,x_2$.
\end{enumerate}
The proof of \eqref{K=0eqdx_0xdxx_0} now splits into two cases corresponding to the forms of $\gamma(t)$.

{\bf Case 1.}  Suppose  $\gamma(t)$ is of the first form \eqref{K=0disi}. Then   $P(x,y)=0$  by \eqref{geodesicBackf''1}. The unit speed of $\gamma$ implies
\[
x_2=\gamma(d_F(x_1,x_2))= d_F(x_1,x_2) \frac{x_2-x_1}{F(x_1, x_2-x_1)}+x_1,
\]
which together with $P=0$ yields
\begin{align*}
d_F(x_1,x_2) = F(x_1, x_2-x_1)=\frac{F(x_1, x_2-x_1)}{1-P(x_1, x_2-x_1)}.
\end{align*}

{\bf Case 2.}   Suppose $\gamma(t)$ is of the second form \eqref{K=0disii}. Substituting $t=0$ into \eqref{star2} (noting $c_{x,y}=c$) gives
\begin{equation}\label{ciswhatK=0}
\frac1{c} = P \left(x_1, \frac{x_2-x_1}{F(x_1, x_2-x_1)}\right) = \frac{P(x_1,  x_2-x_1)}{F(x_1, x_2-x_1) }.
\end{equation}
On the other hand,  $\gamma(d_F(x_1,x_2))=x_2$ implies
\[
x_2 =\left( \frac{c  d_F(x_1,x_2)}{c + d_F(x_1,x_2)}\right) \frac{x_2-x_1}{F(x_1, x_2-x_1)} +x_1,
\]
which combined with \eqref{ciswhatK=0} again yields \eqref{K=0eqdx_0xdxx_0}.

It remains to prove \eqref{coreqd0xdx0}. The first formula \eqref{coreqd0xdx0}$_1$ follows directly from \eqref{K=0eqdx_0xdxx_0}.
To establish \eqref{coreqd0xdx0}$_2$, consider the reverse Finsler metric $\overleftarrow{F}(x,y):=F(x,-y)$. A straightforward verification shows that $(\Omega,\overleftarrow{F})$ is also a projectively flat Finsler manifold  with $\mathbf{K}=0$, and its projective factor satisfies $\overleftarrow{P}(x,y)=-P(x,-y)$.
Hence, by noting
 $\overleftarrow{F}(\mathbf{0},y)=\psi(-y)$ and $\overleftarrow{P}(\mathbf{0},y)= -\phi(-y)$,  the formula \eqref{coreqd0xdx0}$_1$ yields
 \begin{align*}
 d_F(x,\mathbf{0})=d_{\overleftarrow{F}}(\mathbf{0},x)= \frac{\psi(-x)}{1+ \phi(-x)},
 \end{align*}
 which completes the proof of \eqref{coreqd0xdx0}$_2$ and thus the theorem.
\end{proof}



The following result reveals a connection between the completeness  and the domain structure of a projectively flat Finsler manifold with $\mathbf{K}=0$.

\begin{corollary}\label{Phi0uniquess}
Let $(\Omega,F)$ be a projectively flat Finsler manifold with $\mathbf{K}=0$, where $\Omega\subset \mathbb{R}^n$ is a bounded domain containing the origin.  Then $(\Omega,F)$ is forward complete if and only if $\Omega=\{y \in \mathbb{R}^n : \phi(y) < 1\}$, where $\phi(y): = P(\mathbf{0},y)$.
\end{corollary}
\begin{proof} The forward implication ``$\Rightarrow$" is an immediate consequence of Theorem~\ref{thmK=0global}. For the converse ``$\Leftarrow$", since $\phi$ is positively $1$-homogeneous and $\Omega=\{\phi<1\}$ is bounded,  it follows that $\phi>0$ on $\Rno$. Consequently, $\phi(x)\rightarrow 1$ as the standard Euclidean distance $d_{\E}(x,\partial\Omega)\rightarrow 0$. Therefore, \eqref{coreqd0xdx0} implies
$d_F(\mathbf{0},x) \rightarrow +\infty$ as   $d_{\E}(x,\partial\Omega)\rightarrow 0$.
The forward completeness of $(\Omega, F)$ follows from Proposition~\ref{completenessofprojeflat}.
\end{proof}

We now prove Theorem~\ref{thmK=0globalintro}, which is the main characterization theorem for $\mathbf{K}=0$ stated in the Introduction.
\begin{proof}[Proof of Theorem~\ref{thmK=0globalintro}] The forward implication ``$\Rightarrow$'' follows directly from Theorems~\ref{thmK=0}, \ref{thmK=0global} and
Corollary~\ref{reverprojK=0}.
For the converse ``$\Leftarrow$'',  if $\mathbb{R}^n=\{\phi<1\}$ and $\phi\geq 0$, then the positive homogeneity of $\phi$ forces $\phi \equiv 0$.
Condition \eqref{K=0basciciont}, together with Theorem \ref{thmK=0} and Corollary \ref{Phi0uniquess}, then implies that $(\Omega,F)$ is a forward complete projectively flat Finsler manifold with $\mathbf{K}=0$.
The remaining conclusions are immediate consequences of Theorem~\ref{thmK=0global} and Corollary~\ref{reverprojK=0}.
\end{proof}

\begin{remark}\label{K=0condi1ne}
Therefore, Condition \eqref{K=0basciciont} in Theorem~\ref{thmK=0globalintro} is already a necessary and sufficient condition for a Finsler manifold  $(\Omega,F)$ to
be forward complete, projectively flat and of vanishing flag curvature $\mathbf{K} = 0$.
\end{remark}

To conclude this subsection, we construct an explicit example of a projectively flat non-Riemannian Finsler manifold with $\mathbf{K} = 0$, thus verifying the foregoing results.

\begin{example}\label{prlxeK=0nonR}
Let $a = (a_1, 0) \in \mathbb{R}^2$ with $a_1 \in [0, \frac{2\sqrt{2}}{3})$, and define  two Minkowski norms on $\mathbb{R}^n$ by
\[
\psi(y):= |y|+ \langle a, y\rangle=:\phi(y).
\]
Denote by    $P=P(x, y)$  the (unique)
Funk metric on $\Omega := \{ y \in \mathbb{R}^n \,:\,\phi(y) < 1\}$, i.e.,
\begin{equation*}
P=\phi (y+xP).
\end{equation*}
Using the construction in \eqref{thmK=0Eq}, define a function $F=F(x,y)$ on $\Omega\times \mathbb{R}^n$ by
\begin{align*}
F(x,y):=\psi(y+xP) \left\{1+P_{y^k}x^k\right\}= P \left\{1 +  P_{y^k}x^k\right\}.
\end{align*}
A direct calculation yields
\begin{align*}
P(x,y)
	=\frac{\sqrt{\mathcal {A}_a(x,y)} + \mathcal {C}_a(x,y)+ \langle x,y\rangle}{\mathcal {B}_a(x)},
\qquad
1 +  P_{y^k}x^k =   \frac{ \mathcal {C}_a(x,y) + |x|^2 \langle a,y \rangle}{\sqrt{\mathcal {A}_a(x,y)}\mathcal {B}_a(x)}+ \frac{1-\langle a, x \rangle}{\mathcal {B}_a(x)},
\end{align*}
where
\begin{align*}
\mathcal {A}_a(x,y)&:= \left( (1-\langle a, x \rangle)^2-|x|^2 \right)(|y |^2- \langle a, y \rangle^2) + \left(\langle x,y \rangle +(1-\langle a, x \rangle) \langle a, y \rangle \right)^2,\\
\mathcal {B}_a(x)&:=(1-\langle a,x  \rangle)^2-|x|^2,\qquad \mathcal {C}_a(x,y):= (1-\langle a,x  \rangle)\langle a, y\rangle.
 \end{align*}
A straightforward (though somewhat lengthy) verification confirms that $F$ is a projectively flat Finsler metric  on $\Omega$ with $\mathbf{K}=0$, and that $P$ is  the projective factor.
Furthermore,
\begin{align*}
d_F(\mathbf{0},x)
=   \frac{|x| +  \langle a, x \rangle}{1-|x| - \langle a, x \rangle}=\frac{\psi(x)}{1-\phi(x)},
\qquad d_F(x, \mathbf{0})
 =  \frac{|x| -  \langle a, x \rangle}{1+|x|- \langle a, x \rangle}=\frac{\psi(-x)}{1+ \phi(-x)}.
\end{align*}
From these, we see that $d_F(\mathbf{0},x) \rightarrow +\infty$ as $\phi(x) \rightarrow 1$, while $d_F(x, \mathbf{0})< 1$ for all $x\in \Omega$.
 Proposition~\ref{completenessofprojeflat} therefore  implies
that $(\Omega,F)$ is forward complete but not backward complete. Moreover, it follows from \eqref{distandreversi} that $\lambda_F(\Omega)=+\infty$.

In the special case $a = 0$, the metric reduces to Berwald's metric (cf.~\cite{Be1,Be2}):
\begin{equation*}
\mathsf{B}(x,y)= \frac{ (\sqrt{ (1-|x|^2) |y|^2 + \langle x, y \rangle^2} +\langle x, y \rangle )^2}{ (1-|x|^2)^2 \sqrt{ (1-|x|^2) |y|^2 + \langle x, y \rangle^2} }.
\end{equation*}
A qualitatively different behavior occurs when $a_1 \in [\frac{2\sqrt{2}}{3}, 1)$; see Proposition~\ref{K=0Ex1} in the next section.
\end{example}

\subsection{Global existence of metrics}\label{SecglobalexK=0}
In view of   Theorems \ref{thmK=0globalintro}, every forward complete projectively flat Finsler manifold $(\Omega,F)$ satisfying $\mathbf{K}=0$ is determined by a Minkowski norm $\psi$ together with either a zero function or a weak Minkowski norm $\phi$ such that $\Omega=\{\phi<1\}$.
This naturally leads to the following inverse problem:

\medskip
\textit{Given such a pair $(\psi, \phi)$, is the associated function $F$ always a Finsler metric on $\Omega = \{\phi < 1\}$?}
\medskip

\noindent
An affirmative answer would render superfluous the initial assumption that $(\Omega, F)$ is a Finsler manifold.
 While this assertion holds trivially when $\phi=0$ (yielding $F=\psi$), intriguing phenomena emerge when $\phi$ is a weak Minkowski norm.

We first introduce some notation used throughout this subsection.
\begin{notation}\label{basciassmup} Let $\psi$ be a Minkowski norm and $\phi$   a weak Minkowski norm on $\mathbb{R}^n$, and set
\[
\Omega:=\{x\in \mathbb{R}^n\,:\,\phi(x)<1 \}.
\]
Define a function $F:\Omega\times \mathbb{R}^n\rightarrow \mathbb{R}$ by
\begin{equation}\label{defiFK=0}
F(x,y) := \psi {\left( y + x P(x,y)\right)} \,\left\{ 1 +  x^k P_{y^k}(x,y) \right\},
\end{equation}
where  $P = P(x,y)$ is the unique solution (the weak Funk metric on $\Omega$) to the following equation
\begin{equation}\label{equationP}
P = \phi (y + x P),\quad \forall (x,y)\in \Omega\times \mathbb{R}^n.
\end{equation}
Define a set
\[
\PD := \PD(F):=\left\{x\in \Omega\,:\, \text{$F(x,\cdot)$ is a Minkowski norm at $x$}\right\}.
\]
\end{notation}

\begin{remark}
Theorem~\ref{lemsolexi} and Corollary~\ref{basicproper} imply that $F$ is a well-defined function on $\Omega\times \mathbb{R}^n$ and satisfies
\begin{equation}\label{postivitylemma}
0<F|_{\Omega\times \Rno}\in C^\infty\left(\Rno\right), \qquad F(x,\alpha y)=\alpha F(x,y) \text{ for any $\alpha>0$}.
\end{equation}
Thus, $F$ is a Finsler metric on $\PD$, and
\begin{equation}\label{Fposidome}
\PD=\left\{x\in \Omega\,:\, \left([F^2]_{y^iy^j}(x,y)\right) \text{ is positive definite for any }y\in \mathbb{S}^{n-1}\right\}.
\end{equation}
Moreover, if the matrix $([F^2]_{y^i y^j}(x, y))$ is positive definite, then by \cite[(1.2.9)]{BCS} we have
\begin{equation}\label{smedefF}
[F]_{y^iy^j}(x,y)\theta^i\theta^j\geq 0, \quad \forall \theta\in \mathbb{R}^n,
\end{equation}
with equality if and only if   $ \theta= \alpha y$ for some $\alpha \in  \mathbb{R}$.
\end{remark}


\begin{proposition}\label{Finmeanscomplete}
Under Notation \ref{basciassmup}, the pair $(\Omega,F)$ is a forward complete projectively flat Finsler manifold with $\mathbf{K}=0$ if and only if   $F$ is a Finsler metric on $\Omega$, i.e., $\PD=\Omega$.
\end{proposition}
\begin{proof}
This is an immediate consequence of Corollary~\ref{Phi0uniquess}.
\end{proof}

Proposition  \ref{Finmeanscomplete} establishes the equivalence between the positive-definiteness of the ``fundamental tensor''  and the forward completeness of the manifold.
This equivalence would seem to imply that the function
$F$ defined in \eqref{defiFK=0} must always be a Finsler metric. Surprisingly, this is not the case, as the following proposition demonstrates, even when both
  $\psi$ and $\phi$ are Minkowski norms.

\begin{proposition}\label{K=0Ex1}
 Let  $a=(a_1,0)\in \mathbb{R}^2$ with $a_1\in [0,1)$, and define two Minkowski norms on $\mathbb{R}^2$ by
 \[
 \psi(y):= |y|+ \langle a, y\rangle=:\phi(y).
 \]
  Then, under   Notation \ref{basciassmup},
 the domain $\PD$ exhibits three distinct configurations within $\Omega$ (see Figure~1):
\begin{enumerate}[ \ \ \ \ \rm (1)]
\item  \label{K=0RandersnormCase1} if $a_1\in [0, \frac{2\sqrt{2}}{3})$, then  $\PD= \Omega$;

\smallskip

\item  \label{K=0RandersnormCase2} if $ a_1 = \frac{2\sqrt{2}}{3}$, then $\PD=  \Omega \, \backslash  \big\{ x_1= - \frac{3\sqrt{2}}{2}   \big\}$;

\smallskip

\item  \label{K=0RandersnormCase34}
if   $a_1\in (\frac{2\sqrt{2}}{3}, 1)$, then
\begin{align*}
\PD = \Omega\,  \bigcap \, \left \{ (x_1, x_2) \, : \, 8(1- a_1^2) \left( x_1 + \frac{a_1}{4(1-a_1^2)} \right)^2 - (9 a_1^2-8) x_2^2 > \frac{9 a_1^2 -8}{2(1-a_1^2)} \right\}.
\end{align*}
\end{enumerate}


\begin{figure}[h]
\begin{tikzpicture}[scale=1]
   \filldraw[fill = gray!50,fill opacity=0.5, draw = black][dash pattern={on 4.5pt off 4.5pt}] (0,0) ellipse (2.25  and 1.5 );
\draw (0,-2) node [inner sep=0.6pt]    {$ (1) \  a_1 = \frac{\sqrt{5}}{3}$};

   \begin{scope} 
   \clip   (17.8*0.3,0) ellipse (9*0.3+ 0.5  and 3*0.3+0.5 );
   \filldraw[ fill = gray!50,fill opacity=0.5, draw = black][dash pattern={on 4.5pt off 4.5pt}] (17.8*0.3,0) ellipse (9*0.3  and 3*0.3 );
   \draw[ draw =white] plot ({(17.8 + 6.36)*0.3},{\x});
    \draw[fill opacity=0.5, draw = black][dash pattern={on 4.5pt off 4.5pt}] plot ({(17.8 + 6.36)*0.3},{\x});
    \end{scope}
\draw (17.8*0.3,-2) node [inner sep=0.6pt]    {$(2) \  a_1 = \frac{2 \sqrt{2}}{3}$};

    \begin{scope} 
   \clip   (11.5,0) ellipse (18*0.17  and 7.242*0.17 );
    \filldraw[fill = gray!50,fill opacity=0.5, domain = -1.2 : 1.2, draw=white] plot ( {11.5 + 17.49*0.157 + 0.17*(1.5*1.414*sec(\x r) - 2.915)/0.667 } ,{ 0.17* 3*tan(\x r)} );
    \filldraw[fill = gray!50,fill opacity=0.5, domain = -1.5 : 1.5, draw=white] plot ( {11.5 + 17.49*0.157 + 0.17*(-1.5*1.414*sec(\x r) - 2.915)/0.667 } ,{ -0.17* 3*tan(\x r)} );
    \end{scope}
     \draw[ draw = black][dash pattern={on 4.5pt off 4.5pt}] (11.5,0) ellipse (18*0.17  and 7.242*0.17 );
     \draw[ domain = -1.1 : 1.1][dash pattern={on 4.5pt off 4.5pt}] plot ( {11.5 + 17.49*0.157 + 0.17*(1.5*1.414*sec(\x r) - 2.915)/0.667 } ,{ 0.17* 3*tan(\x r)} );
     \draw[ domain = -1.2 : 1.2][dash pattern={on 4.5pt off 4.5pt}] plot ( {11.5 + 17.49*0.157 + 0.17*(-1.5*1.414*sec(\x r) - 2.915)/0.667 } ,{ -0.17* 3*tan(\x r)} );
\draw (11.5,-2) node [inner sep=0.6pt]    {$(3) \  a_1 = \frac{ \sqrt{34}}{6}$};

\end{tikzpicture}
\caption{ \small Evolution of $\PD$ (shaded) within $\Omega$.}
\label{K=0fig}
\end{figure}
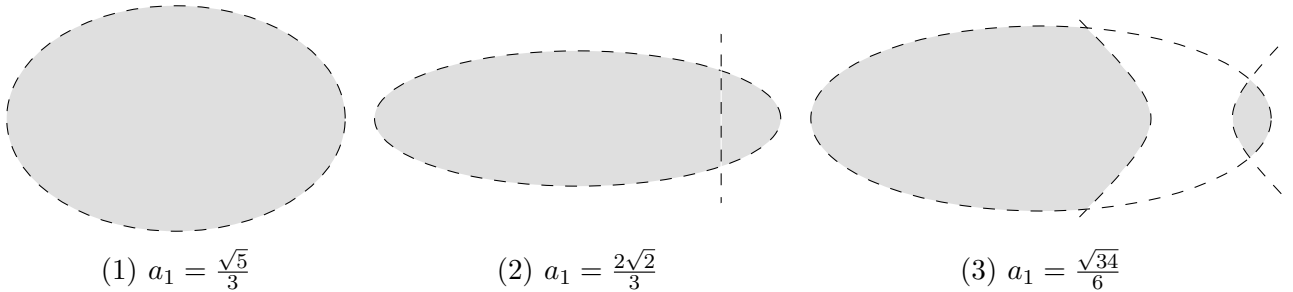
\end{proposition}

To the best of our knowledge, this striking phenomenon has not been previously documented in literature on Finsler geometry.
While a direct verification of Proposition~\ref{K=0Ex1} is possible, it is computationally cumbersome. We therefore provide a more insightful proof in Appendix~\ref{propergenerlengappex0}.

This example motivates a fundamental question:

\medskip
\begin{minipage}{0.9\textwidth}
\itshape Under what conditions does the function $F$ associated with a pair $(\psi, \phi)$ fail to be a Finsler metric on the entire $\Omega$?
\end{minipage}
\medskip

 \noindent To answer this, we prove the following key lemma in Appendix~\ref{propergenerlengappex0}.
\begin{lemma}\label{keylemmak=0}
Under Notation~\ref{basciassmup}, the following hold:
\begin{enumerate}[\rm (a)]
\item \label{keylemmak=0a}
for any $y\in \Rno$ and any $\mu\in \mathbb{R}$ such that $\mu y\in \Omega$, the matrix $(  [F^2]_{y^i y^j}(\mu y,y) )$ is positive definite;

\item \label{keylemmak=0b}
 for a fixed point $\bar{x}\in \Omega$ and for all $(x,y)\in \Omega\times \mathbb{R}^n$, we have
\begin{equation}\label{newFK=0ex}
F(x,y)=\bar{\psi}(y + (x-\bar{x}) \bar{P}(x-\bar{x},y)) \Big\{ 1 + (x^k-\bar{x}^k) \bar{P}_{y^k}(x-\bar{x},y)  \Big\},
\end{equation}
where $\bar{\psi}(y) := F(\bar{x},y)$,
and $\bar{\phi},\bar{P}$ are the unique solutions to the following equations
\begin{equation}\label{twoequations}
\bar{\phi}(y) = \phi(y+ \bar{x} \bar{\phi}(y)), \quad \bar{P}(x, y) = \bar{\phi}(y + x \bar{P}(x,y)), \quad \forall y\in \mathbb{R}^n,\ x\in \Omega-\bar{x}.
\end{equation}
\end{enumerate}
\end{lemma}

\begin{remark}
Suppose in Notation \ref{basciassmup} that $\psi$ is only a positively $1$-homogeneous function satisfying $0<\psi|_{\Rno}\in C^\infty(\Rno)$. Then, by the same argument as in part (a) of Lemma~\ref{keylemmak=0} and the fact $F(\mathbf{0},\cdot)=\psi$, one can show that for any $y\in \Rno$,
\begin{align*}
\left(  [\psi^2]_{y^i y^j}(y) \right) \text{ is positive definite} \ \ & \Longleftrightarrow \ \
\left(  [F^2]_{y^i y^j}(\mu y,y) \right) \text{ is positive definite for all $\mu\in \mathbb{R}$ with $\mu y\in \Omega$}.
 \end{align*}
\end{remark}

The next lemma follows directly from
 Rademacher's Theorem (cf. Evans \cite[Theorem 6]{Eva}) and Grisvard \cite[Corollary 1.2.2.3]{Gri}.
\begin{lemma}\label{lipschismooofboundry}
Let $U$ be a bounded convex open set in $\mathbb{R}^n$. Then the collection of nonsmooth points of the boundary $\partial U$ is a $\mathcal{H}^{n-1}$-negligible set, where $\mathcal{H}^{n-1}$ denotes the $(n-1)$-Hausdorff measure. In particular, if $n=2$, then any nonempty connected part of $\partial U$ has positive $1$-Hausdorff measure.
\end{lemma}

The following result describes the behavior of  $F$  on a convex component of $\mathscr{O}$.

\begin{proposition}\label{K=0conneccomconvex}
Under Notation \ref{basciassmup}, let $U$ be a convex connected component of $\PD$. Then the following hold:
\begin{enumerate}[\rm (a)]

\item\label{conneccompconvex3} for any smooth point  $x\in \partial U\cap \Omega$ and any vector $y$  not tangent to $\partial U$ at $x$, the matrix $([F^2]_{y^iy^j}(x,y))$ is positive definite;

\item\label{conneccompconvex2} for  any smooth point  $x\in \partial U\cap \Omega$, there exists a nonzero vector $y$ tangent to $\partial U$ at $x$ such that $([F^2]_{y^iy^j}(x,y))$ has a zero eigenvalue. Moreover, the corresponding eigenvector $\theta = (\theta^i)$ is not parallel to $y$ and satisfies
\begin{align} \label{lemK=0lem2}
F_{y^i}(x,y)  \theta^i =0,  \qquad  F_{y^i y^j}(x,y)   \theta^i \theta^j  =0.
 \end{align}
If, in addition, the vector $x$ is not tangent to $\partial U$ at the point $x$, then for all sufficiently small  $|t|\neq0$,
\begin{equation}\label{lemK=0lem3}
F_{y^i y^j}({x},y+ t {x})  \theta^i \theta^j>0.
\end{equation}
\end{enumerate}
\end{proposition}
\begin{proof}
Let $x_o$ be a smooth point of $\partial U\cap \Omega$. Since $\partial U\subset \partial {\PD}$, it follows from \eqref{Fposidome} that
there exists $y_o\in T_{x_o}\mathbb{R}^n\backslash\{\mathbf{0}\}$ such that the  positive semi-definite  matrix $([F^2]_{y^iy^j}(x_o,y_o))$ has a zero eigenvalue.

We now show that $y_o$ is tangent to $\partial U$ at $x_o$.
If not, the convexity of  ${U}$ yields a point $\bar{x}\in {U}$ such that
 \begin{equation}\label{newxbary_o}
  x_o-\bar{x}=\mu y_o, \quad  \text{for some } \mu\neq0.
  \end{equation}
Define $\bar{\psi}(y) := F(\bar{x}, y)$ and $\bar{\phi}(y) := P(\bar{x}, y)$. Then $\bar{\psi}$ is a Minkowski norm while $\bar{\phi}$ is a weak Minkowski norm
(see Remark~\ref{weakFuniswekmonk}). Moreover, it follows from \eqref{equationP} and Theorem \ref{lemsolexi} that $\bar{\phi}$ is the unique solution to the equation
$\bar{\phi}(y) = \phi(y+ \bar{x} \bar{\phi}(y))$. Now define
\begin{equation*}
\bar{F}(x,y) := \bar{\psi}\left(y + x \bar{P}(x,y)\right) \Big\{ 1+   x^k\bar{P}_{y^k}(x,y) \Big\},\quad x\in \{\bar{\phi}<1\},
\end{equation*}
where $\bar{P}=\bar{P}(x,y)$ is the unique solution to $\bar{P} = \bar{\phi}(y + x \bar{P})$. Theorem \ref{lemsolexi} yields $\{\bar{\phi}<1\}=\Omega-\bar{x}$.
Applying Lemma \ref{keylemmak=0}/\eqref{keylemmak=0a} and \eqref{newxbary_o} to $\bar{F}$, the matrix
\[
\left( [ \bar{F}^2 ]_{y^i y^j}  (x_o-\bar{x},y_o)\right) =\left( [ \bar{F}^2 ]_{y^i y^j} (\mu y_o,y_o) \right)
\]
is positive definite. On the other hand, Lemma \ref{keylemmak=0}/\eqref{keylemmak=0b} implies that  $F(x_o,y) = \bar{F}(x_o-\bar{x},y)$, and hence 
\[
\left( [  {F}^2 ]_{y^i y^j} (x_o,y_o) \right) =\left( [ \bar{F}^2 ]_{y^i y^j} (x_o-\bar{x},y_o) \right)
\]
is positive definite. This contradicts the positive
semi-definiteness of $\left( [  {F}^2 ]_{y^i y^j}  (x_o,y_o) \right)$. Then $y_o$ is tangent to $\PD$ at $x_o$, which therefore establishes \eqref{conneccompconvex3}.

It remains to prove \eqref{lemK=0lem2} and \eqref{lemK=0lem3}. Since $([F^2]_{y^iy^j}(x,y))$ is positive definite for all $(x,y)\in TU$,
its continuity and \eqref{smedefF} imply that
 \begin{align} \label{K=0globlepsi_o_ijnew1lem}
 \left( [F]_{y^i  y^j}(x_o,y_o) \right) \text{ is positive semi-definite}.
\end{align}
Let $\theta = (\theta^i)$ be a nonzero vector such that
\begin{equation}\label{brevepsideg}
[ F^2 ]_{y^i y^j}(x_o,y_o) \theta^i \theta^j =0.
\end{equation}
Then $\theta$ cannot be parallel to ${y_o}$ due to the homogeneity and positivity of $F^2$ (see \eqref{postivitylemma}). Therefore, \eqref{brevepsideg}
 and \eqref{K=0globlepsi_o_ijnew1lem} together yield
\begin{align} \label{k=0globlepsi_o_ithetanew1lem}
F_{y^i}(x_o,{y_o}) \theta^i =0, \qquad F_{y^i y^j}(x_o,{y_o}) \theta^i \theta^j =0.
\end{align}
Furthermore, if ${x_o}$ is not tangent to $\partial U$ at $x_o$,
then ${y_o}+ t {x_o}$ is also not tangent to $\partial U$ at $x_o$ for all sufficient small $|t|>0$. In this case,
Proposition \ref{K=0conneccomconvex}/\eqref{conneccompconvex3} combined with \eqref{smedefF} gives $F_{y^iy^j}(x_o,y_o+tx_o)\theta^i\theta^j>0$.
\end{proof}

In the two-dimensional case, the following theorem provides a complete characterization of the geometric properties of $\PD$, thereby furnishing an alternative proof for
Proposition~\ref{K=0Ex1} (see Appendix~\ref{propergenerlengappex0}).
\begin{theorem}\label{n=2k=0deger}
Under Notation \ref{basciassmup} with the dimension $n=2$, the following hold:
\begin{enumerate}[\rm (i)]
\item\label{K=0Dim2conneccompconvex1} every connected component of $\PD$  is  a convex domain in $\mathbb{R}^2$;

\item\label{K=0Dim2conneccompconvex3} for $\mathcal{H}^{1}$-a.e.  $x\in \partial {\PD}\cap \Omega$, if the matrix $([F^2]_{y^iy^j}(x,y))$ is degenerated for some $y\in \mathbb{R}^2_\circ$,   then $y$ is tangent to $\partial{\PD}$ at $x$. In particular, such a pair $(x,y)$ exists whenever $\partial {\PD}\cap \Omega$ is nonempty.
\end{enumerate}
\end{theorem}
\begin{proof} We prove only part \eqref{K=0Dim2conneccompconvex1}, as part \eqref{K=0Dim2conneccompconvex3} follows directly from  Lemma~\ref{lipschismooofboundry} and
Proposition~\ref{K=0conneccomconvex}.

Suppose, for contradiction, that some connected  component
 ${U}\subset \PD $ is not convex.
Then there exist   $x_1,x_2\in {U}$ such that
the line segment $l_{x_1 x_2}$ connecting them is not entirely contained in  ${U}$. Consequently, there is a point $x_o \in l_{x_1 x_2}\cap \partial{U}$ with
$x_o \notin  \PD$.

For any nonzero vector $y_o \in T_{x_o}\mathbb{R}^2$, the non-convexity of $U$ implies the existence of a point $\bar{x} \in U$ and a scalar $\mu \neq 0$ such that $x_o - \bar{x} = \mu y_o$.
Define $\bar{\psi}(y) := F(\bar{x}, y)$ and $\bar{\phi}(y) := P(\bar{x}, y)$. The same argument as in the proof of Proposition~\ref{K=0conneccomconvex} then shows  the positive definiteness of  $( [  {F}^2 ]_{y^i y^j} (x_o,y_o) ) $.
Thus, the arbitrariness of $y_o$ yields
 $x_o\in \PD$,  which leads to a contradiction. Therefore, every connected component of $\PD$  is  convex.
\end{proof}



\vskip 10mm
\section{Projectively flat Finsler manifolds with $\mathbf{K} = -1$}\label{sectionK=-1}

This section studies projectively flat Finsler manifolds with $\mathbf{K} = -1$. The first subsection establishes a global characterization of these manifolds, culminating in the proofs of Theorems~\ref{uniquenssFinsler} and \ref{thmK=-1globalintro}. The second subsection addresses the global existence of such metrics on a given domain.

\subsection{Global characterization of manifolds} \label{subsK=-1_1}
The starting point for our global analysis is the following local structure theorem. The sufficiency and necessity are due to Shen \cite[Theorem 1.2]{Sh1} and Li \cite[Theorem 1.2]{Li}, respectively, while the reversible case was treated in Cheng--Li \cite[Theorems 1.1 and 3.4]{ChengLi}.
\begin{theorem}[\cite{ChengLi,Li,Sh1}]\label{localstrucK}
Let $F=F(x,y)$ be a Finsler metric on a neighborhood $\mathcal {U}$ of the origin $\mathbf{0}\in \mathbb{R}^n$.  Thus,
$F$  is    projectively flat  with $\mathbf{K}=-1$
 if and only  if there exist  a  Minkowski norm $\psi=\psi(y)$  and a
positively $1$-homogeneous function $\phi =\phi(y)$  on $\mathbb{R}^n$, with $\phi|_{\Rno}\in C^\infty(\Rno)$,
 such that
\begin{equation} \label{thmK=-1eq}
F  = \frac{1}{2}\big\{ \Phi_{+} - \Phi_{-} \big\}, \qquad P  = \frac{1}{2}\big\{ \Phi_{+} + \Phi_{-} \big\},
\end{equation}
where $P=P(x,y)$ is the projective factor and $\Phi_\pm=\Phi_\pm(x,y)$ are the solutions to
\begin{equation} \label{Phiplusminus}
\Phi_{\pm} = (\phi \pm \psi)(y + x \Phi_{\pm}).
\end{equation}
In particular,  $F$ and $P$ are uniquely determined near the origin by their initial data:
\[
\psi(y):=F(\mathbf{0},y),\qquad  \phi(y):= P(\mathbf{0},y).
\]
Moreover,  $F$ is reversible if and only if $\psi$ is reversible and $\phi$ is odd, i.e.,
\[
 \psi(\alpha y)=|\alpha|\psi(y), \qquad \phi(\alpha y)=\alpha \phi(y), \quad \forall \alpha\in \mathbb{R},
\]  in which case $F$ is a Hilbert metric.
\end{theorem}

\begin{remark}\label{diffundeK=-1}
By Lemma \ref{Sh1Lemma5.1}, the solutions $\Phi_\pm$ to equation \eqref{Phiplusminus} are unique in a sufficiently small ball $\mathbb{B}^n_\delta(\mathbf{0})\subset \mathcal {U}$, and hence so are $F$ and $P$.
However, it remains an open question whether this uniqueness holds on the entire $\mathcal{U}$.
\end{remark}

The following result is also needed in the sequel.

\begin{lemma}\label{lempropFunk}
Let $(\Omega,F)$   be a  forward complete projectively flat  Finsler manifold with $\mathbf{K} =-1$, where $\Omega\subset \mathbb{R}^n$ is a domain. If  $P = \mu F$ for some constant $\mu>0$,  then
$(\Omega,\frac{1 + \mu^2}{\mu} F)$ is a Funk metric space. In particularly, $(\Omega,F)$ is not backward complete.
\end{lemma}
\begin{proof}
Since $\mathbf{K}=-1$, equation \eqref{Berwald2}$_2$ gives
\begin{equation*}
P_{x^k} = P P_{y^k} + F F_{y^k}. 
\end{equation*}
Combining this with the assumption $P = \mu F$ yields $\mu F_{x^k} = (1 + \mu^2) F F_{y^k}$, which is equivalent to $\Theta_{x^k} = \Theta \Theta_{y^k}$, where $\Theta := \frac{1 + \mu^2}{\mu} F$.
Therefore,
$\Theta$ is a Funk metric by \eqref{Funk}.
Finally, since $F$ and $\Theta$ differ only by a constant scaling factor, they share the same geodesics and hence have equivalent completeness properties.
\end{proof}

\begin{theorem}\label{globalK=-1}
Let $(\Omega,F)$ be an $n$-dimensional forward complete projectively flat Finsler manifold with $\mathbf{K} =-1$, where $\Omega\subset \mathbb{R}^n$ is a domain  containing the origin.  Define
\begin{equation}\label{defphsiK=-1}
\psi(y):=F(\mathbf{0},y), \qquad \phi(y):= P(\mathbf{0},y).
\end{equation}
Then $\phi + \psi$ is a weak Minkowski norm, and $\Omega$ is a bounded strictly convex domain given by
\begin{equation}\label{K=-1region}
 \Omega = \{ x \in \mathbb{R}^n \,:\, \phi(x) + \psi(x) < 1 \}.
\end{equation}
Furthermore, the following hold:
\begin{enumerate}[{\rm (i)}]
\item  \label{globalK=-1(i)} if $(\Omega,F)$ is also backward complete, then $F$ is a Hilbert metric and for all $(x,y)\in \Omega\times \mathbb{R}^n$,
\begin{equation}\label{PFK=-1complete}
-F(x,y)\leq P(x,y) \leq F(x,y) \text{ with equality if and only if $y=\mathbf{0}$};
\end{equation}

\item \label{globalK=-1(ii)}
if $(\Omega, F)$ is not backward complete, then exactly one of the following cases holds for all $(x,y)\in \Omega\times \mathbb{R}^n$:

\begin{enumerate}[{\rm (1)}]
\item  \label{PgreatFK=-1}
$F(x,y)\leq  P(x,y)$, with equality if and only if $y=\mathbf{0}$;

\item \label{K=-1(2)} $\ds
F(x,y)= P(x,y)$ and particularly,
 $(\Omega,2F)$ is  a Funk metric space;

\item \label{PlessFK=-1}
$\ds -F(x,y)\leq P(x,y)\leq   F(x,y)$, with equality if and only if $y=\mathbf{0}$.
\end{enumerate}
\end{enumerate}
\end{theorem}
\begin{proof} Let $\gamma(t)$ be a unit-speed geodesic from $x\in \Omega$ with initial velocity $y\in S_{x}\Omega$.
By Theorem \ref{lemgeodesicK} and  $\mathbf{K}=-1$, we have $\gamma(t)=f(t){y}+{x}$, where $f(t)$ satisfies the initial condition \eqref{geodesicintial} and the ODE
\begin{equation} \label{OdeK=-1}
2 \left( \frac{f''(t)}{f'(t)}  \right)' - \left( \frac{f''(t)}{f'(t)}  \right)^2 = -4.
\end{equation}
The solution to equation \eqref{OdeK=-1} with \eqref{geodesicintial} is given by one of the following forms:
\begin{align}\label{fspecialK=-1}
f(t) = \frac{e^{2 t} -1}2, \quad
f(t) =  \frac{1-e^{-2 t} }2, \quad
f(t) =  \frac{(c_{x,y}-1) ( e^{2 t}-1)}{2(c_{x,y} e^{2 t} -1) },
\end{align}
where $c_{x,y}$ is a constant depending continuously on ${x}$ and ${y}$, with either $c_{x,y}<0$ or $c_{x,y}>1$, as required by forward completeness.
Since $f(t) = f(t; x, y)$ depends continuously on $x$ and $y$, these three types of geodesics cannot coexist on $(\Omega, F)$. Consequently,
 all geodesics must be of a single type from \eqref{fspecialK=-1}.

We now rule out the first form, \eqref{fspecialK=-1}$_1$. Note that
\[
\lim_{t\rightarrow +\infty}  \frac{e^{2 t} -1}2 = + \infty \quad \text{while} \quad  \inf_{t\in (-\infty,0)}\frac{e^{2 t} -1}2=\lim_{t\rightarrow-\infty}  \frac{e^{2 t} -1}2 = -\frac12,
\]
which contradicts Lemma~\ref{lemfortocom}. Therefore, $f(t)$ must be given by either \eqref{fspecialK=-1}$_2$ or \eqref{fspecialK=-1}$_3$.
We analyze these two cases separately:

\medskip\noindent
{\bf Case (a).} Suppose that the geodesic $\gamma(t)$ is given by
\[
\gamma(t)=f(t)y+x=\frac{1-e^{-2 t} }2{y}+{x},\quad t\in [0,+\infty).
\]
Then \eqref{geodesicBackf''1} implies
$
P(\gamma(t), \gamma'(t)) =  1 = F(\gamma(t), \gamma'(t)).
$
Evaluating at $t = 0$ gives
\begin{equation}\label{f=pcond}
P({x},{y})=F({x},{y}).
\end{equation}
By Lemma~\ref{lempropFunk}, it follows that $(\Omega, 2F)$ is a Funk metric space and $(\Omega, F)$ is not backward complete.
Theorem~\ref{funkmetrisufficnecessy} then implies that $\Omega$ is a bounded strictly convex domain of the form
\[
\Omega=\{ \xi\in \mathbb{R}^n\,:\, 2F(\mathbf{0},\xi)<1 \}=\{ \xi\in \mathbb{R}^n\,:\, \psi(\xi)+\phi(\xi)<1 \}.
\]
This establishes part \eqref{globalK=-1(ii)}/(2).

\medskip\noindent
 {\bf Case (b).} Now suppose that the geodesic has the form
\begin{equation}\label{geodesiccase2k=-1}
\gamma(t)=f(t)y+x=\frac{(c_{x,y}-1) ( e^{2 t}-1)}{2(c_{x,y} e^{2 t} -1) }{y}+{x},\quad t\in [0, +\infty),
\end{equation}
with $c_{x,y}<0$ or $c_{x,y}>1$.
From \eqref{geodesicBackf''1} and \eqref{geodesiccase2k=-1}, we have
\begin{equation}\label{P0K=-1}
P({x},{y})=P(\gamma(0), \gamma'(0))= \left.\frac{c_{x,y}e^{2t}+1}{c_{x,y}e^{2t}-1}\right|_{t=0}= \frac{c_{x,y}+1}{c_{x,y}-1}.
\end{equation}
Combining this with \eqref{defphsiK=-1} and choosing $x=\mathbf{0}$ yields
\begin{equation}\label{bsiscK=-1psishi}
\phi(y)=P(\mathbf{0},{y})= \frac{c_{\mathbf{0},y}+1}{c_{\mathbf{0},y}-1}, \qquad \psi(y)=F(\mathbf{0},{y})=1.
\end{equation}
Since $\frac{c_{\mathbf{0},y}-1}{2c_{\mathbf{0},y}}>0$ for $c_{\mathbf{0},y}<0$ or $c_{\mathbf{0},y}>1$, equation \eqref{bsiscK=-1psishi} and homogeneity imply
\begin{equation}\label{boundardyphsi}
\lim_{t\rightarrow +\infty}(\phi+ \psi)(\gamma(t))=(\phi+ \psi)\left( \frac{c_{\mathbf{0},y}-1}{2c_{\mathbf{0},y}}{y} \right)
= P \left(\mathbf{0},\frac{c_{\mathbf{0},y}-1}{2c_{\mathbf{0},y}}{y}\right)+F\left(\mathbf{0},\frac{c_{\mathbf{0},y}-1}{2c_{\mathbf{0},y}}{y}\right)=1.
\end{equation}
The projective flatness indicates $\Omega=\{\phi+\psi<1 \}$.
Lemma~\ref{Regionstrconvex} then shows that $\phi + \psi$ is a weak Minkowski norm and $\Omega$ is a bounded strictly convex domain.

Since $c_{x,y}$ depends continuously on $x$ and $y$, the cases $c_{x,y} < 0$ and $c_{x,y} > 1$ cannot occur simultaneously. We therefore consider two subcases.

\medskip\noindent
{\bf Subcase (b1).} Assume $ c_{x,y}>1$ for all ${x}\in \Omega$ and ${y}\in S_{{x}}\Omega$.  Then \eqref{P0K=-1} gives
\begin{equation}\label{xoaopfvalue}
P({x},{y})=\frac{c_{x,y}+1}{c_{x,y}-1} > 1 = F({x}, {y}).
\end{equation}
By homogeneity, this holds on $\Omega\times \Rno$.
Moreover, \eqref{geodesiccase2k=-1} shows that
$\gamma(t)$ can not be extended to $(-\infty,1]$, so $(\Omega,F)$ is not backward complete. This proves part \eqref{globalK=-1(ii)}/(1).



\medskip\noindent
{\bf Subcase (b2).} Assume $c_{x,y}<0$ for all   ${x}\in \Omega$ and ${y}\in S_{{x}}\Omega$.
Given that $F(x, y) = 1$ by the unit-speed condition, \eqref{P0K=-1} yields
 \[
-F({x},{y})< P({x},{y})<F({x},{y}).
 \]
By homogeneity, the above inequality holds  on $ \Omega\times \Rno$, proving part \eqref{globalK=-1(ii)}/(3).

Finally, to prove part \eqref{globalK=-1(i)}, assume $(\Omega, F)$ is also backward complete. It remains to show that $F$ is a Hilbert metric.
The   completeness together with \eqref{geodesiccase2k=-1} implies
\begin{equation}\label{-inftboudnk=-1}
\frac{c_{x,y}-1}{2c_{x,y}}{y}+{x}=\lim_{t \rightarrow +\infty}\gamma(t) \in \partial\Omega,\qquad \frac{c_{x,y}-1}{2}{y}+{x}
=\lim_{t \rightarrow -\infty}\gamma(t) \in \partial\Omega.
\end{equation}

Now consider the reversed initial velocity $\tilde{y} := -y / F(x, -y) \in S_x \Omega$. By \eqref{geodesiccase2k=-1}, the geodesic $\tilde{\gamma}$ with this initial data
is given by
\[
\tilde{\gamma}(t)=\frac{(c_{x,\tilde{y}}-1) ( e^{2 t}-1)}{2(c_{x,\tilde{y}} e^{2 t} -1) }\tilde{{y}}+{x},
\]
with parameter $c_{x,\tilde{y}} < 0$, since a positive value is incompatible with completeness.

Since the images of $\tilde{\gamma}$ and $\gamma$ coincide,  equation  \eqref{-inftboudnk=-1} implies
\begin{align*}
\frac{c_{x,y}-1}{2c_{x,y}}{y}+{x}&=\gamma(+\infty) =\tilde{\gamma}(-\infty)=\frac{c_{x,\tilde{y}}-1}{2}\tilde{{y}}+{x}=-\frac{c_{x,\tilde{y}}-1}{2F({x},-{y})} {{y}}+{x},\\
\frac{c_{x,y}-1}{2}{y}+{x}&=\gamma(-\infty)=\tilde{\gamma}(+\infty)=\frac{c_{x,\tilde{y}}-1}{2c_{x,\tilde{y}}}\tilde{y}+x=-\frac{c_{x,\tilde{y}}-1}{2 c_{x,\tilde{y}}F(x,-y)}y+x,
\end{align*}
which furnish
\[
c_{x,\tilde{y}}=c_{x,y}^{-1},\quad F({x},-{y})=1=F({x},{y}),\quad \forall {y}\in S_{x}\Omega.
\]
Hence, $F$ is reversible  and therefore a Hilbert metric by Theorem \ref{localstrucK}.
\end{proof}

The following result demonstrates that each case in Theorem~\ref{globalK=-1} is non-vacuous.
\begin{proposition}\label{ExampleK=-1amazing}
Define two positively $1$-homogeneous functions on $\mathbb{R}^n$ by
 $\psi(y) := |y| $  and $\phi(y):= c |y|$, where $c>0$ is a constant.
 For $(x, y) \in \Omega \times \mathbb{R}^n$, let $\Phi_{\pm}(x, y)$ be the unique solutions to the equations
\[ \Phi_{+}(x,y) = (\phi+\psi)(y+ x \Phi_{+}(x,y) ), \qquad \Phi_{-}(x,y) = (\phi-\psi)(y+ x \Phi_{-}(x,y) ), \]
which are explicitly given by
\begin{align*}
\Phi_{+}(x,y) & =  \frac{(c+1)\sqrt{\left( 1- (c+1)^2 |x|^2 \right) |y|^2 + (c+1)^2\langle x,y \rangle^2}  + (c+1)^2 \langle x,y \rangle }{1 - (c+1)^2 |x|^2}, \\
\Phi_{-}(x,y) &= \frac{(c-1) \sqrt{ \left( 1- (c-1)^2 |x|^2 \right) |y|^2 + (c-1)^2\langle x,y \rangle^2}  + (c-1)^2 \langle x,y \rangle }{1 - (c-1)^2 |x|^2}.
\end{align*}
Set
\begin{equation} \label{suffK=-1new1Randersnorm}
F(x,y): = \frac{1}{2} \big\{ \Phi_{+}(x,y) - \Phi_{-}(x,y) \big\},\qquad
 \Omega:= \big\{ x\in \mathbb{R}^n   \ :  \ \phi(x)+\psi(x)  < 1 \big\}.
\end{equation}
Then $(\Omega,F)$ is a forward complete projectively flat Finsler manifold with $\mathbf{K}= - 1$. In particular,
\begin{enumerate}[\rm (a)]
\item\label{c=0} if $c=0$, then $F$ is a Hilbert metric and $-F(x,y)\leq P(x,y)\leq   F(x,y)$, with equality if and only if $y=\mathbf{0}$;
\item\label{c=1} if $c=1$, then $2F$ is a Funk metric with $F(x,y)\equiv P(x,y)$;
\item\label{c<1} if $c\in (1,+\infty)$, then  $F(x,y)\leq P(x,y)$, with equality if and only if $y=\mathbf{0}$;
\item\label{c>1} if $c\in (0,1)$, then  $-F(x,y)\leq P(x,y)\leq   F(x,y)$, with equality if and only if $y=\mathbf{0}$.
\end{enumerate}
\end{proposition}
\begin{remark}
In Case \eqref{c>1} of Proposition~\ref{ExampleK=-1amazing}, the metric $F$ is not reversible and hence cannot be a Hilbert metric. Therefore, the inequality $-F\leq P\leq   F$ can be realized in both Hilbert and non-Hilbert settings.
\end{remark}
Although Proposition~\ref{ExampleK=-1amazing} can be verified by direct computation, such an approach is extremely challenging. Instead, we will provide a simpler proof using Proposition~\ref{thmK=-1convex} below.

Based on Theorem \ref{globalK=-1}, the following result characterizes the Hilbert metric from various perspectives.
\begin{corollary}\label{corK=-1S}
Let $(\Omega,F)$ and $\psi,\phi$ be as in Theorem \ref{globalK=-1}. Thus the following statements are equivalent:
\begin{enumerate}[\rm (i)]
\item\label{backwcomopK=-1} $(\Omega,F)$ is backward complete;

\item\label{reverFK=-1} $F$ is reversible (i.e., a Hilbert metric), i.e., $\lambda_F(\Omega)=1$;

\item\label{reversK=-1co} $\ds\Omega=\{ x \in \mathbb{R}^n \,:\, -1< \phi(-x) - \psi(-x)\leq 0 \}$.

\end{enumerate}
\end{corollary}
\begin{proof}
Theorem~\ref{globalK=-1} immediately gives the equivalence \eqref{backwcomopK=-1} $\Leftrightarrow$ \eqref{reverFK=-1}.
It remains to prove \eqref{reverFK=-1} $\Leftrightarrow$ \eqref{reversK=-1co}. For notational convenience, set
\[
\Omega_-:=\big\{ x \in \mathbb{R}^n \,:\, -1< \phi(-x) - \psi(-x)\leq 0 \big\}.
\]

\noindent
\textbf{\eqref{reverFK=-1}$\Rightarrow$\eqref{reversK=-1co}:}
If $F$ is reversible (i.e., a Hilbert metric), then by Theorem~\ref{localstrucK}, $\psi$ is reversible (even) and $\phi$ is odd. Consequently, $\phi(-x) - \psi(-x) = -(\phi(x) + \psi(x))$, and thus $\Omega_- = \Omega$ by \eqref{K=-1region}.

\noindent
\textbf{\eqref{reverFK=-1}$\Leftarrow$\eqref{reversK=-1co}:}
Assume $\Omega = \Omega_-$. Then for every $x \in \partial\Omega$, we have
\begin{equation}\label{K=-1Fstarphipsi-x}
 \phi(-x) - \psi(-x) = -1, \qquad \phi(x) + \psi(x) =1.
\end{equation}
By Theorem~\ref{globalK=-1} and \eqref{defphsiK=-1}, the pair $(\psi, \phi)$ must satisfy one of the following:
\begin{enumerate}[\rm \ \ \ \ (a)]
\item\label{PFK=-11} $\phi(y)\geq \psi(y)$ for all $y$, with equality only at $y = \mathbf{0}$;

\item\label{PFK=-12} $\phi(y)=\psi(y)$ for all $y\in \mathbb{R}^n$;

\item\label{PFK=-13} $-\psi(y)\leq \phi(y)\leq \psi(y)$ for all $y$, with equality only at $y = \mathbf{0}$.
\end{enumerate}
Case \eqref{PFK=-11} implies $\Omega_-=\{\mathbf{0}\}\subsetneq \Omega$, while Case \eqref{PFK=-12} gives $\Omega_-=\mathbb{R}^n\supsetneq \Omega$.
Both contradict $\Omega = \Omega_-$.
Therefore, only Case \eqref{PFK=-13} is  possible, which corresponds to Case  (3) in Theorem~\ref{globalK=-1}/\eqref{globalK=-1(ii)}.

Now, let $\gamma(t)$ be the unit-speed geodesic from the origin $\mathbf{0}$ with initial velocity $y \in S_{\mathbf{0}}\Omega$. By the proof of Theorem~\ref{globalK=-1}, it has the form
\[
\gamma(t)=\frac{(c_y-1) ( e^{2 t}-1)}{2(c_y e^{2 t} -1) }{y},\qquad c_y:=c_{\mathbf{0},{y}}<0.
\]
Since $\frac{c_y-1}{2c_y}{y}=\lim_{t\rightarrow +\infty}\gamma(t)\in \partial\Omega$,
it follows by \eqref{K=-1Fstarphipsi-x} and \eqref{defphsiK=-1} that
\[
\phi\left( -\frac{c_y-1}{2c_y}{y}  \right)-\psi\left(  -\frac{c_y-1}{2c_y}{y} \right)=-1,\quad \phi\left( \frac{c_y-1}{2c_y}{y}  \right)+\psi\left(  \frac{c_y-1}{2c_y}{y} \right)=1, \quad \psi({y})=1.
\]
By homogeneity, these  relations yield
\begin{equation}\label{basiccc1}
\phi\left( - {y}  \right)-\psi\left(  - {y} \right)=-\frac{2c_y}{c_y-1},\qquad \phi({y})=\frac{c_y+1}{c_y-1}.
\end{equation}

On the other hand, we replace $y$ with $\tilde{{y}}:=-{y}/\psi(-{y})\in S_{\mathbf{0}}\Omega$ in \eqref{basiccc1}  and obtain
\begin{equation}\label{basiccc2}
 \phi({y})-1=-\frac{2c_{\tilde{y}}}{c_{\tilde{y}}-1}\psi(-{y}),  \qquad \phi(-{{y}})=\frac{c_{\tilde{y}}+1}{c_{\tilde{y}}-1}\psi(-{y}).
\end{equation}
From \eqref{basiccc1} and \eqref{basiccc2}, it is easy to see
\[
c_{\tilde{y}}=c_y^{-1},\quad  \psi(-{y})=1=\psi({y}),\quad  \phi(-{y})=-\phi({y}).
\]
Hence, $\psi$ is even and $\phi$ is odd, so $F$ is a reversible Hilbert metric by Theorem~\ref{localstrucK}.
\end{proof}

The forward completeness of the manifold guarantees the global uniqueness of solutions to equation \eqref{Phiplusminus}, in contrast to the local uniqueness discussed in Remark~\ref{diffundeK=-1}.

\begin{proposition}\label{uniqunesslemmaK=-1}
Let $(\Omega,F)$ and $\psi,\phi$ be  as in Theorem \ref{globalK=-1}.
Then for any $(x,y)\in \Omega\times \mathbb{R}^n$, the equations
\begin{equation} \label{K=-1phi-psisolution}
\Phi_{+}(x,y) = (\phi + \psi)(y + x \Phi_{+}(x,y)),\qquad \Phi_{-}(x,y) = (\phi - \psi)(y + x \Phi_{-}(x,y))
\end{equation}
admit unique solutions $\Phi_{+}(x,y)$, $\Phi_{-}(x,y)$.
Moreover,
\begin{equation}\label{uniquessPhipm1111}
\Phi_\pm(x,y)=P(x,y)\pm F(x,y), \quad \forall (x,y)\in \Omega\times \mathbb{R}^n.
\end{equation}

\end{proposition}
\begin{proof}
If the solutions $\Phi_{\pm}(x,y)$ are unique, then \eqref{uniquessPhipm1111} follows directly from Theorem~\ref{localstrucK}.
Furthermore, since $\phi+\psi$ is a weak Minkowski norm,
Theorem~\ref{lemsolexi} immediately implies the uniqueness of $\Phi_+$.
Thus, it suffices to prove the uniqueness of $\Phi_{-}$.

By Theorem~\ref{globalK=-1} and \eqref{defphsiK=-1}, the functions $\psi$ and $\phi$ must satisfy one of the following:
\begin{enumerate}[\rm \ \ \ \  (a)]
\item\label{phi_-uniq1} $\phi(y) - \psi(y) \geq 0$ for all $y$, with equality if and only if $y=\mathbf{0}$;

\smallskip

\item\label{phi_-uniq2} $-2\psi(y) \leq \phi(y) - \psi(y) \leq  0$ for all $y$, with equality if and only if $y=\mathbf{0}$;

\smallskip

\item\label{phi_-uniq3} $\phi(y) - \psi(y)=0$ for all $y\in \mathbb{R}^n$.

\end{enumerate}
In Case \eqref{phi_-uniq3} the only solution to \eqref{K=-1phi-psisolution}$_2$ is $\Phi_{-} \equiv 0$.
It remains to establish uniqueness in Cases~\eqref{phi_-uniq1} and~\eqref{phi_-uniq2}.

\medskip\noindent
\textbf{Case  \eqref{phi_-uniq1}.} Note that it is
exactly Case (1) in Theorem \ref{globalK=-1}/\eqref{globalK=-1(ii)}.
The proof is divided into two steps.

{\bf Step 1}.  In this step, we show that for every $\bar{x} \in \Omega$ and $y\in T_{\bar{x}}\mathbb{R}^n\backslash\{\mathbf{0}\}$,
 there exists a unique  $s_1 > 0$ such that
$(\bar{\phi} - \bar{\psi})(s_1 y ) = 1$,
where $\bar{\phi}(y) := P(\bar{x}, y)$ and $\bar{\psi}(y) := F(\bar{x}, y)$.

First consider the case $\bar{x} = \mathbf{0}$.
For a given $y\in \Rno$, let $l_{\mathbf{0} y}$ denote the ray from the origin $\mathbf{0}$ in the direction $y$, i.e.,
\[
 l_{\mathbf{0} y}(s) := s y, \quad s\in [0,+\infty).
 \]
By the convexity of $\Omega$, this ray $ l_{\mathbf{0} y}$ intersects $\partial\Omega$ at a unique point $s_0 y$ for some $s_0>0$. Equation \eqref{K=-1region} then
gives
\begin{equation}\label{s0y1}
\phi(s_0y)+\psi(s_0y)=1.
\end{equation}
From the proof of Theorem~\ref{globalK=-1} (see \eqref{geodesiccase2k=-1}), the ray $l_{\mathbf{0} y}(s)$ for $s\in [0,s_0)$ coincides the image of the geodesic
 \[
 \gamma(t)=\frac{(c_y-1) ( e^{2 t}-1)}{2(c_y e^{2 t} -1) }\frac{y}{\psi(y)},\quad t\in [0,+\infty),
 \]
 where $c_y:=c_{\mathbf{0},y}>1$. Thus, the boundary point is given by
\begin{equation}\label{boundary}
s_0y=\lim_{t\rightarrow +\infty}\gamma(t)=\frac{c_y-1}{2 c_y}\frac{y}{\psi(y)}\in \partial\Omega.
\end{equation}
Using homogeneity  and \eqref{s0y1}, we find that
\[
\phi(c_ys_0 y) - \psi(c_ys_0 y) =c_y\left[ \phi(s_0 y) - \psi(s_0 y)  \right]=c_y\left[  1-2\psi(s_0y) \right]=
c_y\left[ 1 - 2\psi\left(\frac{c_y-1}{2c_y}\frac{y}{\psi(y)}\right) \right]= 1,
\]
so $c_y s_0y$ is an intersection point of $l_{\mathbf{0} y}$ and $(\phi-\psi)^{-1}(1)$. The uniqueness of the intersection point follows by the homogeneity of $(\phi-\psi)$.

For a general point $\bar{x}\in \Omega$, define $\bar{\phi}(y) := P(\bar{x}, y)$ and $\bar{\psi}(y): = F(\bar{x}, y)$.
By  translating $\bar{x}$ to the origin in $\Omega$, the functions $\bar{\psi},\bar{\phi}$ become the initial data of $F,P$ at this new origin.
The same argument as above then yields a unique  $s_1 > 0$ satisfying $(\bar{\phi} - \bar{\psi})(s_1 y ) = 1$.

\smallskip

{\bf Step 2}.  In this step, we show that the solution $\Phi_{-}(x,y)$ to equation \eqref{K=-1phi-psisolution}$_2$ exists uniquely.

Since $\phi+\psi$ is a weak Minkowski norm,
Lemma~\ref{varphi_minus} (with $\varphi := \phi + \psi$ and $\tilde{\varphi} := \phi - \psi$) guarantees the existence of a solution.
It remains to show the uniqueness.

In view of Lemma \ref{Sh1Lemma5.1}, let $\mathscr{D}$ be the
largest connected subset of $\Omega$ containing the origin  such that for every $(x, y) \in  \mathscr{D}\times \mathbb{R}^n$, the solution $\Phi_{-}(x,y)$  is unique.
Then a translation shows that $\mathscr{D}$ is an open set. Moreover, we have
\[
\mathscr{D}\subset \Omega=\{\phi+\psi<1\}\subset \{\phi-\psi<1\}.
\]
Since $(\phi-\psi)|_{\Rno}>0$,   Observation  \ref{K=-1remarkindicatrix} implies that $\Phi_-(x,y)\geq 0$, with equality only when $y=\mathbf{0}$, and
\begin{equation}\label{strixinx}
 \{ y\in\mathbb{R}^n\, :\,  \Phi_{-}(x,y) = 1 \} = \{ y\in\mathbb{R}^n\, :\,  (\phi - \psi)( y ) = 1 \} - x, \quad \forall x\in \mathscr{D}.
 \end{equation}

Suppose, for contradiction,  that  $\mathscr{D} \subsetneq \Omega $. Then
there exists a point $\bar{x} \in \partial\mathscr{D}\cap \Omega$ and some $\bar{y}\in \Rno$ for which the solution $\Phi_-(\bar{x},\bar{y})$  is not unique.
 Lemma \ref{K=-1lemintersectpoints} then implies that the ray $\{\bar{x}+s\bar{y}\,:\,s\geq 0\}$ intersects $(\phi - \psi)^{-1}(1)$  more than once.
 However, Step~1 shows that the  ray $\{s\bar{y}\,:\,s\geq 0\}$
intersects  $(\bar{\phi} - \bar{\psi})^{-1}(1)$ exactly once, where
$\bar{\phi}(y):=P(\bar{x},y)$ and $\bar{\psi}(y):=F(\bar{x},y)$.
Hence, translation by $\bar{x}$ establishes
 \begin{equation} \label{K=1secContra}
 (\phi - \psi)^{-1}(1) \neq  (\bar{\phi} - \bar{\psi})^{-1}(1)+\bar{x}.
\end{equation}

Treating $\bar{x}$  as a ``new origin'', Theorem~\ref{globalK=-1} implies that $ \bar{\phi}+\bar{\psi}$ is a weak Minkowski norm and $\Omega-\bar{x}=\{\bar{\phi}+\bar{\psi}<1 \}$. The existence and uniqueness of the solution $\bar{\Phi}_+(x,y)$ to the equation
\begin{equation*} 
\bar{\Phi}_{+}(x,y)=(\bar{\phi}
+ \bar{\psi})\big( y+(x-\bar{x}) \bar{\Phi}_{+}(x,y) \big), \quad \forall (x,y)\in \Omega\times \mathbb{R}^n,
\end{equation*}
then guaranteed by Theorem~\ref{lemsolexi}.
Similarly, Lemma \ref{Sh1Lemma5.1} and a translation argument yield a ball $\mathbb{B}^n_{\delta}(\bar{x})$ on which the solution $\bar{\Phi}_{-}(x,y)$ to
the corresponding equation for $\bar{\phi} - \bar{\psi}$ is unique.

Now choose a point $\hat{x} \in \mathbb{B}^n_{\delta}(\bar{x})\cap \mathscr{D}\neq\emptyset $. Then  both  $\bar{\Phi}_\pm(\hat{x},y)$ are uniquely determined as solutions to
\begin{equation} \label{K=-1barphi-psiunique}
\bar{\Phi}_{\pm}(\hat{x},y)=(\bar{\phi}
\pm \bar{\psi})\left( y+(\hat{x}-\bar{x}) \bar{\Phi}_{\pm}(\hat{x},y) \right), \quad \forall y\in \mathbb{R}^n.
\end{equation}
 Since $\bar{\phi},\bar{\psi}$ are the initial data at $\bar{x}$,
 Theorem \ref{localstrucK} combined with a translation implies that
 \[
\bar{\Phi}_-(\hat{x},y)=P(\hat{x},y) - F(\hat{x},y).
\]
On the other hand, since $\hat{x}\in \mathscr{D}$, the solutions ${\Phi}_{\pm}(\hat{x},y)$ are also unique, and Theorem \ref{localstrucK} gives
${\Phi}_-(\hat{x},y)=P(\hat{x},y) - F(\hat{x},y)$. Hence,  actually
$\Phi_-(\hat{x},y)=\bar{\Phi}_-(\hat{x},y)$. The uniqueness of $\bar{\Phi}_-(\hat{x},y)$ in \eqref{K=-1barphi-psiunique} then yields
\begin{align}
\{ y\in \mathbb{R}^n\, :\,  \Phi_{-}(\hat{x},y) = 1 \}  =  \{ y\in \mathbb{R}^n\, :\, (\bar{\phi} - \bar{\psi})(y+\hat{x}-\bar{x}) = 1 \},
\end{align}
which together with \eqref{strixinx} implies
\[
 \{ y\in \mathbb{R}^n\, :\, (\phi - \psi)(y) = 1 \} - \hat{x} =  \{ y\in \mathbb{R}^n\, :\, (\bar{\phi} - \bar{\psi})(y) = 1 \} - \hat{x} + \bar{x}.
\]
However, this contradicts \eqref{K=1secContra}.
Therefore, $\mathscr{D}=\Omega$, and the uniqueness of $\Phi_{-}(x, y)$ follows.

\medskip\noindent
\textbf{Case  \eqref{phi_-uniq2}.}
This case is handled analogously to Case~\eqref{phi_-uniq1}, mutatis mutandis for the condition $-2\psi(y) \le \phi(y) - \psi(y) \le 0$ and the analysis of the level sets ${\phi - \psi = -1}$. The details are omitted.
\end{proof}

The following corollary is an immediate consequence of Proposition~\ref{uniqunesslemmaK=-1}.
\begin{corollary}\label{K=-1sufP>=F}
Let $\Omega\subset \mathbb{R}^n$ be a domain containing the origin. Given a Minkowski norm $\psi$  and a positively $1$-homogeneous function $\phi$ on $\mathbb{R}^n$,
 there exists at most one Finsler metric $F$ on $\Omega$ such that $(\Omega,F)$ is a forward complete projectively flat Finsler manifold with $\mathbf{K}=-1$ and
\[
\psi(y)=F(\mathbf{0},y),\qquad \phi(y)=P(\mathbf{0},y).
\]
\end{corollary}

We now prove Theorem~\ref{uniquenssFinsler}, which is the uniqueness result stated in the Introduction.
\begin{proof}[Proof of Theorem \ref{uniquenssFinsler}]
By the Bonnet--Myers theorem (cf. \cite[Theorem 7.7.1]{BCS}),
 every projectively flat Finsler manifold $(\Omega,F)$ of constant positive flag curvature cannot be forward complete because $\Omega$ is noncompact.  Therefore, it suffices to consider the case of  non-positive constant flag curvature, which follows directly from
 Corollaries~\ref{existsK=0metric} and~\ref{K=-1sufP>=F}.
 \end{proof}

We also derive an explicit formula for the distance function on a (not necessarily complete) projectively flat Finsler manifold with $\mathbf{K} = -1$.

\begin{theorem}  \label{lemK=-1dx_0xdx_0}
Let $(\Omega,F)$ be a projectively flat Finsler manifold with $\mathbf{K}=-1$,  where $\Omega\subset \mathbb{R}^n$ is a strictly convex domain containing the origin.
Then for any $x_1,x_2\in \Omega$, 
\begin{equation}\label{eqK=-1dx_0xdxx_0}
d_F(x_1,x_2) =  \frac12\ln \left(  \frac{1-P(x_1, x_2-x_1)+F(x_1, x_2-x_1)}{1-P(x_1, x_2-x_1)-F(x_1, x_2-x_1)} \right).
\end{equation}
In particular, with $\psi(y):=F(\mathbf{0},y)$ and $\phi(y):= P(\mathbf{0},y) $, this implies
\begin{equation}\label{coreqK=-1dx_0xdxx_01122}
d_F(\mathbf{0},x) =  \frac{1}{2} \ln \left( \frac{1 -  \phi(x) + \psi(x)}{1 - \phi(x) - \psi(x)}\right),
\qquad d_F(x,\mathbf{0})=\frac12 \ln \left( \frac{1+\phi(-x)+\psi(-x)}{1+\phi(-x)-\psi(-x)} \right).
\end{equation}
\end{theorem}

\begin{proof}
The identity \eqref{coreqK=-1dx_0xdxx_01122} follows directly from \eqref{eqK=-1dx_0xdxx_0} by the same argument as in the proof of Theorem~\ref{K=0lemdx_0xdxx_0}. We therefore focus on proving \eqref{eqK=-1dx_0xdxx_0}.

The case $x_1 = x_2$ is trivial. Now fix $x_2 \in \Omega \backslash \{x_1\}$. According to the proof of Theorem~\ref{globalK=-1}, the unit-speed geodesic $\gamma(t)$ from $x_1$ to $x_2$  falls into one of the following three types:
\begin{enumerate}[\rm (i)]
\item \label{K=-1disi}
$\ds\gamma(t)   =\frac{1-e^{-2 t} }2 \frac{x_2-x_1}{F(x_1, x_2-x_1)} + x_1$;

\smallskip

\item \label{K=-1disiii}
$\ds\gamma(t)   =\frac{e^{2t}-1 }2 \frac{x_2-x_1}{F(x_1, x_2-x_1)} + x_1$;

\smallskip
\item \label{K=-1disii}
$\ds
\gamma(t) = \frac{(c-1) ( e^{2 t}-1)}{2(c e^{2 t} -1) } \frac{x_2-x_1}{F(x_1, x_2-x_1)} + x_1$,  where  $c=c(x_1,x_2)$ is a constant with  $c<0$ or $c>1$.
\end{enumerate}
Since $(\Omega,F)$ may not be forward complete, Type \eqref{K=-1disiii} cannot be ruled out. We proceed by considering each case separately.

\medskip\noindent
{\bf Case 1.} Suppose $\gamma$ is of Type  \eqref{K=-1disi}. Then \eqref{f=pcond} implies
\begin{equation}\label{pfequlxx}
P(x_1, x_2-x_1)=F(x_1, x_2-x_1).
\end{equation}
Since $\gamma$ has unit speed, we have
\[
x_2=\gamma(d_F(x_1,x_2))=\frac{1-e^{-2 d_F(x_1,x_2)} }2 \frac{x_2-x_1}{F(x_1, x_2-x_1)} + x_1,
\]
which together with \eqref{pfequlxx} yields
\begin{align*}
d_F(x_1,x_2)&=-\frac12\ln \big(1-2 F(x_1, x_2-x_1)\big)
=\frac12\ln \left(  \frac{1-P(x_1, x_2-x_1)+F(x_1, x_2-x_1)}{1-P(x_1, x_2-x_1)-F(x_1, x_2-x_1)} \right).
\end{align*}

\medskip\noindent
{\bf Case 2.} If $\gamma$ is of Type \eqref{K=-1disiii}, then \eqref{geodesiceq} and
 \eqref{geodesicBackf''1} give
\begin{equation*}
P(\gamma(t), \gamma'(t)) =  -1 = -F(\gamma(t), \gamma'(t)).
\end{equation*}
Evaluating at $t = 0$ yields $P(x_1, x_2 - x_1) = -F(x_1, x_2 - x_1)$. The distance formula then follows by a computation analogous to Case~1.

\medskip\noindent
{\bf Case 3.}
If $\gamma$ is of Type~\eqref{K=-1disii}, then \eqref{P0K=-1} with $x = x_1$ and $y = (x_2 - x_1)/F(x_1, x_2 - x_1)$ gives
\begin{equation*}
\frac{P(x_1,x_2-x_1)}{F(x_1,x_2-x_1)}=P\left(x_1,\frac{x_2-x_1}{F(x_1, x_2-x_1)}\right)=\frac{c+1}{c-1},
\end{equation*}
which determines the constant $c$ as
\begin{equation}\label{ciswhat11}
c=\frac{P(x_1, x_2-x_1)+F(x_1, x_2-x_1)}{P(x_1, x_2-x_1)-F(x_1, x_2-x_1)}.
\end{equation}
On the other hand, since $\gamma(d_F(x_1,x_2))=x_2$, the expression \eqref{K=-1disii} for $\gamma$  implies
\[
d_F(x_1,x_2)=\frac12 \ln\left( \frac{1-c+2F(x_1, x_2-x_1)}{1-c+2c F(x_1, x_2-x_1)}  \right).
\]
Combining this with \eqref{ciswhat11} establishes \eqref{eqK=-1dx_0xdxx_0}.
\end{proof}

The following corollary can be proved analogously to Corollary~\ref{Phi0uniquess}.

\begin{corollary}\label{Phi-uniquess}
Let $(\Omega,F)$ be a projectively flat Finsler manifold with $\mathbf{K}=-1$, where $\Omega\subset \mathbb{R}^n$ is a bounded domain containing the origin.  Then $(\Omega,F)$ is forward complete if and only if $\Omega=\{\phi+\psi<1\}$, where $\psi(y): = F(\mathbf{0},y)$ and $\phi(y): = P(\mathbf{0},y)$.
\end{corollary}
We now prove Theorem~\ref{thmK=-1globalintro}, which is the main characterization theorem for $\mathbf{K}=-1$ stated in the Introduction.
\begin{proof}[Proof of Theorem \ref{thmK=-1globalintro}]
The forward implication ``$\Rightarrow$'' is a direct consequence of Theorem~\ref{localstrucK}, \ref{globalK=-1} and Proposition~\ref{uniqunesslemmaK=-1}.

For the converse ``$\Leftarrow$'', Condition~\eqref{K=-1bascicondition1} together with Theorem~\ref{localstrucK} and Corollary~\ref{Phi-uniquess} imply that $(\Omega,F)$ is a forward complete projectively flat Finsler manifold with $\mathbf{K}=-1$.
 The remaining statements then follow immediately from Theorems~\ref{localstrucK} and~\ref{globalK=-1}.
\end{proof}

\begin{remark}\label{K=-1condi1ne}
Therefore, Condition~\eqref{K=-1bascicondition1} in Theorem~\ref{thmK=-1globalintro} provides a complete characterization of forward complete, projectively flat Finsler manifolds of constant flag curvature $\mathbf{K} = -1$.
\end{remark}

\subsection{Global existence of metrics}\label{subsK=-1_1}
By Theorem \ref{thmK=-1globalintro}, any forward complete projectively flat Finsler manifold
($\Omega$, $F$) with $\mathbf{K} = -1$ is determined by a Minkowski norm $\psi$ and a  positively $1$-homogeneous function $\phi$ such that
$\psi+ \phi$ is a weak Minkowski norm and $\Omega  = \{  \psi+\phi < 1 \}$.
This naturally leads to the following inverse problem:

\medskip
\textit{Given such a pair $(\psi, \phi)$, is the associated function $F$ always a Finsler metric on $\Omega = \{\psi + \phi < 1\}$?}
\medskip

\noindent
By Theorem \ref{funkmetrisufficnecessy}, the answer is affirmative when $\phi=\psi$, in which case
$F = \Theta/2$ for the unique Funk metric $\Theta$ on $\Omega$, with its symmetrization yielding the Hilbert metric.
 For other pairs $(\psi, \phi)$, however, phenomena analogous to the $\mathbf{K} = 0$ case occur.
 We now introduce the notation for this question.

\begin{notation}\label{basciassmupK=-1} Let $\psi$ be a Minkowski norm and $\phi$  a positively $1$-homogeneous function on $\mathbb{R}^n$ such that $\psi+\phi$ is a weak Minkowski norm and  one of the following holds:
\begin{enumerate}[\rm \ \ \ \  (i)]
\item\label{K=-1phi_-uniq1con} $\phi(y) \geq \psi(y) $ for all $y$, with equality only at $y=\mathbf{0}$;

\item\label{K=-1phi_-uniq2con} $-\psi(y) \leq  \phi(y) \leq \psi(y)$ for all $y$, with equality  only at $y=\mathbf{0}$.



\end{enumerate}
Define the domain
\begin{equation}\label{K=-1positive=completeOmega}
  \Omega : = \{ x \in \mathbb{R}^n \,:\, \phi(x) + \psi(y) < 1 \},
\end{equation}
and the functions $F,P:\Omega\times \mathbb{R}^n\rightarrow \mathbb{R}$  by
\begin{equation}\label{thmK=-1eq2}
F(x,y) := \frac{1}{2}\big\{ \Phi_{+} - \Phi_{-} \big\}, \qquad P(x,y): = \frac{1}{2}\big\{ \Phi_{+} + \Phi_{-} \big\},
\end{equation}
where  $\Phi_\pm=\Phi_\pm(x,y)$ are the unique solutions to
\begin{equation} \label{Phiplusminus233}
\Phi_{\pm} = (\phi \pm \psi)(y + x \Phi_{\pm}), \quad \forall (x,y)\in \Omega\times \mathbb{R}^n.
\end{equation}
Define the set on which $F$ is a Finsler metric as
\[
\PD:=\PD(F):=\{x\in \Omega\,:\, \text{$F(x,\cdot)$ is a Minkowski norm at $x$}\}.
\]
\end{notation}

\begin{remark}
The case $\phi=\psi$ is omitted since $F$ is always a Finsler metric on $\Omega$.
Note that in both cases above, we have
\begin{equation}\label{FpositiveK=-1}
\text{$F(x,y)\geq 0$ with equality if and only if $y=\mathbf{0}$.}
\end{equation}
Indeed, by letting $\varphi:=\phi+\psi$, Theorem \ref{lemsolexi} implies   $\Phi_+(x,y)\geq 0$, with equality only when $y=\mathbf{0}$. For Condition~\eqref{K=-1phi_-uniq2con},  equation \eqref{Phiplusminus233} shows  $\Phi_-(x,y)\leq 0$, so \eqref{FpositiveK=-1} follows immediately.
Under Condition~\eqref{K=-1phi_-uniq1con}, by letting $\tilde{\varphi} := \phi - \psi$, Lemma~\ref{varphi_minus} implies $\Phi_+(x, y) \ge \Phi_-(x, y)$, with equality only when $y = \mathbf{0}$, which again yields \eqref{FpositiveK=-1}.
\end{remark}

\begin{proposition} \label{K=-1positive=complete}
Under  Notation \ref{basciassmupK=-1}, the pair $(\Omega,F)$ is a forward complete projectively flat Finsler manifold with $\mathbf{K}=-1$ if and only if   $F$ is a Finsler metric on all of $\Omega$, i.e., $\PD = \Omega$.
\end{proposition}
\begin{proof}
The necessity is immediate from the definitions, and the sufficiency follows  from Corollary~\ref{Phi-uniquess}.
\end{proof}


Proposition~\ref{K=-1positive=complete} establishes that a projectively flat Finsler metric with $\mathbf{K} = -1$ exists on the entire  $\Omega$ if and only if the manifold is forward complete. However, analogous to the case $\mathbf{K} = 0$, the function $F$ defined in Notation~\ref{basciassmupK=-1} may still fail to be a Finsler metric.
Indeed, the following example demonstrates  that the fundamental tensor can indeed fail to be positive definite globally,
even for natural perturbations of the Euclidean metric such as Randers-type data.

\begin{example} \label{K=-1ExampleAmazing}
Let $a= (a_1, 0)\in \mathbb{R}^2$ with $a_1\in [0,1)$, and let $c>1$ be a constant.
Define two Minkowski norms (i.e., Randers norms) on $\mathbb{R}^2$ by
\[
\psi(y):= |y|+ \langle a, y\rangle,\qquad \phi(y):= c \psi(y).
\]
 Let
 $\Phi_{\pm}(x,y)$  be  the unique solutions to  the  equations
\[
\Phi_{+}(x,y) = (\phi+\psi)(y+ x \Phi_{+}(x,y) ), \qquad   \Phi_{-}(x,y) = (\phi-\psi)(y+ x \Phi_{-}(x,y) ).
\]
A direct calculation yields explicit expressions:
\begin{align*}
\Phi_{+}(x,y)& = \frac{\sqrt{\mathcal {A}_a((c+1)x,(c+1)y)}  + \mathcal {C}_a((c+1)x,(c+1)y)+ (c+1)^2\langle x,y\rangle}{\mathcal {B}_a((c+1)x)},\\
\Phi_{-}(x,y) &= \frac{\sqrt{\mathcal {A}_a((c-1)x,(c-1)y)}  + \mathcal {C}_a((c-1)x,(c-1)y)+ (c-1)^2\langle x,y\rangle}{\mathcal {B}_a((c-1)x)},
\end{align*}
where the functions $\mathcal{A}_a, \mathcal{B}_a, \mathcal{C}_a$ are given by
\begin{align}
\mathcal{A}_a(x,y)&:= \left( (1-\langle a, x \rangle)^2-|x|^2 \right)(|y |^2- \langle a, y \rangle^2) + \left(\langle x,y \rangle + (1-\langle a, x \rangle) \langle a, y \rangle \right)^2, \label{K=-1Ex2A} \\
\mathcal {B}_a(x)&:=(1-\langle a,x  \rangle)^2-|x|^2,\qquad \mathcal {C}_a(x,y):= (1-\langle a,x  \rangle)\langle a, y\rangle.
\notag
 \end{align}
Define
\begin{align} \label{K=-1Ex2}
F(x,y):=\frac{1}{2}\big\{ \Phi_{+}(x,y) -\Phi_{-}(x,y)\big\} ,\quad \Omega := \big\{ x =(x^1, x^2) \in \mathbb{R}^2 \,:\, \phi(x)+ \psi(x)  < 1\big\}.
\end{align}
The set $\Omega$ is an open elliptical disk. However, the domain $\PD$ can exhibit three distinct configurations depending on the parameter $a_1$. Define the critical value
\begin{align*}
\Lambda := \frac{1}{2}\left( (c+1)^{\frac{2}{3}} -(c-1)^{\frac{2}{3}} \right)\sqrt{ (c+1)^{\frac{2}{3}} + (c-1)^{\frac{2}{3}} }.
\end{align*}
It can be verified that $\Lambda<1$.
Then the following statements hold:
\begin{enumerate}[\rm (i)]

\item if $a_1\in [0,\Lambda)$,
then $\PD=\Omega$; that is, $F$ is positive definite on the whole $\Omega$;

\item if  $a_1=\Lambda$,
then $\PD$ is obtained by removing from $\Omega$ a line segment parallel to the $x^2$-axis;
\item if $a_1 \in (\Lambda,1)$,
then $\PD$ consists of  two disjoint convex components
 of $\Omega$.
\end{enumerate}
\end{example}

\begin{remark}
The statements above can be verified using the geometric approach developed in the proof of Proposition~\ref{K=0Ex1}, rather than by direct computation. See also Proposition~\ref{thmK=-1convex} and Theorem~\ref{K=-1Dim2theorem} below.
\end{remark}

This example underscores the importance of investigating the structure of the  domain $\PD$.
To proceed, we require the following result, whose proof is deferred to  Appendix \ref{propergenerlengappex0}.
\begin{lemma} \label{lemK=-1suffeqposilem}
Let $\bar{\psi}$ and $\bar{\phi}$ be positively $1$-homogeneous functions on $\mathbb{R}^n$ such that
\[
0<\bar{\psi}|_{\Rno}\in C^\infty(\Rno),\qquad \bar{\phi}|_{\Rno}\in C^\infty(\Rno),
\]
and $(\bar{\phi}\pm \bar{\psi})|_{\Rno}$ is nowhere vanishing.
Suppose that $\bar{\Phi}_{\pm}(x,y)$ are the unique solutions to the equations
\begin{align}\label{lemK=-1Phipm}
 \bar{\Phi}_{\pm}(x,y) = (\bar{\phi}\pm\bar{\psi})\big(y+ x \bar{\Phi}_{\pm}(x,y)\big),  \quad \forall (x,y)\in \Omega\times \mathbb{R}^n,
 \end{align}
where $\Omega\subset \mathbb{R}^n$ is a domain containing the origin.
Then for every
 $y\in \Rno$, the following equivalence holds:
\begin{align*}
\left(  [\bar{\psi}^2]_{y^i y^j}(y) \right) \text{ is positive definite} \ \ & \Longleftrightarrow  \ \
\left(  [\bar{F}^2]_{y^i y^j}(\mu y,y) \right) \text{ is positive definite for all $\mu \in \mathbb{R}$ with $\mu y\in \Omega$},
 \end{align*}
where
$\bar{F}(x,y): = \frac{1}{2} \big\{ \bar{\Phi}_{+} (x,y) - \bar{\Phi}_{-} (x,y)\big\}$.
\end{lemma}

\begin{proposition}\label{thmK=-1convex}
Under  Notation \ref{basciassmupK=-1},  let $U$ be a convex connected component of $\PD$. Then the following hold:
\begin{enumerate}[{\rm (a)}]

\item \label{suffK=-1new1lemCasei} for any smooth point  $x\in \partial U\cap \Omega$ and any vector $y$  not tangent to $\partial U$ at $x$, the matrix $([F^2]_{y^iy^j}(x,y))$ is positive definite;

\item  \label{suffK=-1new1lemCaseii}
for any smooth point  $x\in \partial U\cap \Omega$, there exists a nonzero vector $y$ tangent to $\partial U$ at $x$ such that $([F^2]_{y^iy^j}(x,y))$ has a zero eigenvalue.  Moreover, the corresponding eigenvector  $\theta = (\theta^i)$ is not parallel to $y$ and satisfies
\begin{align} \label{lemK=-1sufflem2(ii)1}
 F_{y^i}({x},y)  \theta^i =0, \qquad F_{y^i y^j}({x},y)   \theta^i \theta^j  =0.
\end{align}
If, in addition, the vector $x$ is not tangent to $\partial U$ at the point $x$, then    for all sufficiently small  $|t|\neq0$,
\[
F_{y^i y^j}({x},y+ t {x})  \theta^i \theta^j>0.
\]
\end{enumerate}
\end{proposition}
\begin{proof}
Let $x_o$ be a smooth point of $\partial U\cap \Omega$.  Since $\partial U \subset \partial {\PD}$,
there exists $y_o\in T_{x_o}\mathbb{R}^n\backslash\{\mathbf{0}\}$ such that the  positive semi-definite  matrix $([F^2]_{y^iy^j}(x_o,y_o))$ has a zero eigenvalue.

We now show that $y_o$ is tangent to $\partial U$ at $x_o$.
If not, the convexity of  ${U}$ yields a point $\bar{x}\in {U}$ such that
 \begin{equation}\label{newxbary_o22}
  x_o-\bar{x}=\mu y_o, \quad  \text{for some } \mu\neq0.
  \end{equation}
Define $\bar{\psi}(y) := F(\bar{x}, y)$ and $\bar{\phi}(y) := P(\bar{x}, y)$. Then $\bar{\psi}$ is a Minkowski norm and $\bar{\phi}$ is a positively $1$-homogeneous function.
It follows from \eqref{thmK=-1eq2} and \eqref{Phiplusminus233} that $\bar{\psi},\bar{\phi}$ are the solutions to the equations
\begin{align*}
 (\bar{\phi}+\bar{\psi})(y) = (\phi+\psi)(y+ \bar{x} (\bar{\phi}+\bar{\psi})(y)),
 \qquad (\bar{\phi}-\bar{\psi})(y) = (\phi-\psi)(y+ \bar{x} (\bar{\phi}-\bar{\psi})(y)), \quad \forall y\in\mathbb{R}^n.
 \end{align*}
By Conditions \eqref{K=-1phi_-uniq1con} and \eqref{K=-1phi_-uniq2con} in Notation \ref{basciassmupK=-1},
it follows that $(\bar{\phi}\pm\bar{\psi})(y)=0$ if and only if $y=\mathbf{0}$. Moreover,
Proposition \ref{transformforK=-1} and Notation \ref{basciassmupK=-1} guarantee that
 for any $(x,y)\in  (\Omega-\bar{x})\times \mathbb{R}^n$, the equations
 \begin{equation*}
\bar{\Phi}_+ (x,y)= (\bar{\phi} + \bar{\psi})\big(y + \bar{x} \bar{\Phi}_+(x,y)\big), \quad  \bar{\Phi}_-(x,y) = \big(\bar{\phi} - \bar{\psi})(y + \bar{x} \bar{\Phi}_-(x,y)\big),
\end{equation*}
have unique solutions  $ \bar{\Phi}_\pm(x,y)$.
For any $x\in \Omega-\bar{x}$,    relation \eqref{tranK=-1} implies that
\begin{equation}\label{twoFbarF22}
\bar{F}(x,y): = \frac{1}{2} \Big\{ \bar{\Phi}_{+}(x,y) - \bar{\Phi}_{-}(x,y) \Big\} =\frac{1}{2} \Big\{ {\Phi}_{+}(x+\bar{x},y) - {\Phi}_{-}(x+\bar{x},y) \Big\}=F(x+\bar{x},y).
\end{equation}
Since $\bar{F}(\mathbf{0},y)=\bar{\psi}(y)$ is a Minkowski norm, it follows by Lemma \ref{lemK=-1suffeqposilem} that  $([\bar{F}^2]_{y^iy^j}(\mu y_o,y_o))$ is positive definite. Thus, \eqref{twoFbarF22} together with \eqref{newxbary_o22} implies that
\begin{equation}\label{therrethy1}
([{F}^2]_{y^iy^j}(x_o,y_o))=([\bar{F}^2]_{y^iy^j}(x_o-\bar{x},y_o))=([\bar{F}^2]_{y^iy^j}(\mu y_o,y_o))
\end{equation}
is positive definite,  which contradicts the positive
semi-definiteness of $\left( [  {F}^2 ]_{y^i y^j}  (x_o,y_o) \right)$. Hence, $y_o$ is tangent to $\PD$ at $x_o$, and therefore Statement \eqref{suffK=-1new1lemCasei} is established.

The rest of the proof follows from the same argument as in Proposition~\ref{K=0conneccomconvex}, and we omit the details.
\end{proof}


The following theorem provides a complete characterization of the geometry of $\PD$
in the two-dimensional case. Its proof follows the same strategy as that of Theorem~\ref{n=2k=0deger}, using Proposition~\ref{thmK=-1convex}.
\begin{theorem}\label{K=-1Dim2theorem}
Under  Notation \ref{basciassmupK=-1} with the dimension $n=2$, the following hold:
\begin{enumerate}[\rm (i)]
\item\label{K=-1Dim2conneccompconvex1} every connected component of $\PD$  is  a convex domain in $\mathbb{R}^2$;

\item\label{K=-1Dim2conneccompconvex3} for $\mathcal{H}^{1}$-a.e.  $x\in \partial {\PD}\cap \Omega$, if the matrix $([F^2]_{y^iy^j}(x,y))$ is degenerated at some $y\in \mathbb{R}^2_\circ$,   then $y$ is tangent to $\partial{\PD}$ at $x$. In particular, such a pair $(x,y)$ exists whenever $\partial {\PD}\cap \Omega$ is nonempty.
\end{enumerate}
\end{theorem}

Proposition~\ref{thmK=-1convex} and Theorem~\ref{K=-1Dim2theorem} provide an effective criterion for determining
 whether the function $F$ defined in Notation \ref{basciassmupK=-1} is a Finsler metric on $\Omega$.
As an application, we now give a conceptual proof of Proposition~\ref{ExampleK=-1amazing}
that avoids direct computation.


\begin{proof}[Proof of Proposition \ref{ExampleK=-1amazing}]
We prove only case \eqref{c>1}, as the other cases are similar.
A local verification shows that
  $F$ is a projectively flat Finsler metric with $\mathbf{K} =-1$ near the origin.
  Let $U$ denote the connected component of $\PD$ containing the origin. We claim $U=\Omega$.

Suppose, for contradiction, that $U \subsetneq \Omega$.
Let $\breve{x}\in \Omega$ be a point in $\partial U$ that is closest to the origin, i.e., $|\breve{x}|=\inf_{x\in \partial U}|x|=:\delta>0$;
such a point exists since $\Omega$ is a Euclidean ball centered at the origin.
Then $\mathbb{B}^n_{\delta}(\mathbf{0})\subset U$, and there exists a vector $\breve{y}\in T_{\breve{x}}\mathbb{R}^n\backslash\{\mathbf{0}\}$  such that the matrix $([F^2]_{y^iy^j}(\breve{x}, \breve{y}))$ is not positive definite.
Hence, there exists a vector $\breve{\theta} = (\breve{\theta}^i) \neq \mathbf{0}$, not parallel to $\breve{y}$, satisfying
\begin{equation}\label{K=-11Euclideanxytheta}
 [F^2]_{y^i y^j}(\breve{x}, \breve{y}) \breve{\theta}^i \breve{\theta}^j \leq 0.
 \end{equation}

Since $F$  is spherical symmetric (i.e.,
$F(\mathbf{Q}x, \mathbf{Q}y) =F(x, y)$ for any $\mathbf{Q}\in O(n)$), the inequality \eqref{K=-11Euclideanxytheta} holds
 for all orthogonal transformations of the triple $(\breve{x}, \breve{y}, \breve{\theta})$.
This implies that $F$ fails to be a Finsler metric outside $\mathbb{B}^n_{\delta}(\mathbf{0})$, so
 $U=\mathbb{B}^n_{\delta}(\mathbf{0}) \Subset  \Omega$ and $\breve{x}$  is a smooth boundary point of $U$. By Proposition~\ref{thmK=-1convex}/\eqref{suffK=-1new1lemCasei}, the vector $\breve{y}$ must be tangent to the Euclidean ball $U$ at $\breve{x}$, and therefore
 \begin{align}
\langle \breve{x}, \breve{y} \rangle =0. \label{K=-1Ex_xy0}
\end{align}

 Moreover,
since the vector $\breve{x}$ is not tangent to $\partial{U}$ at $\breve{x}$, Proposition~\ref{thmK=-1convex}/\eqref{suffK=-1new1lemCaseii} implies the
existence of a nonzero vector $\theta = (\theta^i)$, not parallel to $\breve{y}$, such that
\begin{equation} \label{K=-1amazinggijthetaithetaj}
 F_{y^i}(\breve{x}, \breve{y}) \theta^i =0, \quad  F_{y^i y^j}(\breve{x}, \breve{y}) \theta^i \theta^j =0, \quad  \left.\frac{\dd }{{\dd} t}\right|_{t=0}F_{y^i y^j}(\breve{x},\breve{y}+ t \breve{x})  \theta^i \theta^j  =0.
 \end{equation}
We now show that these conditions lead to a contradiction.

A direct calculation gives
\begin{align}\label{firsPhi_+}
[\Phi_{+}]_{y^i}(x,y) =
\frac{(c+1)}{1 - (c+1)^2 |x|^2} \left\{ \frac{\left( 1- (c+1)^2 |x|^2 \right) y^i + (c+1)^2 \langle x,y \rangle x^i}{\sqrt{\left( 1- (c+1)^2 |x|^2 \right) |y|^2 + (c+1)^2\langle x,y \rangle^2}}  + (c+1)  x^i \right\},
\end{align}
which together with \eqref{K=-1Ex_xy0} yields
\begin{align*}
[\Phi_{+}]_{y^i}(\breve{x},\breve{y}) \theta^i =
\frac{(c+1)\langle \breve{y},\theta \rangle}{|\breve{y}|\sqrt{\left( 1- (c+1)^2 |\breve{x}|^2 \right)}} + \frac{(c+1)^2\langle \breve{x},\theta \rangle }{1 - (c+1)^2 |\breve{x}|^2}.
\end{align*}
Replacing $c+1$ by $c-1$ gives
\begin{align*}
[\Phi_{-}]_{y^i}(\breve{x},\breve{y}) \theta^i =
\frac{(c-1)\langle \breve{y},\theta \rangle}{|\breve{y}|\sqrt{\left( 1- (c-1)^2 |\breve{x}|^2 \right)}} + \frac{(c-1)^2\langle \breve{x},\theta \rangle }{1 - (c-1)^2 |\breve{x}|^2}.
\end{align*}
Combining these  with \eqref{suffK=-1new1Randersnorm}$_1$ and \eqref{K=-1amazinggijthetaithetaj}$_1$ leads to
\begin{equation}\label{K=-1ExFithetai}
\begin{split}
0=2 F_{y^i}(\breve{x},\breve{y}) \theta^i
=& \left\{ \frac{c+1}{\sqrt{\left( 1- (c+1)^2 |\breve{x}|^2 \right)}} -\frac{c-1}{\sqrt{\left( 1- (c-1)^2 |\breve{x}|^2 \right)}} \right\}
\frac{\langle \breve{y},\theta \rangle}{|\breve{y}|} \\
&+ \left\{ \frac{(c+1)^2 }{1 - (c+1)^2 |\breve{x}|^2} -  \frac{(c-1)^2 }{1 - (c-1)^2 |\breve{x}|^2}\right\} \langle \breve{x},\theta \rangle.
\end{split}
\end{equation}

Next, using \eqref{K=-1Ex_xy0} and \eqref{firsPhi_+}, we compute
\begin{equation} \label{K=-1Ex_Phi+ij}
\begin{split}
& \quad\  [\Phi_{+}]_{y^i y^j}(\breve{x},  \breve{y} + t \breve{x}) \theta^i \theta^j  \\
& =  \frac{(c+1)}{1 - (c+1)^2 |\breve{x}|^2}\Bigg\{ \frac{\left( 1- (c+1)^2 |\breve{x}|^2 \right)|\theta|^2 + (c+1)^2 \langle \breve{x},\theta \rangle^2}{ \sqrt{\left( 1- (c+1)^2 |\breve{x}|^2 \right)|\breve{y}|^2 + t^2|\breve{x}|^2 }} \\
 &
 \quad \quad  \quad \quad  \quad \quad
  \quad \quad  \quad \quad  \quad \quad
 -   \frac{  [\left( 1- (c+1)^2 |\breve{x}|^2 \right) \langle \breve{y},\theta \rangle + t \langle \breve{x},\theta \rangle]^2 }{ [ \left( 1- (c+1)^2 |\breve{x}|^2 \right)|\breve{y}|^2 + t^2|\breve{x}|^2 ]^{\frac{3}{2}}  }
\Bigg\},
\end{split}
\end{equation}
from which we obtain
\begin{align} \label{K=-1Ex_dtPhi+}
\left.\frac{\dd }{{\dd} t}\right|_{t=0}\Big\{ [\Phi_{+}]_{y^i y^j}(\breve{x},  \breve{y} + t \breve{x}) \theta^i \theta^j \Big\}
 = - \frac{2(c+1)\langle \breve{x},\theta \rangle \langle \breve{y},\theta \rangle}{( 1 - (c+1)^2 |\breve{x}|^2)^{\frac{3}{2}}  |y|^3 } .
\end{align}
Replacing $c+1$ by $c-1$ and using \eqref{suffK=-1new1Randersnorm} gives the analogous expression for $\Phi_{-}$, and hence
\begin{align*}
\left.\frac{\dd }{{\dd} t}\right|_{t=0}\Big\{ F_{y^i y^j}(\breve{x},\breve{y} + t \breve{x}) \theta^i \theta^j \Big\}
 = \left\{ \frac{(c-1)}{( 1 - (c-1)^2 |\breve{x}|^2)^{\frac{3}{2}}} - \frac{(c+1)}{( 1 - (c+1)^2 |\breve{x}|^2)^{\frac{3}{2}}} \right\} \frac{\langle \breve{x},\theta \rangle \langle \breve{y},\theta \rangle}{ |y|^3 }.
\end{align*}
Together with \eqref{K=-1amazinggijthetaithetaj} and \eqref{K=-1ExFithetai}, this implies
\begin{align}\label{K=-1Ex_theta}
\langle \breve{x},\theta \rangle =0, \qquad \langle \breve{y},\theta \rangle =0.
\end{align}

Substituting \eqref{K=-1Ex_theta} into \eqref{K=-1Ex_Phi+ij} and  evaluating at $t=0$ yields
\begin{align*}
 [\Phi_{+}]_{y^i y^j}(\breve{x},\breve{y})   \theta^i \theta^j
 =  \frac{(c+1)|\theta|^2}{|\breve{y}|\sqrt{\left( 1- (c+1)^2 |\breve{x}|^2 \right)}}. \notag
\end{align*}
Again replacing $c+1$ by $c-1$ gives
\begin{align*}
 [\Phi_{-}]_{y^i y^j}(\breve{x},\breve{y}) \theta^i \theta^j
 =  \frac{(c-1)|\theta|^2}{|\breve{y}|\sqrt{\left( 1- (c-1)^2 |\breve{x}|^2 \right)}}. \notag
\end{align*}
From the definition of $F$ and \eqref{K=-1amazinggijthetaithetaj}, it follows that
\begin{align*}
 0=2F_{y^i y^j}(\breve{x},\breve{y}) \theta^i \theta^j
 = \left\{ \frac{(c+1)}{\sqrt{\left( 1- (c+1)^2 |\breve{x}|^2 \right)}} - \frac{(c-1)}{\sqrt{\left( 1- (c-1)^2 |\breve{x}|^2 \right)}}  \right\} \frac{|\theta|^2}{|\breve{y}|},
\end{align*}
which forces  $\theta = \mathbf{0}$, contradicting the assumption that $\theta \neq \mathbf{0}$.
Thus, no such point $\breve{x}$ exists, so $U = \Omega =\PD$.

The forward completeness now follows by Proposition~\ref{K=-1positive=complete}.
Furthermore, since $c > 1$, equation \eqref{thmK=-1eq2} gives $-F \le P \le F$. Moreover, because $\phi$ is not an odd function, Theorem~\ref{localstrucK} shows that $F$ is not a Hilbert metric, whence Theorem~\ref{globalK=-1} implies that $F$ is not backward complete.
\end{proof}



\vskip 10mm
\section{Projectively flat Finsler manifolds with   $\mathbf{K}=1$}\label{sectionK=1}

This section is devoted to the study of projectively flat Finsler manifolds with $\mathbf{K} = 1$. It is divided into two subsections. The first subsection derives an explicit formula for the distance function, which completes the proof of Theorem~\ref{DisInto} and leads to a maximum diameter theorem. The second subsection proves that the metric completion of such a manifold is a sphere, thus establishing Theorem~\ref{thmK=1globalintro}.

\subsection{Distance and diameter}


\begin{theorem}  \label{lemK=1dx_0xdx_0}
Let $(\Omega,F)$ be a projectively flat Finsler manifold with $\mathbf{K}=1$, where $\Omega\subset \mathbb{R}^n$ is a strictly convex domain containing the origin.
Then for any distinct points $x_1, x_2\in \Omega$, 
\begin{align}\label{eqK=1dx_0xdxx_0}
d_F(x_1,x_2) =  \arctan{   \frac{F^2(x_1, x_2-x_1)+ P^2(x_1, x_2-x_1)-P(x_1, x_2-x_1)}{F(x_1, x_2-x_1)}     }
 + \arctan{ \frac{P(x_1, x_2-x_1)}{ F(x_1, x_2-x_1) }}.
\end{align}
In particular, with $\psi(y):=F(\mathbf{0},y)$ and $\phi(y):= P(\mathbf{0},y) $, this implies
\begin{align}
d_F(\mathbf{0},x) &=   \arctan{   \frac{\psi^2(x)+ \phi^2(x)-\phi(x)}{\psi(x)}     }
 + \arctan{ \frac{\phi(x)}{ \psi(x) }}, \label{coreqK=1d0x}\\
 d_F(x,\mathbf{0}) &=   \arctan{   \frac{\psi^2(-x)+ \phi^2(-x)+\phi(-x)}{\psi(-x)}     }
 - \arctan{ \frac{\phi(-x)}{ \psi(-x) }}.\label{coreqK=1dx0}
\end{align}
\end{theorem}
\begin{proof}
Let $\gamma(t)$ be a unit-speed geodesic from $x\in \Omega$ with initial velocity $y \in S_{x}\Omega$. By Theorem \ref{lemgeodesicK} and $\mathbf{K}=1$,  we have $\gamma(t)=f(t)y+x$, where $f(t)$ satisfies  the initial condition \eqref{geodesicintial} and the ODE
\[
2 \left( \frac{f''(t)}{f'(t)}  \right)' - \left( \frac{f''(t)}{f'(t)}  \right)^2 = 4.
\]
Solving this equation yields
\begin{equation}\label{fK=1}
f(t) =  \cos^2c  \left[ \tan(  t + c) - \tan c \right],
\end{equation}
where   $c =c(x,y)  \in (-\frac{\pi}{2}, \frac{\pi}{2}) $ is a constant.

Now, for two distinct points $x_1,x_2\in \Omega$,
  the unit-speed geodesic $\gamma(t)$ from $x_1$ to $x_2$ is given by
 \begin{align}\label{K=1dis1}
 \ds\gamma(t)   =\cos^2c  \left[ \tan(  t + c) - \tan c \right] \frac{x_2-x_1}{F(x_1, x_2-x_1)} + x_1.
 \end{align}
Moreover, from \eqref{geodesicBackf''1} evaluated at $t=0$, we obtain
\begin{align} \label{K=1pfequlxx}
\tan c=  - \frac{P(x_1, x_2-x_1)}{F(x_1, x_2-x_1)}.
\end{align}

Since $\gamma$ has unit speed, equation \eqref{K=1dis1} implies that
\[
x_2=\gamma(d_F(x_1,x_2))=\cos^2c  \left[ \tan(  d_F(x_1,x_2) + c) - \tan c \right] \frac{x_2-x_1}{F(x_1, x_2-x_1)} + x_1,
\]
which yields
\[
\tan(  d_F(x_1,x_2) + c) - \tan c = \frac{F(x_1, x_2-x_1)}{\cos^2 c}.
\]
Combining this with \eqref{K=1pfequlxx} establishes
\eqref{eqK=1dx_0xdxx_0} and \eqref{coreqK=1d0x}.
The identity \eqref{coreqK=1dx0} is then derived by considering the reverse metric $\overleftarrow{F}$.
\end{proof}

\begin{remark}
By \eqref{eqK=1dx_0xdxx_0} and homogeneity, it is easy to verify that
\[
\lim_{x_1\rightarrow x_2}d_F(x_1,x_2)=0=\lim_{x_2\rightarrow x_1}d_F(x_1,x_2).
\]
\end{remark}

\begin{proof}[Proof of Theorem \ref{DisInto}] The theorem is now evident from Theorems \ref{K=0lemdx_0xdxx_0}, \ref{lemK=-1dx_0xdx_0} and \ref{lemK=1dx_0xdx_0}.
\end{proof}

The {\it diameter} of a Finsler manifold $(M,F)$ is defined by
\[
\diam_F(M):=\sup\big\{ d_F(x_1,x_2)\,:\, x_1,x_2\in M \big\}.
\]
By the classical results of Cheng~\cite{Ch} and Toponogov~\cite{To}, the diameter of a complete Riemannian manifold with $\mathbf{K}=1$ is at most $\pi$, with equality if and only if
the manifold is isometric to the unit Euclidean sphere.  In the setting of projectively flat Finsler manifolds, we
prove the following  maximum diameter theorem.

\begin{theorem}\label{thmK=1global}
Let $(\Omega,F)$ be a projectively flat  Finsler manifold with   $\mathbf{K} = 1$, where $\Omega\subset \mathbb{R}^n$ is a convex domain containing the origin.
Then the following hold:
\begin{enumerate}[\rm (i)]
\item\label{diam1spher1} $\diam_F(\Omega)\leq \pi$;

\item\label{diam1spher2} $\Omega=\mathbb{R}^n$ if and only if  every straight line in $\Omega$ has length $\pi$; in this case, $\diam_F(M)=\pi$.
\end{enumerate}
\end{theorem}
\begin{proof}
Let $\gamma_{x,y}(t)$ be the unit-speed geodesic with $\gamma_{x,y}(0)=x$ and $\gamma_{x,y}'(0)=y \in S_{x}\Omega$.
By equation \eqref{fK=1}, we have
\begin{equation}\label{equgammxyK=1}
\gamma_{x,y}(t)=x+f(t)y  =x+\cos^2c  \left[ \tan(  t + c) - \tan c \right]y.
\end{equation}
Let $I = (a, b)$ be the maximal interval of existence for $\gamma_{x,y}$.
 Since $0 \in I$, it follows that $(a,b)\subset (-\frac{\pi}{2 } - c,\frac{\pi}{2 } - c)$,
 and hence the length of $\gamma_{x,y}$ satisfies
 \[
 L_F(\gamma_{x,y})\leq b-a\leq  \left(\frac{\pi}2-c\right)-\left(-\frac{\pi}{2 } - c\right)=\pi.
 \]
 As $x$ and $y$ are arbitrary, this proves part \eqref{diam1spher1}.

Moreover, a direct argument based on \eqref{equgammxyK=1} shows that $\Omega=\mathbb{R}^n$ is equivalent to
 \[
 \lim_{t\rightarrow a^+ }|\gamma_{x,y}(t)-x|= \lim_{t\rightarrow a^+}|f(t)||y|=+\infty,\qquad \lim_{t\rightarrow b^- }|\gamma_{x,y}(t)-x|= \lim_{t\rightarrow b^- }|f(t)||y|=+\infty,
 \]
 for any $(x,y)\in S\Omega$,
which implies  $a=-\frac{\pi}{2 } - c$, $b=\frac{\pi}{2 } - c$ and
 $L_F(\gamma_{x,y})=  \pi$. This establishes part \eqref{diam1spher2}.
\end{proof}



By the Bonnet--Myers theorem, any forward complete Finsler manifold with $\mathbf{K} = 1$ is compact and has diameter $\le \pi$.
Hence, a projectively flat Finsler manifold $(\mathbb{R}^n, F)$ with $\mathbf{K} = 1$ cannot be forward complete.
Theorem~\ref{thmK=1global}  yields the maximal diameter $\pi$, and we next show its completion is a sphere under a natural condition.

\subsection{Completion of    manifolds}

Let  $\mathfrak{p}: \mathbb{R}^n\rightarrow \mathbb{S}^n_+\subset \mathbb{R}^{n+1}$ denote a diffeomorphism  defined by
\begin{equation}\label{projeticpm}
\mathfrak{p}(x):=\left( \frac{x}{\sqrt{1+|x|^2}}, \frac{  1}{\sqrt{1+|x|^2}}  \right)= \frac{x^\alpha e_\alpha}{\sqrt{1+|x|^2}}+\frac{  e_{n+1}}{\sqrt{1+|x|^2}}.
\end{equation}
This map identifies $\mathbb{R}^n$ with the open upper hemisphere $\mathbb{S}^n_+ := \{ w \in \mathbb{S}^n : w^{n+1} > 0 \}$. Hereafter,  $\{ e_i\}_{i=1}^{n+1}$ is the standard orthonormal basis of $\mathbb{R}^{n+1}$,
 where $\{ e_\alpha \}_{\alpha=1}^n$ forms the standard orthonormal basis of $\mathbb{R}^n$.



In the following, we identify the sphere $\mathbb{S}^n$ with its standard embedding in $\mathbb{R}^{n+1}$. Under this identification, a tangent vector $V \in T_{w}\mathbb{S}^n$ is represented by a vector in $\mathbb{R}^{n+1}$ orthogonal to $w$. This representation allows us to parallel translate $V$ to the antipodal point $-w$, which is denoted by $(-w, V)$. In addition, $\partial\mathbb{S}^n_+$ denotes the {\it equator} of $\mathbb{S}^n$.

\begin{lemma}\label{spherelemma1}
Let $w=w^ie_i\in  \partial \mathbb{S}^n_+$ and $V=V^ie_i\in T_w\mathbb{S}^n\subset \mathbb{R}^{n+1}$ with $V^{n+1} <0$. Then there exists $(x,y)\in \mathbb{R}^n\times \mathbb{R}^n$  such that the constant-speed  geodesic $\gamma_{x,y}(t)$ with $\gamma_{x,y}(0)=x$ and $\gamma_{x,y}'(0)=y$ satisfies
\begin{equation}\label{tangecrivle22}
\lim_{t\rightarrow \pm \frac{\pi}{2 F(x,y)}-\frac{c}{F(x,y)}}\mathfrak{p}(\gamma_{x,y}(t))=\pm w, \qquad \lim_{t\rightarrow
\pm \frac{\pi}{2 F(x,y)}-\frac{c}{F(x,y)}}\mathfrak{p}_{*\gamma_{x,y}(t)}\bigg(    \left(\frac{ \cos c}{F(x,y)}\right)^2  \gamma'_{x,y}(t) \bigg)=\pm V,
\end{equation}
\begin{multicols}{2}
\noindent where $c=c(x,y/F(x,y))$ is defined in \eqref{fK=1} and  ${\Big(}{-}\frac{\pi+2c}{2F(x,y)},\frac{\pi-2c}{2F(x,y)}\Big)$ is the maximal domain of $\gamma_{x,y}(t)$.
\vskip 2mm
Moreover, the set of all $(x, y)$ satisfying \eqref{tangecrivle22} forms a one-dimensional linear space given by
\begin{equation}\label{solutonxy}
 (x,y)=\left(\frac{V^\alpha e_\alpha}{  V^{n+1} }-\lambda \frac{w^\alpha e_\alpha}{  V^{n+1} }, \ -\frac{w^\alpha e_\alpha}{  V^{n+1} } \right),
\end{equation}
for $\lambda\in \mathbb{R}$.
In particular, all such geodesics $\gamma_{x,y}$ have the same image, which is a straight line.
\begin{center}
\tikzset{every picture/.style={line width=0.75pt}} 

\begin{tikzpicture}[x=0.75pt,y=0.75pt,yscale=-1,xscale=1, scale=0.9]

\draw  [draw opacity=0][dash pattern={on 4.5pt off 4.5pt}] (83.64,96.97) .. controls (98.09,75.08) and (122.15,60.82) .. (149.29,61.01) .. controls (174.21,61.19) and (196.37,73.51) .. (210.77,92.63) -- (148.7,144.9) -- cycle ;
\draw  [dash pattern={on 4.5pt off 4.5pt}] (83.64,96.97) .. controls (98.09,75.08) and (122.15,60.82) .. (149.29,61.01) .. controls (174.21,61.19) and (196.37,73.51) .. (210.77,92.63) ;
\draw  [draw opacity=0] (228.04,140.04) .. controls (224.8,154.52) and (190.77,165.94) .. (149.26,166) .. controls (105.72,166.06) and (70.38,153.59) .. (70.15,138.12) -- (149.22,137.87) -- cycle ; \draw   (228.04,140.04) .. controls (224.8,154.52) and (190.77,165.94) .. (149.26,166) .. controls (105.72,166.06) and (70.38,153.59) .. (70.15,138.12) ;
\draw  [draw opacity=0][dash pattern={on 4.5pt off 4.5pt}] (227.64,136.66) .. controls (219.26,123.06) and (187.76,112.98) .. (150.21,113.04) .. controls (114.32,113.09) and (83.99,122.37) .. (74.09,135.08) -- (150.25,143.99) -- cycle ;
\draw  [dash pattern={on 4.5pt off 4.5pt}] (227.64,136.66) .. controls (219.26,123.06) and (187.76,112.98) .. (150.21,113.04) .. controls (114.32,113.09) and (83.99,122.37) .. (74.09,135.08) ;
\draw   (82.41,15) -- (282.82,15) -- (204.74,97) -- (4.33,97) -- cycle ;
\draw  [draw opacity=0] (124.89,165.95) .. controls (98.62,135.72) and (85.82,104.04) .. (93.07,86.8) -- (167.74,149.3) -- cycle ; \draw   (124.89,165.95) .. controls (98.62,135.72) and (85.82,104.04) .. (93.07,86.8) ;
\draw  [draw opacity=0][dash pattern={on 4.5pt off 4.5pt}] (171.45,115.16) .. controls (140.55,85.03) and (110.8,68.81) .. (100.47,77.92) .. controls (99.21,79.04) and (98.29,80.48) .. (97.69,82.22) -- (166.23,152.47) -- cycle ;
\draw  [dash pattern={on 4.5pt off 4.5pt}] (171.45,115.16) .. controls (140.55,85.03) and (110.8,68.81) .. (100.47,77.92) .. controls (99.21,79.04) and (98.29,80.48) .. (97.69,82.22) ;
\draw    (61.33,78) -- (105.33,33) ;
\draw [shift={(83.33,55.5)}, rotate = 314.36] [color={rgb, 255:red, 0; green, 0; blue, 0 }  ][fill={rgb, 255:red, 0; green, 0; blue, 0 }  ][line width=0.75]      (0, 0) circle [x radius= 1.34, y radius= 1.34]   ;
\draw    (125.33,166) -- (146.04,190.47) ;
\draw [shift={(147.33,192)}, rotate = 229.76] [color={rgb, 255:red, 0; green, 0; blue, 0 }  ][line width=0.75]    (10.93,-3.29) .. controls (6.95,-1.4) and (3.31,-0.3) .. (0,0) .. controls (3.31,0.3) and (6.95,1.4) .. (10.93,3.29)   ;
\draw    (81.83,56.5) -- (62.71,76.55) ;
\draw [shift={(61.33,78)}, rotate = 313.64] [color={rgb, 255:red, 0; green, 0; blue, 0 }  ][line width=0.75]    (10.93,-3.29) .. controls (6.95,-1.4) and (3.31,-0.3) .. (0,0) .. controls (3.31,0.3) and (6.95,1.4) .. (10.93,3.29)   ;
\draw  [dash pattern={on 4.5pt off 4.5pt}]  (43.83,96) -- (122.83,15) ;
\draw  [draw opacity=0] (70.15,138.12) .. controls (71.03,122.85) and (75.78,108.7) .. (83.35,96.72) -- (149.3,143.28) -- cycle ;
\draw   (70.15,138.12) .. controls (71.03,122.85) and (75.78,108.7) .. (83.35,96.72) ;
\draw  [draw opacity=0] (210.31,92.02) .. controls (221.04,106) and (227.6,123.71) .. (227.99,143) -- (148.7,144.9) -- cycle ;
\draw   (210.31,92.02) .. controls (221.04,106) and (227.6,123.71) .. (227.99,143) ;
\draw  [dash pattern={on 4.5pt off 4.5pt}]  (171.45,115.16) -- (124.89,165.95) ;
\draw [shift={(124.89,165.95)}, rotate = 132.51] [color={rgb, 255:red, 0; green, 0; blue, 0 }  ][fill={rgb, 255:red, 0; green, 0; blue, 0 }  ][line width=0.75]      (0, 0) circle [x radius= 1.34, y radius= 1.34]   ;
\draw [shift={(148.17,140.55)}, rotate = 132.51] [color={rgb, 255:red, 0; green, 0; blue, 0 }  ][fill={rgb, 255:red, 0; green, 0; blue, 0 }  ][line width=0.75]      (0, 0) circle [x radius= 1.34, y radius= 1.34]   ;

\draw (152,185) node [anchor=north west][inner sep=0.75pt]    {$V$};
\draw (70,40) node [anchor=north west][inner sep=0.75pt]    {$x$};
\draw (47,60) node [anchor=north west][inner sep=0.75pt]    {$y$};
\draw (55,165) node [anchor=north west][inner sep=0.75pt]    {\qquad\quad $w$};
\draw (150,135) node [anchor=north west][inner sep=0.75pt]    {$O$};
\end{tikzpicture}
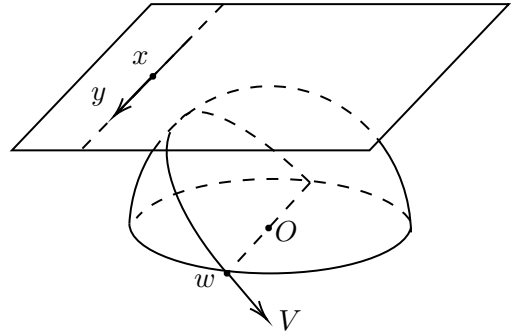
\caption{\small The geometric interpretation of Lemma \ref{spherelemma1}.}
\label{K=1fig1}
\end{center}
\end{multicols}

\end{lemma}

\begin{proof}
By \eqref{expon} and \eqref{fK=1}, the geodesic
 $\gamma_{x,y}(t)$ admits the representation
\begin{align*}
\gamma_{x,y}(t)&= \, \exp_x(ty)=\exp_x\left(F(x,y)t\frac{y}{F(x,y)}\right)=x+f(F(x,y)t)\frac{y}{F(x,y)}\\
& = \,x+ y\frac{\cos^2c  \left[ \tan( F(x,y) t + c) - \tan c \right]}{F(x,y)}=:x+h(t)y,
\end{align*}
where $c=c(x,y/F(x,y))$ is a constant in $(-\pi/2,\pi/2)$ and particularly, $  {\Big(}{-}\frac{\pi+2c}{2F(x,y)},\frac{\pi-2c}{2F(x,y)}\Big)$ is the maximal domain of $\gamma_{x,y}(t)$.
From \eqref{projeticpm}, we obtain
\begin{equation}\label{expssionPlim}
\mathfrak{p}(\gamma_{x,y}(t))=\left( \frac{x+h(t)y}{\sqrt{1+|x+h(t)y|^2}}, \frac{1}{\sqrt{1+|x+h(t)y|^2}}   \right),
\end{equation}
and consequently,
\begin{align}\label{expressionP}
&\mathfrak{p}_{*\gamma_{x,y}(t)}( \gamma'_{x,y}(t)   )\notag\\
=& \left( \frac{(1+|x+h(t)y|^2)h'(t)y-\langle x+h(t)y,h'(t)y   \rangle(x+h(t)y) }{(1+|x+h(t)y|^2)^{3/2}}, - \frac{ \langle x+h(t)y,h'(t)y    \rangle  }{(1+|x+h(t)y|^2)^{3/2}}   \right).
\end{align}

To find $(x, y)$ satisfying \eqref{tangecrivle22}, we solve the limiting equations:
\begin{align}
\lim_{t\rightarrow   \frac{\pi-2c}{2F(x,y)}}\mathfrak{p}(\gamma_{x,y}(t))&= \left( \frac{y}{|y|},0 \right)=  w,\label{tangecrivle1}\\
 \lim_{t\rightarrow \frac{\pi-2c}{2F(x,y)}}\mathfrak{p}_{*\gamma_{x,y}(t)}\bigg( \left(\frac{ \cos c}{F(x,y)}\right)^2 \gamma'_{x,y}(t)  \bigg)&= \left( \frac{\langle x,y\rangle y-|y|^2x}{|y|^3 }, \frac{-1}{|y| } \right)=  V.\label{tangecrivle2}
\end{align}
Writing $w = (w^\alpha, 0)$ and $V = (V^\alpha, V^{n+1})$, equations \eqref{tangecrivle1} and \eqref{tangecrivle2} imply
\[
V^\alpha e_\alpha=\frac{\langle x,y\rangle y-|y|^2x}{|y|^3},\quad V^{n+1}=-\frac{1}{|y| }<0, \quad w\perp V \  \Leftrightarrow \  y\perp V^\alpha e_\alpha.
\]
From \eqref{tangecrivle1}, we deduce
\begin{equation}\label{expression y}
 y=y^\alpha e_\alpha=|y|w^\alpha e_\alpha=-\frac{w^\alpha e_\alpha}{V^{n+1}}, \qquad  |y|=-\frac{1}{   V^{n+1} }\neq 0.
\end{equation}
Equation \eqref{tangecrivle2} yields $\langle x,y\rangle y^\alpha-|y|^2x^\alpha=|y|^3  V^\alpha$, which forms a linear system in $x$:
\begin{equation}\label{linerasys}
A^\alpha_\beta x^\beta= |y|^3 V^\alpha, \quad 1\leq \alpha,\beta\leq n,
\end{equation}
where $A:=(A^\alpha_\beta):=(\delta_{\beta\iota}y^\iota y^\alpha-|y|^2\delta^\alpha_\beta)$.
A direct computation shows that $\text{rank}(A)=n-1$ and $A y=0$, so the solution space $\mathcal {S}$ of \eqref{linerasys} is one-dimensional
and parallel to $y$.
Since $y \perp V^\alpha e_\alpha$, the vector $\tilde{x} := -|y| V^\alpha e_\alpha$ satisfies \eqref{linerasys}. Combined with \eqref{expression y}, we obtain
 \[
 x\in \mathcal {S}=\Big\{ \tilde{x}+\lambda y\,: \, \lambda\in \mathbb{R} \Big\}=\left\{ \frac{(V^\alpha-\lambda w^\alpha)e_\alpha}{V^{n+1}} \,:\, \lambda\in \mathbb{R} \right\}.
 \]
which establishes \eqref{solutonxy}.
For any such $(x, y)$, the image of $\gamma_{x,y}$ coincides with $\mathcal{S}$. Finally, \eqref{tangecrivle22} follows from \eqref{expssionPlim}--\eqref{tangecrivle2}.
\end{proof}

\begin{remark}
Geometrically, for each $w \in \partial \mathbb{S}^n_+$ and $V \in T_w\mathbb{S}^n$ with $V^{n+1} < 0$, there exists a ``unique'' straight line $\gamma_{x,y}$ whose projection $\mathfrak{p}(\gamma_{x,y})$ is the upper half of a great circle, and whose tangent vectors at $w$ and $-w$ are $V$ and $-V$, respectively. See Figure~\ref{K=1fig1} for an illustration.
\end{remark}

The maximum diameter theorem in the Riemannian setting  states that a complete Riemannian manifold with  sectional curvature $1$ and diameter $\pi$ must be a sphere. In view of Theorem \ref{thmK=1global}, a natural question arises:

\medskip
\textit{ Is the completion of a projectively flat Finsler manifold $(\mathbb{R}^n,F)$ with $\mathbf{K}=1$ also a sphere? }
\medskip

\noindent The answer is affirmative under a natural necessary condition.

\begin{multicols}{2}
To establish this, we recall the  standard spherical coordinates $(\zeta^i)=(\varphi, \theta^1, \theta^2, ..., \theta^{n-1})$    on $\mathbb{S}^n$.
For $\mathfrak{p}(x)=(\varphi, \theta^1, \theta^2, ..., \theta^{n-1}) \in \mathbb{S}^{n}_+$, we have
  $\sqrt{1+|x|^2}  = \sec \varphi$ and

\begin{align}\label{K=1spherical_x}
\begin{cases}
x^1 \circ \mathfrak{p}^{-1} = \tan\varphi \cos\theta^{1}, \\
\quad\vdots \\
x^i \circ \mathfrak{p}^{-1} = \tan\varphi \cos\theta^i \prod_{\alpha=1}^{i-1}\sin\theta^\alpha, \\
\phantom{x^i \circ \mathfrak{p}^{-1} = {}}\mathllap{\scriptstyle 2 \leq i \leq n-1,} \\
\quad\vdots \\
x^{n} \circ \mathfrak{p}^{-1} = \tan\varphi \sin\theta^{n-1} \prod_{\alpha=1}^{n-2}\sin\theta^{\alpha},
\end{cases}
\end{align}
where  $\theta^{s}\in(0, \pi)$ for $s=1,\cdots,n-2$, $\theta^{n-1}\in[0, \pi)$, and $\varphi \in (-\frac{\pi}{2}, \frac{\pi}{2} )$.
\begin{center}
\tikzset{every picture/.style={line width=0.75pt}} 
\begin{tikzpicture}[x=0.75pt,y=0.75pt,yscale=-1,xscale=1, scale=0.8]

\draw   (228.15,115.84) -- (479.95,115.14) -- (423.12,158.99) -- (171.31,159.7) -- cycle ;
\draw  [draw opacity=0] (363.87,161.2) .. controls (384.95,177.11) and (398.68,202.01) .. (399.08,230.18) .. controls (399.77,279.05) and (360.06,319.24) .. (310.38,319.94) .. controls (260.71,320.64) and (219.87,281.59) .. (219.18,232.73) .. controls (218.79,205.3) and (231.14,180.6) .. (250.82,164.09) -- (309.13,231.45) -- cycle ; \draw   (363.87,161.2) .. controls (384.95,177.11) and (398.68,202.01) .. (399.08,230.18) .. controls (399.77,279.05) and (360.06,319.24) .. (310.38,319.94) .. controls (260.71,320.64) and (219.87,281.59) .. (219.18,232.73) .. controls (218.79,205.3) and (231.14,180.6) .. (250.82,164.09) ;
\draw  [draw opacity=0][dash pattern={on 4.5pt off 4.5pt}] (249.43,165.69) .. controls (265.43,152.04) and (286.44,143.79) .. (309.42,143.86) .. controls (327.88,143.92) and (345.03,149.34) .. (359.3,158.6) -- (309.13,231.45) -- cycle ; \draw  [dash pattern={on 4.5pt off 4.5pt}] (249.43,165.69) .. controls (265.43,152.04) and (286.44,143.79) .. (309.42,143.86) .. controls (327.88,143.92) and (345.03,149.34) .. (359.3,158.6) ;
\draw  [draw opacity=0] (399.42,234.83) .. controls (396.98,243.13) and (358.14,249.95) .. (310.24,250.28) .. controls (260.48,250.62) and (219.79,243.81) .. (219.35,235.08) .. controls (219.29,233.93) and (219.94,232.8) .. (221.22,231.71) -- (309.44,234.47) -- cycle ; \draw   (399.42,234.83) .. controls (396.98,243.13) and (358.14,249.95) .. (310.24,250.28) .. controls (260.48,250.62) and (219.79,243.81) .. (219.35,235.08) .. controls (219.29,233.93) and (219.94,232.8) .. (221.22,231.71) ;
\draw  [draw opacity=0][dash pattern={on 4.5pt off 4.5pt}] (219.66,233.95) .. controls (222,225.34) and (260.79,218.32) .. (308.65,218.07) .. controls (358.3,217.82) and (398.9,224.96) .. (399.33,234.01) .. controls (399.39,235.16) and (398.79,236.29) .. (397.61,237.38) -- (309.44,234.47) -- cycle ; \draw  [dash pattern={on 4.5pt off 4.5pt}] (219.66,233.95) .. controls (222,225.34) and (260.79,218.32) .. (308.65,218.07) .. controls (358.3,217.82) and (398.9,224.96) .. (399.33,234.01) .. controls (399.39,235.16) and (398.79,236.29) .. (397.61,237.38) ;
\draw  [draw opacity=0] (372.41,169.9) .. controls (367.98,174.36) and (340.14,177.7) .. (306.13,177.6) .. controls (278.57,177.52) and (254.65,175.2) .. (243.79,171.93) -- (305.3,168.33) -- cycle ; \draw   (372.41,169.9) .. controls (367.98,174.36) and (340.14,177.7) .. (306.13,177.6) .. controls (278.57,177.52) and (254.65,175.2) .. (243.79,171.93) ;
\draw  [draw opacity=0][dash pattern={on 4.5pt off 4.5pt}] (246.52,167.45) .. controls (258.2,164.71) and (281.18,163.01) .. (307.67,163.28) .. controls (338.06,163.6) and (363.99,166.42) .. (373.63,170.08) -- (307.97,172.43) -- cycle ; \draw  [dash pattern={on 4.5pt off 4.5pt}] (246.52,167.45) .. controls (258.2,164.71) and (281.18,163.01) .. (307.67,163.28) .. controls (338.06,163.6) and (363.99,166.42) .. (373.63,170.08) ;
\draw  [dash pattern={on 4.5pt off 4.5pt}]  (373.63,170.08) -- (307.5,234.8) ;
\draw [shift={(307.5,234.8)}, rotate = 135.62] [color={rgb, 255:red, 0; green, 0; blue, 0 }  ][fill={rgb, 255:red, 0; green, 0; blue, 0 }  ][line width=0.75]      (0, 0) circle [x radius= 1.34, y radius= 1.34]   ;
\draw    (386.28,158.17) -- (373.63,170.08) ;
\draw [shift={(373.63,170.08)}, rotate = 136.74] [color={rgb, 255:red, 0; green, 0; blue, 0 }  ][fill={rgb, 255:red, 0; green, 0; blue, 0 }  ][line width=0.75]      (0, 0) circle [x radius= 1.34, y radius= 1.34]   ;
\draw  [dash pattern={on 4.5pt off 4.5pt}]  (398.33,146) -- (386.28,158.17) ;
\draw [shift={(398.33,146)}, rotate = 134.72] [color={rgb, 255:red, 0; green, 0; blue, 0 }  ][fill={rgb, 255:red, 0; green, 0; blue, 0 }  ][line width=0.75]      (0, 0) circle [x radius= 1.34, y radius= 1.34]   ;
\draw  [dash pattern={on 4.5pt off 4.5pt}]  (307.41,135.16) -- (307.42,233.17) ;
\draw    (307.34,82) -- (307.41,143.16) ;
\draw [shift={(307.41,143.16)}, rotate = 89.93] [color={rgb, 255:red, 0; green, 0; blue, 0 }  ][fill={rgb, 255:red, 0; green, 0; blue, 0 }  ][line width=0.75]      (0, 0) circle [x radius= 1.34, y radius= 1.34]   ;
\draw [shift={(307.33,80)}, rotate = 89.93] [color={rgb, 255:red, 0; green, 0; blue, 0 }  ][line width=0.75]    (4.37,-1.32) .. controls (2.78,-0.56) and (1.32,-0.12) .. (0,0) .. controls (1.32,0.12) and (2.78,0.56) .. (4.37,1.32)   ;
\draw  [draw opacity=0] (306.3,210.53) .. controls (310.05,207.91) and (314.29,207.03) .. (317.99,208.53) .. controls (321.36,209.89) and (323.67,212.98) .. (324.71,216.91) -- (310.82,226.21) -- cycle ; \draw   (306.3,210.53) .. controls (310.05,207.91) and (314.29,207.03) .. (317.99,208.53) .. controls (321.36,209.89) and (323.67,212.98) .. (324.71,216.91) ;

\draw (314,188) node [anchor=north west][inner sep=0.75pt]   [align=left] {$\varphi$};
\draw (400,125) node [anchor=north west][inner sep=0.75pt]   [align=left] {($x$, $1$)};

\draw (380,160) node [anchor=north west][inner sep=0.75pt]   [align=left] {$\mathfrak{p}(x)$};
\draw (234,196) node [anchor=north west][inner sep=0.75pt]   [align=left] {$\mathbb{S}^n$};
\draw (310.82,226.21) node [anchor=north west][inner sep=0.75pt]   [align=left] {O};
\draw (200,140) node [anchor=north west][inner sep=0.75pt]   [align=left] {$\mathbb{R}^n$};
\draw (314,73) node [anchor=north west][inner sep=0.75pt]   [align=left] {$x^{n+1}$};
\end{tikzpicture}
\caption{\small The map $\mathfrak{p}(x)=(\varphi, \theta^1, \theta^2, ..., \theta^{n-1})$. }
\end{center}
\end{multicols}

For a point in $\mathbb{S}^{n}_+$ ($n\geq 3$) with  $\theta^s=0$ or $\pi$ for some $1\leq s\leq n-2$, its spherical coordinates are not unique.
We exclude such degenerate cases in what follows to simplify the presentation. We also require the following observation, whose proof is deferred to Appendix \ref{propergenerlengappexK=-1}.

\begin{observation}\label{informobvers2} The following statements hold:
\begin{enumerate}[\rm (1)]
\item\label{ob21}  the image of any  straight line in $\mathbb{R}^n$ under $\mathfrak{p}$ is the upper half of some great circle in $\mathbb{S}^n$;

\item \label{ob22}
 any sequence of great circles in $\mathbb{S}^n$ (parametrized  on the same interval)
 admits a uniformly convergent subsequence with respect to the standard sphere metric, and the limit is again a great circle.
\end{enumerate}

\end{observation}


\begin{theorem}\label{plancetospereK=1}
Let $F$ be a projectively flat Finsler metric  on $\mathbb{R}^n$ with $\mathbf{K}=1$. If the pull-back metric $(\mathfrak{p}^{-1})^*F$ on $\mathbb{S}^n_+$ can be extended to a Finsler metric on $\overline{\mathbb{S}^n_+}$, there exists a locally projectively flat Finsler metric $\mathscr{F}$ on $\mathbb{S}^n$  such that:
\begin{enumerate}[\rm (i)]

\item\label{sphere1} $(\mathbb{S}^n,\mathscr{F})$ is a complete Finsler manifold with $\mathbf{K}=1$ and $\mathscr{F}|_{\mathbb{S}^n_+}=\mathfrak{p}^*F$;

\item\label{sphere2} $\mathscr{F}(w,V)=\mathscr{F}(-w,-V)$ for all $w\in \mathbb{S}^n$ and $V\in T\mathbb{S}^n\subset \mathbb{R}^{n+1}$;

 \item\label{sphere3}  every great circle is a closed geodesic of length $2\pi$, and  the injectivity radius is $\pi$;

 \item\label{sphere4}  every geodesic in  $(\mathbb{S}^n,\mathscr{F})$ is contained in some great circle.

  \end{enumerate}
\end{theorem}

\begin{proof} The proof proceeds in three steps.

\smallskip
\noindent
{\bf Step 1.} In this step, we construct the required metric $\mathscr{F}$ on $\mathbb{S}^n$.

\smallskip

By assumption,
the pull-back metric $\mathscr{F}_+:=(\mathfrak{p}^{-1})^*F$ extends to $\overline{\mathbb{S}^n_+}$.
We first prove that for any antipodal points $\pm w \in \partial \mathbb{S}^n_+$ and any $V \in T_w\mathbb{S}^n \subset \mathbb{R}^{n+1}$,
\begin{equation}\label{pmyvaluem}
\mathscr{F}_+(w, V) = \mathscr{F}_+(-w, -V).
\end{equation}

First, suppose $V^{n+1} \neq 0$. If $V^{n+1} < 0$, choose a geodesic $\gamma_{x,y}$ as in Lemma~\ref{spherelemma1}. Since $\gamma_{x,y}$ has constant speed, for any $t \in \left( -\frac{\pi + 2c}{2F(x,y)}, \frac{\pi - 2c}{2F(x,y)} \right)$ we have
\[
\mathscr{F}_+\left( \mathfrak{p}(\gamma_{x,y}(t)), \mathfrak{p}_{*\gamma_{x,y}(t)} \left( \left( \frac{\cos c}{F(x,y)} \right)^2 \gamma'_{x,y}(t) \right) \right) = \left( \frac{\cos c}{F(x,y)} \right)^2 F(\gamma_{x,y}(t), \gamma'_{x,y}(t)) = \frac{\cos^2 c}{F(x,y)},
\]
and taking the limit as $t$ approaches the endpoints yields \eqref{pmyvaluem} via \eqref{tangecrivle22}.

If $V^{n+1} > 0$, set $\widetilde{V} := -V$ and $\tilde{w} := -w$. Then $\widetilde{V}^{n+1} < 0$, and the same argument gives $\mathscr{F}_+(\tilde{w}, \widetilde{V}) = \mathscr{F}_+(-\tilde{w}, -\widetilde{V})$, which is equivalent to \eqref{pmyvaluem}.

Now suppose $V^{n+1} = 0$. Choose a sequence $(V_\alpha)_\alpha \subset T_w\mathbb{S}^n$ converging to $V$ with $V_\alpha^{n+1} \neq 0$. By the continuity of $\mathscr{F}_+$, we have
\[
\mathscr{F}_+(w, V) = \lim_{\alpha \to +\infty} \mathscr{F}_+(w, V_\alpha) = \lim_{\alpha \to +\infty} \mathscr{F}_+(-w, -V_\alpha) = \mathscr{F}_+(-w, -V).
\]
\begin{multicols}{2}
\noindent
This establishes \eqref{pmyvaluem} in all cases.

We now construct $\mathscr{F}$ on $\mathbb{S}^n$. Take two copies $\overline{\mathbb{S}^n_{+\ell}}$, $\ell = 1, 2$, each equipped with the metric $\mathscr{F}_+$, and glue them along their boundaries by identifying $w \in \partial\mathbb{S}^{n}_{+1}$ with $-w \in \partial\mathbb{S}^{n}_{+2}$ (see Figure~\ref{K=1fig2}). By \eqref{pmyvaluem}, this yields a well-defined locally projectively flat Finsler metric $\mathscr{F}$ on $\mathbb{S}^{n}$. Since $\mathbb{S}^n$ is compact, $(\mathbb{S}^n, \mathscr{F})$ is complete. It follows directly that $\mathbf{K} = 1$ and $\mathfrak{p} : (\mathbb{R}^n, F) \to (\mathbb{S}^n_+, \mathscr{F})$ is an isometry. Moreover, the construction implies
\begin{equation}\label{globalmetric}
\mathscr{F}(w, V) = \mathscr{F}(-w, -V), \quad \forall (w, V) \in T\mathbb{S}^{n},
\end{equation}
which proves Statements \eqref{sphere1} and \eqref{sphere2}.
\vskip 2mm
Let $(\zeta^i) = (\varphi, \theta) = (\varphi, \theta^1, \ldots, \theta^{n-1})$  be the standard spherical coordinates on $\mathbb{S}^n$, and let $(\zeta^i, V^j)$ be the induced coordinates on $T\mathbb{S}^n$, where $V = V^j \frac{\partial}{\partial \zeta^j}$.

\begin{center}
\tikzset{every picture/.style={line width=0.75pt}} 

\begin{tikzpicture}[x=0.75pt,y=0.75pt,yscale=-1,xscale=1, scale=0.8]

\draw   (205.33,232.5) .. controls (205.33,214.27) and (249.21,199.5) .. (303.33,199.5) .. controls (357.46,199.5) and (401.33,214.27) .. (401.33,232.5) .. controls (401.33,250.73) and (357.46,265.5) .. (303.33,265.5) .. controls (249.21,265.5) and (205.33,250.73) .. (205.33,232.5) -- cycle ;
\draw  [draw opacity=0] (202.37,132.02) .. controls (203.28,76.84) and (245.82,32.06) .. (298.82,31.3) .. controls (352.08,30.54) and (396.04,74.5) .. (398.15,130) -- (300.29,134) -- cycle ; \draw   (202.37,132.02) .. controls (203.28,76.84) and (245.82,32.06) .. (298.82,31.3) .. controls (352.08,30.54) and (396.04,74.5) .. (398.15,130) ;
\draw  [draw opacity=0][dash pattern={on 4.5pt off 4.5pt}] (202.36,134) .. controls (204.09,114.77) and (246.94,98.82) .. (299.78,98.07) .. controls (352.2,97.32) and (395.22,111.79) .. (398.14,130.74) -- (300.3,134) -- cycle ; \draw  [dash pattern={on 4.5pt off 4.5pt}] (202.36,134) .. controls (204.09,114.77) and (246.94,98.82) .. (299.78,98.07) .. controls (352.2,97.32) and (395.22,111.79) .. (398.14,130.74) ;
\draw  [draw opacity=0] (202.41,131.14) .. controls (204.86,148.46) and (248.04,161.69) .. (300.73,160.94) .. controls (354.23,160.17) and (397.52,145.27) .. (398.22,127.53) -- (300.26,128.33) -- cycle ; \draw   (202.41,131.14) .. controls (204.86,148.46) and (248.04,161.69) .. (300.73,160.94) .. controls (354.23,160.17) and (397.52,145.27) .. (398.22,127.53) ;
\draw  [dash pattern={on 4.5pt off 4.5pt}]  (267.76,161.33) -- (327.76,98.33) ;
\draw [shift={(297.76,129.83)}, rotate = 313.6] [color={rgb, 255:red, 0; green, 0; blue, 0 }  ][fill={rgb, 255:red, 0; green, 0; blue, 0 }  ][line width=0.75]      (0, 0) circle [x radius= 1.34, y radius= 1.34]   ;
\draw    (255.33,43) .. controls (226.67,60) and (268.33,158) .. (267.76,161.33) ;
\draw  [dash pattern={on 4.5pt off 4.5pt}]  (255.33,43) .. controls (280.33,31) and (328.33,97) .. (327.76,98.33) ;
\draw  [draw opacity=0] (205.51,237.28) .. controls (208.03,290.73) and (251.79,332.92) .. (304.76,332.16) .. controls (357.99,331.4) and (400.68,287.56) .. (401.26,233.7) -- (303.33,232.5) -- cycle ; \draw   (205.51,237.28) .. controls (208.03,290.73) and (251.79,332.92) .. (304.76,332.16) .. controls (357.99,331.4) and (400.68,287.56) .. (401.26,233.7) ;
\draw    (280,263.5) -- (333.33,201) ;
\draw [shift={(306.67,232.25)}, rotate = 310.48] [color={rgb, 255:red, 0; green, 0; blue, 0 }  ][fill={rgb, 255:red, 0; green, 0; blue, 0 }  ][line width=0.75]      (0, 0) circle [x radius= 1.34, y radius= 1.34]   ;
\draw    (346,321.5) .. controls (317.33,338.5) and (280.57,260.17) .. (280,263.5) ;
\draw  [dash pattern={on 4.5pt off 4.5pt}]  (351,318.5) .. controls (376,306.5) and (333.9,199.67) .. (333.33,201) ;
\draw    (327.76,98.33) -- (312.15,76.14) ;
\draw [shift={(311,74.5)}, rotate = 54.88] [color={rgb, 255:red, 0; green, 0; blue, 0 }  ][line width=0.75]    (10.93,-3.29) .. controls (6.95,-1.4) and (3.31,-0.3) .. (0,0) .. controls (3.31,0.3) and (6.95,1.4) .. (10.93,3.29)   ;
\draw    (267.76,161.33) -- (278.24,186.65) ;
\draw [shift={(279,188.5)}, rotate = 247.53] [color={rgb, 255:red, 0; green, 0; blue, 0 }  ][line width=0.75]    (10.93,-3.29) .. controls (6.95,-1.4) and (3.31,-0.3) .. (0,0) .. controls (3.31,0.3) and (6.95,1.4) .. (10.93,3.29)   ;
\draw    (333.33,201) -- (323.83,180.32) ;
\draw [shift={(323,178.5)}, rotate = 65.33] [color={rgb, 255:red, 0; green, 0; blue, 0 }  ][line width=0.75]    (10.93,-3.29) .. controls (6.95,-1.4) and (3.31,-0.3) .. (0,0) .. controls (3.31,0.3) and (6.95,1.4) .. (10.93,3.29)   ;
\draw    (280,263.5) -- (294.92,286.82) ;
\draw [shift={(296,288.5)}, rotate = 237.38] [color={rgb, 255:red, 0; green, 0; blue, 0 }  ][line width=0.75]    (10.93,-3.29) .. controls (6.95,-1.4) and (3.31,-0.3) .. (0,0) .. controls (3.31,0.3) and (6.95,1.4) .. (10.93,3.29)   ;
\draw    (449,168.5) .. controls (410.4,165.61) and (400.67,158.99) .. (385.66,149.54) ;
\draw [shift={(384,148.5)}, rotate = 32.01] [color={rgb, 255:red, 0; green, 0; blue, 0 }  ][line width=0.75]    (10.93,-3.29) .. controls (6.95,-1.4) and (3.31,-0.3) .. (0,0) .. controls (3.31,0.3) and (6.95,1.4) .. (10.93,3.29)   ;
\draw    (438,198.5) .. controls (426.18,198.5) and (419.21,193.65) .. (392.25,213.57) ;
\draw [shift={(391,214.5)}, rotate = 323.13] [color={rgb, 255:red, 0; green, 0; blue, 0 }  ][line width=0.75]    (10.93,-3.29) .. controls (6.95,-1.4) and (3.31,-0.3) .. (0,0) .. controls (3.31,0.3) and (6.95,1.4) .. (10.93,3.29)   ;

\draw (299,125) node [anchor=north west][inner sep=0.75pt]  [font=\footnotesize]  {$O$};
\draw (289,225) node [anchor=north west][inner sep=0.75pt]  [font=\footnotesize]  {$O$};
\draw (251,158) node [anchor=north west][inner sep=0.75pt]    {$w$};
\draw (280,165) node [anchor=north west][inner sep=0.75pt]    {$\mathnormal{V}$};
\draw (325,85) node [anchor=north west][inner sep=0.75pt]    {$-w$};
\draw (297,276) node [anchor=north west][inner sep=0.75pt]    {$-\mathnormal{V}$};
\draw (257,245) node [anchor=north west][inner sep=0.75pt]    {$-w$};
\draw (336,190) node [anchor=north west][inner sep=0.75pt]    {$w$};
\draw (308,60) node [anchor=north west][inner sep=0.75pt]    {$-\mathnormal{V}$};
\draw (328,170) node [anchor=north west][inner sep=0.75pt]    {$\mathnormal{V}$};
\draw (372,23) node [anchor=north west][inner sep=0.75pt]    {$\overline{\mathbb{S}_{+1}^{n}}$};
\draw (389,290) node [anchor=north west][inner sep=0.75pt]    {$\overline{\mathbb{S}_{+2}^{n}}$};
\draw (442,139) node [anchor=north west][inner sep=0.75pt]    {$\partial\mathbb{S}_{+1}^{n}$};
\draw (440,188.5) node [anchor=north west][inner sep=0.75pt]    {$\partial\mathbb{S}_{+2}^{n}$};

\end{tikzpicture}
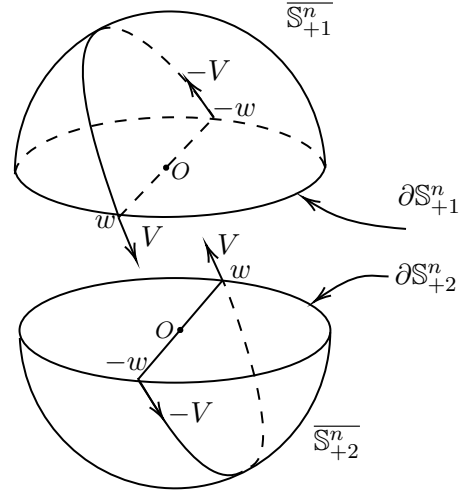
\caption{\small Gluing   $\overline{\mathbb{S}^n_{+1}}$ and $\overline{\mathbb{S}^n_{+2}}$.}
\label{K=1fig2}
\end{center}
\end{multicols}

For $w = (\varphi, \theta)$ and $V|_w = V^i \frac{\partial}{\partial \zeta^i}|_w$, Lemma~\ref{differenobversion} then gives
\[
-w = (\varphi - \sgn(\varphi) \pi, \theta), \quad -V|_{-w} = V^i \left. \frac{\partial}{\partial \zeta^i} \right|_{-w}.
\]
In these coordinates, \eqref{globalmetric} becomes
\[
\mathscr{F}(\varphi, \theta, V^i) = \mathscr{F}(\varphi - \sgn(\varphi) \pi, \theta, V^i),
\]
and together with \eqref{goedcooff} this implies
\begin{align}
G^j(w, V)  = G^j(\varphi, \theta, V^i)
= G^j(\varphi - \sgn(\varphi) \pi, \theta, V^i)
= G^j(-w, -V). \label{geodesicconstant}
\end{align}

\smallskip
\noindent
{\bf Step 2.} By the construction in Step~1, we identify $\mathbb{S}^n_+=\mathbb{S}^n_{+1}$ and refer to $\partial\mathbb{S}^n_+=\partial\mathbb{S}^n_{+1}$ as the {\it equator}  in what follows. In this step,  we prove the following statements:
\begin{enumerate}[\rm (a)]
\item\label{subecir1} if a great circle $\sigma$ is not entirely contained in the equator,  then it is a closed geodesic of  length $2\pi$;

\item\label{subecir2} if a geodesic $\gamma$ is not entirely contained in the equator, then
it lies on some great circle;

\item\label{subecir3} the injectivity radius is $\pi$ and  the distance between a pair of antipodal points is $\pi$.

\end{enumerate}

\smallskip

By \eqref{fK=1}, every straight line in $(\mathbb{R}^n,F)$ can  be parameterized as a unit-speed geodesic
\[
\gamma(t)=x+f(t)y=x+ y\cos^2c  \left[ \tan(  t + c) - \tan c \right], \quad t\in \left(-\frac{\pi}2-c,\frac{\pi}2-c\right)
\]
for some $x\in \mathbb{R}^n$ and $y\in S_x\mathbb{R}^n$, where $c=c(x,y) \in (-\pi/2,\pi/2)$.
A direct calculation (see \eqref{tangecrivle1} and \eqref{tangecrivle2}) gives
\begin{align}
&\lim_{t\rightarrow \pm \frac{\pi}{2}-c}\mathfrak{p}(\gamma(t))=\pm\left( \frac{y}{|y|},0 \right)=:\pm w,\label{tangecrivle33}\\
 &\lim_{t\rightarrow \pm \frac{\pi}{2}-c}\mathfrak{p}_{*\gamma(t)}(\gamma'(t))=\frac{\pm 1}{\cos^2 c} \left( \frac{\langle x,y\rangle y-|y|^2x}{|y|^3}, \frac{-1}{|y| } \right)=:\pm V.\label{tangecrivle44}
\end{align}
If a great circle $\sigma$ is not contained in $\partial \mathbb{S}^n_+$,  then $\sigma$  can be represented as the union of two  copies of some $\mathfrak{p}(\gamma(t))$ as above (see Figure \ref{K=1fig2}).
By \eqref{geodesicconstant}--\eqref{tangecrivle44} and \eqref{geodesequ}, this union is still a geodesic, and hence $\sigma$ is a closed geodesic.
Moreover, Theorem \ref{thmK=1global}/\eqref{diam1spher2} implies that $\sigma$ has length $2\pi$. This proves \eqref{subecir1}.

For \eqref{subecir2},
let  $\gamma:I\rightarrow \mathbb{S}^n$ be a geodesic not entirely contained in $\partial\mathbb{S}^n_+$. Without loss of generality, there exists a  nonempty subinterval $I_1\subset I$ such that  $\gamma:I_1\rightarrow \mathbb{S}^n_+$ is in the upper half sphere.
By the preceding construction, $\gamma|_{I_1}$ is the image under $\mathfrak{p}$ of some segment in $\mathbb{R}^n$.
Observation \ref{informobvers2}/\eqref{ob21} then implies that $\gamma|_{I_1}$ lies on some great circle. Similarly, the part of $\gamma$ in the lower hemisphere also lies on some great circle.
By smoothness, these two parts must belong to the same great circle, so $\gamma$ itself lies on some great circle.
This establishes \eqref{subecir2}.

It remains to show  \eqref{subecir3}.
From the above, every great semicircle avoiding the equator is a geodesic of length $\pi$, and every geodesic not entirely contained in the equator with length less than $\pi$ is minimizing.
Define the unit sphere bundle
\[
S\mathbb{S}^n:=\cup_{w\in \mathbb{S}^n} S_w\mathbb{S}^n=\{(w,V)\in T\mathbb{S}^n\,:\, \mathscr{F}(w,V)=1 \}.
\]
For any $(w,V)\in S\mathbb{S}^n$ with $V^{n+1}\neq 0$,
 the injectivity radius satisfies $i_V=i_{(w,V)}=\pi$. 
 For the remaining case where  $(w,V)\in S\mathbb{S}^n$ with  $w\in \partial \mathbb{S}^n_+$ and $V^{n+1}= 0$, choose a sequence $(V_\alpha)_\alpha\subset S_w\mathbb{S}^n$ converging to $V$. The continuity of the injectivity radius (cf.~\cite[Proposition 8.4.1]{BCS}) then yields $i_V=\lim_{\alpha\rightarrow +\infty}i_{V_\alpha}=\pi$. Hence, the injectivity radius is $\pi$ everywhere.
 Consequently, $d_\mathscr{F}(w, -w) = \pi$ for all $w \in \mathbb{S}^n$, since there always exists a geodesic (a great semicircle) not entirely contained in $\partial \mathbb{S}^n_+$ from $w$ to $-w$ of length $\pi$.

\smallskip
\noindent
{\bf Step 3.} In this step,  we complete the proof.

We first prove Statement \eqref{sphere3}.
By \eqref{subecir1} and \eqref{subecir3} in Step~2, it suffices to show that any great circle $\sigma$ entirely contained in $\partial\mathbb{S}^n_+$ is a geodesic of length $2\pi$.
From \eqref{subecir3}, we have $L_\mathscr{F}(\sigma)\geq 2\pi$.
Conversely, by  reparameterization and \eqref{subecir1}, we can choose a sequence of  great circles (unit speed geodesics) $(\sigma_{\alpha}(t))_{\alpha}$, $t\in [0,2\pi]$, such that  each $\sigma_{\alpha}$ is not entirely contained in $\partial\mathbb{S}^n_+$ and $(\sigma_{\alpha}(t))_{\alpha}$ converges pointwise to $\sigma(t)$ for each $t$.
Since each $\sigma_{\alpha}(t)$ can be written as $\exp_{w_\alpha}(tV_\alpha)$  with $(w_\alpha,V_\alpha)\in S\mathbb{S}^n$, the compactness of $S\mathbb{S}^n$ allows us to assume, passing to a subsequence, that $(w_\alpha,V_\alpha)\rightarrow (w,V)$.
The continuity of  $\exp$ then implies $\sigma(t)=\exp_w(tV)$ for $t\in [0,2\pi]$, so $\sigma(t)$ is   a unit-speed closed geodesic of length $2\pi$.

We now prove Statement \eqref{sphere4}. By \eqref{subecir2} in Step 2,
it remains to consider a geodesic $\gamma:I\rightarrow \mathbb{S}^n$ entirely contained in the equator. Without loss of generality,  assume $\gamma(t)=\exp_w(tV)$ for $t\in [0,1]$ with $V\in S_w\mathbb{S}^n$ and $V^{n+1}=0$.
Choose a sequence $(V_\alpha) \subset S_w\mathbb{S}^n$ converging to $V$ with $V_\alpha^{n+1} \neq 0$.
By  \eqref{subecir2}, each geodesic $\gamma_\alpha(t):=\exp_w(tV_\alpha)$ for $t\in [0,1]$ lies on some great circle $\sigma_\alpha(t)=\exp_w(tV_\alpha)$ for $t\in [0,2\pi]$.
By Observation~\ref{informobvers2}/\eqref{ob22}, after passing to a subsequence, we may assume that $(\sigma_\alpha)_\alpha$  converges uniformly to some great circle $\sigma(t)$ on $t\in[0,2\pi]$. Therefore,
\[
\gamma(t)=\lim_{\alpha\rightarrow +\infty}\gamma_\alpha(t)=\lim_{\alpha\rightarrow +\infty}\sigma_\alpha(t)=\sigma(t), \quad \forall\, t\in [0,1],
\]
so $\gamma$ lies on the great circle $\sigma$.
\end{proof}

\begin{remark}

If $F$ is  Riemannian, the pull-back metric $\mathscr{F}_+:=(\mathfrak{p}^{-1})^*F$ always extends to the equator $\partial\mathbb{S}^n_+$ without requiring the explicit form of $F$. In fact, since  parallel transportation is a linearly isometry in the Riemannian setting,   Jacobi field theory implies that
for every $(w,V)\in T\mathbb{S}^n\backslash\{0\}$ with $w\in \partial\mathbb{S}^n_+$,  there exists a unique triple $(r,y,X)\in (0,\pi/2)\times S_N\mathbb{S}^n_+ \times T_N\mathbb{S}^n_+$ such that
\begin{align}\label{resctrIEM}
(w,V)=(\exp_N)_{*ry}(rX),\quad
\mathscr{F}_+((\exp_N)_{*ry}(rX))=\| r \langle y,X\rangle y+\sin r \left(  X-\langle y, X \rangle y \right)  \|,
\end{align}
where $N$ is the north pole of $\mathbb{S}^n_+$, $\|y\|:=\mathscr{F}_+(N,y)$ is an Euclidean norm and $\langle\cdot,\cdot\rangle$ is its inner product.
One can then verify that \eqref{resctrIEM} extends continuously to $r = \pi/2$, and the resulting $\mathscr{F}_+$ remains Riemannian. See also Example~\ref{Riemsphere} for an alternative approach via direct computation.



For non-Riemannian $F$, however, the extension of $\mathscr{F}_+$ to the equator $\partial\mathbb{S}^n_+$ is generally difficult. In fact, according to Bryant~\cite{Bry} and Shen~\cite{Sh1},  if
$F$ is additionally assumed to be reversible, then $\mathscr{F}_+$ can be extended to $\partial\mathbb{S}^n_+$ if and only if $F$ is Riemannian.
 In the irreversible case, there exists a projectively flat Finsler metric on $\mathbb{R}^n$ with $\mathbf{K}=1$ but its pull-back metric cannot extend to the equator; see Remark \ref{K=1remarkconditionextension}. Therefore, the extension hypothesis in Theorem \ref{plancetospereK=1} is necessary and natural as well.
\end{remark}

We now prove Theorem~\ref{thmK=1globalintro}, which is the main diameter bound and completion theorem for $\mathbf{K}=1$ stated in the Introduction.

\begin{proof}[Proof of Theorem \ref{thmK=1globalintro}] Conclusion \eqref{K=1theo1} follows immediately from the Bonnet--Myers theorem (cf.~\cite[Theorem 7.7.1]{BCS}) combined with the reverse metric, while Conclusions \eqref{K=1theo2} and \eqref{K=1theo3} are direct consequences from Theorems~\ref{thmK=1global} and \ref{plancetospereK=1}.
\end{proof}

We conclude this subsection  with concrete examples illustrating Theorem \ref{plancetospereK=1}.
Using \eqref{K=1spherical_x}, a direct computation yields
\begin{align}
|x|^2 & =  {\sec^2}{\varphi} -1,  \qquad
\sum^{n}_{i=1} x^i {\dd} x^i  =   {\sec^2}\varphi \tan\varphi \,  {\dd}\varphi, \label{spherex}\\
\sum^{n}_{i=1} ({\dd }x^i)^2 & =
{\sec^4}\varphi\, ({\dd}\varphi)^2 + {\tan^2}\varphi\, ({\dd}\theta^1)^2 +   {\tan^2}\varphi\sum_{i=2}^{n-1}  \Big(  \prod_{\alpha=1}^{i-1}{\sin^2} \theta^\alpha \Big)  ({\dd}\theta^{i})^2. \label{spheredx}
\end{align}
\begin{example}[Riemannian case]\label{Riemsphere}
The classical Riemannian metric on $\mathbb{R}^n$ with $\mathbf{K}=1$ is given by
\begin{align*}
F(x,y) = \frac{\sqrt{ (1+|x|^2)|y|^2 - \langle x,y \rangle^2 }}{1+|x|^2},\quad \forall y\in T_x\mathbb{R}^n.
\end{align*}
Using \eqref{K=1spherical_x}, \eqref{spherex}, and \eqref{spheredx}, the pull-back metric on $\mathbb{S}^n_+$ is
\begin{align*}
\mathscr{F}_+:=(\mathfrak{p}^{-1})^* F & =
 \sqrt{({\dd}\varphi)^2 + {\sin^2}\varphi \left[ ({\dd}\theta^1)^2 + \sum_{i=2}^{n-1}\Big( \prod_{\alpha=1}^{i-1}\sin^2\theta^\alpha  \Big)({\dd}\theta^i)^2 \right]}\\
 &=\sqrt{({\dd}\varphi)^2+ {\sin^2}\varphi\, g_{\mathbb{S}^{n-1}}}=\sqrt{g_{\mathbb{S}^n}}.
\end{align*}
Thus, $\mathscr{F}_+$ is exactly the restriction of the standard Riemannian metric $g_{\mathbb{S}^n}$ to the upper hemisphere and hence  can be extended to the equator naturally.
This can also be verified directly by setting $\varphi = \pm \pi/2$, which yields $\mathscr{F}_+ = \sqrt{(\mathrm{d}\varphi)^2 + g_{\mathbb{S}^{n-1}}}$ and confirms that the fundamental tensor remains positive definite across the equator, ensuring a smooth extension.
The resulting glued metric $\mathscr{F}$  is exactly $g_{\mathbb{S}^n}$, which satisfies all the properties in Theorem \ref{plancetospereK=1}.
\end{example}

\begin{example}
[{Bryant's metric\cite{Bry,Bry2,Sh1}}]\label{K=1exampleBryant}
Given $\alpha\in (0,\pi/4)$ and $n\geq 2$, define  $F:\mathbb{R}^n\times \mathbb{R}^n\rightarrow \mathbb{R}_+$ by
\begin{equation}\label{Bryant}
 F(x,y)= \mathsf{Im}\left[ \frac{-\langle x,y\rangle + i \sqrt{(e^{2i\alpha} + |x|^2)|y|^2 - \langle x,y\rangle^2 }}{e^{2i\alpha} + |x|^2}\right]=\sqrt{\frac{\sqrt{\mathcal {A}}+\mathcal
{B}}{2\mathcal {D}}+\left(\frac{\mathcal {C}}{\mathcal
{D}}\right)^2}+\frac{\mathcal {C}}{\mathcal {D}},
\end{equation}
where
\begin{align}
& \mathcal {A}:=(|y|^2\cos(2\alpha)+|x|^2|y|^2-\langle
x,y\rangle^2 )^2+(|y|^2\sin(2\alpha))^2, \label{BryA} \\
& \mathcal {B}:=|y|^2\cos(2\alpha)+|x|^2|y|^2-\langle x,y\rangle^2, \quad \mathcal {C}:=\langle x,y\rangle \sin(2\alpha), \quad  \mathcal{D}:=|x|^4+2|x|^2\cos(2\alpha)+1. \notag
\end{align}
Then $F$ is a projectively flat Finsler metric on $\mathbb{R}^n$ with $\mathbf{K} = 1$.
Using \eqref{spherex} and \eqref{spheredx}, the pull-back metric on $\mathbb{S}^n_+$ is given by
\begin{align*}
\mathscr{F}_+:=(\mathfrak{p}^{-1})^* F = \sqrt{\frac{\sqrt{\widetilde{\mathcal {A}}}+\widetilde{\mathcal
{B}}}{2\widetilde{\mathcal {D}}}+\left(\frac{\widetilde{\mathcal {C}}}{\widetilde{\mathcal
{D}}}\right)^2}+\frac{\widetilde{\mathcal {C}}}{\widetilde{\mathcal {D}}},
\end{align*}
where
\begin{align*}
\widetilde{\mathcal {A}}: & ={\sec^8} {\varphi} \left[ ({\dd}\varphi)^2 + {\sin^2}\varphi \,\left( \cos(2\alpha) {\sin^2}\varphi + {\cos^2}\varphi \right)  g_{\mathbb{S}^{n-1}} \right]^2
 + {\tan^8}\varphi\, {\sin^2}(2\alpha) g^2_{\mathbb{S}^{n-1}},
 \\
 \widetilde{\mathcal {B}} :&= \cos(2\alpha)\, {\sec^4}\varphi ({\dd}\varphi)^2  +  {\tan^2}\varphi \, \left(\cos(2\alpha) + {\tan^2}\varphi \right)g_{\mathbb{S}^{n-1}},\quad \\
 \widetilde{\mathcal {C}} : &=  \sin(2\alpha) \sin\varphi\, {\sec^3}\varphi\, {\dd}\varphi,
 \qquad
 \widetilde{\mathcal {D}} := 2 \cos(2\alpha) {\tan^2}\varphi + {\tan^4}\varphi +1.
\end{align*}
On the equator ($\varphi=\pm\frac{\pi}{2}$), the limit is
\begin{align} \label{K=1equatormetric}
\lim_{\varphi \rightarrow \pm\frac{\pi}{2}} \mathscr{F}_+ :=  \sqrt{ \frac{ \cos(2\alpha)  ({\dd}\varphi)^2  +  g_{\mathbb{S}^{n-1}} + \sqrt{ \left[ ({\dd}\varphi)^2 + \cos(2\alpha)\,  g_{\mathbb{S}^{n-1}} \right]^2 +  {\sin^2}(2\alpha) g^2_{\mathbb{S}^{n-1}} } }{2} },
\end{align}
which is a spherically symmetric Minkowski norm.
Since $(\mathfrak{p}^{-1})^* F$ can be extended to a Finsler metric on $\overline{\mathbb{S}^n_+}$, we obtain a Finsler metric $\mathscr{F}$ on $\mathbb{S}^n$ by gluing. This metric coincides with Bryant's original construction~\cite{Bry} and satisfies all properties in Theorem~\ref{plancetospereK=1}.
Moreover, one can  verify directly that the geodesics are precisely the great circles, all of length $2\pi$.
\end{example}

\begin{remark}\label{K=1remarkconditionextension}
To illustrate the necessity of the extension condition in Theorem~\ref{plancetospereK=1}, we examine the two-dimensional case, which is computationally more transparent.
When $\alpha = \pi/4$, the function $F$ defined in \eqref{Bryant} is a Finsler metric on $\mathbb{R}^2$. However, the fundamental tensor of the pull-back $(\mathfrak{p}^{-1})^* F$ fails to remain positive definite at the equator.
In fact, at $\varphi = \pm\pi/2$, equation \eqref{K=1equatormetric} gives
\begin{align*} 
\breve{F}:= \lim_{\varphi \rightarrow \pm\frac{\pi}{2}}  \mathscr{F}_+|_{\alpha=\frac{\pi}4} = \sqrt{\frac{({\dd}\theta)^2+\sqrt{({\dd}\varphi)^4+({\dd}\theta)^4}}2}.
\end{align*}
For a vector $y = y^1 \frac{\partial}{\partial \varphi} + y^2 \frac{\partial}{\partial \theta}$, this becomes
\[ \breve{F}( y ) = \sqrt{\frac{(y^1)^2+\sqrt{(y^1)^4+(y^2)^4}}2}, \]
whose fundamental tensor $g_{ij}$ can be given by a direct computation, i.e.,
\begin{equation*}
(g_{ij})  =
\begin{bmatrix}
\frac{1}{2} +\frac{3 (y^1)^2}{2 \sqrt{(y^1)^4+(y^2)^4} } - \frac{(y^1)^6}{ \left(\sqrt{(y^1)^4+(y^2)^4} \right)^3}
&  - \frac{(y^1)^3 (y^2)^3}{ \left(\sqrt{(y^1)^4+(y^2)^4} \right)^3 }   \\[15pt]
- \frac{(y^1)^3 (y^2)^3}{ \left(\sqrt{(y^1)^4+(y^2)^4} \right)^3 }
 &   \frac{3(y^2)^2}{2 \sqrt{(y^1)^4+(y^2)^4} } - \frac{(y^2)^6}{ \left(\sqrt{(y^1)^4+(y^2)^4} \right)^3}
\end{bmatrix}.
\end{equation*}
Then it is degenerated when $y^2 =0$, since
\begin{equation*}
(g_{ij})|_{y^2 =0}  =
\begin{bmatrix}
1
&  0  \\
0
 &   0
\end{bmatrix}.
\end{equation*}
This shows that the extension condition in Theorem~\ref{plancetospereK=1} is indeed necessary.
\end{remark}



\vskip 10mm
\section{Sobolev spaces over projectively flat Finsler manifolds} \label{SecSobolev}

Sobolev spaces over Euclidean space and Riemannian manifolds are well known to be vector spaces. In Finsler geometry, however, this linearity property may fail, as demonstrated by the Funk metric (cf.~Krist\'aly--Rudas~\cite{KR}). For projectively flat Finsler manifolds of constant flag curvature, we show that the linear structure of Sobolev spaces is intrinsically tied to the backward completeness of the manifold.

This section is organized as follows. In the first subsection, we establish basic properties of Sobolev spaces in the Finsler setting. The next two subsections analyze the cases $\mathbf{K} = 0$ and $\mathbf{K} = -1$, respectively. We then combine these results to prove Theorem~\ref{thmSobolevIntro}.


\subsection{ Basic properties }\label{Sobolevsapce}
In this subsection, we recall  properties of Sobolev spaces in the Finsler setting, maintaining
notation consistent with Section~\ref{Preliminaries}; see  \cite{KLZ,KR} for more details.

A triple $(M, F, \m)$ is called a {\it Finsler metric measure manifold} (abbreviated $\FMMM$) if $(M, F)$ is a forward complete Finsler manifold endowed with a smooth positive measure $\m$.

\begin{definition}\label{soboloevespace}Let $(M,F,\m)$ be  an  $\FMMM$.
For $u\in C^\infty_0(M)$ and $p \in [1,+\infty)$, define a pseudo-norm as
\begin{equation}\label{W1p}
\|u\|_{W^{1,p}_{\m}}:=\|u\|_{L^p_{\m}}+\|{\dd} u\|_{L^p_{\m}}:=\left(\int_M |u|^p \dm \right)^{1/p}+\left(\int_M F^{*p}({\dd} u)  \dm  \right)^{1/p}.
\end{equation}
The  {\it  Sobolev space} $W^{1,p}_0(M,F,\m)$ is defined as the closure of $C^\infty_0(M)$ with respect to the {\it backward topology} induced by $\|\cdot\|_{W^{1,p}_{\m}}$, i.e.,
\[
u\in W^{1,p}_0(M,F,\m) \quad \Longleftrightarrow \quad \exists \, (u_k)\subset C^\infty_0(M) \text{ \ with }\lim_{k\rightarrow +\infty}\|u-u_k\|_{W^{1,p}_{\m}}= 0.
\]
\end{definition}

\begin{remark}\label{sobonormprop}
By the triangle inequality for the co-metric $F^*$ defined in \eqref{dualmetric}, one can verify that for any $u,v\in  W^{1,p}_0(M,F,\m)$:
 \begin{itemize}
\item $\|u\|_{W^{1,p}_{\m}}\geq 0$ with equality if and only if $u=0$;

\item $\|\alpha\,u\|_{W^{1,p}_{\m}}= \alpha \|u\|_{W^{1,p}_{\m}}$ for any $\alpha\geq 0$;

\item $\|u+v\|_{W^{1,p}_{\m}}\leq \|u\|_{W^{1,p}_{\m}}+\|v\|_{W^{1,p}_{\m}}$.

\end{itemize}
Although $W^{1,p}_0(M,F,\m)$ is a standard Sobolev space in the Riemannian setting, it
may fail to be  linear in the Finsler setting; see \cite{KR} for examples.
\end{remark}

A necessity condition for the nonlinearity of $W^{1,p}_0(M,F,\m)$ is given by the following result (cf.~\cite[Proposition 3.1]{KLZ}).
\begin{proposition}[\cite{KLZ}]\label{sobolevspaceline}
Let $(M,F,\m)$ be an $\FMMM$.
If $W^{1,p}_0(M,F,\m)$ is not a vector space for some $p\in [1, +\infty)$, then $\lambda_F(M)= +\infty$.
 \end{proposition}

We also require the following lemma, proved in Appendix~\ref{propesoboleve}.


\begin{lemma}\label{SobolevLemma1}
Let $(M,F,\m)$ be an $n$-dimensional noncompact $\FMMM$ with total measure $\m(M) < +\infty$. Given a point $o\in M$, set $r(x):=d_F(o,x)$.
 Then for any $\alpha>0$, the functions
 \[
 u_\alpha(x):=-e^{-\alpha r(x)},\qquad u(x): = - \Big[\ln(2+r(x))\Big]^{-\frac{1}{n}}
 \]
 belong  to $ W^{1,p}_0(M,F,\m)$ for every $p\in [1, +\infty)$.
\end{lemma}


In what follows, we investigate  Sobolev spaces over {\it forward complete} projectively flat Finsler manifolds $(\Omega,F)$ of constant flag curvature,
where
\[
  \text{$\Omega\subset \mathbb{R}^n$ is a domain  containing the origin $\mathbf{0}$.}
  \]
By Theorem~\ref{thmK=1globalintro}, no such manifold exists with $\mathbf{K} = 1$.
 Hence, we focus on the cases when $\mathbf{K}=0,-1$. For convenience, define
\begin{equation}\label{rreverrdefs}
r(x):=d_F(\mathbf{0}, x),\quad  \forall x\in \Omega.
\end{equation}
See Theorems~\ref{K=0lemdx_0xdxx_0} and~\ref{lemK=-1dx_0xdx_0}  for the
 expressions of $r$  when $\mathbf{K}=0,-1$ respectively. In both cases, $r$ is smooth in $\Omegao:=\Omega\backslash\{\mathbf{0}\}$ because $(\Omega, F)$ is a Cartan--Hadamard manifold (see \eqref{geommeaingofr}$_2$).

\subsection{ Sobolev spaces for $\mathbf{K}=0$}
This subsection studies  Sobolev spaces over projectively flat manifolds with $\mathbf{K}=0$.

\begin{lemma}\label{lemK=0Fstar} Let $(\Omega,F)$ be an $n$-dimensional forward complete but not backward complete projectively flat Finsler manifold  with $\mathbf{K}=0$. Then there
 exists a constant $\mathcal{C} > 0$ such that
\begin{equation*}
F^*(x, - {\dd}r) \geq \mathcal{C} r^2, \quad \forall x\in \Omegao.
\end{equation*}
\end{lemma}
\begin{proof}
By Theorem \ref{thmK=0globalintro}/\eqref{K=0additoncomp}(2), we have
\begin{equation} \label{K=0FstarF}
F = \psi ( y + x P)\,\left\{ 1 +  P_{y^k} x^k \right\},\qquad \Omega=\{\phi<1\},
\end{equation}
where  $\psi(y)=F(\mathbf{0},y)$ is a Minkowski norm, $\phi(y)=P(\mathbf{0},y)$ is a weak Minkowski norm, and the projective factor  $P=P(x,y)$ is a weak Funk metric satisfying
\begin{equation} \label{thmK=0P7}
P = \phi (y + x P).
\end{equation}

We first prove that
\begin{equation}\label{baiscinequaho}
P(x, -x) \leq 1, \quad x\in \Omegao.
\end{equation}
Suppose, for contraction, that $P(x, -x) > 1$ for some $x\in\Omegao$. Taking $y = -x$ in \eqref{thmK=0P7} and using the homogeneity of $\phi$, we obtain
\[ P(x, -x) = \phi(-x + x P(x, -x)) = (P(x, -x) - 1) \phi(x). \]
Since $0 < \phi(x) < 1$ by \eqref{K=0FstarF}$_2$, this implies
\begin{equation*}
1<P(x,-x) = \frac{- \phi(x)}{1-\phi(x)} < 0,
\end{equation*}
a contradiction. Hence, \eqref{baiscinequaho} holds.
Now, taking $y = -x$ in \eqref{thmK=0P7} gives
\[ P(x, -x) = \phi(-x + x P(x, -x)) = (1 - P(x, -x)) \phi( - x), \]
so that
\[ P(x, -x) = \frac{\phi(-x)}{1+ \phi(-x)}. \]
Substituting this into \eqref{K=0FstarF}$_1$ with $y=-x$, we obtain
\begin{align}
F (x, -x) & = \psi ( -x + x P(x,-x)) \big\{ 1 +  P_{y^k}(x, -x) x^k \big\} = (1 - P(x, -x))\psi(-x)\big\{ 1 - P_{y^k}(x, -x)(-x^k) \big\}
\notag \\
& = (1 - P(x, -x))^2 \psi(-x) = \frac{\psi(-x)}{(1+ \phi(-x))^2}. \label{K=0Fx-x}
\end{align}

Since $r(x):=d_F(\mathbf{0}, x)$ is smooth on $\Omegao$,
equations \eqref{coreqd0xdx0}$_1$ and \eqref{Eulerhter} yield
\begin{equation}\label{K=0minusdrx}
\langle- x , - {\dd}r \rangle = \left\langle x^i \frac{\partial }{\partial x^i},  \frac{\psi_{x^j}(x) - \psi_{x^j}(x)\phi(x)+\psi(x)\phi_{x^j}(x)}{(1- \phi(x))^2} {\dd}x^j  \right\rangle =  \frac{\psi(x)}{(1 - \phi(x))^2} = \frac{r^2(x)}{\psi(x)}.
\end{equation}
From \eqref{dualmetric}, \eqref{K=0Fx-x}, and \eqref{K=0minusdrx}, it follows that for any $x\in \Omegao$,
\begin{align}\label{F^*consta}
F^*(x,-{\dd}r)=\sup_{y \in T_x\Omega \setminus \{0\}} \frac{ \langle y , - {\dd}r \rangle}{ F(x,y)} \geq \frac{ \langle- x , - {\dd}r \rangle}{ F(x,-x)} = \frac{(1 + \phi(-x))^2}{\psi(x) \psi(-x)} r^2(x).
\end{align}
On the other hand,
 the homogeneity of $\phi$ and $\psi$ imply the existence of a finite number $\rho>1$ such that
\begin{equation}\label{homeophsiphi}
 \frac{\phi(-x)}{\phi(x)}\leq \rho, \quad  \rho^{-1}\leq \frac{ \psi(x)}{\phi(x)}\leq \rho, \quad \forall x\in \Rno,
\end{equation}
which combined with \eqref{K=0FstarF}$_2$ gives
\begin{equation*}
 \frac{(1 + \phi(-x))^2}{\psi(x) \psi(-x)}  \geq   \frac{(1 + \phi(-x))^2}{\rho^2 \phi(x) \phi(-x)}   \geq  \frac{1}{\rho^3 }=:\mathcal {C}.
\end{equation*}
Together with \eqref{F^*consta}, this completes the proof.
\end{proof}

\begin{lemma}\label{K=0nabla}
Let $(\Omega,F)$ be as in Lemma \ref{lemK=0Fstar}, and define $\psi(y): = F(\mathbf{0},y)$ and  $\phi(y) := P(\mathbf{0},y)$. Then the following hold:
\begin{enumerate}[\rm (i)]
\item \label{nabla_r}
the gradient of $r(x)$  is given by
\begin{equation}\label{eqnabla_r}
\nabla r(x) = \frac{  \psi(x) }{(\psi(x) + r(x) \phi(x))^2} x,\quad  \forall x\in \Omegao;
\end{equation}
\item \label{K=0_S}
the S-curvature $\mathbf{S}$ with respect to the Lebesgue measure $\mathscr{L}^n$ satisfies
$\mathbf{S} =  (n+1) P\geq 0$. In particular,
\begin{equation}\label{Socntk=0}
0 < \mathbf{S}(x,\nabla r(x)) = (n+1) \frac{\phi(x)}{r(x)} \leq \frac{n+1}{\varrho +  r(x)}, \quad \forall x\in \Omegao,
\end{equation}
where $\varrho: =\inf_{y\in \Rno}\frac{\psi(y)}{\phi(y)} > 0$.
\end{enumerate}
\end{lemma}
\begin{proof} Let $(r,y)$ denote  the polar coordinate system of $(\Omega,F)$ around $\mathbf{0}$.

\smallskip

\textbf{\eqref{nabla_r}.}
Given $x\in \Omegao$, let $(r,y)=x$ be its polar coordinates and let $\gamma_y(t)$ be the unit-speed  geodesic with $\gamma'_y(0)=y\in S_{\mathbf{0}}\Omega$. Then $\gamma_y(r)=x$, and $x$ is parallel to $y$ since $F$ is projectively flat.
Hence, we may write $x^i = h(r,y) y^i$ for some nonnegative function $h(r,y)$.
Substituting this into  \eqref{coreqd0xdx0}$_1$ yields
\[
 r  = \frac{\psi(x)}{1- \phi(x)}=\frac{h(r,y) \psi(y)}{1- h(r,y) \phi(y)},
 \]
 which implies
 \[
 h(r,y) = \frac{r}{ \psi(y) +r \phi(y)}, \quad x^i=\frac{r}{\psi(y) + r \phi(y)} y^i, \quad \gamma_y(t)=\frac{t y}{\psi(y) + t  \phi(y)}.
 \]
By \eqref{geommeaingofr}$_2$ and the homogeneity of $\phi$ and $\psi$, we obtain
\[
\nabla r(x) = \gamma_y'(r)=\left.\frac{{\dd}\gamma_y^i}{{\dd}t}\right|_{t=r} \frac{\partial}{\partial x^i} = \frac{\psi(y)}{(\psi(y) +r(x) \phi(y))^2} y^i  \frac{\partial}{\partial x^i}= \frac{\psi(x)}{(\psi(x) +r(x) \phi(x))^2} x^i  \frac{\partial}{\partial x^i}.
\]

{\bf \eqref{K=0_S}.}
 Since the density function $\sigma(x)$ of the Lebesgue measure $\mathscr{L}^n$ is  constant,
\eqref{Scurvature} and \eqref{prjoactor}  yield
\begin{equation}\label{K=osp1}
\mathbf{S}(x,y) = \frac{\partial (P(x,y) y^m)}{\partial y^m} = (n+1)P(x,y).
\end{equation}
Given $x\in \Omegao$, let $(r,y)=x$ be its polar coordinates.
By \eqref{thmK=0P7} and \eqref{eqnabla_r}, we have
\[
P(x, \nabla r(x)) = \phi\left( \frac{  \psi(x) }{(\psi(x) + r(x) \phi(x))^2} x + x P(x,  \nabla r(x))\right) = \left[ \frac{  \psi(x) }{(\psi(x) + r(x) \phi(x))^2} + P(x,  \nabla r(x)) \right] \phi(x),
\]
which combined with \eqref{coreqd0xdx0}$_1$ implies
\begin{equation}\label{K=0Pnablar}
P(x,  \nabla r(x)) = \frac{  \psi(x)\phi(x) }{(\psi(x) + r(x) \phi(x))^2 (1 -\phi(x))} = \frac{\phi(x)(1-\phi(x))}{\psi(x)} = \frac{\phi(x)}{r(x)}>0.
\end{equation}
On the other hand, \eqref{homeophsiphi}$_2$ yields  $\varrho :=\inf_{y\in \Rno}\frac{\psi(y)}{\phi(y)} > 0$, which together  with
\eqref{coreqd0xdx0}$_1$ gives
\[
  \phi(x) = \frac{r(x) - \psi(x)}{r(x)} \leq \frac{r(x) - \varrho \phi(x)}{r(x)}.
\]
This inequality implies $\phi(x) \leq \frac{r(x)}{\varrho + r(x)}$, which combined with \eqref{K=osp1} and \eqref{K=0Pnablar} establishes \eqref{Socntk=0}.
\end{proof}



The main result of this subsection is the following theorem.
\begin{theorem}\label{thmK=0Sobolevall}
Let $(\Omega,F)$ be an $n$-dimensional  forward complete projectively flat Finsler manifold with $\mathbf{K}=0$.
Then for any  $p \in (1, +\infty)$,
\begin{enumerate}[\rm (i)]
\item\label{linerK=01} if $(\Omega,F)$ is backward complete, then $W^{1,p}_0(\Omega,F,\mathscr{L}^n)$  is a vector space;

\item\label{linerK=02} if $(\Omega,F)$ is not backward complete, then $W^{1,p}_0(\Omega,F,\mathscr{L}^n)$  is not a vector space .
\end{enumerate}
 \end{theorem}
\begin{proof}
{\bf \eqref{linerK=01}.}
 Since $(\Omega,F)$ is complete, Theorem~\ref{thmK=0globalintro}/\eqref{K=0additoncomp}(1) implies $\lambda_F(\Omega) < +\infty$.
 Proposition~\ref{sobolevspaceline} then ensures that $W^{1,p}_0(\Omega,F,\mathscr{L}^n)$ is a vector space.

{\bf \eqref{linerK=02}. }
Let $(r,y)$ denote the polar coordinate system around $\mathbf{0} \in \Omega$, and let $\gamma_y(t)$ be a unit-speed geodesic with $\gamma_y'(0)=y\in S_{\mathbf{0}}\Omega$. Since $(\Omega,F)$ is a Cartan--Hadamard manifold, we have $i_y = +\infty$ and hence
\[
r(\gamma_y(t))=d_F(\mathbf{0},\gamma_y(t))=t,\quad \forall t\in [0,+\infty).
\]
Equations
\eqref{distsdef}$_2$, \eqref{geommeaingofr}, and \eqref{Socntk=0} then yield
\begin{align*}
\tau(\gamma_y(r),{\gamma}_y'(r))-\tau(\mathbf{0},y)=\int^r_0 \mathbf{S}(\gamma_y(t),{\gamma}_y'(t)){\dd}t \leq \int^r_0 \frac{n+1}{\varrho +  t}{\dd}t=  (n+1)\ln\left(\frac{\varrho + r}{\varrho}\right),
\end{align*}
where $\varrho$ is the positive constant from Lemma~\ref{K=0nabla}.
 Theorem~\ref{bascivolurcompar} (with $k=0$) then gives
\begin{equation}\label{K=0sigmaall}
\hat{\sigma}_o(r,y)  \geq    e^{-\tau\big(\gamma_y(r),{\gamma}_y'(r)\big)}  r^{n-1}\geq e^{-\tau(\mathbf{0},y)} \frac{\varrho^{n+1}r^{n-1}}{(\varrho + r)^{n+1}},\quad \forall r\in (0,+\infty),\ \forall y\in S_{\mathbf{0}}\Omega.
\end{equation}

Define
$u(x): = - [\ln(2+r(x))]^{-\frac{1}{n}}$. By Theorem \ref{thmK=0globalintro}/\eqref{K=0additoncomp}(2) and Lemma \ref{SobolevLemma1}, we have   $u\in W^{1,p}_0(\Omega,F,\mathscr{L}^n)$. To complete the proof, we show that $-u\notin W^{1,p}_0(\Omega,F,\mathscr{L}^n)$, i.e., $F^*(-{\dd} u)\notin L^p(\Omega,\mathscr{L}^n)$.
 Lemma \ref{lemK=0Fstar}  provides a constant $\mathcal{C}>0$ such that
 \begin{align}\label{K=0Fstar>}
 F^*(  -{\dd} u) = \frac{\partial u}{\partial r} F^*( -{\dd} r) \geq  \frac{\mathcal{C} r^2}{n(2+r)} \big[\ln(2+r)\big]^{-(1+\frac{1}{n})} \quad  \text{ in } \Omegao.
 \end{align}
For
$p>1$, choose $\alpha \in (0, 1)$ such that
\begin{equation}\label{K=0palphscondion}
p-\alpha p\left(1+\frac{1}{n}\right)-1>0.
\end{equation}
There exists  $R_\alpha\geq \max\{2,\varrho\}$ such that
for all $r \in (R_\alpha, +\infty)$,
\begin{equation}\label{K=0estimateall}
\left[ \ln \left( 2+r \right)   \right]^{-p(1+\frac{1}{n})} \geq  \left( 2+r \right)^{-\alpha p(1+\frac1n)}.
\end{equation}

Now it follows by \eqref{inffexprexx},   \eqref{cocont}, \eqref{K=0sigmaall}--\eqref{K=0estimateall}  that
\begin{align*}
 \int_\Omega F^{*p}( -{\dd} u) {\dd} x
 &\geq \left(\int_{S_{\mathbf{0}}\Omega}e^{-\tau(\mathbf{0},y)}{\dd}\nu_{\mathbf{0}}(y) \right)\left( \int^{+\infty}_0 \frac{\mathcal{C}^p \varrho^{n+1} r^{n-1+2p}}{n^p(\varrho + r)^{n+1}(2+r)^p}  [\ln(2+r)]^{-p(1+\frac{1}{n})} {\dd}r \right)\\
 &\geq
  \frac{ \mathcal{C}^p\varrho^{n+1}\mathscr{I}_{\mathscr{L}^n}(\mathbf{0})}{n^p}  \int^{+\infty}_{R_\alpha} \frac{  r^{n-1+2p} \left( 2+r \right)^{-\alpha p(1+\frac{1}{n})-p}}{(\varrho + r)^{n+1}}  {\dd}r\\
  & \geq
 \frac{ \mathcal{C}^p\varrho^{n+1}\mathscr{I}_{\mathscr{L}^n}(\mathbf{0})}{ 2^{\alpha p(1+\frac1n)+p} n^p}  \int^{+\infty}_{R_\alpha} \frac{  r^{n-1+p -\alpha p(1+\frac{1}{n})} }{ (2r)^{n+1}}  {\dd}r
 =+\infty,
\end{align*}
which shows  $F^*(-{\dd} u)\notin L^p(\Omega,\mathscr{L}^n)$.
\end{proof}

\subsection{ Sobolev spaces for $\mathbf{K}=-1$}

In this subsection, we study  Sobolev spaces over projectively flat Finsler manifolds with $\mathbf{K}=-1$ and prove Theorem~\ref{thmSobolevIntro}.

\begin{lemma}\label{lemK=-1Fstarall}
Let $(\Omega,F)$ be a forward complete but not backward complete projectively flat Finsler manifold with $\mathbf{K}=-1$. Then there exist  a constant $\mathcal {C}>0$ and a non-empty open subset $\mathcal {V}\subset \Omegao$ such that
\[
F^*(x, - {\dd}r) \geq\mathcal{C} e^{2 r},\quad \forall  x\in \mathcal {V}.
\]
\end{lemma}
\begin{proof} Set $\phi(y):=P(\mathbf{0},y)$ and $\psi(y):=F(\mathbf{0},y)$.
By  Theorem~\ref{thmK=-1globalintro}, we have
\begin{equation}\label{bask=-1equat}
(\phi\pm \psi)|_{\Omega}<1  , \quad F = \frac{1}{2}(\Phi_{+} - \Phi_{-}), \quad P = \frac{1}{2}(\Phi_{+} + \Phi_{-}),
\end{equation}
where  $\Phi_\pm$ are the unique solutions to
\begin{equation}\label{K=-1FstarPhi_pm}
\Phi_{\pm} = (\phi \pm \psi)(y + x \Phi_{\pm}).
\end{equation}

Using \eqref{K=-1FstarPhi_pm} and \eqref{bask=-1equat}$_1$, and proceeding as in the proof of  \eqref{baiscinequaho}, one can show that
\[
1 - \Phi_{\pm}(x, -x) > 0,\quad \forall x\in \Omegao.
\]
Taking $y = -x$ in \eqref{K=-1FstarPhi_pm} yields
\begin{align*}
\Phi_{+}(x, -x) &=(\phi +  \psi)(-x + x \Phi_{+}(x,-x))= [1 - \Phi_{+}(x, -x)] (\phi +  \psi)(-x), \\
 \Phi_{-}(x, -x) &=(\phi -  \psi)(-x + x \Phi_{-}(x,-x))= [1 - \Phi_{-}(x, -x)] (\phi -  \psi)(-x),
\end{align*}
which combined with \eqref{bask=-1equat}$_{2,3}$ give
\begin{align*} 
 \frac{(\phi +  \psi)(-x)}{1 + (\phi +  \psi)(-x)}&= \Phi_{+}(x, -x) =P(x,-x)+F(x,-x),\\
\frac{(\phi -  \psi)(-x)}{1+ (\phi - \psi)(-x)}&= \Phi_{-}(x, -x) =P(x,-x)- F(x,-x) .
\end{align*}
Therefore, we obtain
\begin{equation} \label{K=-1FstarF-x}
F(x, -x) = \frac{\psi(-x)}{(1+ \phi(-x) - \psi(-x))(1+ \phi(-x) + \psi(-x))}.
\end{equation}

On the other hand,
a direct calculation  analogous to \eqref{K=0minusdrx}, combined with \eqref{coreqK=-1dx_0xdxx_01122}$_1$, gives
\begin{align}
 \langle -x,- {\dd}r \rangle
  = \frac{\psi(x)}{(1 - \phi(x) - \psi(x))(1 - \phi(x) + \psi(x))}. \label{K=-1-dr-x}
\end{align}
Thus, it follows from \eqref{dualmetric}, \eqref{K=-1FstarF-x}, and \eqref{K=-1-dr-x} that
  for any $x\in \Omegao$,
\begin{align}
F^*(x,-{\dd}r)&=\sup_{y \in T_x\Omega \setminus \{0\}} \frac{ \langle y,- {\dd}r \rangle}{ F(x,y)}   \notag\\
&\geq  \frac{ \langle -x,- {\dd}r \rangle}{ F(x,-x)}
=  \frac{\psi(x)(1+ \phi(-x) - \psi(-x))(1+ \phi(-x) + \psi(-x))}{\psi(-x)(1- \phi(x) - \psi(x))(1- \phi(x) + \psi(x))}. \label{K=-1Fstar-dr}
\end{align}

According to cases (2)--(4) in Theorem~\ref{thmK=-1globalintro}/\eqref{K=-1additoncondtion},
the  remainder of the proof is divided into two cases.

\smallskip

\textbf{Case 1.} Suppose $P \geq F$ on $T\Omega\backslash\{0\}$. In this case, we  prove
\begin{equation}\label{K=-1Fstar(i)}
F^*(x, - {\dd}r) \geq  \frac{e^{2 r}}{\lambda_\psi},\quad \forall x\in \Omegao,
\end{equation}
where  $\lambda_\psi:=\sup_{y\in \mathbb{S}^{n-1}}\frac{\psi(-y)}{\psi(y)}\in [1,+\infty)$ is the reversibility of  $\psi$.

In fact, the assumption
$P\geq F$ implies $\phi(x)\geq \psi(x)$  for any $x \in \Rno$. Hence, we obtain
\[
\phi(-x) - \psi(-x)\geq 0\geq  - \phi(x) + \psi(x), \quad  \phi(-x) + \psi(-x)> 0 \geq -  \phi(x) + \psi(x).
\]
Combined this with \eqref{K=-1Fstar-dr}, \eqref{bask=-1equat}$_1$ (i.e., $1-\phi>\psi\geq 0$), and  \eqref{coreqK=-1dx_0xdxx_01122}$_1$, we get
\begin{align*}
F^*(x, - {\dd}r)\geq  \frac{(1 - \phi(x) + \psi(x))(1-  \phi(x) + \psi(x))}{\lambda_\psi (1- \phi(x) - \psi(x))(1- \phi(x) + \psi(x))}
=\frac{e^{2r}}{\lambda_\psi}. 
\end{align*}

\smallskip

\textbf{Case 2.}
Suppose $-F< P< F$ on $T\Omega\backslash\{0\}$. In this case, we show that
there exist  a non-empty open domain $\mathcal {V} \subset \Omegao$ and a constant $C>0$ such that
\begin{equation}\label{cae2k=-1sol}
F^*(x, - {\dd}r) \geq C e^{2 r},\quad \forall  x\in \mathcal {V}.
\end{equation}

First,  the assumption together with homogeneity implies
\begin{equation}\label{phi-psi}
  \phi(x)-\psi(x)=P(\mathbf{0},x) -  F(\mathbf{0},x) < 0, \qquad \phi(-x) + \psi(-x)\geq 0, \quad \forall x\in \Omegao.
\end{equation}

On the other hand, since $\Phi_-$ is well-defined on $\Omega$,
  Proposition \ref{propMainEq}/\eqref{MainEqii} (with $\varphi:=\phi-\psi$) implies
\begin{equation*}
\Omega\subset \mathcal {D}(\Phi_-)  \subset  \{ x \in \mathbb{R}^n \,:\, -1 < \phi(-x) - \psi(-x) \leq 0 \}=: \Omega_{-}.
\end{equation*}
According to Corollary~\ref{corK=-1S}, we must have $\Omega\subsetneq  \Omega_-$.
Hence, there exists a point $x_0\in \partial\Omega \cap \Omega_-$ because $\Omega_-$ is star-like. Since both $\Omega$ and $\Omega_-$ are open in $\mathbb{R}^n$, one can choose a small Euclidean open ball $\mathbb{B}^n_\epsilon(x_0)$ such that
$\mathbf{0}\notin \mathbb{B}^n_\epsilon(x_0)\Subset \Omega_-$.
Now set $\mathcal {V}:=\Omega\cap \mathbb{B}^n_\epsilon(x_0)$. This construction of $\mathcal {V}$, together with  \eqref{phi-psi}$_1$, implies
\begin{equation}\label{munumer1}
\mu_1:=\inf_{x\in \mathcal {V}}\left[ \phi(-x)-\psi(-x)\right]\in (-1,0],\qquad \mu_2:=\sup_{x\in \mathcal {V}}\left[- \phi(x) + \psi(x)\right]\in [0,+\infty).
\end{equation}

Thus, for any $x\in \mathcal {V}\subset \Omega$, from \eqref{K=-1Fstar-dr}, \eqref{phi-psi}$_2$, \eqref{munumer1}, and \eqref{coreqK=-1dx_0xdxx_01122},  we have
\begin{align*}
F(x,-{\dd}r) & \geq  \frac{ (1+ \mu_1) }{\lambda_\psi(1+\mu_2)(1- \phi(x) - \psi(x))}\geq   \frac{(1+\mu_1)(1- \phi(x) + \psi(x))}{\lambda_\psi(1+\mu_2)^2 (1- \phi(x) - \psi(x))} = \frac{(1+\mu_1)}{ \lambda_\psi(1+\mu_2)^2}e^{2r},
\end{align*}
which  is exactly \eqref{cae2k=-1sol}.

Therefore, the proof is completed by \eqref{K=-1Fstar(i)} and \eqref{cae2k=-1sol}.
\end{proof}

\begin{remark}\label{Fderk=1-exprs}
By the construction of $\mathcal {V}$, there exist a sufficiently large $\delta>0$ and a  nonempty open subset set  $\mathcal {O}\subset S_{\mathbf{0}}\Omega$ such that
\[
\gamma_y(t)=\exp_{\mathbf{0}}(ty)\subset \mathcal {V},\quad \forall  t\in [\delta,+\infty),\ \forall y\in \mathcal {O}.
\]
Hence, there exists a constant $\mathcal {C}>0$ such that
\[
F^*(x, - {\dd}r) \geq\mathcal{C} e^{2 r},\quad \forall  x\in \left\{\gamma_y(t)\,:\, t\in [\delta,+\infty),\ y\in \mathcal {O}  \right\}.
\]
Note that $\mathcal {V}$ is   unbounded   in $(\Omega,F)$ while it is bounded in $(\mathbb{R}^n,|\cdot|)$.
\end{remark}

\begin{lemma} \label{K=-1nabla}
Let $(\Omega,F)$ be as in Lemma \ref{lemK=-1Fstarall} and define $\psi(y):=F(\mathbf{0},y)$ and $\phi:=P(\mathbf{0},y)$. Then the following hold:
\begin{enumerate}[\rm (i)]
\item \label{K=-1nabla_r}
the gradient of $r(x)$ is given by
\begin{equation}\label{K=-1eqnabla_r}
\nabla r(x) = \frac{(1-\phi(x) + \psi(x))(1 - \phi(x) - \psi(x))}{\psi(x)} x,\quad \forall  x\in \Omegao;
\end{equation}
\item \label{K=-1_S}
the $S$-curvature $\mathbf{S}$ of the Lebesgue measure $\mathscr{L}^n$ is given by
\begin{equation} \label{K=-1_SP}
\mathbf{S} =  (n+1) P,
\end{equation}
which satisfies for any $x\in \Omegao$,
\begin{equation*}
 {\mathbf{S} (x,\nabla r(x))} = (n+1)\left\{  \frac{e^{2r(x)} -1}{e^{2r(x)} + 1} + \left( \frac{e^{2r(x)} + 1}{e^{2r(x)} - 1} - \frac{e^{2r(x)} -1}{e^{2r(x)} + 1} \right) \phi(x) \right\}.
\end{equation*}
In particular,
\begin{equation}\label{K=-1Sestimate}
 \lim_{r(x) \rightarrow +\infty} \mathbf{S}(x, \nabla r(x)  ) = n+1.
\end{equation}
\end{enumerate}
\end{lemma}
\begin{proof}
Both \eqref{K=-1eqnabla_r} and  \eqref{K=-1_SP} follow from \eqref{coreqK=-1dx_0xdxx_01122}$_1$ and an argument analogous to the  proof of
Lemma~\ref{K=0nabla}/\eqref{nabla_r}. We now prove the expression for $\mathbf{S} (x,\nabla r(x))$ in Statement \eqref{K=-1_S}.

In view of \eqref{K=-1eqnabla_r},  we  set $\nabla r(x) =: \ell(x) x$ for simplicity, where
\begin{equation}\label{hexpression}
\ell(x)=\frac{(1-\phi(x) + \psi(x))(1 - \phi(x) - \psi(x))}{\psi(x)} >0.
\end{equation}
Since $P+F\geq 0$ (see Theorem \ref{thmK=-1globalintro}), by \eqref{bask=-1equat} and \eqref{K=-1FstarPhi_pm}, we obtain
\begin{align*}
P(x, \nabla r(x)) + F(x, \nabla r(x)) &=(\phi+\psi)\Big( \nabla r(x) + x \big(P(x,  \nabla r(x))+ F(x,\nabla r(x))\big) \Big)\\
&= (\phi+\psi)\Big( \ell(x) x + x \big(P(x,  \nabla r(x))+ F(x,\nabla r(x))\big) \Big) \\
&= \big( \ell(x) + P(x,  \nabla r(x))+ F(x,\nabla r(x)) \big) (\phi+\psi)(x).
\end{align*}
Combining this with \eqref{hexpression} and the identity
 $F(x, \nabla r(x)) = 1$ yields
\begin{align}\label{K=-1Pnablar}
P(x,  \nabla r(x)) = \frac{\phi(x) + \psi^2(x) - \phi^2(x)}{\psi(x)}.
\end{align}

On the other hand,
  a direct computation using \eqref{coreqK=-1dx_0xdxx_01122}$_1$ gives
\begin{equation*} 
\psi(x) = \frac{e^{2r} -1}{e^{2r} + 1}(1-\phi(x)).
\end{equation*}
Substituting this into \eqref{K=-1Pnablar}, we have
\begin{align}\label{K=-1Pnablar1}
P(x,  \nabla r(x)) = \frac{e^{2r} -1}{e^{2r} + 1} + \left( \frac{e^{2r} + 1}{e^{2r} - 1} - \frac{e^{2r} -1}{e^{2r} + 1} \right) \phi(x).
\end{align}

By  Theorem~\ref{thmK=-1globalintro}, $\Omega\subset \mathbb{R}^n$ is a bounded domain, so $\sup_{x\in \Omega} |\phi(x)| < +\infty$. Therefore, \eqref{K=-1Pnablar1} and  \eqref{K=-1_SP}  imply
$\lim_{r(x) \rightarrow  +\infty} \mathbf{S} (x,\nabla r(x)) = n+1$, which
 completes the proof of \eqref{K=-1Sestimate} and hence the lemma.
\end{proof}

\begin{theorem}\label{thmK=-1Sobolevall}
Let $(\Omega,F)$ be an $n$-dimensional  forward complete projectively flat Finsler manifold with $\mathbf{K}=-1$.
Then for any  $p \in (1, +\infty)$, the following statements hold:
\begin{enumerate}[\rm (i)]
\item\label{linerK=-11} if $(\Omega,F)$ is backward complete, then $W^{1,p}_0(\Omega,F,\mathscr{L}^n)$  is a vector space;

\item\label{linerK=-12} if $(\Omega,F)$ is not backward complete, then $W^{1,p}_0(\Omega,F,\mathscr{L}^n)$  is not a vector space.
\end{enumerate}
 \end{theorem}

\begin{proof}  By Theorem~\ref{thmK=-1globalintro}/\eqref{K=-1additoncondtion}(1), $F$ is reversible  when $(\Omega,F)$ is complete. Hence,  Statement~\eqref{linerK=-11} follows directly from Proposition~\ref{sobolevspaceline}.

We now prove Statement~\eqref{linerK=-12}.
 Given $p\in (1,+\infty)$,
define
$u_{\alpha}(x):=-e^{-\alpha r(x)}$, where $\alpha\in (0, 2- {2}/{p} )$ is a constant. Theorem~\ref{thmK=-1globalintro}/\eqref{K=-1bascicondition1}  implies $\mathscr{L}^n(\Omega)< +\infty$. Combining this with
Lemma~\ref{SobolevLemma1} yields $u_\alpha \in W^{1,p}_0(\Omega,F,\mathscr{L}^n)$. To complete the
proof, it suffices to show
\begin{equation}\label{w1puibfty}
\|-u_\alpha \|^p_{W^{1,p}_{\mathscr{L}^n}} = +\infty.
\end{equation}

To establish \eqref{w1puibfty}, first note that \eqref{bask=-1equat}$_1$ gives
\[
\phi(x)\leq \phi(x)+\psi(x)<1,\quad \forall  x\in \Omegao,
\]
Together with Lemma~\ref{K=-1nabla}/\eqref{K=-1_S}, this implies
\begin{align}
\mathbf{S}(x,\nabla r(x))&=    (n+1)\left\{  \frac{e^{2r} -1}{e^{2r} + 1} + \left( \frac{e^{2r} + 1}{e^{2r} - 1} - \frac{e^{2r} -1}{e^{2r} + 1} \right) \phi(x) \right\}<  (n+1)\left( \frac{2e^{2r} }{e^{2r} - 1} -1 \right).\label{K=-1_Snabla2}
\end{align}

Let $(r,y)$ denote the polar coordinate system around $\mathbf{0}$. Note that $\mathfrak{i}_\mathbf{0}=+\infty$ since $(\Omega,F)$ is a Cartan--Hadamard manifold.
By  Remark~\ref{Fderk=1-exprs},
 there exist  constants $\mathcal {C},\delta>0$ and a subset $\mathcal {O} \subset S_{\mathbf{0}}\Omega$ such that
\begin{equation}\label{K=-1Fstar<}
F^*(x, - {\dd}r) \geq\mathcal{C} e^{2 r},\quad \forall  x\in \big\{ (r,y):\, r\in [\delta,+\infty),\ y\in \mathcal {O}  \big\}=:\mathcal {V}\subset \Omega.
\end{equation}
Now set
\[
C_1:=C_1(\delta):=  -(n+1)\left[\ln(e^{2\delta} -1)-\delta\right],\qquad C_2:=C_2({\delta}):=\sup_{y\in S_\mathbf{0}\Omega} {\tau(\gamma_y(\delta),{\gamma}_y'(\delta))}< +\infty.
\]
For $r\in [\delta,+\infty)$,
it follows from \eqref{distsdef}$_2$, \eqref{geommeaingofr}$_2$  and \eqref{K=-1_Snabla2}  that for any $y\in S_{\mathbf{0}}\Omega$,
\begin{align*}
\tau\big(\gamma_y(r),{\gamma}_y'(r)\big)-C_2&\leq \tau\big(\gamma_y(r),{\gamma}_y'(r)\big)-\tau\big(\gamma_y(\delta),{\gamma}_y'(\delta)\big) =\int^r_\delta \mathbf{S}\big(\gamma_y(s),{\gamma}_y'(s)\big){\dd}s\\
&\leq \int^r_\delta (n+1)\left( \frac{2e^{2s} }{e^{2s} - 1} -1 \right){\dd}s
=
 (n+1)\left[ \ln(e^{2r} - 1) -r\right]+C_1. 
\end{align*}
Combining this with  Theorem~\ref{bascivolurcompar} (for $k=-1$) yields
\begin{align}
\hat{\sigma}_{\mathbf{0}}(r,y) & \geq  e^{-\tau\big(\gamma_y(r),{\gamma}_y'(r)\big)}  \mathfrak{s}_{-1}^{n-1}(r)\geq  e^{-(C_1+C_2)} \frac{(e^{r} - e^{-r})^{n-1}e^{(n+1)r} }{2^{n-1} (e^{2r} -1)^{n+1}}\notag\\
&=   \frac{e^{2r}}{2^{n-1}e^{C_1+C_2} (e^{2r}-1)^2}\geq \frac{e^{-2r}}{2^{n-1}e^{C_1+C_2} },  \label{K=-1sigmaall}
\end{align}
for all $r\in [\delta, +\infty)$ and $y\in S_{\mathbf{0}}\Omega$.
Finally, from \eqref{inffexprexx}, \eqref{K=-1Fstar<}, and \eqref{K=-1sigmaall}, we obtain
\begin{align*}
 \int_\Omega F^{*p}(x, -{\dd} u_\alpha) {\dd} x &\geq  \int_{\mathcal {V}} F^{*p}(x, -{\dd} u_\alpha) {\dd} x
 =\int_{y\in \mathcal {O}}\left( \int^{+\infty}_{\delta} F^{*p}(x, -{\dd} u_\alpha) \,\hat{\sigma}_{\mathbf{0}}(r,y){\dd}r  \right){\dd}\nu_{\mathbf{0}}(y)\\
 &=\int_{y\in \mathcal {O}}\left( \int^{+\infty}_{\delta} \left(\frac{\partial u_\alpha}{\partial r}\right)^p F^{*p}(x, -{\dd} r)\, \hat{\sigma}_{\mathbf{0}}(r,y){\dd}r  \right){\dd}\nu_{\mathbf{0}}(y)\\
 &\geq \nu_{\mathbf{0}}({\mathcal {O}}) \,   \int^{+\infty}_{\delta} (\alpha e^{-\alpha r} )^p (\mathcal{C} e^{2 r})^p \frac{e^{-2r}}{2^{n-1}e^{C_1+C_2} }{\dd}r    = +\infty,
\end{align*}
where the last equality follows by $\alpha\in (0, 2- {2}/{p} )$. Therefore, \eqref{w1puibfty} holds, which completes the proof.
\end{proof}

We conclude this subsection by proving Theorem \ref{thmSobolevIntro}.

\begin{proof}[Proof of Theorem \ref{thmSobolevIntro}]
The equivalence \eqref{soli1} $\Leftrightarrow$ \eqref{soli2} follows from Theorems \ref{thmK=0globalintro} and \ref{thmK=-1globalintro}. The equivalences \eqref{soli2} $\Leftrightarrow$ \eqref{soli3} and \eqref{soli2} $\Leftrightarrow$ \eqref{soli4}, in turn, follow directly from Theorems \ref{thmK=0Sobolevall} and \ref{thmK=-1Sobolevall}.
\end{proof}

\vskip 10mm
\section{Asymmetric metrics on Euclidean spaces}\label{sectiondistance}
For projectively flat Finsler metrics of constant flag curvature, Funk \cite{Funk1} established a relation between the Finslerian distance function and the Euclidean distance in the two-dimensional case. Berwald \cite{Be2} later extended this result to higher dimensions.

In this section, we investigate a broader class of structures: asymmetric metrics on Euclidean spaces induced by pseudo-norm functions.
For related work--some employing different terminology--we refer to \cite{BBI,BM,KZ,RMS,Ohta1,ShenLecture}.

Let $\Omega \subset \mathbb{R}^n$ be a convex domain.
A continuous function $d:\Omega\times \Omega$ is called an {\it  asymmetric metric} on $\Omega$ if it satisfies the following conditions for all $x_1,x_2,x_3\in \Omega$:
\begin{enumerate}[\rm (i)]
\item  $d(x_1,x_2)\geq 0$, with equality if and only if $x_1=x_2$;

\item  $d(x_1,x_3)\leq d(x_1,x_2)+d(x_2,x_3)$.

\end{enumerate}
An asymmetric metric $d$ is said to be {\it smooth} if it is smooth on $(\Omega\times \Omega)\backslash \Delta$, where
\[
\Delta:=\{(x,y)\in \Omega\times \Omega\,:\, x=y \}.
\]
Note that $d$ cannot be smooth on the entire $\Omega\times \Omega$, even when induced by a Riemannian metric; a basic example is the Euclidean distance $d_{\E}(x_1,x_2):=|x_2-x_1|$.

Given a   continuous curve $\gamma:[0,1]\rightarrow \Omega$, the {\it length} of $\gamma$ induced by $d$ is defined as
\begin{equation}\label{lengthstru1}
 L_d(\gamma) :=\sup\left\{ \sum_{i=1}^N d\big( \gamma(t_{i-1}),\gamma(t_i) \big)
 \,\middle|\, N \in \mathbb{N},\ a=t_0<\cdots <t_N=b \right\}.
\end{equation}
We say that
 $d$ is  {\it projectively flat} on $\Omega$ if  the minimizing curve between two points is
a straight line segment, that is, $$d(x_1,x_2)=L_d(\gamma),$$
where $\gamma$ is the straight segment from $x_1$ to $x_2$.

A continuous
function $F:T\Omega\rightarrow [0,+\infty)$ is called a {\it pseudo-norm function} if $\|\cdot\|_x:=F(x,\cdot)$ is a pseudo-norm on $T_x\Omega$ for every $x\in \Omega$ (see Section \ref{elemFinslergeo}). A pseudo-norm function $F$ is said to be {\it smooth} if $F\in C^\infty(T\Omega\backslash\{0\})$.

The {\it length} $L_F(c)$ of a piecewise smooth curve $c$ and the induced {\it distance} $d_F$ are defined analogously to the Finsler case (see \eqref{lengthinducedbyF} and \eqref{distanceinducedbyF}).
Then $d_F$ is an asymmetric metric on $\Omega$ satisfying the {\it Busemann--Mayer theorem} (cf.~Bao et al. \cite[pp.\,160--161]{BCS}, Busemann--Mayer \cite[Theorem 4.3, p.\,186]{BM}), i.e.,
 \begin{equation*}
 F(x,y)=\lim_{t\rightarrow 0^+}\frac{d_F(\gamma(0),\gamma(t))}{t}=\lim_{t\rightarrow 0^-}\frac{d_F(\gamma(t),\gamma(0))}{-t},
 \end{equation*}
 for any smooth curve $\gamma:(-\epsilon,\epsilon)\rightarrow \Omega$ with $\gamma'(0)=y\in T_x\Omega$.  Moreover, it follows by \cite[Theorem 4.2, p.\,186]{BM} that $L_F(\gamma)=L_{d_F}(\gamma)$
 for every  piecewise smooth curve $\gamma:[0,1]\rightarrow \Omega$. Hence, we say that $F$ is {\it projectively flat} if $d_F$ is so, which coincides with the definition in the Finsler setting.

However, not every asymmetric metric $d$ on $\Omega$ is induced by a  pseudo-norm function.
\begin{example}\label{asymetrexample}
On $\mathbb{R}^n$, define a function $d$ by
\[
d(x_1,x_2):=|x_2-x_1|+ \langle a ,x_2-x_1\rangle+\sqrt{|x_2-x_1|},
\]
where $a\in \mathbb{R}^n$ is a fixed vector with $|a|<1$.
One may verify that $d$ is a smooth asymmetric metric. However, it
cannot be induced by a pseudo-norm function, since it fails to satisfy the Busemann--Mayer theorem. Indeed, $\lim_{t\rightarrow 0^+}\frac{d(\mathbf{0}, yt)}{t}=+\infty$ for any $y\in \Rno:=\mathbb{R}^n\backslash\{\mathbf{0} \}$.
\end{example}

Following  \cite{BM},  we  study the distance $d_F$ induced by a  pseudo-norm function $F$ under the {\it weakest}  Lipschitz condition: for every $x\in \Omega$, there exist a neighbourhood $U(x)$ of $x$  and a constant $C=C(x)>0$ such that
\begin{equation}\label{weaklIP22}
\left|F(x_1,y)-F(x_2,y)\right|\leq C |y||x_1-x_2|, \quad \forall  x_1,x_2\in U(x), \ \forall y\in \mathbb{R}^n.
\end{equation}

\begin{theorem}\label{dandFequ} Let $d$ be a
  projectively flat asymmetric metric on a convex domain $\Omega\subset \mathbb{R}^n$. Then
   $d$ is induced by a pseudo-norm function $F$ satisfying \eqref{weaklIP22} if and only if,   for every $x\in \Omega$, there exist positive constants $r=r(x)$, $c=c(x)$, and $C=C(x)$ such that
\begin{align}\label{baseequ}
 c|x_2-x_1|\leq d(x_1,x_2)\leq C|x_2-x_1|, \qquad \left|d(x_1,x_1+y)-d(x_2,x_2+y) \right|\leq  C |y| |x_2-x_1|,
 \end{align}
for all $x_1,x_2,x_1+y,x_2+y\in  {\mathbb{B}^n_r(x)}$. In this case, $F$ is also projectively flat.
Moreover,  $d$ is smooth if and only if  $F$ is smooth.
\end{theorem}

\begin{proof}
{\bf The ``$\Rightarrow$'' part.} Suppose that $d$ is induced by a  pseudo-norm function $F$ satisfying \eqref{weaklIP22}. By the assumption, $F$ is projectively flat.
Consider the function
\[
\varphi(x,y):=\frac{F(x,y)}{|y|}, \quad \forall (x,y)\in T\Omega\backslash\{0\}\cong\Omega\times \Rno.
\]
Then $\varphi$ can be viewed as a continuous function on $\Omega\times \mathbb{S}^{n-1}$. Thus, for any $x\in
\Omega$, there exist positive constants $r=r(x)$, $c=c(x)$, and $C=C(x)$ such that
\begin{equation}\label{cnomr}
c\,|y|\leq F(z,y)\leq C\, |y|, \quad \forall  (z,y)\in {\mathbb{B}^n_r(x)}\times \mathbb{R}^n.
\end{equation}
Combined with the projective flatness of $d$, this inequality immediately implies
\eqref{baseequ}$_1$.

To establish \eqref{baseequ}$_2$, choose a smaller $r = r(x) > 0$ and a larger $C = C(x)$ if necessary, so that \eqref{weaklIP22} holds on $\mathbb{B}^n_r(x) \subset U(x) \cap \Omega$.
Now, for any $x_1,x_2,x_1+y,x_2+y\in  {\mathbb{B}^n_r(x)}$,   set  $s:=|y|$ and $\bar{y}:=y/|y|\in \mathbb{S}^{n-1}$.
Since $d$ is  projectively flat and ${\mathbb{B}^n_r(x)}$ is convex, by \eqref{weaklIP22} we have
\begin{align*}
&\left| d(x_1,x_1+y)-d(x_2,x_2+y)  \right|=\left| d(x_1,x_1+s\bar{y})-d(x_2,x_2+s\bar{y})  \right|\\
\leq &\int^s_0 \left| F(x_1+t\bar{y},\bar{y})- F(x_2+t\bar{y},\bar{y})  \right|{\dd}t\leq C|x_1-x_2| s=C |x_1-x_2||y|,
\end{align*}
which is  \eqref{baseequ}$_2$.
Furthermore, if $F$ is smooth, then the projective flatness implies that $d$ is also smooth.

\smallskip

{\bf The ``$\Leftarrow$" part.} Now suppose  \eqref{baseequ} holds. Then for any piecewise smooth curve $\gamma:[0,1]\rightarrow \Omega$ with $\gamma(0)=x_1$ and $\gamma(1)=x_2$, the compactness of $\gamma([0,1])$ yields a constant $\mathfrak{C}=\mathfrak{C}(\gamma)>0$ such that
\begin{equation}\label{ss}
d(\gamma(s),\gamma(t))\leq \mathfrak{C}\,(t-s), \quad \forall [s,t]\subset [0,1].
\end{equation}
Thus, $\gamma$ is an absolutely continuous curve in the sense of Rossi--Mielke--Savar\'e \cite[p.\,106]{RMS}. By \cite[Proposition 2.2]{RMS}, the limit
\begin{equation}\label{exrgammazero}
|\gamma'|(t):=\lim_{h\rightarrow 0^+}\frac{d(\gamma(t),\gamma(t+h))}{h}=\lim_{h\rightarrow 0^+}\frac{d(\gamma(t-h),\gamma(t))}{h}
\end{equation}
exists for $\mathscr{L}^1$-a.e. $t\in (0,1)$.
We now show that this limit is independent of   $\gamma$ (only depending on the point $x:=\gamma(t)$ and the vector $y:=\gamma'(t)$).
Let $\gamma_i : (-\varepsilon, \varepsilon) \to \mathbb{B}^n_r(x)$, $i = 1, 2$, be two smooth curves such that $\gamma_i(0) = x$, $\gamma_i'(0) = y$, and  $|\gamma_i'|(0)$ exist.
It suffices to show
\begin{equation}\label{Fconstur}
|\gamma'_1|(0)=|\gamma'_2|(0).
\end{equation}
Since $\gamma_i$ is differentiable at $t=0$, we have
$\gamma_i(t)=x+ty+o_i(t)$.
For $t>0$, the triangle inequality of $d$ combined with \eqref{baseequ}$_1$ yields
\[
\frac{|d(x,\gamma_1(t))-d(x,\gamma_2(t))|}{t}\leq C \frac{|\gamma_2(t)-\gamma_1(t)|}{t}=C \frac{|o_1(t)-o_2(t)|}{t},
\]
which establishes \eqref{Fconstur}.

Therefore, for any  piecewise smooth curve $\gamma(t)$, $t\in (-\varepsilon,\varepsilon)$ such that $|\gamma'|(0)$ exists with $\gamma(0)=x$ and $\gamma'(0)=y$, we define
\begin{equation}\label{perFnorm}
F(x,y):=\lim_{t\rightarrow 0^+}\frac{d(\gamma(0),\gamma(t))}{t}=\lim_{t\rightarrow 0^+}\frac{d(\gamma(-t),\gamma(0))}{t}.
\end{equation}
A standard argument (cf.~Burago--Burago--Ivanov \cite[Theorem 2.7.4, Theorem 2.7.6]{BBI}) then shows that
\begin{equation}\label{lengthequals}
L_{d}(\gamma)= L_F(\gamma),
\end{equation}
 for every piecewise smooth curve $\gamma:[0,1]\rightarrow \Omega$.
 Thus, the projective flatness of $d$ together with \eqref{lengthequals} implies that $d$ is induced by $F$, i.e.,
 \begin{equation*}
 d(x_1,x_2)=\inf L_F(\gamma),
 \end{equation*}
 where  the infimum is over all piecewise smooth curves $\gamma:[0,1]\rightarrow \Omega$ with $\gamma(0)=x_1$ and $\gamma(1)=x_2$.

We claim that $F$ is actually defined on a dense set of $T\Omega\cong \Omega\times \mathbb{R}^n$. In fact, for any $(x,y)\in T\Omega$,  choose a sequence of straight segments $\gamma_m(t)=x+ty_m$ with $t\in [1/2^{m+1},1/2^m]$ with $y_m\rightarrow y$ as $m\rightarrow +\infty$.
For each $m$, by \eqref{exrgammazero}, we find that $F(\gamma_m(t),\gamma_m'(t))=|\gamma'_m|(t)$ exists for $\mathscr{L}$-a.e. $t\in (1/2^{m+1},1/2^m)$. Hence,  choose a $t_m\in (1/2^{m+1},1/2^m)$ such that $F(\gamma_m(t_m),\gamma_m'(t_m))$ is defined. Thus, the claim is true due to
\[
(\gamma_m(t_m),\gamma_m'(t_m))=(x+t_my_m,y_m)\rightarrow (x,y), \quad \text{ as }m\rightarrow +\infty.
\]

Next, we extend $F$ continuously to the whole $T\Omega$. Given $(x,y)\in T\Omega$, choose a sequence $((x_m,y_m))_m$ such that each $F(x_m,y_m)$ is defined and $(x_m,y_m)\rightarrow (x,y)$ as $m\rightarrow +\infty$.
For a sufficiently small $t>0$, the triangle inequality of $d$ combined with \eqref{baseequ} yields
\begin{align*}
&\quad \,\, |d(x_l,x_l+ty_l)-d(x_m,x_m+ty_m)|\notag\\
&\leq |d(x_l,x_l+ty_l)-d(x_m,x_m+ty_l)|+|d(x_m,x_m+ty_l)-d(x_m,x_m+ty_m)|\notag\\
&\leq  Ct|y_l||x_m-x_l|+\max\left\{ d(x_m+ty_m,x_m+ty_l), \ d(x_m+ty_l,x_m+ty_m)\right \}\notag\\
&\leq  Ct|y_l||x_m-x_l|+C t|y_m-y_l|.
\end{align*}
This inequality together with \eqref{perFnorm} gives
\begin{equation}\label{Flips}
|F(x_l,y_l)-F(x_m,y_m)|\leq C |y_l||x_m-x_l|+C|y_m-y_l|.
\end{equation}
Then $(F(x_m,y_m))_m$ is a Cauchy sequence and hence,  the limit $\lim_{m\rightarrow +\infty}F(x_m,y_m)$ exists.
 This limit is independent of the choice of of $((x_m,y_m))_m$.
 We  therefore define a continuous extension $\overline{F}$ on $T\Omega$ by
\[
\overline{F}(x,y):=\lim_{m\rightarrow +\infty}F(x_m,y_m).
\]

We now verify that $\overline{F}$ is a pseudo-norm function. Since $\overline{F}$ is continuous, it suffices to verify
\begin{enumerate}[\rm (i)]
\item\label{overF1} $\overline{F}(x,y)\geq 0$, with equality if and only $y=\mathbf{0}$;

\item\label{overF2} $\overline{F}(x,\lambda y)=\lambda \overline{F}(x,y)$ for any $\lambda> 0$;

\item\label{overF3} $\overline{F}(x, y_1+y_2)\leq \overline{F}(x,y_1)+\overline{F}(x,y_2)$.
\end{enumerate}

To prove \eqref{overF1}, let $((x_m, y_m))_m$ be a sequence as above. Then \eqref{baseequ}$_1$ implies
 $c t |y_m|\leq d(x_m,x_m+ty_m)$, and hence $c|y_m|\leq F(x_m,y_m)$, which establishes the positivity in \eqref{overF1}.
  The identity $\overline{F}(x, \mathbf{0}) = 0$ follows directly from the definition.

For \eqref{overF2}, observe that
 \begin{equation*}
F(x_m,\lambda y_m)=\lim_{t\rightarrow 0^+}\frac{d(x_m,x_m+t (\lambda  y_m))}{t}=\lim_{t\rightarrow 0^+}\frac{d(x_m,x_m+(\lambda t) y_m)}{\lambda t}\lambda =\lambda F(x_m,y_m).
\end{equation*}
Passing to the limit as $m \to +\infty$ gives $\overline{F}(x, \lambda y) = \lambda \overline{F}(x, y)$.

To prove \eqref{overF3}, let $((x_m,y_{\alpha,m}))_m$, $\alpha=1,2$ be two sequences  such that each $F(x_m,y_{\alpha,m})$ is defined and $(x_m,y_{\alpha,m})\rightarrow (x,y_\alpha)$ as $m\rightarrow +\infty$.
Then the triangle inequality of $d$ combined with \eqref{baseequ}$_2$ implies
\begin{align*}
&F(x_m,y_{1,m}+y_{2,m})=\lim_{t\rightarrow 0^+}\frac{d(x_m,x_m+t(y_{1,m}+y_{2,m}))}{t}\\
\leq& \lim_{t\rightarrow 0^+}\frac{d(x_m,x_m+ty_{1,m})}{t}+\lim_{t\rightarrow 0^+}\frac{d(x_m+ty_{1,m},x_m+t(y_{1,m}+y_{2,m}))}{t}\\
=& \lim_{t\rightarrow 0^+}\frac{d(x_m,x_m+ty_{1,m})}{t}+\lim_{t\rightarrow 0^+}\frac{d(x_m,x_m+ty_{2,m})}{t}=F(x_m,y_{1,m})+F(x_m,y_{2,m}).
\end{align*}
Taking the limit as $m \to +\infty$ yields \eqref{overF3}.

Note that \eqref{lengthequals} remains valid for $\overline{F}$. Since $\overline{F}$ is quasi-regular in $y$,
the Busemann--Mayer theorem \cite[ Theorem 4.3, p.\,186]{BM} implies that
\[
\overline{F}(x,y)=\lim_{t\rightarrow 0^+}\frac{d(x,x+ty)}{t},\quad \forall (x,y)\in T\Omega.
\]
Comparing with \eqref{perFnorm}, we conclude that $F = \overline{F}$ on $T\Omega$.
Moreover, \eqref{weaklIP22} follows directly from \eqref{Flips}.
The projective flatness of $F$ follows directly from that of $d$.

It remains to show that $F$ is smooth if $d$ is smooth. Without loss of generality, we may assume $\Omega=\mathbb{R}^n$ for simplicity.
By the projective flatness, we have
\begin{equation}\label{basicqueaitl}
d(x,x+ty)=\int^t_0 F(x+sy,y){\dd}s.
\end{equation}
Note that $d(x,x+ty)$ is smooth for $(x,t,y)\in \mathbb{R}^n\times (0,+\infty)\times \mathbb{S}^{n-1}=:\mathscr{R}$. Then   \eqref{basicqueaitl} gives
\[
F(x+ty,y)=\frac{\partial}{\partial t}d(x,x+ty)\in C^\infty(\mathscr{R}).
\]
Hence, $F(x+ty,ty)=t F(x+ty,y)$ is also smooth in $\mathscr{R}$.
On the other hand, the map
\[
(x,t,y)\in \mathscr{R}\longmapsto (x+ty,ty)\in \mathbb{R}^n\times \Rno
\]
is a diffeomorphism,  establishing the smoothness of  $F$  on $ \mathbb{R}^n\times \Rno$.
\end{proof}


\begin{remark}
From  a point of analytic view, \eqref{baseequ}$_1$ provides the bounds of $F(y)$ while \eqref{baseequ}$_2$ gives a upper bound of the Euclidean gradient $|\nabla_{\E} F(\cdot,y)|\leq C|y|$.
These bounds are not consequences of the mere smoothness of $d$. In this sense, the hypotheses \eqref{baseequ} cannot be weakened.

Furthermore, $F$ becomes a weak Finsler metric (resp., Finsler metric) if its indicatrices  are all strictly convex  (resp., strongly convex).
 These convexity properties, however, cannot be inferred solely from the corresponding properties of the distance function $d$.
\end{remark}


\appendix
\vskip 8mm
\section{Complementary results for $\mathbf{K}\leq 0$
}\label{propergenerlengappex0}

In this section, we complete the proofs of Lemma~\ref{keylemmak=0}, Proposition~\ref{K=0Ex1}, and Lemma~\ref{lemK=-1suffeqposilem}.

\begin{proof}[Proof of Lemma \ref{keylemmak=0}]
{\bf \eqref{keylemmak=0a}.}
 We begin by computing the Hessian $[F^2]_{y^i y^j}$. For simplicity, given $y\in \Rno$, set
\begin{align}\label{lemK=0suffxi}
\xi : = (\xi^k) := (y^k + x^k P(x,y) )=(y^k+x^kP).
\end{align}
A direct calculation using \eqref{defiFK=0} yields
\begin{align}
[ F^2 ]_{y^i y^j} (x,y)& =  [ \psi ^2(\xi) ]_{y^i y^j} ( 1+  P_{y^k} x^k )^2 +   \psi^2(\xi) [(1+ P_{y^k} x^k )^2]_{y^i y^j}+  [ \psi^2(\xi)]_{y^i}[(1+ P_{y^k} x^k )^2]_{y^j}  \notag\\
& \quad +   [ \psi^2(\xi)]_{y^j}[(1+ P_{y^k} x^k )^2]_{y^i},\label{lemK=0sufffundamental}
\end{align}
where
\begin{align}
[ \psi^2(\xi) ]_{y^i y^j}(x,y)
 & =  [ \psi^2(\xi) ]_{\xi^i \xi^j} + [ \psi^2(\xi) ]_{\xi^i \xi^l} x^l P_{y^j} + [ \psi^2(\xi) ]_{\xi^j \xi^l} x^l P_{y^i}
+ [ \psi^2(\xi) ]_{\xi^k \xi^l} x^k x^l P_{y^i} P_{y^j}\notag\\
& \quad + [ \psi^2(\xi) ]_{\xi^l} x^l P_{y^i y^j}. \label{lemK=0suffbarpsi0}
\end{align}

Now specialize to the case  $x = \mu y$. By Corollary~\ref{basicproper}, we have $\xi\neq0$ and $1+ \mu P(\mu y,y)>0$.
From \eqref{lemK=0suffxi}, we obtain
\begin{equation}\label{lemK=0suffbarxxi}
x^k = \mu y^k = \frac{\mu}{1+ \mu P} \xi^k.
\end{equation}
The homogeneity of $P$ together with \eqref{Eulerhter} implies
\begin{equation*}
P_{y^k}  x^k = \mu P , \quad
P_{y^k y^i}  x^k =0, \quad  P_{y^k y^i y^j} x^k = - \mu P_{y^i y^j}.
\end{equation*}
Substituting these relations along with \eqref{lemK=0suffbarxxi} into \eqref{lemK=0sufffundamental}--\eqref{lemK=0suffbarpsi0}, and using the homogeneity of $\psi$, we obtain after simplification:
\begin{align*}
[ F^2 ]_{y^i y^j}(\mu y,y)
& =   [ \psi^2(\xi) ]_{y^i y^j} ( 1+   \mu P )^2 - 2\mu \psi^2(\xi) (1+   \mu P ) P_{y^i y^j}, \\
[ \psi^2(\xi) ]_{y^i y^j} & = [ \psi^2(\xi) ]_{\xi^i \xi^j} + \frac{\mu}{1+ \mu P} \left([ \psi^2(\xi) ]_{\xi^i} P_{y^j} +  [ \psi^2(\xi) ]_{\xi^j} P_{y^i}\right)
+\frac{2 \mu^2  \psi^2(\xi)}{(1+ \mu P)^2}  P_{y^i} P_{y^j}+ \frac{2 \mu \psi^2(\xi) }{1+ \mu P} P_{y^i y^j},
\end{align*}
Combining these expressions yields
\begin{equation}\label{lemK=0sufffundamentalxbar1}
[ F^2 ]_{y^i y^j}(\mu y,y)
=    {(1+ \mu P)^2} [ \psi^2(\xi) ]_{\xi^i \xi^j} +  {\mu(1+ \mu P)}  ([ \psi^2(\xi) ]_{\xi^i} P_{y^j} +  [ \psi^2(\xi) ]_{\xi^j} P_{y^i}) + 2\mu^2 \psi^2(\xi)  P_{y^i} P_{y^j}  .
\end{equation}

Note that the Hessian $\left(  [\psi^2]_{y^i y^j}(y) \right)$ is  positive definite for any $y\neq \mathbf{0}$. For $\mu\in \mathbb{R}$ with $x:=\mu y\in \Omega$, let $\xi$ be as in \eqref{lemK=0suffxi}, which is a positive scalar multiple of $y$ by \eqref{lemK=0suffbarxxi}.
The positive  definiteness of
$\left(  [\psi^2]_{y^i y^j}(y) \right)=\left( [ \psi^2(\xi) ]_{\xi^i \xi^j}\right)$, together with \eqref{lemK=0suffbarxxi} and the homogeneity, implies the following inequality via the Cauchy--Schwarz argument:
\begin{equation}\label{lemK=0suffSchwartz}
\left[ \psi^2(\xi) \right]_{\xi^i \xi^j} \theta^i \theta^j \geq \frac{ \left( [ \psi^2(\xi) ]_{\xi^i \xi^j} \theta^i \xi^j \right)^2}{ [ \psi^2(\xi) ]_{\xi^i \xi^j} \xi^i \xi^j } 
= \frac{\left( [ \psi^2(\xi) ]_{\xi^i} \theta^i  \right)^2}{\psi^2(\xi)} ,
\quad \forall  \theta \in \Rno,
\end{equation}
with equality if and only if $\theta$ is parallel to $\xi$.

If $\theta = \kappa y$ for some $\kappa \neq 0$, then by homogeneity we have
\[
[ F^2 ]_{y^i y^j}(\mu y,y) \theta^i \theta^j=\kappa^2 F^2(\mu y,y) > 0.
\]

If $\theta$ is not parallel to $\xi$, then the inequality in \eqref{lemK=0suffSchwartz} is strict. Combining this with \eqref{lemK=0sufffundamentalxbar1}, we obtain
\begin{align*} 
[ F^2 ]_{y^i y^j} (\mu y,y)\theta^i \theta^j>
\left(   \frac{ (1+ \mu P)}{\psi(\xi)}  [ \psi^2(\xi) ]_{\xi^i} \theta^i   + \mu \psi(\xi)  P_{y^i}\theta^i \right)^2  \geq 0.
\end{align*}
Therefore, the Hessian $[F^2]_{y^i y^j}(\mu y, y)$ is positive definite for all such $\mu$ and $y$, which completes the proof of
Statement~\eqref{keylemmak=0a}.

\medskip

{\bf \eqref{keylemmak=0b}.}
 Given $(x,y)\in \Omega\times \mathbb{R}^n$,
 it follows from \eqref{equationP}, \eqref{twoequations}$_1$, and Proposition~\ref{transformforK=-1} that
\begin{align}
&P(\bar{x},y)=\bar{\phi}(y),\qquad \bar{P}(x-\bar{x},y)=P(x,y).\label{barphiP}
\end{align}

We first establish the identity
\begin{equation}
P_{y^i}\big(\bar{x}, y + (x-\bar{x}) P(x,y)\big) = \frac{P_{y^i}(x,y) }{ 1+ (x^k -\bar{x}^k) \bar{P}_{y^k}(x-\bar{x},y)}.   \label{barphiP2}
\end{equation}
To prove this, observe that \eqref{twoequations}$_2$ gives
\begin{equation} \label{Funkintialindependent1}
\bar{{P}}(x-\bar{x},y) = \bar{\phi}(y + (x-\bar{x}) \bar{{P}}(x-\bar{x},y)).
\end{equation}
Set $\xi := y + (x-\bar{x}) \bar{{P}}(x-\bar{x},y)$.
Differentiating \eqref{Funkintialindependent1} with respect to $y^i$ yields
\[
\bar{{P}}_{y^i}(x-\bar{x},y) = \bar{\phi}_{y^i}(\xi) + (x^l-\bar{x}^l)\bar{\phi}_{y^l}(\xi) \bar{{P}}_{y^i}(x-\bar{x},y).
\]
Using \eqref{barphiP}$_2$, this becomes
\begin{equation}\label{prephi}
\bar{\phi}_{y^i}(\xi) = \bar{{P}}_{y^i}(x-\bar{x},y) [1- (x^l-\bar{x}^l)\bar{\phi}_{y^l}(\xi)]= {P}_{y^i}(x,y) [1- (x^l-\bar{x}^l)\bar{\phi}_{y^l}(\xi)].
 \end{equation}
Solving $(x^l - \bar{x}^l) \bar{\phi}_{y^l}(\xi)$ from \eqref{prephi}, we obtain
\[
(x^l-\bar{x}^l)\bar{\phi}_{y^l}(\xi) =\frac{(x^i-\bar{x}^i) {P}_{y^i}(x, y)}{ 1+ (x^k-\bar{x}^k) {P}_{y^k}(x, y)}.
 \]
Substituting this back into \eqref{prephi} gives
\begin{equation*} 
\bar{\phi}_{y^i}(\xi) = \frac{{P}_{y^i}(x, y)}{ 1+ (x^k-\bar{x}^k) {P}_{y^k}(x, y)}.
\end{equation*}
In view of \eqref{barphiP}$_1$, this is exactly \eqref{barphiP2}.

Next, we examine the transformation of $\bar{\psi}$. By \eqref{defiFK=0}, we have
\begin{equation*}
\bar{\psi}(y) = F(\bar{x},y) = \psi(y + \bar{x} P(\bar{x},y)) \Big\{ 1 + \bar{x}^k P_{y^k}(\bar{x},y) \Big\}.
\end{equation*}
Replace $y$ by $\xi = y + (x - \bar{x}) \bar{P}(x - \bar{x}, y)$ in this equality.  Using \eqref{barphiP}$_1$  and \eqref{twoequations}$_2$, we compute:
\begin{align*}
 & \quad \ \bar{\psi}(y + (x-\bar{x}) \bar{P}(x-\bar{x},y)) \notag \\
& = \psi(y + (x-\bar{x}) \bar{P}(x-\bar{x},y) + \bar{x} P(\bar{x},y + (x-\bar{x}) \bar{P}(x-\bar{x},y))) \Big\{ 1 + \bar{x}^kP_{y^k}(\bar{x},\xi) \Big\}
\notag\\
& = \psi(y + (x-\bar{x}) \bar{P}(x-\bar{x},y) + \bar{x} \bar{\phi}(y + (x-\bar{x}) \bar{P}(x-\bar{x},y))) \Big\{ 1 +  \bar{x}^kP_{y^k}(\bar{x},\xi)  \Big\}
\notag\\
& = \psi(y + (x-\bar{x}) \bar{P}(x-\bar{x},y) + \bar{x}  \bar{P}(x-\bar{x},y) ) \Big\{ 1 + \bar{x}^k P_{y^k}(\bar{x},\xi)  \Big\}
\notag\\
& = \psi(y + x \bar{P}(x-\bar{x},y) ) \Big\{ 1 + \bar{x}^k P_{y^k}(\bar{x},\xi) \Big\}.
\end{align*}
Using \eqref{barphiP}$_2$ and \eqref{barphiP2}, this becomes
\[
\bar{\psi}(y + (x-\bar{x}) \bar{P}(x-\bar{x},y)) = \psi(y + x P(x,y))  \frac{1+ x^k P_{y^k}(x, y)  }{ 1+ (x^k-\bar{x}^k) \bar{P}_{y^k}(x-\bar{x}, y)}.
\]
The relation \eqref{newFK=0ex} now follows immediately from \eqref{defiFK=0}.
\end{proof}

To prove Proposition \ref{K=0Ex1}, we need the following lemma.
\begin{lemma}\label{f_ijvivj}
Let $f:\mathbb{R}^2\rightarrow \mathbb{R}$ be a positively $1$-homogeneous function such that $f|_{\mathbb{R}^2_\circ} \in C^{\infty}(\mathbb{R}^2_\circ)$, and let $y,v\in \mathbb{R}^2_\circ$ be two non-parallel vectors. Then the following hold:
  \begin{enumerate}[\rm (a)]
  \item $f_{y^i y^j}(y) v^i v^j =0$ if and only if $ f_{y^i y^j}(y) \theta^i \theta^j = 0$  for all $\theta\in \mathbb{R}^2_\circ$;
  \item $f_{y^i y^j}(y) v^i v^j >0$ if and only if $ f_{y^i y^j}(y) \theta^i \theta^j > 0$  for some $\theta\in \mathbb{R}^2_\circ$.
  \end{enumerate}
\end{lemma}
\begin{proof} Note  $f_{y^i y^j}(y) y^i =0 $ due to \eqref{Eulerhter}. Thus, for  $\theta = c_1 y + c_2 v$, we have
\begin{align*}
f_{y^i y^j}(y) \theta^i \theta^j  & = f_{y^i y^j}(y) (c_1 y^i + c_2 v^i) (c_1 y^j + c_2 v^j)  = c_2^2 f_{y^i y^j}(y) v^i v^j,
\end{align*}
which establishes the lemma.
\end{proof}

\begin{proof}[Proof of Proposition \ref{K=0Ex1}]
We begin by computing the derivatives of $F$. A direct calculation using \eqref{defiFK=0} yields
\begin{align}
F_{y^i y^j} = P_{y^i y^j} (1 + P_{y^k} x^k) + P_{y^i} P_{y^j y^k} x^k + P_{y^j} P_{y^i y^k} x^k + P P_{y^i y^j y^k} x^k. \label{K=0RandersnormFi}
\end{align}
Following the expressions in Example~\ref{prlxeK=0nonR}, we decompose the projective factor as
\begin{align}
P(x,y) = \mathcal{S}_a  + \mathcal{T}_a, \label{deopP}
\end{align}
where
\[
\mathcal{S}_a  := \frac{\sqrt{\mathcal{A}_a(x,y)}}{\mathcal{B}_a(x)}, \qquad
\mathcal{T}_a  := \frac{\mathcal{C}_a(x,y) + \langle x, y \rangle}{\mathcal{B}_a(x)},
\]
with $\mathcal{A}_a(x,y)$, $\mathcal{B}_a(x)$, and $\mathcal{C}_a(x,y)$ as defined in Example~\ref{prlxeK=0nonR}.
Note that $[\mathcal{T}_a]_{y^i y^j} = 0$, while the derivatives of $\mathcal{S}_a$ are given by:
\begin{align}
[\mathcal{S}_a]_{y^i} &= \frac{[\mathcal{A}_a(x,y)]_{y^i}}{2 \mathcal{B}_a(x) \sqrt{\mathcal{A}_a(x,y)}}, \label{K=0RandersnormSi} \\
[\mathcal{S}_a]_{y^i y^j y^k} &= - \frac{\mathcal{B}_a(x)}{\sqrt{\mathcal{A}_a(x,y)}} \left( [\mathcal{S}_a]_{y^i y^j} [\mathcal{S}_a]_{y^k} + [\mathcal{S}_a]_{y^j y^k} [\mathcal{S}_a]_{y^i} + [\mathcal{S}_a]_{y^i y^k} [\mathcal{S}_a]_{y^j} \right). \label{K=0RandersnormSijk}
\end{align}
Contracting \eqref{K=0RandersnormFi} with $x^i$ and $x^j$, and using \eqref{deopP}--\eqref{K=0RandersnormSijk}, we obtain:
\begin{align}
F_{y^i y^j} x^i x^j
&= P_{y^i y^j} x^i x^j (1 + 3 P_{y^k} x^k) + P P_{y^i y^j y^k} x^i x^j x^k \notag \\
&= [\mathcal{S}_a]_{y^i y^j} x^i x^j \left( 1 + 3 [\mathcal{T}_a]_{y^k} x^k - \frac{3 \mathcal{T}_a \mathcal{B}_a}{\sqrt{\mathcal{A}_a}} [\mathcal{S}_a]_{y^k} x^k \right). \label{K=0RandersnormFijxixj}
\end{align}

Since $\mathcal{S}_a$ is induced by a Riemannian metric, we have $[\mathcal{S}_a]_{y^i y^j} x^i x^j > 0$ when $x$ and $y$ are not parallel. Thus, \eqref{K=0RandersnormFijxixj} implies
\begin{align}
F_{y^i y^j}(x,y) x^i x^j = 0 \quad \Longleftrightarrow \quad 1 + 3 [\mathcal{T}_a]_{y^k} x^k - \frac{3 \mathcal{T}_a \mathcal{B}_a}{\sqrt{\mathcal{A}_a}} [\mathcal{S}_a]_{y^k} x^k = 0. \label{K=0Randersnormequiva}
\end{align}
A direct computation uisng \eqref{K=0Randersnormequiva} yields, for non-parallel vectors $x$ and $y$:
\begin{align}
& F_{y^i y^j}(x,y) x^i x^j = 0 \quad \Longleftrightarrow \notag \\
& (1 + \langle a,x \rangle - 2\langle a,x \rangle^2 + 2|x|^2)|y|^2 - 2|x|^2 \langle a,y \rangle^2 + (4\langle a,x \rangle - 1)\langle a,y \rangle \langle x,y \rangle - 2\langle x,y \rangle^2 = 0. \label{K=0RandersnormEq1}
\end{align}

We now study the structure of $\mathscr{O}$. By Theorem~\ref{n=2k=0deger}/\eqref{K=0Dim2conneccompconvex1}, every connected component of $\PD$ is convex.

Let $\breve{x} \in \partial{\PD} \cap \Omega$ be a smooth boundary point. We claim that $\breve{x}$ cannot be tangent to $\partial\mathscr{O}$ at $\breve{x}$; otherwise, Proposition~\ref{K=0conneccomconvex}/\eqref{conneccompconvex2} would imply that $[F^2]_{y^i y^j}(\breve{x}, \mu \breve{x})$ is not positive definite for some $\mu \neq 0$, contradicting Lemma~\ref{keylemmak=0}/\eqref{keylemmak=0a}.

By this claim and Proposition~\ref{K=0conneccomconvex}/\eqref{conneccompconvex2}, there exist non-parallel vectors $\breve{y} \in T_{\breve{x}}\partial{\PD} \backslash \{\mathbf{0}\}$ and $\theta = (\theta^i)$ such that
\[
F_{y^i y^j}(\breve{x}, \breve{y}) \theta^i \theta^j = 0, \qquad F_{y^i y^j}(\breve{x}, \breve{y} + t \breve{x}) \theta^i \theta^j > 0 \quad \text{for small } |t| > 0.
\]
Since $n = 2$, $\theta$ is a linear combination of $\breve{x}$ and $\breve{y}$. Lemma~\ref{f_ijvivj} then implies
\begin{equation}
F_{y^i y^j}(\breve{x}, \breve{y}) \breve{x}^i \breve{x}^j = 0, \qquad F_{y^i y^j}(\breve{x}, \breve{y} + t \breve{x}) \breve{x}^i \breve{x}^j > 0 \quad \text{for small } |t| > 0, \label{FBRWXDER}
\end{equation}
and consequently,
\begin{equation}
\left.\frac{{\dd}}{{\dd}t}\right|_{t=0} F_{y^i y^j}(\breve{x}, \breve{y} + t \breve{x}) \breve{x}^i \breve{x}^j = 0. \label{defi11}
\end{equation}

Using \eqref{K=0RandersnormEq1}, we can rewrite \eqref{FBRWXDER}$_1$ and \eqref{defi11} as:
\begin{align}
\left\{1 + 2(1 - a_1^2) \breve{x}_2^2\right\} \breve{y}_1^2 - \breve{x}_2 \left\{4(1 - a_1^2)\breve{x}_1 + a_1\right\} \breve{y}_1 \breve{y}_2  + \left\{2(1 - a_1^2)\breve{x}_1^2 + a_1 \breve{x}_1 + 1\right\} \breve{y}_2^2 &= 0, \label{K=0RandersMainEq1_1} \\
(2\breve{x}_1 - a_1 \breve{x}_2^2) \breve{y}_1 + \breve{x}_2 (a_1 \breve{x}_1 + 2) \breve{y}_2 &= 0, \label{K=0RandersMainEq2_1}
\end{align}
where $\breve{x} = (\breve{x}_1, \breve{x}_2)$ and $\breve{y} = (\breve{y}_1, \breve{y}_2)$.

We claim that $\breve{y}_2 \neq 0$. In fact, if $\breve{y}_2 = 0$, then by \eqref{K=0RandersMainEq1_1},
\[ \left\{ 1 + 2 (1 - a_1^2) \breve{x}_2^2 \right\} \breve{y}_1^2 =0,  \]
which combined with $a_1\in [0,1)$
implies $\breve{y}_1 =0$ and hence, $\breve{y}=\mathbf{0}$, which is a contradiction. Thus, $\breve{y}_2 \neq 0$.
Therefore,
 \eqref{K=0RandersMainEq1_1} can be rewritten as
\begin{align}\label{K=0RandersMainEq1_1v2}
\left\{ 1 + 2 (1 - a_1^2) \breve{x}_2^2 \right\} \left( \frac{\breve{y}_1}{\breve{y}_2} \right)^2 - \breve{x}_2 &\left\{  4(1- a_1^2)\breve{x}_1 + a_1  \right\}\frac{\breve{y}_1}{\breve{y}_2}
+   2(1 - a_1^2)\breve{x}_1^2 + a_1 \breve{x}_1 + 1 =0.
\end{align}
The discriminant $\Delta_{\eqref{K=0RandersMainEq1_1v2}}$ of the above quadratic equation of $\frac{\breve{y}_1}{\breve{y}_2}$ is given by
\begin{align}
\Delta_{\eqref{K=0RandersMainEq1_1v2}} & =  \breve{x}^2_2   \left\{  4(1- a_1^2)\breve{x}_1 + a_1  \right\}^2-4\left\{ 1 + 2 (1 - a_1^2) \breve{x}_2^2 \right\} \left\{ 2(1 - a_1^2)\breve{x}_1^2 + a_1 \breve{x}_1 + 1\right\}
\notag\\
&= - ( a_1  \breve{x}_1 +2)^2 - (8- 9 a_1^2) (\breve{x}_1^2+ \breve{x}_2^2).
\label{K=0exampdiscriminant}
\end{align}

We now proceed by case analysis based on the value of $a_1$.

\noindent\textbf{Case 1: $a_1 \in \left[0, \frac{2\sqrt{2}}{3}\right)$.}
In this case,
 $\Delta_{\eqref{K=0RandersMainEq1_1v2}} <0$, so no solution $\breve{y}$ exists. By Theorem~\ref{n=2k=0deger}/\eqref{K=0Dim2conneccompconvex3}, we conclude that $\PD = \Omega$, establishing Statement \eqref{K=0RandersnormCase1}.

\noindent\textbf{Case 2: $a_1 = \frac{2\sqrt{2}}{3}$.}
In this case,
$
\Delta_{\eqref{K=0RandersMainEq1_1v2}} = -4\left(\frac{\sqrt{2}}{3} \breve{x}_1 + 1\right)^2 \leq 0,
$
with equality only when $\breve{x}_1 = -\frac{3\sqrt{2}}{2}$. Then \eqref{K=0RandersMainEq1_1v2} forces $\breve{y}_1 = 0$, and both \eqref{K=0RandersMainEq1_1} and \eqref{K=0RandersMainEq2_1} hold for any $\breve{x}_2 \in \mathbb{R}$ and $\breve{y}_2 \neq 0$. Thus, $\partial{\PD}$ is the vertical line $\Omega \cap \{x_1 = -3\sqrt{2}/2\}$, which divides the open elliptical disk $\Omega$ into two open parts.

Since $F$ is a Finsler metric near the origin $\mathbf{0}$, the right part
$\left\{(x_1, x_2) \in \Omega : x_1 > -{3\sqrt{2}}/{2}\right\} \subset \PD.$
We now show that the left part $\Omega_L:=\left\{(x_1, x_2) \in \Omega : x_1 < -{3\sqrt{2}}/{2}\right\} $ is also contained in $\PD$. It suffices to verify that $F(\bar{x}, \cdot)$ is a Minkowski norm at some point $\bar{x} = (\bar{x}_1, 0) \in \Omega_L$. By \eqref{postivitylemma} and homogeneity, we need only check that
$F_{y^i y^j}(\bar{x}, \bar{y}) \bar{x}^i \bar{x}^j > 0 $
for any $\bar{y} = (\bar{y}_1, \bar{y}_2)$ not parallel to $\bar{x}$. Following the derivations from \eqref{K=0RandersnormFijxixj} to \eqref{K=0RandersnormEq1},
 it is equivalent to prove
\begin{align}\label{K=0RandersnormCaseiiiIneq}
 (1 + \langle a,\bar{x}  \rangle -2 \langle a,\bar{x}  \rangle^2+2 |\bar{x}|^2) |\bar{y}|^2 -2 |\bar{x}|^2 \langle a, \bar{y}\rangle^2 +  (4\langle a,\bar{x}  \rangle-1) \langle a, \bar{y}\rangle \langle \bar{x}, \bar{y}\rangle - 2\langle \bar{x}, \bar{y}\rangle^2 > 0.
\end{align}
It is exactly
$\bar{y}_1^2 + \left(\tfrac{\sqrt{2}}{3} \bar{x}_1 + 1\right)^2 \bar{y}_2^2 > 0,$
which holds since $\bar{y} \neq \mathbf{0}$ and $\bar{x}_1 < -3\sqrt{2}/2$. Hence $\Omega_L \subset \PD$, proving Statement \eqref{K=0RandersnormCase2}.

\noindent\textbf{Case 3: $a_1 \in \left(\frac{2\sqrt{2}}{3}, 1\right)$.}
We first show $a_1 \breve{x}_1 + 2 \neq 0$. If instead $a_1 \breve{x}_1 + 2 = 0$, then $\breve{x}_1 = -2/a_1$, and \eqref{K=0RandersMainEq2_1} gives $\breve{y}_1 = 0$. Substituting it into \eqref{K=0RandersMainEq1_1} yields
\begin{equation}
2(1 - a_1^2) \breve{x}_1^2 + a_1 \breve{x}_1 + 1 = 0, \label{K=0RandersnormCasei}
\end{equation}
which together with $\breve{x}_1 = -2/a_1$ implies $a_1 = 2\sqrt{2}/3 \notin (2\sqrt{2}/3, 1)$, a contradiction.

Now consider $\breve{x}_2 = 0$. Since $\breve{x} \neq \mathbf{0}$, \eqref{K=0RandersMainEq2_1} gives  $\breve{y}_1 = 0$. Then \eqref{K=0RandersMainEq1_1} again reduces to \eqref{K=0RandersnormCasei}, yielding
\begin{equation}
\breve{x} = \left( \frac{-a_1 \pm \sqrt{9a_1^2 - 8}}{4(1 - a_1^2)},\ 0 \right). \label{K=0examp_twopoints}
\end{equation}

For $\breve{x}_2 \neq 0$, \eqref{K=0RandersMainEq2_1} combined with $y\neq \mathbf{0}$ implies
\[
\breve{y}_2 = -\frac{(2\breve{x}_1 - a_1 \breve{x}_2^2) \breve{y}_1}{\breve{x}_2 (a_1 \breve{x}_1 + 2)},\qquad \breve{y}_1 \neq 0.
\]
Substituting  into \eqref{K=0RandersMainEq1_1} yields
\[
\left[8(1 - a_1^2) \breve{x}_1^2 + 4a_1 \breve{x}_1 + (8 - 9a_1^2) \breve{x}_2^2 + 4\right] \frac{\breve{x}_1^2 + \breve{x}_2^2}{\breve{x}_2^2 (a_1 \breve{x}_1 + 2)^2} \breve{y}_1^2 = 0.
\]
Since $\breve{x} \neq \mathbf{0}$ and $\breve{y}_1 \neq 0$, we obtain
\begin{equation}
8(1 - a_1^2) \left( \breve{x}_1 + \frac{a_1}{4(1 - a_1^2)} \right)^2 + (8 - 9a_1^2) \breve{x}_2^2 = \frac{9a_1^2 - 8}{2(1 - a_1^2)}. \label{K=0Randersnormhyper}
\end{equation}

For $a_1 \in (2\sqrt{2}/3, 1)$, equation \eqref{K=0Randersnormhyper} defines a hyperbola, which contains the points \eqref{K=0examp_twopoints}. This hyperbola divides $\Omega = \{|x| + \langle a, x \rangle < 1\}$ into three open parts. The middle part is non-convex and hence cannot lie in $\PD$ by Theorem~\ref{n=2k=0deger}/\eqref{K=0Dim2conneccompconvex1}. The right part contains the origin $\mathbf{0} \in \mathscr{O}$ and thus belongs to $\PD$.

It remains to check the left part. Take $\bar{x} = (\bar{x}_1, 0)$ in the left part, i.e.,
\begin{equation}
\bar{x}_1 + \frac{a_1}{4(1 - a_1^2)} < -\frac{\sqrt{9a_1^2 - 8}}{4(1 - a_1^2)}. \label{K=0Randersnormtildex}
\end{equation}
As in Case 2, it suffices to verify   \eqref{K=0RandersnormCaseiiiIneq}   for any $\bar{y} = (\bar{y}_1, \bar{y}_2)$ not parallel to $\bar{x}$. The left-hand side of \eqref{K=0RandersnormCaseiiiIneq} is exactly
\[
\bar{y}_1^2 + \left[ 2(1 - a_1^2) \left( \bar{x}_1 + \frac{a_1}{4(1 - a_1^2)} \right)^2 + \frac{8 - 9a_1^2}{8(1 - a_1^2)} \right] \bar{y}_2^2,
\]
which is positive by \eqref{K=0Randersnormtildex}. This establishes Statement \eqref{K=0RandersnormCase34}.

In both Cases 2 and 3, the smooth points of $\partial{\PD} \cap \Omega$ form a continuous (actually smooth) curve. Hence every point in $\partial{\PD} \cap \Omega$ is smooth.
\end{proof}


\begin{proof}[Proof of Lemma \ref{lemK=-1suffeqposilem}] The ``$\Leftarrow$'' part is immediate since $\bar{F}(\mathbf{0},y)=\bar{\psi}(y)$.
To prove the ``$\Rightarrow$'' part, we first
calculate $[\bar{F}^2]_{y^iy^j}$.
For simplicity, define
\begin{align}
{\varphi}_{+} := \bar{\phi}+ \bar{\psi},& \qquad  \varphi_{-}  := \bar{\phi} -  \bar{\psi}, \label{phcipm}\\
\xi  = (\xi^k) := (y^k + x^k \bar{\Phi}_+(x,y) ), & \qquad \eta  = (\eta^k) := (y^k + x^k \bar{\Phi}_-(x,y) ).  \label{lemK=-1suffxieta}
\end{align}
A direct computation yields the Hessian of $\bar{F}^2$:
\begin{equation}\label{lemK=-1sufffundamental}
\begin{split}
[ \bar{F}^2 ]_{y^i y^j}(x,y) & = \frac{1}{2}\left [ \bar{\Phi}_{+}(x,y)-  \bar{\Phi}_{-}(x,y) \right]\left[ \bar{\Phi}_{+}(x,y) -  \bar{\Phi}_{-}(x,y) \right]_{y^i y^j} \\
& \quad + \frac{1}{2}\left[ \bar{\Phi}_{+}(x,y) -  \bar{\Phi}_{-}(x,y) \right]_{y^i}  \left[ \bar{\Phi}_{+}(x,y) -  \bar{\Phi}_{-}(x,y) \right]_{y^j}.
\end{split}
\end{equation}
We now compute the derivatives of $\bar{\Phi}_{\pm}(x,y)$. From \eqref{lemK=-1Phipm}, we obtain
\begin{align*}
[\bar{\Phi}_{+}(x,y)]_{y^i} = [\varphi_{+}(\xi)]_{\xi^i} + [\varphi_{+}(\xi)]_{\xi^l} {x}^l [\bar{\Phi}_{+}(x,y)]_{y^i},\qquad
[\bar{\Phi}_{-}(x,y)]_{y^i} = [\varphi_{-}(\eta)]_{\eta^i} + [\varphi_{-}(\eta)]_{\eta^l}  {x}^l [\bar{\Phi}_{-}(x,y)]_{y^i},
\end{align*}
which imply
\begin{align}
[\bar{\Phi}_{+}(x,y)]_{y^i} &= \frac{[\varphi_{+}(\xi)]_{\xi^i}}{1- [\varphi_{+}(\xi)]_{\xi^l}  {x}^l},
\qquad
[\bar{\Phi}_{-}(x,y)]_{y^i} = \frac{[\varphi_{-}(\eta)]_{\eta^i}}{1- [\varphi_{-}(\eta)]_{\eta^l}  {x}^l}.
\label{lemK=-1suffPhiyi}
\end{align}
Differentiating these expressions with respect to $y^j$ gives the second derivatives:
\begin{equation}
\begin{split}
[\bar{\Phi}_{+}(x,y)]_{y^i y^j} & =  \frac{1}{\left(1- [\varphi_{+}(\xi)]_{\xi^l} x^l \right)^2} \Big\{ \left( [\varphi_{+}(\xi)]_{\xi^i \xi^j} +[\varphi_{+}(\xi)]_{\xi^i \xi^l}x^l [\bar{\Phi}_{+}(x,y)]_{y^j} \right) \left( 1- [\varphi_{+}(\xi)]_{\xi^l} x^l \right)
\\
& \quad + [\varphi_{+}(\xi)]_{\xi^i}\left( [\varphi_{+}(\xi)]_{\xi^j \xi^l} x^l + [\varphi_{+}(\xi)]_{\xi^k \xi^l} x^k x^l [\bar{\Phi}_{+}(x,y)]_{y^j} \right) \Big\} ,\label{lemK=-1sufPhiyiyj_1}
\end{split}
\end{equation}
\begin{equation}
\begin{split}
[\bar{\Phi}_{-}(x,y)]_{y^i y^j} & = \frac{1}{\left(1- [\varphi_{-}(\eta)]_{\eta^l} x^l \right)^2} \Big\{ \Big( [\varphi_{-}(\eta)]_{\eta^i \eta^j} +[\varphi_{-}(\eta)]_{\eta^i \eta^l}x^l [\bar{\Phi}_{-}(x,y)]_{y^j} \Big) \Big(1- [\varphi_{-}(\eta)]_{\eta^l} x^l \Big)
 \\
& \quad + [\varphi_{-}(\eta)]_{\eta^i} \Big( [\varphi_{-}(\eta)]_{\eta^j \eta^l} x^l + [\varphi_{-}(\eta)]_{\eta^k \eta^l} x^k x^l [\bar{\Phi}_{-}(x,y)]_{y^j} \Big) \Big\}. \label{lemK=-1sufPhiyiyj_2}
\end{split}
\end{equation}

Now specialize to the case $x=\mu y$. Then \eqref{lemK=-1suffxieta} implies
\begin{equation}\label{lemK=-1suffbarxxi}
x = \mu y = \frac{\mu}{1+ \mu \bar{\Phi}_{+}(\mu y,y)} \xi^k  = \frac{\mu}{1+ \mu \bar{\Phi}_{-}(\mu y,y)} \eta^k.
\end{equation}
By Corollary~\ref{basicproper}, we have
\begin{equation}\label{lemK=-1suff1+muPhi>0}
1+ \mu \bar{\Phi}_{+}(\mu y,y) > 0, \qquad 1+ \mu \bar{\Phi}_{-}(\mu y,y) > 0,
\end{equation}
which together with \eqref{lemK=-1Phipm} yields
\begin{align}
\varphi_{+}(\xi)|_{x=\mu y} &  = \left(1 + \mu  \bar{\Phi}_{+}(\mu y,y) \right) \varphi_{+}( y )= \bar{\Phi}_{+} (\mu y, y),
\label{lemK=-1suffPhi+muy1} \\
\varphi_{-}(\eta)|_{x=\mu y} & = \left(1 + \mu  \bar{\Phi}_{-}(\mu y,y) \right) \varphi_{-}(y)=\bar{\Phi}_{-} (\mu y, y)  .
\label{lemK=-1suffPhi-muy1}
\end{align}
Thus, $\varphi_{+}(\xi)|_{x=\mu y}$ and $\varphi_{+}( y )$ (resp., $\varphi_{-}(\eta)|_{x=\mu y}$ and $\varphi_{-}( y )$) share the same sign.
Solving \eqref{lemK=-1suffPhi+muy1} and \eqref{lemK=-1suffPhi-muy1} yields
\begin{equation}\label{lemK=-1suffvarphi+-}
\varphi_{+}(\xi)|_{x=\mu y}= \bar{\Phi}_{+} (\mu y, y)=\frac{\varphi_{+}(y)}{1 - \mu \varphi_{+}(y)}, \qquad
\varphi_{-}(\eta)|_{x=\mu y} =\bar{\Phi}_{-} (\mu y, y)=  \frac{ \varphi_{-}(y)}{1 - \mu \varphi_{-}(y)}.
\end{equation}
Substituting these expressions into \eqref{lemK=-1suff1+muPhi>0} implies
\begin{align}\label{lemK=-1suff1-muvarphi>0}
1 - \mu \varphi_{+}(y) > 0, \qquad 1 - \mu \varphi_{-}(y)>0.
\end{align}

Using the homogeneity of $\varphi_{\pm}$ and Euler's relation \eqref{Eulerhter}, along with \eqref{lemK=-1suffbarxxi} and \eqref{lemK=-1suffvarphi+-}, we compute:
\begin{align*}
[\varphi_{+}(\xi)]_{\xi^k} x^k |_{x=\mu y}= \left.\frac{\mu \varphi_{+}(\xi)}{1+ \mu \varphi_{+}(\xi)}\right|_{x=\mu y} = \mu \varphi_{+}(y) , &\qquad
[\varphi_{-}(\eta)]_{\eta^k} x^k |_{x=\mu y}= \left.\frac{\mu \varphi_{-}(\eta)}{1+ \mu \varphi_{-}(\eta) }\right|_{x=\mu y} =\mu \varphi_{-}(y), \notag\\
[\varphi_{+}(\xi)]_{\xi^k \xi^l} x^k |_{x=\mu y}= 0, &\qquad [\varphi_{-}(\eta)]_{\eta^k \eta^l} x^k |_{x=\mu y}= 0.
\end{align*}
Since $[\varphi_{+}(\xi)]_{\xi^i}$ and $[\varphi_{-}(\eta)]_{\eta^i}$ are positively $0$-homogeneous, while
$[\varphi_{+}(\xi)]_{\xi^i \xi^j}$ and $[\varphi_{-}(\eta)]_{\eta^i \eta^j}$ are positively $(-1)$-homogeneous,
we apply the homogeneity relations along with \eqref{lemK=-1suffbarxxi}, \eqref{lemK=-1suff1+muPhi>0}, and \eqref{lemK=-1suffvarphi+-} to obtain:
\begin{align*}
[\varphi_{+}(\xi)]_{\xi^i}|_{x=\mu y}=[\varphi_{+}(y)]_{y^i},&\qquad [\varphi_{-}(\eta)]_{\eta^i}|_{x=\mu y}=[\varphi_{-}(y)]_{y^i},\\
[\varphi_{+}(\xi)]_{\xi^i \xi^j}|_{x =\mu y}   = \left. \frac{1}{1 + \mu \varphi_{+}(\xi)} \right|_{x =\mu y} [\varphi_{+}(y)]_{y^i y^j},
&\qquad
[\varphi_{-}(\eta)]_{\eta^i \eta^j}|_{x =\mu y}   = \left. \frac{1}{1 + \mu \varphi_{-}(\eta)} \right|_{x =\mu y} [\varphi_{-}(y)]_{y^i y^j}.
\end{align*}

Substituting the derived relations into \eqref{lemK=-1suffPhiyi}--\eqref{lemK=-1sufPhiyiyj_2} and using \eqref{lemK=-1suffvarphi+-}, we obtain after simplification:
\begin{align}
[\bar{\Phi}_{+}]_{y^i}(\mu y,y) =
\frac{ [\varphi_{+}(y)]_{y^i}}{1- \mu \varphi_{+}(y)}, &\qquad
[\bar{\Phi}_{-}]_{y^i}(\mu y,y) = \frac{ [\varphi_{-}(y)]_{y^i}}{1- \mu \varphi_{-}(y)},
\label{lemK=-1suffPhi+-1_2}\\
[\bar{\Phi}_{+}]_{y^i y^j} (\mu y,y) = [\varphi_{+}(y)]_{y^i y^j},
&\qquad
[\bar{\Phi}_{-}]_{y^i y^j} (\mu y,y) = [\varphi_{-}(y)]_{y^i y^j}.
\label{lemK=-1suffPhi+-22}
\end{align}
Combining these expressions with \eqref{lemK=-1suffPhi+muy1}--\eqref{lemK=-1suffvarphi+-}, \eqref{lemK=-1suffPhi+-1_2}--\eqref{lemK=-1suffPhi+-22}, and \eqref{phcipm}, we have
\[
[\bar{\Phi}_{+}-\bar{\Phi}_{-} ](\mu y,y)
= \frac{2\bar{\psi}(y)}{(1 - \mu \varphi_{+}(y))(1 - \mu \varphi_{-}(y))},
\]
\[ [\bar{\Phi}_{+}-\bar{\Phi}_{-} ]_{y^i}(\mu y,y)
=\frac{2 (1-\mu \bar{\phi}(y)) \bar{\psi}_{y^i}(y) + 2\mu\bar{\psi}(y)\bar{\phi}_{y^i}(y)  }{(1-\mu \varphi_{+}(y))(1-\mu \varphi_{-}(y))},
\qquad
[\bar{\Phi}_{+}-\bar{\Phi}_{-} ]_{y^i y^j} (\mu y,y)
 = 2 \bar{\psi}_{y^i y^j}(y). \]
Substituting these into \eqref{lemK=-1sufffundamental} at $x = \mu y$ yields the fundamental identity:
\begin{align}\label{K=-1suffFund1}
[ \bar{F}^2 ]_{y^i y^j}(\mu y,y) \theta^i \theta^j = \frac{ 2\bar{\psi}(y) \bar{\psi}_{y^i y^j}(y) \theta^i \theta^j}{(1 - \mu \varphi_{+}(y))(1 - \mu \varphi_{-}(y))} + 2 \frac{ \left[ (1-\mu \bar{\phi}(y)) \bar{\psi}(y)_{y^i}\theta^i + \mu \bar{\psi}(y)\bar{\phi}_{y^i}(y) \theta^i\right]^2  }{(1-\mu \varphi_{+}(y))^2 \left(1-\mu \varphi_{-}(y)\right)^2},
\end{align}
for any $\theta = (\theta^i) \in \Rno$.

We now complete the proof of   the ``$\Rightarrow$" part using \eqref{K=-1suffFund1}. Suppose that $([\bar{\psi}^2]_{y^iy^j}(y))$ is positive definite for some $y\in \Rno$.
If $\theta$ is parallel  to $y$, i.e.,  $\theta=\lambda y$, then \eqref{K=-1suffFund1} and $\bar{\psi}(y)>0$ imply
\[
[ \bar{F}^2 ]_{y^i y^j}(\mu y,y) \theta^i \theta^j=   \frac{ 2\lambda^2 \bar{\psi}^2(y)  }{(1-\mu \varphi_{+}(y))^2 \left(1-\mu \varphi_{-}(y)\right)^2} > 0.
 \]
 Alternatively, provided $\theta$ is not parallel to $y$, the positive definiteness of
$\left( [ \bar{\psi}^2 ]_{y^i y^j}(y) \right)$   implies
$   \bar{\psi}_{y^i y^j} (y)  \theta^i \theta^j > 0$ (cf.\,\cite[(1.2.9)]{BCS}).
Thus $[ \bar{F}^2 ]_{y^i y^j} \theta^i \theta^j > 0$ follows by \eqref{K=-1suffFund1} and $\bar{\psi}(y) > 0$,
 which completes the proof.
\end{proof}

\vskip 8mm
\section{Additional results for $\mathbf{K}=1$}\label{propergenerlengappexK=-1}

In this section, we provide the proofs of Observation~\ref{informobvers2} and Lemmas~\ref{standardfromRntoSn} and \ref{differenobversion}.
\begin{proof}[Proof of Observation~\ref{informobvers2}]
Statement \eqref{ob21} follows directly from \eqref{expssionPlim}. We now prove  Statement~\eqref{ob22}.

Let $\sigma_\alpha:[0,2\pi]\rightarrow \mathbb{S}^{n}$, $\alpha\in \mathbb{N}$, be a sequence of great circles.
By the Arzela--Ascoli theorem (cf.~Burago et al. \cite[Theorem 2.5.14]{BBI}), there exists a subsequence $(\sigma_{\alpha_k})_k$  convergent uniformly to a closed curve $\sigma:[0,2\pi]\rightarrow \mathbb{S}^n$.
Each $\sigma_{\alpha_k}$ is the intersection of $\mathbb{S}^n$ with a plane $\pi_{\alpha_k}\subset\mathbb{R}^{n+1}$ through the origin. Thus, for each $k$ and all $t \in [0, 2\pi]$, we have
\begin{equation}\label{curveinplane}
\langle\mathbf{n}_{\alpha_k} ,\sigma_{\alpha_k}(t)\rangle=0,\qquad |\sigma_{\alpha_k}(t)|=1,
\end{equation}
where $\mathbf{n}_{\alpha_k}$ is the Euclidean unit normal vector to $\pi_{\alpha_k}$, and $\sigma_{\alpha_k}(t)$ is viewed as a vector in $\mathbb{R}^{n+1}$. Passing to a further subsequence if necessary, we may assume that $\mathbf{n}_{\alpha_k}$ converges to some unit vector $\mathbf{n}$.
Hence, \eqref{curveinplane} implies that the limit curve $\sigma(t)$ lies in some plane $\pi$ through the origin  with normal vector $\mathbf{n}$.
Thus, $\sigma$ has to be a great circle.
\end{proof}

In the remainder of this subsection,  let $(x^i)$ denote the standard Euclidean coordinates on $\mathbb{R}^n$ and $(\zeta^j)=(\varphi, \theta^1, \theta^2, ..., \theta^{n-1})$ the standard spherical coordinates  on $\mathbb{S}^n$, where  $\theta^{s}\in(0, \pi)$ for $s=1,\cdots,n-2$, $\theta^{n-1}\in[0, \pi)$, and $\varphi \in (-\frac{\pi}{2}, \frac{3 \pi}{2} ]$.
A direct (though computationally involved) calculation yields the following transformation formulas between the coframes $({\dd}x^i)$ and $({\dd}\zeta^j)$; see \eqref{K=1spherical_x} for details.
\begin{lemma}\label{standardfromRntoSn}Define $\mathfrak{p}:\mathbb{R}^n\rightarrow \mathbb{S}^n_+$ by \eqref{projeticpm}, and let
\begin{equation}\label{maxtricformsphere}
\left(\mathfrak{p}^{-1}\right)^*{\dd}x^k= \mathcal{J}_i^k {\dd}\zeta^i.
\end{equation}
Then the matrix $\mathcal {J}=[\mathcal {J}_i^k]_{n\times n}$ has the following structure:
\begin{itemize}
\item \textbf{Case $n = 2$:}
\[
\mathcal{J} =
\begin{bmatrix}
{\sec^2}\varphi \cos\theta^1 & -\tan\varphi \sin\theta^1 \\
{\sec^2}\varphi \sin\theta^1 & \tan\varphi \cos\theta^1
\end{bmatrix}.
\]

\item \textbf{Case $n \geq 3$:}
\begin{align*}
\mathcal{J}_i^1 &=
\begin{cases}
{\sec^2}\varphi \cos\theta^1, & i = 1; \\
{\sec^2}\varphi \cos\theta^i \prod_{\alpha=1}^{i-1} \sin\theta^\alpha, & i = 2, \cdots, n-1; \\
{\sec^2}\varphi \prod_{\alpha=1}^{n-1} \sin\theta^\alpha, & i = n;
\end{cases} \\
\mathcal{J}_i^2 &=
\begin{cases}
-\tan\varphi \sin\theta^1, & i = 1; \\
\tan\varphi \cos\theta^1 \cos\theta^2, & i = 2; \\
\tan\varphi \cos\theta^1 \cos\theta^i \prod_{\alpha=2}^{i-1} \sin\theta^\alpha, & i = 3, \cdots, n-1; \\
\tan\varphi \cos\theta^1 \prod_{\alpha=2}^{n-1} \sin\theta^\alpha, & i = n;
\end{cases} \\
\mathcal{J}_i^n &=
\begin{cases}
0, & i = 1, \cdots, n-2; \\
-\tan\varphi \prod_{\alpha=1}^{n-1} \sin\theta^\alpha, & i = n-1; \\
\tan\varphi \cos\theta^{n-1} \prod_{\alpha=1}^{n-2} \sin\theta^\alpha, & i = n;
\end{cases}
\end{align*}
and for $3 \leq k \leq n-1$,
\[
\mathcal{J}_i^k =
\begin{cases}
0, & i = 1, \cdots, k-2; \\
-\tan\varphi \prod_{\alpha=1}^{i} \sin\theta^\alpha, & i = k-1; \\
\tan\varphi \cos\theta^i \cos\theta^{k-1} \prod_{\substack{\alpha=1, \alpha \neq k-1}}^{i-1} \sin\theta^\alpha, & i = k, \cdots, n-1; \\
\tan\varphi \cos\theta^{k-1} \prod_{\substack{\alpha=1, \alpha \neq k-1}}^{n-1} \sin\theta^\alpha, & i = n.
\end{cases}
\]
\end{itemize}

\end{lemma}

\begin{lemma}\label{differenobversion}
Let $w\in \partial \mathbb{S}^n_+\subset \mathbb{R}^{n+1}$ and $V=V^i\frac{\partial}{\partial \zeta^i}|_{w}\in T_{w}\mathbb{S}^n$. Regarding $V$ as a vector in $\mathbb{R}^{n+1}$, parallel transport it to the antipodal point $-w\in \partial\mathbb{S}^n_+\subset\mathbb{R}^{n+1}$, and denote the resulting vector by  $V|_{-w}=:\widetilde{V}^i\frac{\partial}{\partial \zeta^i}|_{-w}$. Then
\begin{equation}\label{negivereverspheretange}
V^i=-\widetilde{V}^i.
\end{equation}
\end{lemma}

\begin{proof}
Let $\left(x^1, \cdot\cdot\cdot, x^{n},x^{n+1} \right)$ be the Euclidian coordinates of $\mathbb{R}^{n+1}$. For a point $w\in \mathbb{S}^n\subset \mathbb{R}^{n+1}$, its coordinates $(x^i)$ and $(\zeta^j)$ are related by
\begin{align*}
\begin{cases}
x^{n+1} & =\cos \varphi,\\
x^1 & = \sin\varphi \cos\theta^{1}, \\
 &\vdots\\
x^{i} &  = \sin\varphi \cos\theta^{i} \prod_{\alpha=1}^{i-1}\sin\theta^{\alpha} , \quad 2\leq i\leq n-1,\\
& \vdots\\
x^{n} & =\sin\varphi \sin\theta^{n-1} \prod_{\alpha=1}^{n-2}\sin\theta^{\alpha},
\end{cases}
\end{align*}
which
implies the coordinate transformation
 \begin{align}\label{K=1transportmatRnpartialX}
 \left( \frac{\partial}{\partial \varphi},  \frac{\partial}{\partial \theta^1},\cdots,\frac{\partial}{\partial \theta^{n-1}}   \right)= \left( \frac{\partial}{\partial x^{n+1}}, \frac{\partial}{\partial x^1},\cdots,\frac{\partial}{\partial x^n}   \right)\widetilde{\mathcal{J}},
 \end{align}
where $\widetilde{\mathcal{J}}$ is an $ (n+1) \times n$ matrix  given by
\begin{align*}
\widetilde{\mathcal{J}} =
\begin{bmatrix}
& -\sin{\varphi} & 0& \cdot\cdot\cdot & 0 \\
& {\cos^3}{\varphi}\ \mathcal{J}^1&  \cos{\varphi}\ \mathcal{J}^2  & \cdot\cdot\cdot & \cos{\varphi}\ \mathcal{J}^n
\end{bmatrix}
,
\end{align*}
and $\mathcal{J}^k=[\mathcal{J}_{i}^k]_{n\times 1}$ defined by Lemma~\ref{standardfromRntoSn}.

The $(n-1)$-tuple $(\theta^1,\ldots,\theta^{n-1})$ in  the spherical coordinate system   $(\varphi,\theta^1,\ldots,\theta^{n-1})$  corresponds to the standard spherical coordinates on $\mathbb{S}^{n-1}$. Thus,
for $w=(\pi/2,\theta^1,\ldots,\theta^{n-1})$, we have $-w=(-\pi/2,\theta^1,\ldots,\theta^{n-1})$. A direct calculation shows  $\widetilde{\mathcal{J}}|_{w } = - \widetilde{\mathcal{J}}|_{-w}$,
which together with \eqref{K=1transportmatRnpartialX} implies  \eqref{negivereverspheretange}.
\end{proof}

\vskip 10mm
\section{Auxiliary properties of Sobolev spaces
}\label{propesoboleve}
In this section, we prove Lemma~\ref{SobolevLemma1} using approximations by Lipschitz functions.
\begin{definition}
A function $f:M\rightarrow \mathbb{R}$ is called {\it Lipschitz} if there exists a constant $C\geq 0$ such that
\begin{equation}\label{lispchin}
|f(x_1)-f(x_2)|\leq  C\,d_F(x_1,x_2),\quad \forall  x_1,x_2\in M.
\end{equation}
In this case, $f$ is also said to be {\it $C$-Lipschitz}.
The minimal constant $C$ satisfying  \eqref{lispchin}  is called the {\it dilatation} of $f$, denoted by $\dil(f)$. 
\end{definition}

For convenience,  we denote by $\Lip(M)$  the {\it collection of Lipschitz functions on $M$}, and let $\Lip_0(M):=C_0(M)\cap \Lip(M)$ be the {\it collection of Lipschitz functions  with compact support}.
The following approximation (see \cite[Proposition 3.3]{KLZ}) is useful in the sequel.
\begin{lemma}[\cite{KLZ}]\label{lipsconverppax}
Let $(M,F,\m)$ be an $\FMMM$. For every $u\in \Lip_0(M)$,   there exists a sequence of smooth Lipschitz functions $u_n\in C^\infty_0(M)\cap \Lip_0(M)$  such that
\[
\supp u\cup \supp u_n\subset \mathscr{K},\quad |u_n|\leq C,\quad F^*({\dd} u_n)\leq   \dil(u_n)\leq C,\quad u_n\rightrightarrows u,\quad \|u-u_n \|_{W^{1,p}_{\m}}\rightarrow 0,
\]
where $\mathscr{K}\subset M$ is a compact set  and $C>0$ is  a constant. In particular, if $u$ is nonnegative, so are $u_n$'s.
\end{lemma}

\begin{proof}[Proof of Lemma \ref{SobolevLemma1}] The proof processes in two steps.

\smallskip

{\bf Step 1.} In this step, we show $u_\alpha=-e^{-\alpha r(x)}\in W^{1,p}_0(M,F,\m)$.

Note that  ${F}^{*}({\dd}u_\alpha)=\alpha e^{-\alpha r}=\alpha|u_\alpha|$, since $F^*({\dd}r)=1$. By $|u_\alpha|\leq 1$ and $\m(M)<+\infty$, we have
\begin{align*}
\int_M |u_\alpha|^p \dm \leq \m(M)<+\infty,\qquad  \int_M {F}^{*p}({\dd}u_\alpha) \dm=\alpha^p \int_M|u_\alpha|^p\dm  < + \infty,
\end{align*}
which implies $\|u_\alpha\|_{W^{1,p}_{\m}}<+\infty$.

For each small $\varepsilon>0$, define
\[
u_{\alpha,\varepsilon}(x):=
\begin{cases}
-e^{-\alpha\varepsilon}, &\text{ if }r(x)<\varepsilon;\\[3pt]

-e^{-\alpha r(x)}, &\text{ if }r(x)\geq \varepsilon.
\end{cases}
\]
Since $ |u_\alpha-u_{\alpha,\varepsilon}|\leq |u_\alpha|\in L^p(M,\m)$ and $F^{*}({\dd}u_\alpha)\in L^p(M,\m)$, the dominated convergence theorem yields
\begin{align*}
&\lim_{\varepsilon\rightarrow 0^+}\int_M |u_\alpha-u_{\alpha,\varepsilon}|^p \dm    = \lim_{\varepsilon\rightarrow 0^+}\int_M \chi_{\{r<\varepsilon\}}(e^{-\alpha r}-e^{-\alpha\varepsilon})^p \dm=0,\\
&\lim_{\varepsilon\rightarrow 0^+}\int_M F^{*p}({\dd} u_\alpha-{\dd} u_{\alpha,\varepsilon}) \dm  = \lim_{\varepsilon\rightarrow 0^+}\int_{M}\chi_{\{r<\varepsilon\}} F^{*p}({\dd}u_\alpha) \dm =0,
\end{align*}
where $\chi_{\{r<\varepsilon\}}$ denotes the characteristic function of the set $\{r<\varepsilon\}$.
Hence, 
\begin{equation}\label{firstconverge}
\lim_{\varepsilon\rightarrow 0^+}\|u_\alpha-u_{\alpha,\varepsilon}\|_{W^{1,p}_{\m}}= 0,
\end{equation}
 which implies
\[
\|u_{\alpha,\varepsilon}\|_{W^{1,p}_{\m}}=\|u_{\alpha,\varepsilon}\|_{L^p_{\m}}+\|{\dd} u_{\alpha,\varepsilon}\|_{L^p_{\m}}\leq \Big(\|u_{\alpha,\varepsilon}-u_\alpha\|_{L^p_{\m}}+\|u_\alpha\|_{L^p_{\m}}\Big)+\|{\dd} u_\alpha\|_{L^p_{\m}}<+\infty.
\]

Now define $u_{\alpha, \varepsilon,\delta}:=\min\{0, u_{\alpha, \varepsilon} +\delta\}$.
We claim $u_{\alpha, \varepsilon,\delta}\in \Lip_0(M)$. In fact,
\[
\supp u_{\alpha, \varepsilon,\delta}=\overline{\{u_{\alpha, \varepsilon,\delta}<0\}}=\overline{B^+_{ -a^{-1}{\ln\delta}}\left(o\right)},
\]
where $B^+_{ -a^{-1}{\ln\delta}}\left(o\right)=\{ r< -\alpha^{-1}\ln \delta\}$ is a forward ball defined by \eqref{forward/backwradball}$_1$, and its closure is compact due to forward completeness.
Then the construction implies
\[
 F^{*}(\pm {\dd} u_{\alpha, \varepsilon,\delta})\leq F^{*}(\pm {\dd} u_{\alpha, \varepsilon})\leq F^{*}(\pm {\dd} u_{\alpha})\leq \max_{x\in \overline{B^+_{ -a^{-1}{\ln\delta}}\left(o\right)}}\left\{ F^{*}(\pm {\dd} u_{\alpha})\right\} =:C< +\infty.
\]
Given $x_1,x_2 \in  \overline{B^+_{ -a^{-1}{\ln\delta}}\left(o\right)}$, let $\gamma(t)$, $t\in [0,d_F(x_1,x_2)]$ be a minimal unit-speed geodesic from $x_1$ to $x_2$.
Then by \eqref{dualff*}, we have
\begin{align*}
&|u_{\alpha, \varepsilon,\delta}(x_1)-u_{\alpha, \varepsilon,\delta}(x_2)|=\left| \int^{d_F(x_1, x_2)}_0 \frac{{\dd}}{{\dd}t}u_{\alpha, \varepsilon,\delta}(\gamma(t))  {\dd}t\right|=\left| \int^{d_F(x_1, x_2)}_0 \langle \dot{\gamma}(t),{\dd}u_{\alpha, \varepsilon,\delta}\rangle  {\dd}t\right|\\
\leq &  \int^{d_F(x_1, x_2)}_0 \left|\langle \dot{\gamma}(t),{\dd}u_{\alpha, \varepsilon,\delta} \rangle\right|  {\dd}t\leq    \int^{d_F(x_1, x_2)}_0 F(\dot{\gamma}(t)) \max \{F^*(\pm {\dd} u_{\alpha, \varepsilon,\delta}) \}{\dd}t=  C \, d_F(x_1, x_2),
\end{align*}
which establishes the claim.

By Lemma~\ref{lipsconverppax} and \eqref{firstconverge}, it remains to prove
\begin{equation}\label{convertgelipp}
\lim_{\delta\rightarrow 0^+}\|u_{\alpha,\varepsilon}-u_{\alpha,\varepsilon,\delta}\|_{W^{1,p}_{\m}}=0.
\end{equation}
In fact, if \eqref{convertgelipp} holds, then a diagonal argument shows that  $u_\alpha$ can be approximated by smooth functions with compact supports.

To prove \eqref{convertgelipp}, by $\m(M) < +\infty$ we have
\begin{align*}
&\lim_{\delta\rightarrow 0^+}\int_{\{r < -\alpha^{-1}\ln \delta\}}\delta^p \dm \leq \lim_{\delta\rightarrow 0^+}\delta^p\m(M)=0.
\end{align*}
Moreover, since $u_{\alpha,\varepsilon}\in L^p(M,\m)$ and $\chi_{\{r \geq  -\alpha^{-1}\ln \delta\}}\rightarrow 0$, the dominated convergence theorem yields
\[
\lim_{\delta\rightarrow 0^+}\int_{\{r \geq  -\alpha^{-1}\ln \delta\}}|u_{\alpha,\varepsilon}|^p \dm=\lim_{\delta\rightarrow 0^+}\int_M \chi_{\{r \geq  -\alpha^{-1}\ln \delta\}}|u_{\alpha,\varepsilon}|^p \dm = 0.
\]
Combining these, we have
\begin{align}\label{deltaconvergezero1}
\lim_{\delta\rightarrow 0^+}\int_M |u_{\alpha,\varepsilon}-u_{\alpha,\varepsilon,\delta}|^p \dm=\lim_{\delta\rightarrow 0^+}\int_{\{r < -\alpha^{-1}\ln \delta\}}\delta^p \dm +\lim_{\delta\rightarrow 0^+}\int_{\{r \geq  -\alpha^{-1}\ln \delta\}}|u_{\alpha,\varepsilon}|^p \dm =0.
\end{align}
Similarly, since $ F^{*}({\dd} u_{\alpha,\varepsilon})\in L^p(M,\m)$, the dominated convergence theorem yields
\begin{align}\label{deltaconvergezero2}
\lim_{\delta\rightarrow 0^+}\int_M F^{*p}({\dd} u_{\alpha,\varepsilon}-{\dd} u_{\alpha,\varepsilon,\delta})  \dm =\lim_{\delta\rightarrow 0^+}\int_{\{ r\geq -\alpha^{-1}\ln \delta\}} F^{*p}({\dd} u_{\alpha,\varepsilon}) \dm = 0.
\end{align}
Then \eqref{convertgelipp} follows from \eqref{deltaconvergezero1} and \eqref{deltaconvergezero2}.

\smallskip

{\bf Step 2.} In this step, we show $u= - \left[\ln(2+r)\right]^{-\frac{1}{n}}\in W^{1,p}_0(M,F,\m)$.

\smallskip

The argument is analogous to Step 1, so we provide a sketch. First, observe that
\begin{equation*}
 |u|\leq \ln2^{-\frac{1}n}, \quad u_r = \frac{\partial u}{\partial r} = \frac{1}{n(2+r)} \big[\ln(2+r)\big]^{-1-\frac{1}{n}} \in \left(0,  \frac{1}{2n(\ln2)^{1+\frac{1}{n}}}\right).
 \end{equation*}
Since $\m(M)<+\infty$ and   $F^{*}({\dd} r) =1$,  we have
\begin{align*}
\int_{M} |u|^p {\dd}x
 \leq ( \ln 2 )^{-\frac{p}{n}}\, {\m}({M})<+\infty, \qquad \int_{{M}} {F}^{*p}({\dd}u) {\dd}x=\int_{{M}} u_r^p {\dd}x\leq \frac{{\m}({M})}{(2n)^p(\ln2)^{p+\frac{p}{n}}}<+\infty,\label{K=0ulpberwadl}
\end{align*}
so $\|u\|_{W^{1,p}_{{\m}}}<+\infty$.
For $\varepsilon>0$, define
\[
u_\varepsilon(x):=
\begin{cases}
- [\ln(2+\varepsilon)]^{-\frac{1}{n}}, &\text{ if }r(x)<\varepsilon;\\[3pt]
- [\ln(2+r(x))]^{-\frac{1}{n}}, &\text{ if }r(x)\geq \varepsilon.
\end{cases}
\]
Then
  $0\leq |u-u_\varepsilon|\leq |u|\in L^p({M},{\m})$ and ${F}^{*}({\dd}u_\varepsilon) \in L^p({M},{\m})$. By the dominated convergence theorem,
\begin{equation}\label{K=0Berwaldfirstconverge}
\lim_{\varepsilon \rightarrow 0^+}\|u-u_\varepsilon\|_{W^{1,p}_{{\m}}}=0.
\end{equation}
Hence $u$ can be approximated by $(u_\varepsilon)_{\varepsilon}$ in the sense of $W^{1,p}_0({M},F,{\m})$.

Now set $u_{\varepsilon,\delta}:=\min\{0, u_\varepsilon+\delta\}$ for $\delta\in (0,(\ln 2)^{-\frac1n})$. Then $u_{\varepsilon,\delta} \in \Lip_0({M})$ and
\begin{equation}\label{K=0Berwaldconvertgelipp}
\lim_{\delta\rightarrow 0^+}\|u_\varepsilon-u_{\varepsilon,\delta}\|_{W^{1,p}_{{\m}}}=0.
\end{equation}
By Lemma~\ref{lipsconverppax}, \eqref{K=0Berwaldfirstconverge}, and \eqref{K=0Berwaldconvertgelipp}, a diagonal argument shows that $u$ can be approximated by smooth functions with compact supports, completing the proof.
\end{proof}



\end{document}